\documentclass[12pt]{article}
\usepackage[utf8]{inputenc}
\usepackage{sp}
\usepackage{subfigure}

\usepackage[colorlinks=true, citecolor=orange, linkcolor=black]{hyperref}

\title{New Sequence-Independent Lifting Techniques for Cutting Planes and When They Induce Facets}
\author{Siddharth Prasad\thanks{Computer Science Department, Carnegie Mellon University. {\tt sprasad2@cs.cmu.edu}}\and Ellen Vitercik\thanks{Management Science \& Engineering and Computer Science Departments, Stanford University. {\tt vitercik@stanford.edu}}\and Maria-Florina Balcan\thanks{Computer Science and Machine Learning Departments, Carnegie Mellon University. {\tt ninamf@cs.cmu.edu}} \and
Tuomas Sandholm\thanks{Computer Science Department, Carnegie Mellon University, Strategy Robot, Inc., Strategic Machine, Inc., Optimized Markets, Inc. {\tt sandholm@cs.cmu.edu}}}

\date{}

\begin{document}

\maketitle
\begin{abstract}
Sequence-independent lifting is a procedure for strengthening valid inequalities of an integer program. We generalize the sequence-independent lifting method of Gu, Nemhauser, and Savelsbergh (GNS lifting) for cover inequalities and correct an error in their proposed generalization. We obtain a new sequence-independent lifting technique---{\em piecewise-constant (PC) lifting}---with a number of interesting properties. We derive a broad set of sufficient conditions under which PC lifting is facet defining. To our knowledge, this is the first characterization of facet-defining sequence-independent liftings that are efficiently computable from the underlying cover. Finally, we demonstrate via experiments that PC lifting can be a useful alternative to GNS lifting. We test our new lifting techniques atop a number of novel cover cut generation routines, which prove to be effective in experiments with CPLEX.

\end{abstract}
\section{Introduction}
Lifting is a technique for strengthening cutting planes for integer programs by increasing the coefficients of variables that are not in the cut. We study lifting methods for valid cuts of {\em knapsack polytopes}, which have the form $\conv(P)$ where $$P = \left\{\vec{x}\in\{0,1\}^n : \sum\limits_{j=1}^n a_j x_j\le b\right\}$$ for $a_1,\ldots, a_n, b\in\N$ with $0 < a_1,\ldots, a_n\le b$. We interpret $P$ as the set of feasible packings of $n$ items with weights $a_1,\ldots, a_n$ into a knapsack of capacity $b$. Such {\em knapsack constraints} arise in binary integer programs from various industrial applications such as resource allocation, auctions, and container packing. They are a very general and expressive modeling tool, as any linear constraint involving binary variables admits an equivalent knapsack constraint by replacing negative-coefficient variables with their complements. A {\em minimal cover} is a set $C\subseteq\{1,\ldots,n\}$ such that $\sum_{j\in C}a_j > b$ and $\sum_{j\in C\setminus\{i\}}a_j\le b$ for all $i\in C$. That is, the items in $C$ cannot all fit in the knapsack, but any proper subset of $C$ can. The {\em minimal cover cut} corresponding to $C$ is the inequality $$\sum_{j\in C} x_j\le |C| - 1,$$ which enforces that the items in $C$ cannot all be selected. A {\em lifting} of the minimal cover cut is any valid inequality of the form \begin{equation}\label{eq:lifting}\sum\limits_{j\in C} x_j + \sum\limits_{j\notin C}\alpha_j x_j\le |C| - 1.\end{equation} The lifting coefficients $\alpha_j$ are often computed one-by-one---a process called {\em sequential lifting} that depends on the lifting order. Sequential lifting can be expensive since one must solve an optimization problem for each coefficient. Furthermore, one must reckon with the question of what lifting order to use. To lessen this computational burden, the lifting coefficients can be computed simultaneously. This method is called {\em sequence-independent lifting} and is the focus of this work. Our contributions include: (i) a generalization of the seminal sequence-independent lifting method developed by Gu et al.~\cite{Gu00:Sequence} and a correction of their proposed generalization; (ii) the first broad conditions under which sequence-independent liftings that are efficiently computable from the underlying cover---via our new techniques---define facets of $\conv(P)$; and (iii) new cover cut generation methods that, together with our new lifting techniques, display promising practical performance in experiments.

\subsection{Preliminaries on sequence-independent lifting}
    We begin with an overview of the \emph{lifting function} $f:[0,b]\to\R$ associated with a minimal cover $C$, defined by $$f(z) = |C| - 1 - \max\left\{\sum_{j\in C} x_j : \sum_{j\in C}a_jx_j\le b-z, x_j\in\{0,1\}\right\}.$$ For $i\notin C$, the value $f(a_i)$ is the maximum possible coefficient $\alpha_i$ such that $\sum_{j\in C}x_j + \alpha_i x_i\le |C| - 1$ is valid for $\conv(P)$. The lifting function has a more convenient closed form due to Balas~\cite{Balas75:Facets}. First, relabel the items so $C = \{1,\ldots,t\}$ and $a_1\ge\cdots\ge a_t$. Let $\mu_0 = 0$ and for $h = 1,\ldots, t$ let $\mu_h = a_1+\cdots+a_h$. Let $\lambda = \mu_t - b > 0$ be the cover's excess weight. Then, 
    $$f(z) = \begin{cases} 0 & 0\le z\le \mu_1 - \lambda \\ h & \mu_h - \lambda < z \le \mu_{h+1} - \lambda.\end{cases}$$ 
The lifting function has an intuitive interpretation: $f(z)$ is the maximum $h$ such that an item of weight $z$ cannot be brought into in $C$ and fit in the knapsack, even if we are allowed to discard any $h$ items from $C$. The lifting function $f$ may be used to maximally lift a {\em single} variable not in the cover. To lift a second variable, a new lifting function must be computed. This (order-dependent) process can be continued to lift all remaining variables, and is known as {\em sequential lifting}. Conforti et al.~\cite{Conforti14:Integer} and Hojny et al.~\cite{Hojny20:Knapsack} contain further details.

\paragraph{Superadditivity and sequence-independent lifting.} A function $g:D\to\R$ is superadditive if $g(u+v)\ge g(u)+g(v)$ for all $u,v, u+v\in D$. If $g\le f$ is superadditive, $$\sum\limits_{j\in C}x_j + \sum\limits_{j\notin C}g(a_j)x_j\le |C|-1$$ is a valid {\em sequence-independent} lifting for $\conv(P)$. This result is due to Wolsey~\cite{Wolsey77:Valid}; Gu et al.~\cite{Gu00:Sequence} generalize to mixed 0-1 integer programs. The lifting function $f$ is generally not superadditive. Gu et al.~\cite{Gu00:Sequence} construct a superadditive function $g\le f$ as follows. Let $\rho_h = \max\{0, a_{h+1} - (a_1 - \lambda)\}$ be the excess weight of the cover if the heaviest item is replaced with a copy of the $(h+1)$-st heaviest item. For $h\in\{0,\ldots, t-1\}$, let $F_h = (\mu_h-\lambda+\rho_h, \mu_{h+1}-\lambda]$ and for $h\in\{1,\ldots, t-1\}$, let $S_h = (\mu_h-\lambda, \mu_h-\lambda+\rho_h]$. $S_h$ is nonempty if and only if $\rho_h>0$. For $w:[0,\rho_1]\to [0,1]$, Gu et al. define $$g_w(z) = \begin{cases} 0 & z = 0 \\ h & z\in F_h\qquad h = 0,\ldots, t-1\;\\ h - w(\mu_h-\lambda+\rho_h-z) & z\in S_h\hfill h = 1,\ldots, t-1.\end{cases}$$ Gu et al. prove that for $w(x) = x/\rho_1$, $g_w$ is superadditive. We call this particular lifting function the {\em Gu-Nemhauser-Savelsbergh (GNS) lifting function}. Furthermore, $g_w$ is undominated, that is, there is no superadditive $g'$ with $f\ge g'\ge g_w$ and $g'(z') > g_w(z')$ for some $z'\in [0,b]$. 

\subsection{Our contributions}

In Section~\ref{section:theory}, we prove that under a certain condition, $g_w$ is superadditive for any linear symmetric function $w$. This generalizes Gu et al.'s result~\cite{Gu00:Sequence} 
for $w(x) = x/\rho_1$ and furthermore corrects an error in their proposed generalization, which incorrectly claims $w$ can be any symmetric function. Of particular interest is the constant function $w = 1/2$; we call the resulting lifting {\em piecewise-constant (PC) lifting}. In Section~\ref{section:comparisons} we give a thorough comparison of PC and GNS lifting. We show that GNS lifting can be arbitrarily worse than PC lifting, and characterize the full domination criteria between the two methods. In Section~\ref{section:facets}, we provide a broad set of conditions under which PC lifting defines facets of $\conv(P)$. {\em To our knowledge, these are the first conditions for facet-defining sequence-independent liftings that are efficiently computable from the underlying cover. \footnote{\citet{Balas75:Facets} proved that under certain conditions, (sequential) lifting coefficients are fully determined independent of the lifting order. This can be viewed as a set of sufficient conditions under which sequence-independent (or sequential; they are one and the same here) lifting would yield a facet-defining cut. However, when sequence-independent lifting can be non-trivially performed, ours is the first such result.}}

In Section~\ref{section:experiments}, we experimentally evaluate our 
lifting techniques in conjunction with a number of novel cover cut generation techniques. Our cut generation techniques do not solve expensive {\sf NP}-hard separation problems (which has been the norm in prior research~\cite{Kaparis10:Separation}). Instead, we cheaply generate many candidate cover cuts based on qualitative criteria, lift them, and check for separation only before 
adding the cut. This approach is effective in experiments with CPLEX. 

\subsection{Related work}

Cover cuts and their associated separation routines were first shown to be useful in a branch-and-cut framework by Crowder et al.~\cite{Crowder83:Solving}. Since then, there has been a large body of work studying various computational aspects, both theoretical and practical, of cover cuts, separation routines, and lifting. The seminal work of Gu et al.~\cite{Gu00:Sequence} showed how sequence-independent lifting can be performed efficiently using $g_w$ for $w(x) = x/\rho_1$. Gu et al.~\cite{Gu98:Lifted} perform a computational study of sequential lifting, and Wolter~\cite{Wolter06:Implementation} presents some computational results on the interaction between the sequence-independent lifting technique of Gu et al.~\cite{Gu00:Sequence} and different separation techniques. To our knowledge, this is the only computational study of sequence-independent lifting published to date. Our computational study takes a different approach than prior work. Rather than solving separation problems exactly, which involves expensive optimization, we generate large pools of candidate cover cuts, lift them, and check for separation before adding cuts to the formulation. This approach proves to be effective in our experiments. (The separation problem is {\sf NP}-hard~\cite{Klabjan98:Complexity,Gu99:Lifted}, but checking separation is a trivial linear time operation. More on separation can be found in Kaparis and Letchford~\cite{Kaparis10:Separation}.) Marchand et al.~\cite{Marchand02:Cutting} and Letchford and Souli~\cite{Letchford19:Lifted, Letchford20:Lifting} present other sequence-independent lifting functions based on superadditivity. 

\section{New sequence-independent lifting functions: structure and properties}\label{section:theory}

We generalize the result of Gu et al.~\cite{Gu00:Sequence} and also point out an error in their suggested generalization. Gu et al. claim that if $\mu_1-\lambda\ge\rho_1$, then $g_w$ is superadditive for any nondecreasing $w:[0, \rho_1]\to[0,1]$ such that $w(x) + w(\rho_1-x) = 1$. This claim is incorrect (we provide counterexamples in Appendix~\ref{app:erratum}). We show that this claim is correct when restricted to linear $w$.

\begin{theorem}\label{theorem:superadditive}
    For $k\in [0, 1/\rho_1]$, let $w_k(x) = kx + \frac{1-k\rho_1}{2}$, and let $g_k = g_{w_k}$. If $\mu_1-\lambda\ge\rho_1$, $g_{k}$ is superadditive and undominated.
\end{theorem}

The GNS lifting function is given by $g_{1/\rho_1}$. The proof of Theorem~\ref{theorem:superadditive} follows the proof that $g_{1/\rho_1}$ is superadditive~\cite{Gu00:Sequence} with a few key modifications; we defer it to Appendix~\ref{app:superadditive_proof}. If $\mu_1-\lambda <\rho_1$, $g_k$ is superadditive if and only if $k = 1/\rho_1$.

Of particular interest is the superadditive lifting function $g_0$, which we refer to as the {\em piecewise-constant (PC) lifting function}. The following result shows that the lifting obtained via $g_k$ is dominated by the union of the liftings obtained via $g_0$ (PC lifting) and $g_{1/\rho_1}$ (GNS lifting). Thus, in order to get as close to $\conv(P)$ as possible, it suffices to study these two lifting functions.

\begin{prop}\label{prop:convex_comb}
    Let $k\in (0, 1/\rho_1)$. 
    If $\sum_{j \in C} x_j + \sum_{j \notin C} g_0(a_j)x_j \leq |C| - 1$ and $\sum_{j \in C} x_j + \sum_{j \notin C} g_{1/\rho_1}(a_j)x_j \leq |C| - 1$, then $\sum_{j \in C} x_j + \sum_{j \notin C} g_k(a_j)x_j \leq |C| - 1$.
\end{prop}

\begin{proof}
    We have $g_k(z) = k\rho_1g_{1/\rho_1}(z) + (1-k\rho_1)g_0(z)$ by direct computation, so $g_k$ lifting is a convex combination of GNS and PC lifting. 
\end{proof}

\begin{example}\label{ex:running}
    Let $C = \{1,2,3,4\}$ and consider a knapsack constraint of the form $16x_1 + 14x_2 + 13x_3 + 9x_4 + \sum_{j\notin C} a_jx_j\le 44$. $C$ is a minimal cover with $\mu_1 = 16$, $\mu_2 = 30$, $\mu_3=43$, $\mu_4=48$, $\lambda=8$, $\rho_1 = 6$, $\rho_2 = 5$, $\rho_3 = 1$, and $\mu_1-\lambda\ge\rho_1$. Fig.~\ref{fig:theorem} depicts $g_0$ and $g_{1/\rho_1}$ truncated to the domain $[\mu_1-\lambda, \mu_3-\lambda] = [8, 35]$.
\end{example}

\begin{figure}[t]
  \centering
  \includegraphics[width=\columnwidth, trim={2.4cm 5.6cm 3cm, 9.5cm},clip  ]{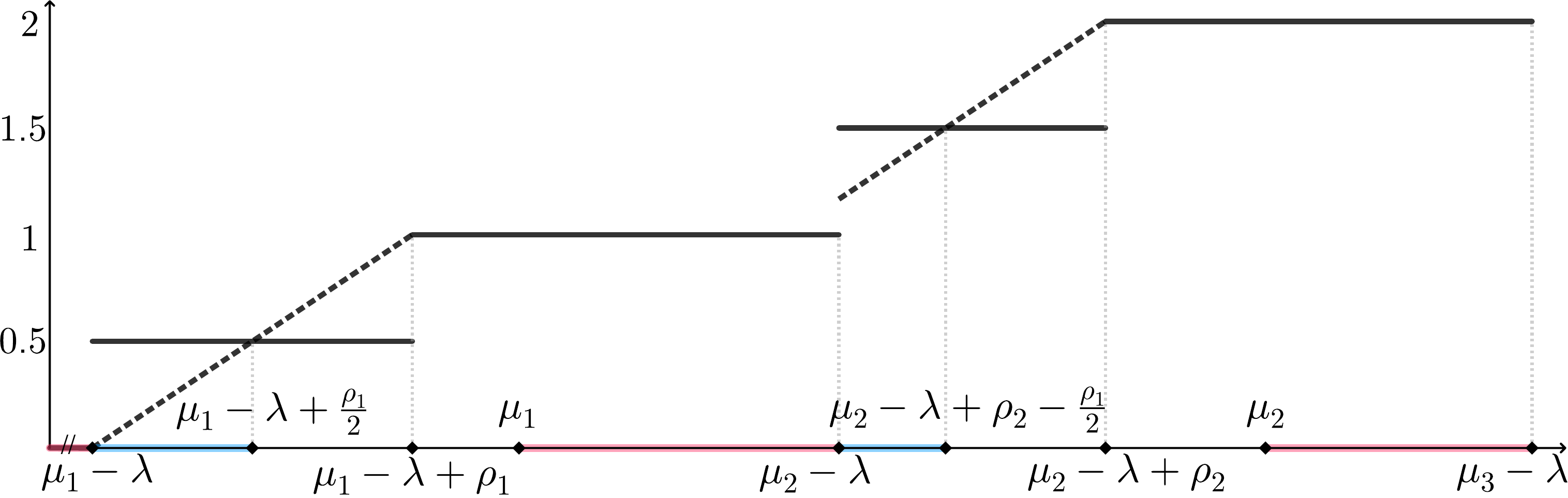}
  \caption{The PC lifting function $g_0$ is the piecewise constant step function depicted by the solid black lines. The GNS lifting function $g_{1/\rho_1}$ is obtained by replacing the solid lines in the intervals $S_h$ with the depicted dashed lines. If all coefficients of variables being lifted lie in the blue and red regions with at least three coefficients in the leftmost blue region, PC lifting is facet-defining and dominates GNS lifting (Theorem~\ref{theorem:facet}).}
  \label{fig:theorem}
\end{figure}

\subsection{Comparisons between PC and GNS lifting}\label{section:comparisons}

The following result shows GNS lifting can be arbitrarily worse than PC lifting.

\begin{prop}\label{prop:domination}
    For any $\varepsilon > 0, t\in\N$ there exists a knapsack constraint with a minimal cover $C$ of size $t$ such that PC lifting yields $$\sum_{j\in C} x_j + \sum_{j\notin C}\frac{1}{2}x_j\le |C|-1$$ and GNS lifting is dominated by $$\sum_{j\in C} x_j + \sum_{j\notin C}\varepsilon x_j\le |C|-1.$$
\end{prop}

The proof is in Appendix~\ref{app:domination}. At a high level, we construct an instance where the length of $S_1$, which is $\rho_1$, is large, and consider coefficients that are at the leftmost part $S_1$. GNS barely lifts such coefficients, while PC yields lifting coefficients of $1/2$. The next proposition fully characterizes the domination criteria between PC and GNS lifting. Its proof is immediate from the plots in Fig.~\ref{fig:theorem}.

\begin{prop}\label{prop:general_domination}
    Assume $\mu_1-\lambda\ge\rho_1$. Furthermore, suppose $\{j\notin C : \exists h\text{ s.t. }a_j\in S_h\}\neq\emptyset$ (else, GNS and PC trivially produce the same cut). If, for all $j\notin C$,
    \begin{enumerate}
        \item $a_j\in S_h\implies\rho_h > \frac{\rho_1}{2}\text{ and } a_j \le \mu_h-\lambda+\rho_h-\frac{\rho_1}{2}$ with at least one such $a_j\in S_h$ satisfying $a_j < \mu_h-\lambda+\rho_h-\frac{\rho_1}{2}$, PC strictly dominates GNS. 
        \item $a_j\in S_h\implies\rho_h > \frac{\rho_1}{2}\text{ and } a_j = \mu_h-\lambda+\rho_h-\frac{\rho_1}{2}$, PC and GNS yield the same cut.
        \item $a_j\in S_h\implies$ ($\rho_h\le\frac{\rho_1}{2}$) or ($\rho_h > \frac{\rho_1}{2}$ and $a_j > \mu_h-\lambda+\rho_h-\frac{\rho_1}{2}$), GNS strictly dominates PC.
        \item Otherwise, neither PC nor GNS dominates the other.
    \end{enumerate}
\end{prop}

\begin{example}\label{example:comparison}
    Consider the constraint $16x_1 + 14x_2 + 13x_3 + 9x_4 + a_5x_5 + a_6x_6 + a_7x_7\le 44$ with minimal cover $C = \{1,2,3,4\}$.
    \begin{enumerate}
        \item Let $a_5=9$, $a_6=10$, $a_7=23$. GNS yields $$x_1 + x_2 + x_3 + x_4 + \frac{1}{6}x_5 + \frac{1}{3}x_6 + \frac{4}{3}x_7\le 3.$$ PC yields the dominant cut $$x_1 + x_2 + x_3 + x_4 + \frac{1}{2}x_5 + \frac{1}{2}x_6 + \frac{3}{2}x_7\le 3.$$
        \item Let $a_5=11$, $a_6=17$, $a_7=24$. GNS and PC both yield the cut $$x_1 + x_2 + x_3 + x_4 + \frac{1}{2}x_5 + x_6 + \frac{3}{2}x_7\le 3.$$
        \item Let $a_5=12$, $a_6=13$, $a_7=26$. GNS yields $$x_1 + x_2 + x_3 + x_4 + \frac{2}{3}x_5 + \frac{5}{6}x_6 + \frac{11}{6}x_7\le 3.$$ PC yields the weaker cut $$x_1 + x_2 + x_3 + x_4 + \frac{1}{2}x_5 + \frac{1}{2}x_6 + \frac{3}{2}x_7\le 3.$$
        \item Let $a_5 = 9$, $a_6 = 13$, $a_7 = 24$. GNS yields $$x_1 + x_2 + x_3 + x_4 + \frac{1}{6}x_5 + \frac{5}{6}x_6 + \frac{3}{2}x_7\le 3.$$ PC yields $$x_1 + x_2 + x_3 + x_4 + \frac{1}{2}x_5 + \frac{1}{2}x_6 + \frac{3}{2}x_7\le 3.$$ Neither cut is dominating.
    \end{enumerate}
\end{example}

\noindent
{\em Open question:} Gu et al.~\cite{Gu99:Lifted} and Hunsaker and Tovey~\cite{Hunsaker05:Simple} construct knapsack problems where branch-and-cut requires a tree of exponential size, even when {\em all lifted cover inequalities} (all inequalities of the form~\eqref{eq:lifting}) are added to the formulation. Do there exist knapsack problems where branch-and-cut requires exponential-size trees when all GNS-lifted cover inequalities are added, but does not when PC-lifted cover inequalities are added (or vice versa)?

\subsection{Facet defining inequalities from sequence-independent lifting}\label{section:facets}

We now provide a broad set of sufficient conditions under which PC lifting yields facet-defining inequalities. Our result relies on the following complete characterization of the facets of the knapsack polytope obtained from lifting minimal cover cuts, due to Balas and Zemel~\cite{Balas78:Facets,Hojny20:Knapsack}, which we restate using the notation and terminology of Gu et al.~\cite{Gu00:Sequence} and Conforti et al.~\cite[sec. 7.2]{Conforti14:Integer}. Given a minimal cover $C$ and $j\notin C$, let $h(j)$ be the index such that $\mu_{h(j)}\le a_j < \mu_{h(j)+1}$.

\begin{theorem}[Balas and Zemel~\cite{Balas78:Facets}]\label{theorem:balas+zemel}
    Let $C$ be a minimal cover. Let $J = \{j\notin C: a_j > \mu_{h(j)+1}-\lambda\}$ and let $I = (\{1,\ldots,n\}\setminus C)\setminus J$. Let $\cQ(J) = \{Q\subseteq J : \sum_{j\in Q}a_j\le b, Q\neq\emptyset\}$. The inequality $$\sum_{j\in C}x_j + \sum_{j\notin C}\alpha_j x_j\le |C| - 1$$ is a facet of $\conv(P)$ if and only if $\alpha_j = h(j) + \delta_j\cdot\mathbf{1}(j\in J)$ where each $\delta_j\in [0,1]$ and $\vec{\delta} = (\delta_j)_{j\in J}$ is a vertex of the polyhedron $$T = \left\{\vec{\delta}\in\R^{|J|} :\sum_{j\in Q}\delta_j\le f\Big(\sum_{j\in Q}a_j\Big) - \sum_{j\in Q}h(j)~\forall~Q\in\cQ(J)\right\}.$$
\end{theorem}

The characterization of facets based on $T$ in Theorem~\ref{theorem:balas+zemel} does not translate to a tractable method of deriving facet-defining inequalities, since one would need to enumerate the vertices of $T$. In fact, 
Hartvigsen and Zemel~\cite{Hartvigsen92:Complexity} show the question of determining whether or not a cutting plane is a valid lifted cover cut is {\sf co-NP}-complete. 
Deciding whether or not a cutting plane is a facet-defining lifted cover cut is in $\mathsf{D^P}$ ($\mathsf{D^P}$ is a complexity class introduced by Papadimitriou and Yannakakis~\cite{Papadimitriou82:Complexity} to characterize the complexity of facet recognition). 
Critically, these complexity results hold when the underlying cover is given as input.

Our result, to the best of our knowledge, provides {\em the first broad set of sufficient conditions under which sequence-independent liftings that can be efficiently computed from the underlying minimal cover---via PC lifting---are facet defining.} We now state and prove our result.

\begin{theorem}\label{theorem:facet}
    Let $C = \{1,\ldots, t\}$, $a_1\ge\cdots\ge a_t$, be a minimal cover such that $\mu_1 - \lambda\ge\rho_1 > 0$. Suppose $|\{j\notin C :  a_j\in S_1 \}|\ge 3$ and for all $j\notin C$: \begin{align*}& a_j\in S_h\implies \rho_h>\frac{\rho_1}{2}\text{ and } a_j\le\mu_h-\lambda+\rho_h - \frac{\rho_1}{2}, \\ & a_j\in F_h\implies a_j\ge\mu_{h}~(\text{equivalently } h(j) \ge h).\end{align*} Then, the cut $$\sum_{j\in C} x_j + \sum_{j\notin C}g_0(a_j)x_j\le |C|-1,$$ obtained via PC lifting, defines a facet of $\conv(P)$.
\end{theorem}

\begin{proof}
    First we show that $J = \cup_{h\ge 1}\{j\notin C : a_j\in S_h\}$ and $I = \cup_{h\ge 0}\{j\notin C : a_j \in F_h \}$. Let $j\notin C$ be such that $a_j\in S_h$. We have $a_j > \mu_h-\lambda > \mu_{h-1}$ (as $\lambda \le a_i$ for any $i\in C$) and $a_j\le \mu_h - \lambda + \rho_h - \frac{\rho_1}{2} < \mu_h-\lambda+\rho_h < \mu_h$ (the final inequality follows directly from expanding out $\mu_h$ and $\rho_h$). So $h(j) = h - 1$, and as $a_j > \mu_h-\lambda = \mu_{h(j)+1}-\lambda$, $j\in J$. Next, let $j\notin C$ be such that $a_j\in F_h$. Then, $a_j\le \mu_{h+1}-\lambda < \mu_{h+1}$, and by assumption $a_j\ge \mu_h$, so $h(j) = h$. Therefore $a_j \le \mu_{h(j)+1}-\lambda$, and so $j\in I$. The facts that $a_j\in S_h\implies h(j) = h-1$ and $a_j\in F_h\implies h(j) = h$ will be used repeatedly in the remainder of the proof.

    We now determine the constraints defining the polyhedron $T\subset \R^{|J|}$ in Theorem~\ref{theorem:balas+zemel}. (For $j\in I$, PC lifting produces a coefficient of $h(j)$, as required by Theorem~\ref{theorem:balas+zemel}.) Partition $J$ into two sets $J = J_1\cup J_{>1}$ where $J_1 = \{j\notin C : a_j\in S_1\}$ and $J_{>1} = \cup_{h > 1}\{j\notin C : a_j\in S_h\}$. Each singleton $Q = \{j\}\in\cQ(J)$ yields the constraint $\delta_j\le 1$. Consider now $Q = \{j_1, j_2\}\in\cQ(J)$. We consider two cases, one where $j_1\in J_1$ and the other where $j_1\in J_{>1}$. First, let $j_1\in J_1$. Let $h$ be such that $a_{j_2}\in S_h$. We have $a_{j_1} + a_{j_2} > a_{j_2} > \mu_h-\lambda$, so $f(a_{j_1} + a_{j_2})\ge h$, and
    \begin{align*}
        a_{j_1} + a_{j_2} &\le \mu_1-\lambda+\frac{\rho_1}{2} + \mu_h-\lambda+\rho_h-\frac{\rho_1}{2} \\
        &= a_1 - \lambda + \mu_h - \lambda + (a_{h+1} - a_1 + \lambda) \\
        &= \mu_h + a_{h+1}-\lambda = \mu_{h+1}-\lambda,
    \end{align*}
    so $f(a_{j_1} + a_{j_2})\le h$. Therefore $f(a_{j_1} + a_{j_2}) = h$, and we get the constraint $\delta_{j_1} + \delta_{j_2} \le f(a_{j_1} + a_{j_2}) - h(j_1) - h(j_2) = h - 0 - (h-1) = 1$. Suppose now that $j_1, j_2\in J_{>1}$ (if $j_2\in J_1$ that is handled by the first case). 
    Let $h_1, h_2$ be the integers such that $a_{j_1}\in S_{h_1}$ and $a_{j_2}\in S_{h_2}$. We have $ a_{j_1} + a_{j_2} > \mu_{h_1}-\lambda + \mu_{h_2}-\lambda = (a_1+\cdots + a_{h_1}) + (a_1+\cdots + a_{h_2}) - 2\lambda > (a_1 + \cdots + a_{h_1-1}) + (a_1 + \cdots + a_{h_2}) - \lambda > \mu_{h_1 + h_2 - 1}-\lambda$ so $f(a_{j_1}+a_{j_2})\ge h_1 + h_2 - 1$, and $f(a_{j_1}+a_{j_2}) - h(j_1) - h(j_2)\ge 1$. So for such pairs, we obtain a constraint $\delta_{j_1} + \delta_{j_2}\le H(j_1, j_2)$, where $H(j_1, j_2)\ge 1$. 
    Finally, we consider $Q\in\cQ(J)$ with $|Q|\ge 3$. For $j\in Q$ let $h_j$ be the integer such that $a_{j}\in S_{h_j}$. We claim that $$\sum\limits_{j\in Q}a_j > \mu_{\sum_{j\in Q}h_j-\lfloor|Q|/2\rfloor}-\lambda.$$ This claim is most succinctly proven using the quantities defined to prove Theorem~\ref{theorem:superadditive} (originally used by Gu et al.~\cite{Gu00:Sequence} to prove superadditivity of $g_{1/\rho_1}$). To avoid notational clutter, we defer the proof of this claim to Appendix~\ref{app:facet_claim}. The claim implies $f(\sum_{j\in Q}a_j)\ge \sum_{j\in Q}h_j-\lfloor|Q|/2\rfloor$, so the constraint induced by $Q$ is of the form $\sum_{j\in Q}\delta_j\le H(Q)$, where 
    $$ H(Q) := f\Big(\sum\limits_{j\in Q}a_j\Big) - \sum\limits_{j\in Q}h(j) \ge \sum\limits_{j\in Q} h_j - \lfloor|Q|/2\rfloor - \sum\limits_{j\in Q}(h_j-1) = \lceil|Q|/2\rceil.$$ We can now write down a complete description of $T$: 
    \begin{equation*}
    T = 
    \left\{
    \vec{\delta}\in \R^{|J|} \;:\; 
    \begin{aligned}
    & (1)\; \delta_j\le 1\;\;\forall\; j\in J\\
    & (2)\; \delta_i + \delta_j\le 1 \;\;\forall\; (i, j) \in J_1\times J\\
    & (3)\; \delta_i + \delta_j\le H(i, j) \;\;\forall\; (i, j)\in J_{>1}\times J_{>1} \\
    & (4)\; \textstyle\sum_{j\in Q} \delta_j\le H(Q) \;\;\forall\; Q\in\cQ(J), |Q|\ge 3
    \end{aligned}
    \right\}
    \end{equation*}
    where $H(i, j)\ge 1$ for all $(i, j)\in J_{>1}\times J_{>1}$ and $H(Q)\ge \left\lceil|Q|/2\right\rceil$ for all $Q\in\cQ(J)$, $|Q|\ge 3$. We argue that $\vec{\delta} = (1/2,\ldots, 1/2)$ is a vertex of $T$. Constraints of type (1), (3), and (4) are satisfied, and type (2) constraints are tight. Fix distinct $j_1, j_2, j_3\in J_1$. The set of $|J|$ type (2) constraints $\{\delta_j + \delta_{j_1}\le 1 \;\forall\; j\in J\setminus\{j_1\}\}\cup\{\delta_{j_2} + \delta_{j_3}\le 1\}$ is linearly independent, and hence $\vec{\delta} = (1/2,\ldots,1/2)$ is a vertex of $T$. PC lifting produces precisely the coefficients prescribed by Theorem~\ref{theorem:balas+zemel}: $g_0(a_j) = h(j)$ for $j\in I$ and $g_0(a_j) = h(j) + \frac{1}{2}$ for $j\in J$, so we are done. 
\end{proof}

Fig.~\ref{fig:theorem} illustrates the sufficient conditions of Theorem~\ref{theorem:facet}. While a facet-defining PC lifting can be efficiently obtained given a minimal cover satisfying the sufficient conditions of Theorem~\ref{theorem:facet}, we do not show how to find a minimal cover satisfying these conditions. We conjecture that finding such a cover is {\sf NP}-hard.

The proof of (the converse direction of) Theorem~\ref{theorem:balas+zemel} maps the binding constraints defining a vertex of $T$ to get $n$ linearly independent points at which the lifting~\eqref{eq:lifting} is binding. One could characterize those points directly to prove Theorem~\ref{theorem:facet}, but we use Theorem~\ref{theorem:balas+zemel} to simplify the proof.

\begin{example}
    Consider the constraint $16x_1+14x_2+13x_3+9x_4+9x_5+10x_6+11x_7+23x_8\le 44$ with minimal cover $C=\{1,2,3,4\}$. GNS lifting yields the cut $$x_1+x_2+x_3+x_4+\frac{1}{6}x_5+\frac{1}{3}x_6+\frac{1}{2}x_7+\frac{4}{3}x_8\le 3$$ and PC lifting yields the strictly dominant facet-defining cut $$x_1+x_2+x_3+x_4+\frac{1}{2}(x_5+x_6+x_7)+\frac{3}{2}x_8\le 3.$$
\end{example}

\noindent
\emph{Remark.} There exist facet-defining lifted cover inequalities with half-integral coefficients that cannot be obtained via PC lifting. Balas and Zemel~\cite{Balas78:Facets} provide an example of a knapsack constraint and minimal cover for which $\vec{\delta} = (1/2,\ldots,1/2)$ is a vertex of $T$, but $\mu_1-\lambda\ge\rho_1$ {\em does not hold}. It is an interesting open question to investigate when such facets arise and how to (efficiently) find them. The lifting procedure of~\citet{Letchford19:Lifted} also produces half-integral coefficients, but it is unclear when it can yield facets (it produces cuts that in general are incomparable with ours).

It would be interesting to derive similar conditions under which GNS lifting defines facets. We provide some sufficient conditions, but it is likely that a stronger result could be derived. We leave this as an open question.

\begin{prop}\label{prop:gns_facet}
    If for all $j\notin C$, $a_j\in S_h\implies a_j = \mu_h-\lambda+\rho_h$, GNS lifting strictly dominates PC lifting and defines a facet of $\conv(P)$.
\end{prop}

\begin{proof}
    The condition implies $g_{1/\rho_1}$ coincides with the lifting function $f$ on all $a_j$, $j\notin C$, which, by {\em sequential lifting} (e.g., Prop. 7.2 in~\cite{Conforti14:Integer}), means GNS lifting is facet defining. PC lifting is strictly weaker due to item 3 of Prop.~\ref{prop:general_domination}. 
\end{proof}

\section{Experimental evaluation}\label{section:experiments}

We evaluate our new sequence-independent lifting techniques in conjunction with a number of methods for generating the minimal cover cuts that are to be lifted. We describe each component of the experimental setup below.

\subsection*{Cover cut generation}

Let $\sum_{j=1}^n a_jx_j\le b$ be a knapsack constraint, let $\vec{x}^{\LP}$ be the LP optimal solution at a given node of the branch-and-cut tree, and let $\cI = \{i : x^{\LP}_i > 0\}$. Enumerate $\cI$ as $\cI = \{1,\ldots, s\}$ (relabeling variables if necessary). We do not include variables not in $\cI$ in any of the minimal covers, since these do not contribute to the violation of $\vec{x}^{\LP}$ (though such variables may be assigned a nonzero coefficient in the lifted cover cut). Next, we present the five cover cut generation techniques that we use in experiments.

\paragraph{Contiguous covers.} First, sort the variables in $\cI$ in descending order of weight; without loss of generality relabel them so that $a_{1}\ge a_{2}\ge\cdots\ge a_{s}$. For each $i\in\{1,\ldots, s\}$, let $j\in\{i+1,\ldots, s\}$ be such that $C = \{i, i+1,\ldots, j\}$ is a minimal cover (if such a $j$ exists). This is the {\em contiguous} cover starting at $i$. We generate all such contiguous cover cuts for each $i$. (Balcan et al.~\cite{Balcan22:Improved} introduced these cover cuts, though they did not lift them nor did they restrict to $x_i^{\LP} > 0$.)

\paragraph{Spread covers.} As before, sort the variables in $\cI$ in descending order of weight; $a_1\ge\cdots\ge a_s$. For each $i$, let $j\in\{i+1,\ldots, s\}$ be {\em maximal} (if such a $j$ exists) such that there exists $k\in\{j+1,\ldots,s\}$ such that $C = \{i\}\cup\{j,\ldots, k\}$ is a minimal cover. Intuitively, $i$ can be thought of as a heavy ``head'', and the ``tail'' from $j$ to $k$ is as light as possible. We coin this the {\em spread} cover with head $i$. We generate all such spread cover cuts for each $i$.

\paragraph{Heaviest contiguous cover.} We define this as the contiguous cover starting at $1$.

\paragraph{Default cover.} Sort (and relabel) the variables in $\cI$ in descending order of LP value so that $x^{\LP}_1\ge\cdots\ge x^{\LP}_s$. Let $j$ be minimal so that $\{1,\ldots, j\}$ is a cover. Then, evict items, lightest first, until the cover is minimal. We coin this the {\em default} cover. These cover cuts are loosely based on the ``default'' separation routine implemented by Gu et al.~\cite{Gu98:Lifted}. (Their routine was for sequential lifting and does not directly carry over to our setting.) Wolter~\cite{Wolter06:Implementation} tested similar routines.

\paragraph{Bang-for-buck cover.} Sort (and relabel) the variables in $\cI$ in descending order of ``bang-for-buck'' so that $\frac{c_1}{a_1}\ge\cdots\frac{c_s}{a_s}$, where $c_i$ is variable $i$'s objective coefficient. Let $j$ be minimal so that $\{1,\ldots, j\}$ is a cover. Then, evict items, lightest first, until the cover is minimal. We coin this the {\em bang-for-buck} cover.

\begin{example}
    Consider the knapsack constraint $10x_1 + 9x_2 + 8x_3 + 7x_4 + 6x_5 + 6x_6 + 5x_7 + 4x_8\le 26$, suppose the LP optimal solution at the current branch-and-cut node is $\vec{x}^{\LP} = (0.1, 0.8, 0.7, 0.4, 0, 1, 0.2, 0.8)$, and suppose $\vec{c} = (5, 7, 9, 1, 2, 6, 6, 5)$. We have $\cI = \{1, 2, 3, 4, 6, 7, 8\}$. The contiguous minimal covers are $\{1, 2, 3\}$, $\{2, 3, 4, 6\}$, $\{3, 4, 6, 7, 8\}$. The spread minimal covers are $\{1,4, 6, 7\}$, $\{2, 4, 6, 7\}$, $\{3, 4, 6, 7, 8\}$. The heaviest contiguous cover is $\{1,2,3\}$. The default cover is $\{2, 3, 6, 8\}$. The bang-for-buck cover is $\{2, 3, 6, 7\}$.
\end{example}

\subsection*{Lifting}

We evaluate three lifting methods. (1) {\em GNS:} The cover cut is lifted using $g_{1/\rho_1}$. (2) {\em PC:} If $\mu_1-\lambda\ge\rho_1$ the cover cut is lifted using $g_0$, and otherwise it is lifted using $g_{1/\rho_1}$. (3) {\em Smart:} If $\mu_1-\lambda\ge\rho_1$, two liftings are generated: one with $g_0$ and one with $g_{1/\rho_1}$. If either lifting dominates the other, the weaker lifting is discarded. If $\mu_1-\lambda < \rho_1$, the cover cut is only lifted using $g_{1/\rho_1}$.

\subsection*{Integration with branch-and-cut}

Algorithm~\ref{alg:covercutgen} is the pseudocode for adding lifted cover cuts at a node of the branch-and-cut tree and is called (once) at every node of the tree. It uses the prescribed lifting method $\texttt{Lift}$ atop the prescribed cover cut generation method $\texttt{CoverCuts}$ for each constraint, and adds the resulting lifted cut to a set of candidate cuts if it separates the current LP optimum $\vec{x}^{\LP}$. It adds the $\ell$ deepest cuts among the candidate cuts to the formulation. (The depth or {\em efficacy} of a cut $\vec{\alpha}^{\top}\vec{x}\le\beta$ is the quantity $\frac{\vec{\alpha}^{\top}\vec{x}^{\LP}-\beta}{\norm{\vec{\alpha}_2}}$ and measures the distance between the cut and $\vec{x}^{\LP}$.) The $\ell$ cuts are added in a single wave/round at the current node, and no further cuts are generated at that node. Algorithm~\ref{alg:covercutgen} does not solve {\sf NP}-hard cover cut separation problems and instead relies on cheap and fast ways of generating many candidate cuts through $\texttt{CoverCuts}$, and only adds those that are separating.

\begin{algorithm}[t]
	\caption{Lifted cover cut generation at a node of branch-and-cut}\label{alg:covercutgen}
	\begin{algorithmic}[1]
\Require IP data $\vec{c}, \vec{A}, \vec{b}$, LP optimum at current node $\vec{x}^{\LP}$, per-node cut limit $\ell$
\State Initialize $\mathsf{cuts} = \emptyset$.
\For{each knapsack constraint $\vec{a}^\top\vec{x}\le b$\label{step:for_begin}}
    \For{each cover cut $\mathsf{cut}\in\texttt{CoverCuts}(\vec{a}, b, \vec{c}, \vec{x}^{\LP})$}
        \For{each lifted cover cut $\mathsf{liftedcut}\in\texttt{Lift}(\mathsf{cut})$}
            \If {$\mathsf{liftedcut}$ separates $\vec{x}^{\LP}$}
                \State $\mathsf{cuts}\leftarrow\mathsf{cuts}\cup\{\mathsf{liftedcut}\}$.
            \EndIf
        \EndFor
    \EndFor
\EndFor\label{step:for_end}
\State Add the top $\min\{\ell, |\mathsf{cuts}|\}$ cuts in $\mathsf{cuts}$ with respect to efficacy.
\end{algorithmic}
\end{algorithm}

\subsection*{Experimental results}

We evaluated our techniques on five different problem distributions: two distributions over winner-determination problems in multi-unit combinatorial auctions ({\em decay-decay} from~\cite{Sandholm02:Winner} and {\em multipaths} from CATS version 1.0~\cite{Leyton00:Towards}) and three distributions over multiple knapsack problems ({\em uncorrelated} and {\em weakly correlated} from~\cite{Fukunaga11:Branch} and {\em Chv\'{a}tal} from~\cite{Balcan22:Improved}). We ran experiments written in C{\tt ++} using the callable library of CPLEX version 20.1 on a 64-core machine, and implemented Algorithm~\ref{alg:covercutgen} with $\ell = 10$ within a cut callback. We generated $100$ integer programs from each distribution. We set a time limit of 1 hour, allocated 16GB of RAM, and used a single thread for each integer program. 

\newpage

\begin{figure}[th]
\centering
\begin{subfigure}
  \centering
  \includegraphics[width=.49\linewidth]{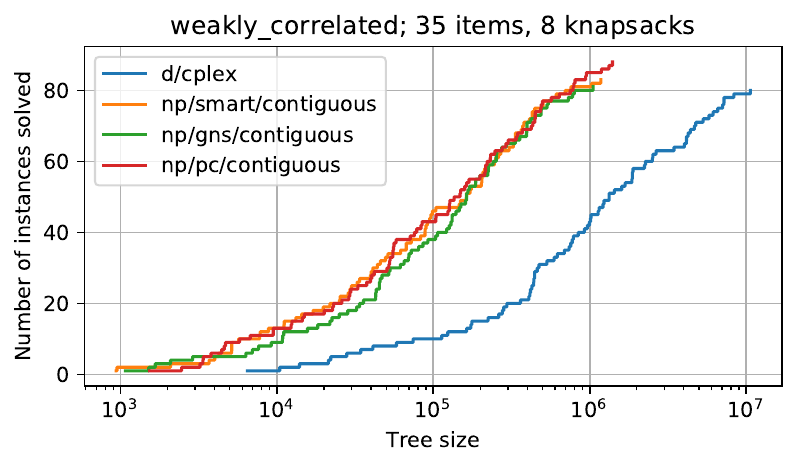}
  \label{fig:blah}
\end{subfigure}
\begin{subfigure}
  \centering
  \includegraphics[width=.49\linewidth]{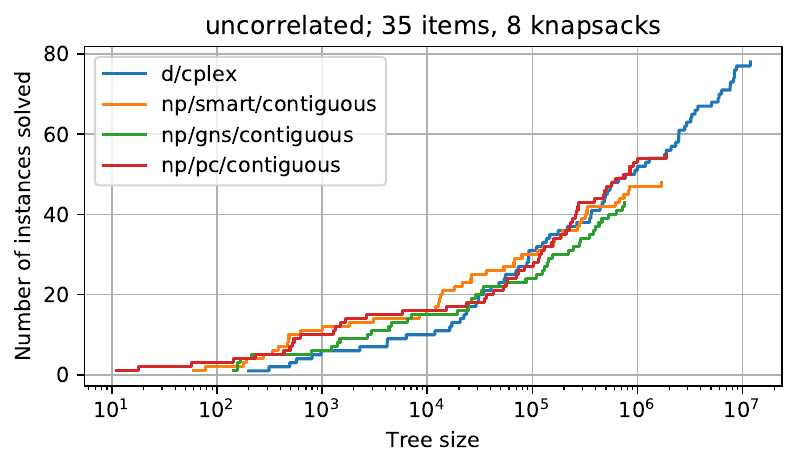}
\end{subfigure}
\begin{subfigure}
  \centering
  \includegraphics[width=0.49\linewidth]{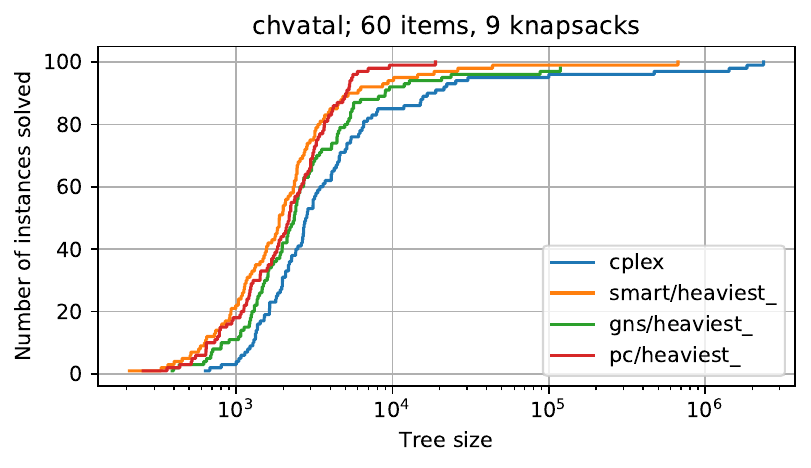}
\end{subfigure}
\begin{subfigure}
  \centering
  \includegraphics[width=0.49\linewidth]{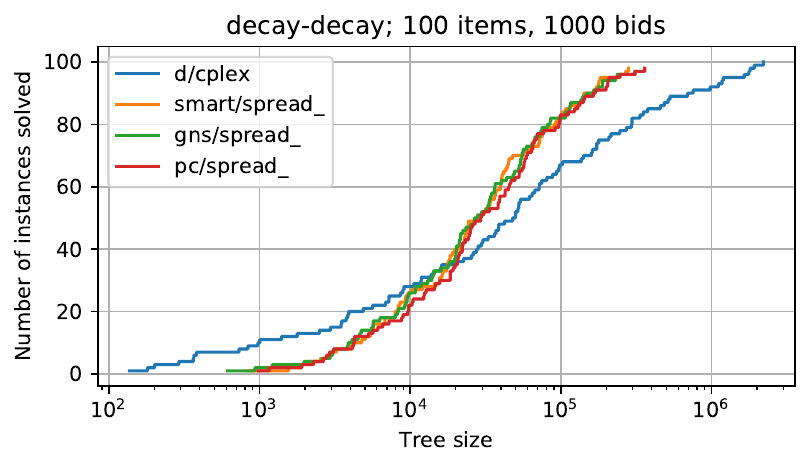}
\end{subfigure}
\begin{subfigure}
  \centering
  \includegraphics[width=0.49\linewidth]{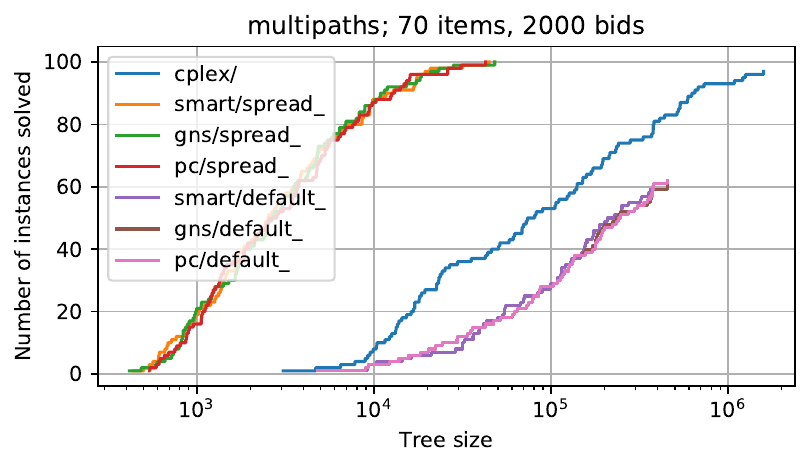}
\end{subfigure}
\begin{subfigure}
  \centering
  \includegraphics[width=0.49\linewidth]{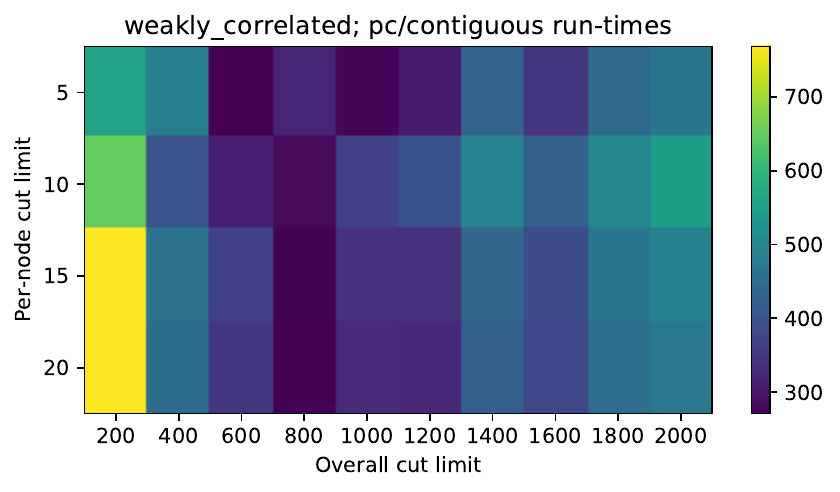}
\end{subfigure}
\caption{Illustrative experiments comparing different lifting methods. The first five plots are performance plots for the five different problem distributions, with various parameter settings. The heatmap illustrates the effect of varying the per-node cut limit and overall cut limit on run-time (avg. over first 10 weakly-correlated instances).}
\label{fig:aggregate}
\end{figure}


We present illustrative results in Fig.~\ref{fig:aggregate}, focusing on settings where PC lifting had significant impact. The full set of experiments are in App.~\ref{app:experiments}. We focus on tree size as the performance measure, but we provide full run-time results in App.~\ref{app:experiments}. Each curve is labeled {\tt setting/lifting/covers} or {\tt setting/CPLEX}. The {\tt setting} parameter is either empty, {\tt np}, or {\tt d}. Empty corresponds to all cuts, all heuristics, and presolve off. A setting of {\tt np} denotes presolve off and all other settings untouched. A setting of {\tt d} denotes the default settings (this setting is incompatible with our lifting implementations since presolve is incompatible with cut callbacks). If an underscore follows the label, CPLEX's internal cover cut generation is turned off; otherwise it is on. To ensure fair comparisons, in all CPLEX runs not involving our new techniques (those labeled {\tt setting/CPLEX}), we register a dummy cut callback that does nothing. This disables proprietary techniques such as ``dynamic search'' which do not support callbacks and thus do not support experiments of the type required in this paper.

We now describe the performance plots (Fig.~\ref{fig:aggregate}), which display how many instances each method was able to solve using trees smaller than the prescribed size on the $x$-axis within the 1 hour time limit. {\em Weakly correlated, contiguous covers (top left):} GNS solves $81$ instances, Smart solves $83$ instances, and PC solves $88$ instances. These all outperform default CPLEX which solves $80$ instances and builds trees an order of magnitude larger. {\em Uncorrelated, contiguous covers (top right):} GNS solves $43$ instances, Smart solves $48$ instances, and PC solves $55$ instances, while default CPLEX solves $78$ instances. {\em Chv\'{a}tal, heaviest covers (middle left):} PC lifting dramatically outperforms the other methods. Here, all CPLEX parameters are turned off, so we are directly comparing our lifted cover cut implementations with CPLEX's own cover cut generation routines. PC and Smart strictly outperform GNS and CPLEX (GNS is the only one unable to solve all 100 problems), and the largest tree size required by PC is an order of magnitude smaller than any of the other methods. This translates into a run-time improvement as well due to PC lifting (Fig.~\ref{fig:chvatal_} in the appendix). On the auction instances, there is little discernible performance difference between the lifting methods. {\em Decay-decay, spread covers (middle right):} our lifting implementations with all other CPLEX parameters off dramatically outperform default CPLEX on problems requiring trees of size $> 10^4$ (though default CPLEX solves all 100 problems while our methods do not.) {\em Multipaths, spread covers (bottom left):} we once again directly compare our lifted cover cut implementations with CPLEX's internal cover cut generation. Here, spread covers yield over an order of magnitude reduction in tree size relative to CPLEX, while default covers perform extremely poorly. We observed that contiguous and spread covers generally resulted in the best performance (with heaviest contiguous covers on par), and default and bang-for-buck covers performed drastically worse. 

While our techniques often led to significantly smaller trees than CPLEX, this did not translate to significant run-time improvements. However, in most settings our techniques were not too much slower than CPLEX and sometimes they were faster. Full run-time performance plots are in App.~\ref{app:experiments}. A possibility for this is that we did not limit the total number of cuts added throughout the tree, causing the formulation to grow very large. We ran an experiment to investigate the run-time impact of varying the number of cuts allowed (i) at each node ($\ell$ in Alg.~\ref{alg:covercutgen}) and (ii) throughout the entire search tree. We plot the average run-time (time limit of 1 hour) over the first 10 weakly-correlated instances using ${\tt pc/contiguous}$, visualized as a heatmap (Fig.~\ref{fig:aggregate}; bottom right). The best settings (limiting the overall number of cuts to 800) yield an average run-time of around 300 seconds and solved all 10 instances. Default CPLEX ({\tt d/CPLEX}) averaged roughly 900 seconds and only solved 9 instances to optimality.

\section{Conclusions and future research}

In this paper we showed that PC lifting can be a useful alternative to GNS lifting. We proved that under some sufficient conditions, PC lifting is facet-defining. To our knowledge, these are the first broad conditions for facet-defining sequence-independent liftings that can be efficiently computed from the underlying cover. We invented new cover cut generation routines, which in conjunction with our new lifting techniques, displayed promising practical performance.

There are a number of important future research directions that stem from our findings. First, a much more extensive experimental evaluation of PC lifting is needed. We have made several simplifying design choices, including (i) adding only one wave of lifted cover cuts at each node; (ii) ranking cuts solely based on efficacy (efficacy is not always the best quality to prioritize~\cite{Balcan22:Improved, Balcan22:Structural, Turner23:Cutting}); (iii) keeping $\ell$ constant across nodes while it could be a tuned hyperparameter, or even adjusted dynamically during search. Our experiments show promise even with these fixed choices, but a more thorough suite of tests could find how to best exploit the potential of our new techniques. A more comprehensive theoretical and practical comparison of PC and GNS lifting is needed as well. For example. PC lifting might possess desirable numerical properties in a solver since it always involves half-integral coefficients (as opposed to GNS). Another direction here is to use machine learning (which has already been used to tune cutting plane selection policies~\citep{Balcan21:Sample, li2023learning}) to decide when to use PC or GNS lifting. There might also be additional ways of determining what lifting method to use based on problem structure, and detecting that automatically.

\subsection*{Acknowledgements}
This material is based on work supported by the Vannevar Bush Faculty
Fellowship ONR N00014-23-1-2876, NSF grants
RI-2312342 and RI-1901403, ARO award W911NF2210266, NIH award
A240108S001 and Defense Advanced Research Projects Agency under cooperative agreement HR00112020003.

\bibliographystyle{plainnat}
\bibliography{references}

\appendix

\section{Counterexample to claim of Gu, Nemhauser, and Savelsbergh}\label{app:erratum}

Gu, Nemhauser, and Savelsbergh~\cite{Gu00:Sequence} remark that if $\mu_1-\lambda\ge\rho_1$, a large family of superadditive lifting functions can be constructed by considering any nondecreasing function $w(x)$ of $x\in [0,\rho_1]$ with $w(x) + w(\rho_1 - x) = 1$. Consider the class of logistic functions centered at $\rho_1/2$ of the form $$w_k(x) = \frac{1}{1+e^{-k(x - \rho_1/2)}}$$ where $k\ge 0$. Each $w_k$ is nondecreasing on $[0,\rho_1]$, and satisfies $w_k(x) + w_k(\rho_1-x) = 1$ for all $x$. In the following example, we show that using $w_{0.9}$ for sequence-independent lifting can yield invalid cuts. Moreover, the lifting function $g_{w_{0.9}}$ is not superadditive.

$$\begin{array}{ll} \text{maximize} & 112x_1 + 108x_2 + 107x_3 + 106x_4+102x_5 + 84x_6 + 82x_7\\
\text{subject to} & 
112x_1 + 108x_2 + 107x_3 + 106x_4+102x_5 + 84x_6 + 82x_7\le 268 \\ 
& x\in\{0,1\}^7
\end{array}$$

An optimal solution is given by $\vec{x}^* = (0, 0, 0,0,1,1,1)$, which has an objective value of $268$, satisfying the single knapsack constraint with equality. The set $C = \{2, 3, 4\}$ is a minimal cover, and the corresponding minimal cover inequality is $x_2 + x_3 + x_4\le 2$. We compute the relevant parameters needed for sequence-independent lifting. We have $\mu_0 = 0, \mu_1 = 108, \mu_2 = 215, \mu_3 = 321, \lambda = 53$, and $\rho_0 = 53, \rho_1 = 52, \rho_2 = 51$. We have that $\mu_1-\lambda\ge\rho_1$ is satisfied. Thus, the lifting function $g_w$ (truncated to the range $[0, 213]$) is given by $$g_w(z) = \begin{cases} 0 & 0 < z\le 55 \\ 1 - w(107-z) & 55 < z\le 107 \\ 1 & 107 < z \le 162 \\ 2 - w(213-z) & 162 < z\le 213.\end{cases}$$

Using $g_{w_{0.9}}$ yields the lifted cover inequality $$x_1 + x_2 + x_3 + x_4 + 0.99x_5 + 0.93 x_6 + 0.71 x_7\le 2.$$ But $0.99 + 0.93 + 0.71  = 2.63 > 2$, so $\vec{x}^*$ violates this inequality, so the lifted cover inequality is invalid for our problem. Furthermore, $g_{w_{0.9}}$ is not superadditive. We have $$g_{w_{0.9}}(82) + g_{w_{0.9}}(82) = 0.71 + 0.71 > g_{w_{0.9}}(82+82)\approx 1.00.$$

One need not look at logistic functions to derive this counterexample. The step function $$w(x) = \begin{cases}0 & x < \rho_1 / 2 \\ 1/2 & x = \rho_1/2 \\ 1 & x>\rho_1/2 \end{cases}$$ disproves the claim as well (in that $g_w$ is not superadditive, and results in an invalid cut in the above example).

\section{Proof of Theorem~\ref{theorem:superadditive}}\label{app:superadditive_proof}

The proof that $g_k$ is superadditive closely follows the proof that $g_{1/\rho_1}$ is superadditive by~\cite{Gu00:Sequence} (Lemma 1 in their paper), with a couple key modifications. As done in~\cite{Gu00:Sequence}, we will establish superadditivity of a function that is defined slightly more generally.

Given $v_1 > 0$, $u_i\ge 0, u_i\ge u_{i+1}, i = 1, 2,\ldots, v_i\ge 0, v_i\ge v_{i+1}, i = 1,2,\ldots$ such that $u_i + v_i > 0$ for all $i$, let $M_0 = 0$ and $M_h = \sum_{i=1}^h (u_i + v_i)$ for $h = 1, 2,\ldots,\infty$. Define $$\tilde{g}_k(z) = \begin{cases}0 & z = 0 \\ h & M_h < z\le M_h + u_{h+1}, \qquad h = 0, 1,\ldots \\ h + 1 - w_k(M_{h+1} - z) & M_h + u_{h+1} < z\le M_{h+1}, \hfill h = 0,1,\ldots\end{cases}$$ where $w_k(x) = kx + \frac{1 - kv_1}{2}$.

The superadditive lifting function $g_k$ is recovered by letting $u_i = a_i - \rho_{i-1}$ for $i\in\{1,\ldots, t\}$, $v_i = \rho_i$ for $i\in\{1,\ldots, t-1\}$. This yields $M_h = \mu_h - \lambda + \rho_h$ and $M_h + u_{h+1} = \mu_{h+1}-\lambda$ (see Gu et al.~\cite{Gu00:Sequence} for further details).

\begin{lemma}
    Let $k \in [0, 1/v_1]$. If $u_1\ge v_1$, the function $\tilde{g}_k$ is superadditive on $[0, \infty)$.
\end{lemma}

\begin{proof}
    We prove that $\max\{\tilde{g}(z_1) + \tilde{g}(z_2) - \tilde{g}(z_1+z_2) : z_1, z_2\in [0, \infty)\}\le 0$, which is equivalent to superadditivity. We break the analysis into cases as done in~\cite{Gu00:Sequence}.

    \noindent
    \underline{Case 1:} $M_{h_1} + u_{h_1+1} < z_1\le M_{h_1+1}$ and $M_{h_2} + u_{h_2+1} < z_2\le M_{h_2+1}$. Then, as in~\cite{Gu00:Sequence}, $z_1 + z_2 \ge M_{h_1+h_2} + u_{h_1+h_2+1}$. The first subcase of~\cite{Gu00:Sequence} (Case 1.1) is $z_1 + z_2\le M_{h_1+h_2+1}$, that is, $z_1+z_2$ lies on a ``sloped" segment of $\tilde{g}$. We show that the assumption $u_1\ge v_1$ rules this case out, that is, $z_1 + z_2 > M_{h_1+h_2+1}$.

    \noindent
    {\em Proof that $z_1 + z_2 > M_{h_1+h_2+1}$}: First, suppose $h_1 = 0$. We have 
    \begin{align*} 
        z_1 + z_2 
        &> u_1 + M_{h_2} + u_{h_2 + 1} \\ 
        &\ge v_{h_2 + 1} + M_{h_2} + u_{h_2 + 1} \qquad (\text{since }u_1\ge v_1\ge v_{h_2+1})\\ 
        &= M_{h_2 + 1},
    \end{align*} 
    as desired. The case where $h_2 = 0$ is symmetric. Thus, suppose $h_1, h_2\ge 1$, and without loss of generality let $h_1\le h_2$ (the case where $h_2\le h_1$ is symmetric). We have 
    \begin{align*}
        z_1 + z_2 
        &> M_{h_1} + u_{h_1+1} + M_{h_2} + u_{h_2 + 1} \\
        &= \sum\limits_{i=1}^{h_1}(u_i + v_i) + u_{h_1+1} + \sum\limits_{i=1}^{h_2}(u_i + v_i) + u_{h_2+1} \\
        &= (u_1 + v_1) + \sum\limits_{i = 2}^{h_1}(u_i + v_i) + u_{h_1+1} + \sum\limits_{i=1}^{h_2}(u_i + v_i) + u_{h_2+1} \\
        &\ge (u_1+v_1) + \sum\limits_{i = h_2 + 1}^{h_1 + h_2 - 1}(u_i + v_i) + u_{h_1+1} + \sum\limits_{i=1}^{h_2}(u_i + v_i) + u_{h_2 + 1} \\
        &= M_{h_1 + h_2 - 1} + u_1 + v_1 + u_{h_1 + 1} + u_{h_2+1}
    \end{align*}
    Now, $u_{h_1+1}\ge u_{h_1+h_2+1}$ and as $h_1, h_2\ge 1$, $u_{h_2+1}\ge u_{h_1+h_2}$. Furthermore, $u_1\ge v_1\ge v_{h_1 + h_2}\ge v_{h_1+h_2+1}$. Therefore, the right-hand-side is at least $$M_{h_1+h_2-1} + (u_{h_1+h_2} + v_{h_1+h_2}) + (u_{h_1+h_2+1} + v_{h_1+h_2+1}) = M_{h_1+h_2+1},$$ as desired. This completes the proof that under Case 1, $z_1 + z_2 > M_{h_1 + h_2 + 1}$. We now proceed with the casework under Case 1, starting our numbering from Case 1.2 so that the labeling of the cases coincides with~\cite{Gu00:Sequence}.

    \noindent
    \underline{Case 1.2:} $M_{h_1+h_2+1} < z_1 + z_2\le M_{h_1+h_2+1} + u_{h_1+h_2+2}$. Then 
    \begin{align*}
        &\tilde{g} (z_1) + \tilde{g}(z_2) - \tilde{g}(z_1+z_2) \\
        &= \left(h_1+1 - k(M_{h_1+1}-z_1) - \frac{1-kv_1}{2}\right) + \left(h_2+1 - k(M_{h_2+1}-z_2) - \frac{1-kv_1}{2}\right) - (h_1+h_2+1) \\
        &= 1 - k(M_{h_1+1} + M_{h_2+1} - z_1 - z_2) - (1 - kv_1) \\
        &\le 1 - k(M_{h_1+1} + M_{h_2+1} - M_{h_1+h_2+1} - u_{h_1+h_2+2}) - (1-kv_1) \\
        &= 1 - k\Big((u_1+v_1) + \sum\limits_{i=2}^{h_1+1}(u_i+v_i) + \sum\limits_{i = 1}^{h_2+1}(u_i+v_i) -\sum\limits_{i=1}^{h_1+h_2+1}(u_i+v_i) - u_{h_1+h_2+2}\Big) - (1-kv_1) \\ 
        &= 1 - k(u_1+v_1 - u_{h_1+h_2+2}) - k\sum\limits_{i=2}^{h_1+1}((u_i+v_i) - (u_{i+h_2} + v_{i+h_2})) - (1-kv_1) \\ 
        &= -k\underbrace{(u_1 - u_{h_1+h_2+2})}_{\ge 0} - k\underbrace{\sum\limits_{i=2}^{h_1+1}((u_i+v_i) - (u_{i+h_2} + v_{i+h_2}))}_{\ge 0} \\
        &\le 0.
    \end{align*}

    \noindent
    \underline{Case 1.3:} $M_{h_1+h_2+1} + u_{h_1+h_2+2} < z_1 + z_2\le M_{h_1+h_2+2}$. Then
    \begin{align*}
        &\tilde{g} (z_1) + \tilde{g}(z_2) - \tilde{g}(z_1+z_2) \\
        &= \left(h_1+1 - k(M_{h_1+1}-z_1) - \frac{1-kv_1}{2}\right) + \left(h_2+1 - k(M_{h_2+1}-z_2)- \frac{1-kv_1}{2}\right) \\
        &\qquad\qquad\qquad  - \left((h_1+h_2+2) - k(M_{h_1+h_2+2} - z_1 - z_2) - \frac{1-kv_1}{2}\right) \\ 
        &= -k(M_{h_1+1} + M_{h_2+1} - M_{h_1+h_2+2}) - \frac{1-kv_1}{2} \\
        & \le -k(M_{h_1+1} + M_{h_2+1} - M_{h_1+h_2+2}) \qquad (\text{since }k\le 1/v_1) \\
        &\le 0.
    \end{align*}

    The remaining cases (1.4, 2, 3.1 - 3.2 from~\cite{Gu00:Sequence}) and their proofs follow~\cite{Gu00:Sequence} verbatim, so we omit them. This completes the proof that $\tilde{g}$ is superadditive. 
\end{proof}

Now, the proofs that $g$ is superadditive, maximal, and undominated are identical to those in~\cite{Gu00:Sequence}.

\section{Proof of Proposition~\ref{prop:domination}}\label{app:domination}

We prove Proposition~\ref{prop:domination}, which states that for any $\varepsilon > 0$ and any $t\in\N$ there exists a knapsack constraint $\vec{a}^\top\vec{x}\le b$ with a minimal cover $C$ of size $t$ such that PC lifting yields $\sum_{j\in C} x_j + \sum_{j\notin C}\frac{1}{2}x_j\le |C|-1$ and GNS lifting is dominated by $\sum_{j\in C} x_j + \sum_{j\notin C}\varepsilon x_j\le |C|-1$.

\begin{proof}
    Let $a_1\ge\cdots\ge a_t$ and let  $b$ be such that $a_1+\cdots + a_t > b$ and $a_1+\cdots+a_{t-1}\le b$. Let $\lambda' = a_1+\cdots+a_t-b$. Furthermore, choose $a_1,\ldots, a_t, b$ so that $a_1-\lambda'\ge a_2-a_1+\lambda'>0$. Let $M\ge\frac{1}{(a_2-a_1+\lambda')\varepsilon}$ and consider the knapsack constraint $$M(a_1x_1+\cdots+a_tx_t) + (1+M(a_1-\lambda'))(x_{t+1}+\cdots+x_{n})\le Mb.$$ $C = \{1,\ldots,t\}$ is clearly a minimal cover with $\mu_1 = Ma_1$, $\lambda = M\lambda'$ and $\rho_1 = M(a_2-a_1+\lambda')$, and by the choice of $a_1,\ldots, a_t$, $\mu_1-\lambda = M(a_1-\lambda')\ge M(a_2-a_1+\lambda') = \rho_1$. Hence, PC lifting can be used to yield a valid lifted cut. As $1+M(a_1-\lambda') = 1+(\mu_1-\lambda)$, we have $g_0(1+M(a_1-\lambda')) = \frac{1}{2}$ and $$g_{1/\rho_1}(1+M(a_1-\lambda')) = 1-\frac{\mu_1-\lambda+\rho_1 - (1+\mu_1-\lambda)}{\rho_1} = \frac{1}{\rho_1} = \frac{1}{M(a_2-a_1+\lambda')}\le\varepsilon.$$
\end{proof}

\section{Proof of claim in Theorem~\ref{theorem:facet}}\label{app:facet_claim}

In the proof of Theorem~\ref{theorem:facet} we critically used the fact that given $Q\in\cQ(J)$, $|Q|\ge 3$, with $h_j$ such that $a_j\in S_{h_j}$, $$\sum\limits_{j\in Q}a_j > \mu_{\sum_{j\in Q}h_j-\lfloor|Q|/2\rfloor}-\lambda.$$ This allowed us to ensure that the constraint induced by every such $Q$ was satisfied by the point $(1/2,\ldots,1/2)$. We prove this claim here.

\begin{proof}
We use the quantities $u_h, v_h, M_h$ defined in Appendix~\ref{app:superadditive_proof}. We break the proof into two cases. We will use the observation that $u_{h}\ge v_{h}$ for any $h$ such that $\rho_h > 0$. This is because for such $h$, $u_{h} = u_1 = a_1-\lambda$, and so $u_{h} = u_1\ge v_1\ge v_h$.

\noindent
{\em Case 1:} $|Q| = 2\ell$ is even. Let (without loss of generality) $Q = \{a_1,\ldots, a_{2\ell}\}$, let $h_j$ be such that $a_j\in S_{h_j}$, and let $H = h_1+\cdots + h_{2\ell}$. We have
\begin{align*}
    \sum\limits_{j=1}^{2\ell}a_j > \sum\limits_{j=1}^{2\ell} \mu_{h_j}-\lambda 
    &= \sum\limits_{j=1}^{2\ell} M_{h_j-1} + u_{h_j} \\
    &\ge M_{H - 2\ell} + \sum\limits_{j=1}^{2\ell} u_{h_j} \\
    &= \sum\limits_{i=1}^{H-2\ell}(u_i + v_i) + \sum\limits_{j=1}^{2\ell} u_{h_j} \\
    &\ge \sum\limits_{i=1}^{H-2\ell}(u_i + v_i) + \sum\limits_{j=\ell+1}^{2\ell} (u_{h_j} + v_{h_j}) \\
    &\ge \sum\limits_{i=1}^{H-2\ell}(u_i + v_i) + \sum\limits_{j=1}^{\ell}(u_{H-2\ell + j}+v_{H-2\ell + j}) \\
    &= \sum\limits_{i = 1}^{H-\ell}(u_i+v_i) \\
    &= M_{H-\ell} \\
    &= \mu_{H - \ell} -\lambda + \rho_{H-\ell} \\
    &\ge \mu_{H - \ell} -\lambda,
\end{align*}
as desired.

\noindent
{\em Case 2:} $|Q|=2\ell+1$ is odd. Let $Q = \{a_1,\ldots, a_{2\ell+1}\}$, let $h_j$ be such that $a_j\in S_{h_j}$, and let $H = h_1+\cdots+h_{2\ell+1}$. We have 
\begin{align*}
    \sum\limits_{j=1}^{2\ell+1}a_j &> \sum\limits_{j=1}^{2\ell+1} \mu_{h_j}-\lambda \\
    &= \sum\limits_{j=1}^{2\ell+1} M_{h_j-1} + u_{h_j} \\
    &\ge M_{H - (2\ell+1)} + \sum\limits_{j=1}^{2\ell+1} u_{h_j} \\
    &= \sum\limits_{i=1}^{H-(2\ell+1)}(u_i + v_i) + \sum\limits_{j=1}^{2\ell+1} u_{h_j} \\
    &\ge \sum\limits_{i=1}^{H-(2\ell+1)}(u_i + v_i) + \sum\limits_{j=\ell+1}^{2\ell} (u_{h_j} + v_{h_j}) + u_{h_{2\ell+1}}\\
    &\ge \sum\limits_{i=1}^{H-(2\ell+1)}(u_i + v_i) + \sum\limits_{j=1}^{\ell}(u_{H-(2\ell+1) + j}+v_{H-(2\ell+1) + j}) + u_{H-(2\ell+1) + (\ell+1)} \\
    &= \sum\limits_{i = 1}^{H-\ell-1}(u_i+v_i) + u_{H-\ell}\\
    &= M_{H-\ell-1} + u_{H-\ell}\\
    &= \mu_{H - \ell} -\lambda,
\end{align*}
as desired. 
\end{proof}

\section{Full experimental results}\label{app:experiments}

We provide the full set of experimental results. In all plots, we evaluate our methods against three CPLEX settings.
\begin{enumerate}
    \item {\tt d/CPLEX} denotes default CPLEX with all settings untouched.
    \item {\tt np/CPLEX} denotes CPLEX with presolve turned off and all other settings untouched.
    \item {\tt CPLEX} denotes CPLEX with all heuristics and presolve turned off. Furthermore, all cuts are turned off with the exception of CPLEX's internal cover cut generation, which is left on.
\end{enumerate}

In all above versions of CPLEX, we register a dummy cut callback that does nothing. This disables proprietary search techniques such as ``dynamic search''. This is needed for a fair comparison against our lifting implementations, since these require cut callbacks. 

\subsection{Direct evaluation of lifted cover cut methods}

First, we evaluate our lifting methods atop bare-bones CPLEX settings; all CPLEX cuts, heuristics, and presolve settings are switched off. We toggle CPLEX's internal cover cut generation routine on and off to measure the impact of our routines with and without the cover cuts added by CPLEX (turning CPLEX cover cuts off corresponds to an underscore label in our plots).

Figures~\ref{fig:mkp_weakly_correlated},~\ref{fig:chvatal},~\ref{fig:muca}, and~\ref{fig:multipaths} contain the tree size and run-time performance plots for the weakly correlated, Chv\'{a}tal, decay-decay and multipaths distributions, respectively, when CPLEX cover cuts are turned on. Figures~\ref{fig:mkp_weakly_correlated_},~\ref{fig:chvatal_},~\ref{fig:muca_}, and~\ref{fig:multipaths_} contain the tree size and run-time performance plots for the weakly correlated, Chv\'{a}tal, decay-decay and multipaths distributions, respectively, when CPLEX cover cuts are turned off. 

\begin{figure}[t]
\centering
\begin{subfigure}
  \centering
  \includegraphics[width=.39\linewidth]{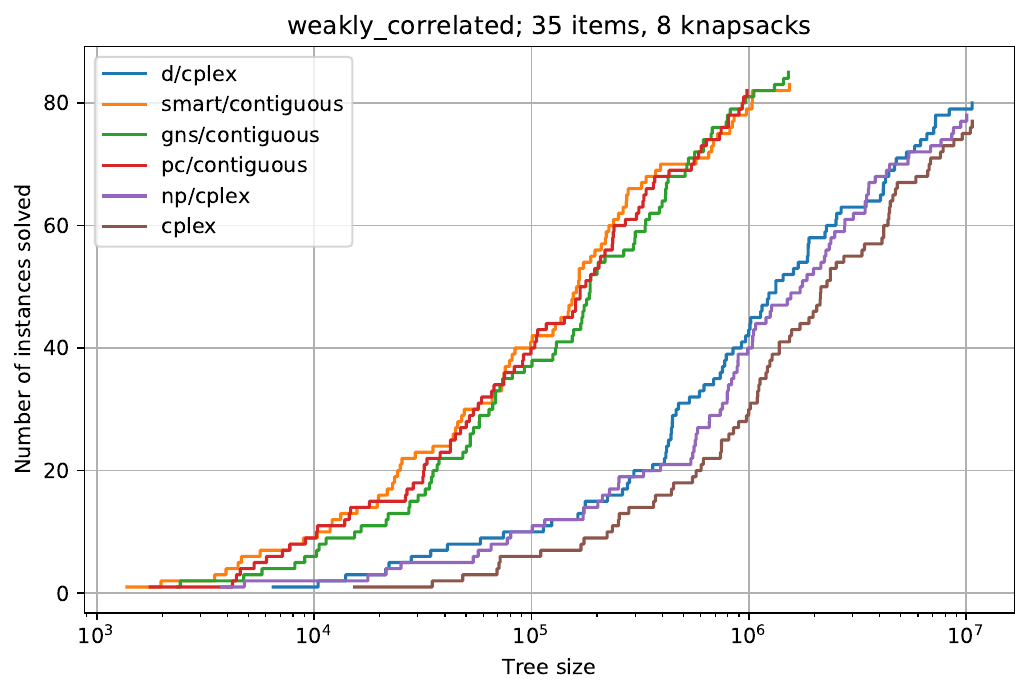}
  \label{fig:blah}
\end{subfigure}
\begin{subfigure}
  \centering
  \includegraphics[width=.39\linewidth]{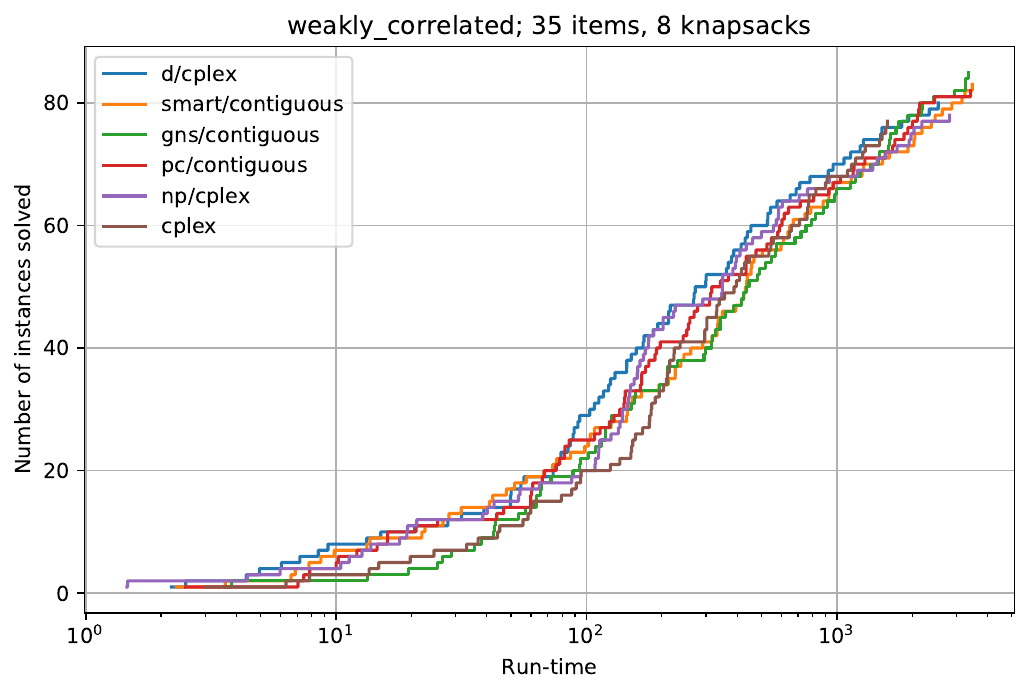}
\end{subfigure}
\begin{subfigure}
  \centering
  \includegraphics[width=.39\linewidth]{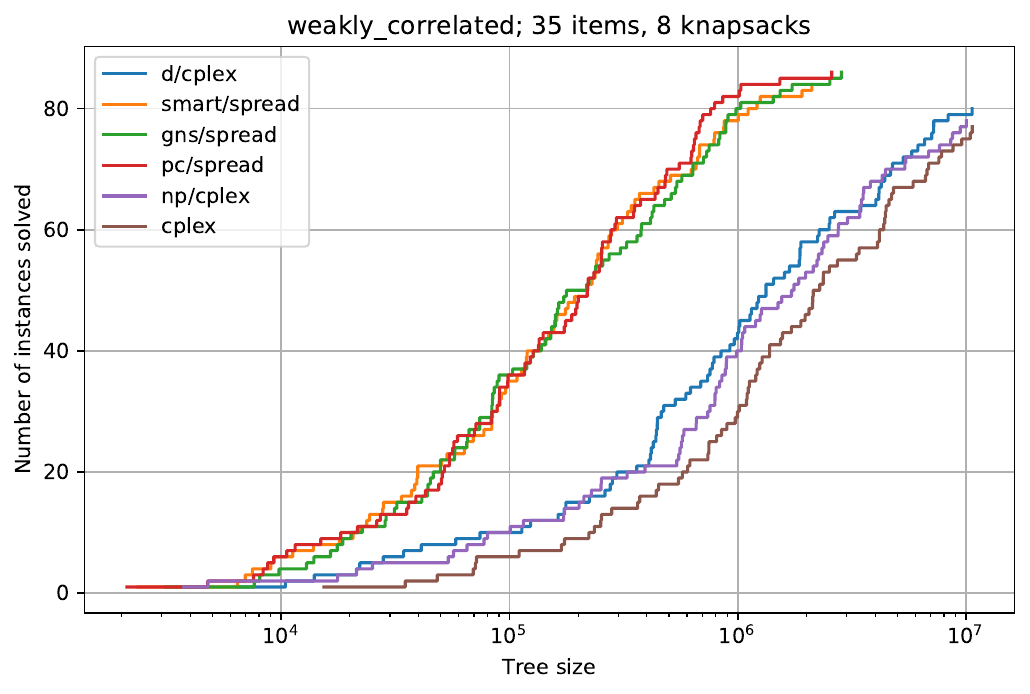}
  \label{fig:blah}
\end{subfigure}
\begin{subfigure}
  \centering
  \includegraphics[width=.39\linewidth]{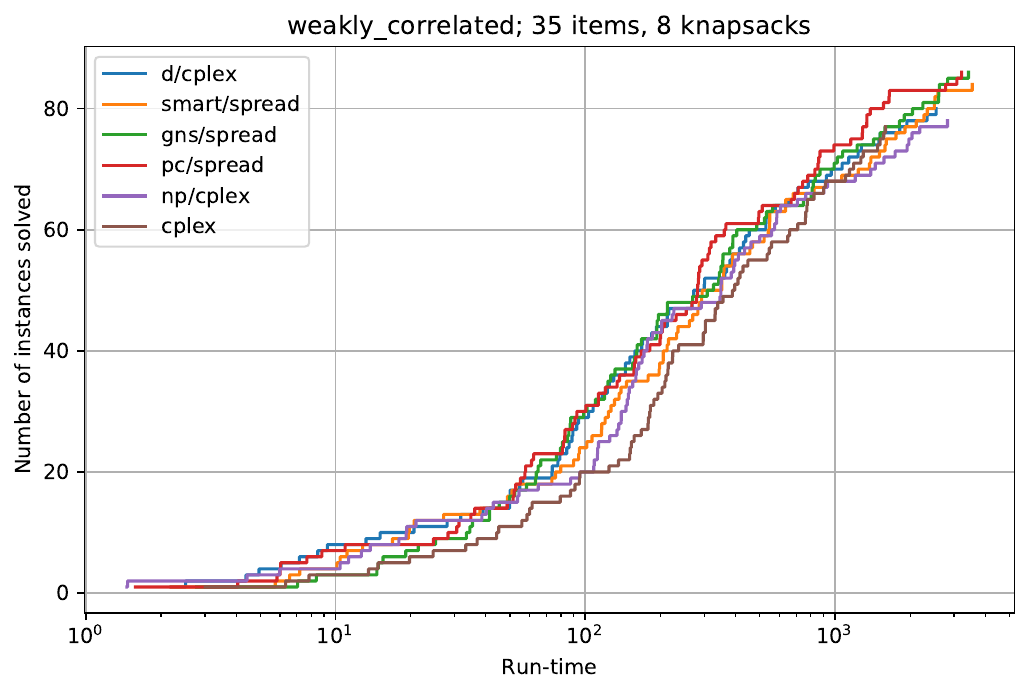}
\end{subfigure}
\begin{subfigure}
  \centering
  \includegraphics[width=.39\linewidth]{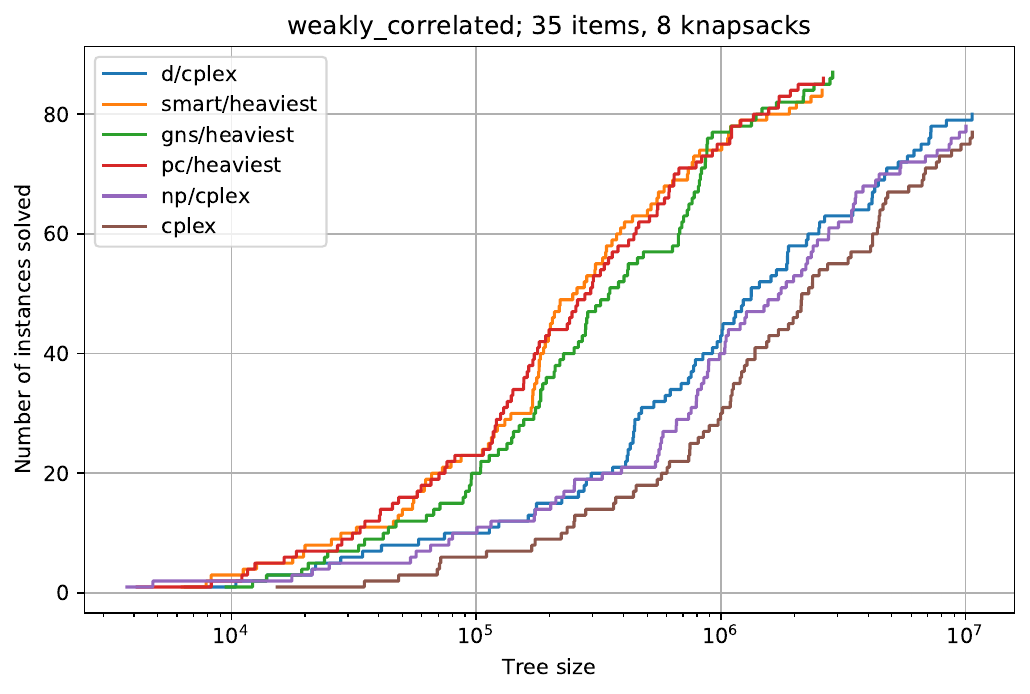}
  \label{fig:blah}
\end{subfigure}
\begin{subfigure}
  \centering
  \includegraphics[width=.39\linewidth]{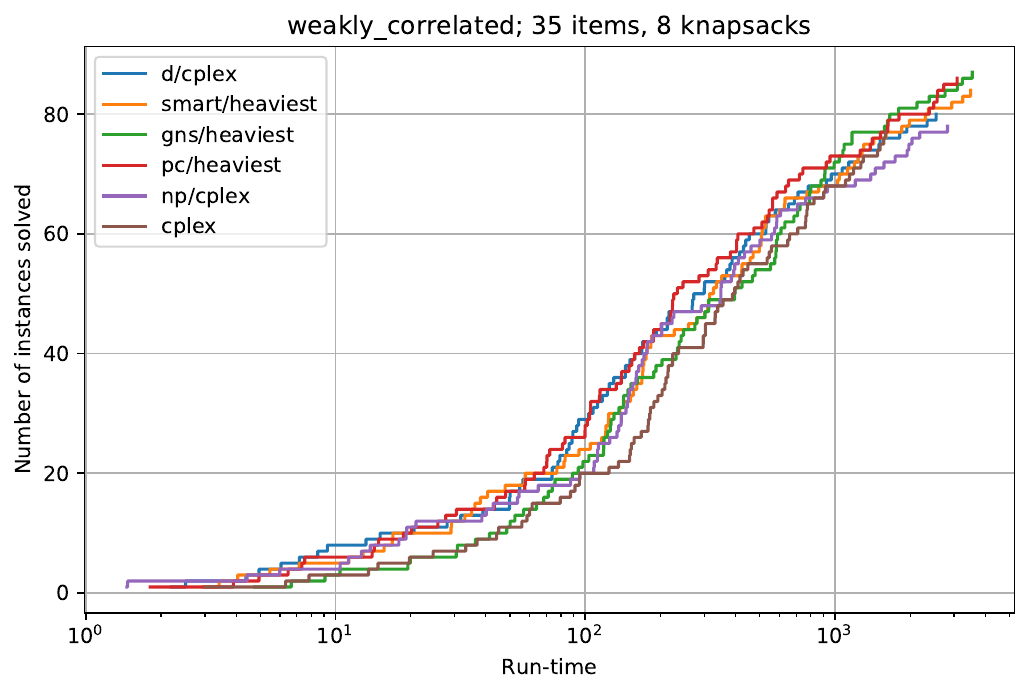}
\end{subfigure}
\begin{subfigure}
  \centering
  \includegraphics[width=.39\linewidth]{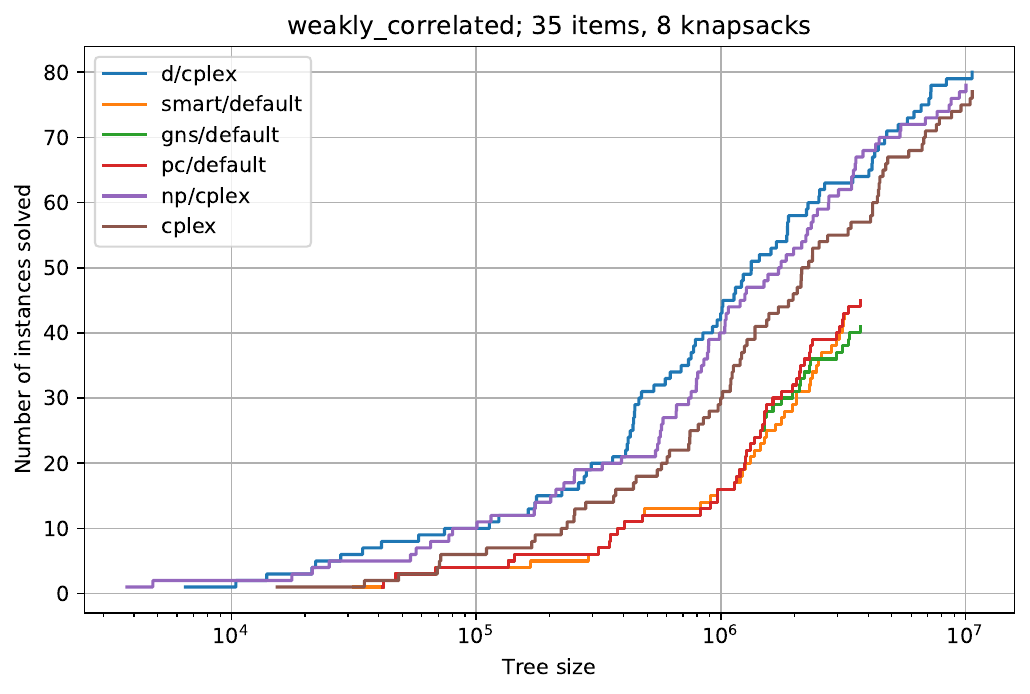}
  \label{fig:blah}
\end{subfigure}
\begin{subfigure}
  \centering
  \includegraphics[width=.39\linewidth]{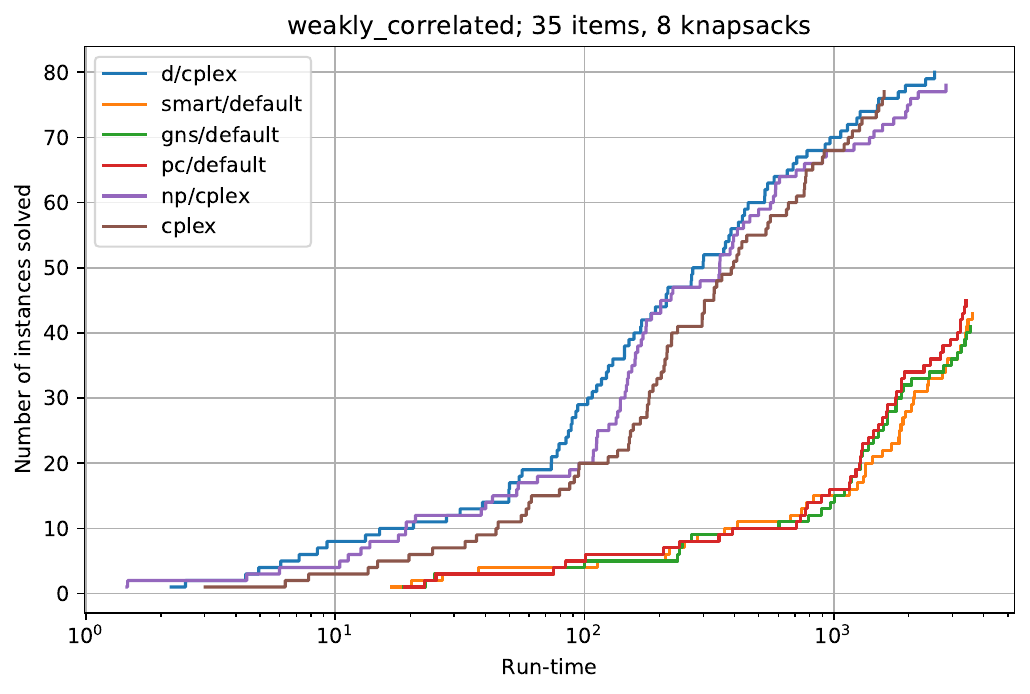}
\end{subfigure}
\begin{subfigure}
  \centering
  \includegraphics[width=.39\linewidth]{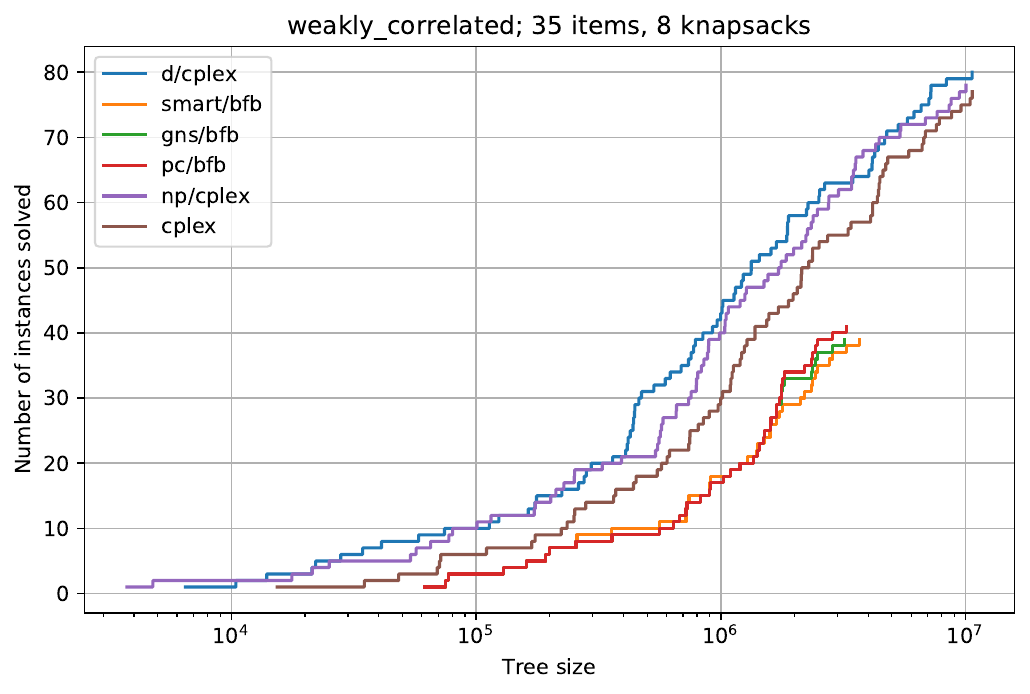}
  \label{fig:blah}
\end{subfigure}
\begin{subfigure}
  \centering
  \includegraphics[width=.39\linewidth]{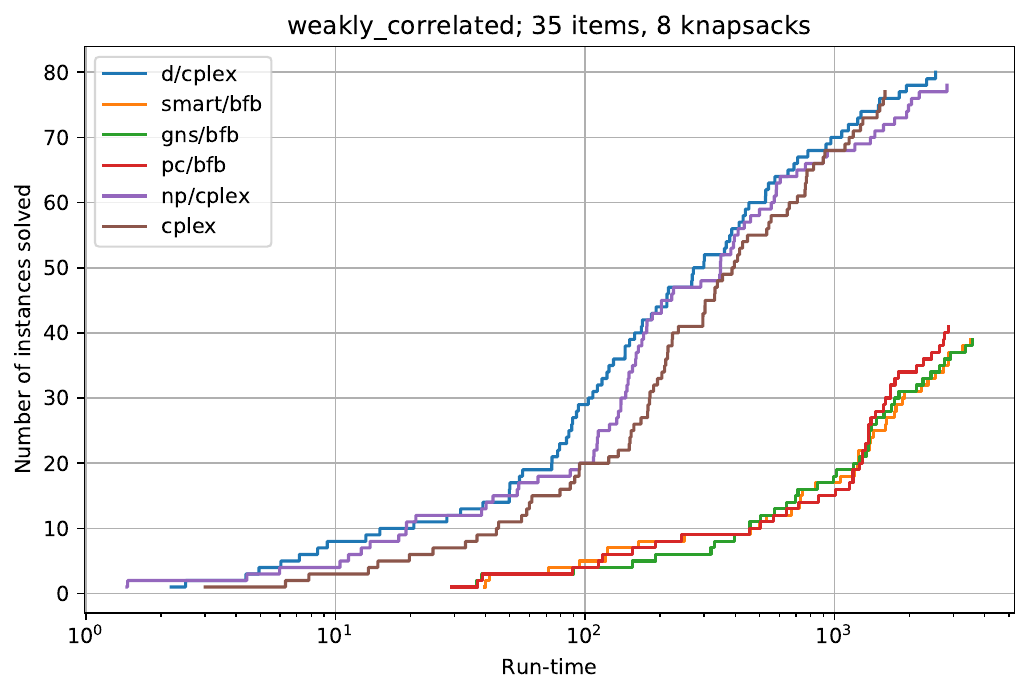}
\end{subfigure}

\caption{Weakly correlated, CPLEX cover cuts on, all other parameters off}
\label{fig:mkp_weakly_correlated}
\end{figure}

\begin{figure}[t]
\centering
\begin{subfigure}
  \centering
  \includegraphics[width=.39\linewidth]{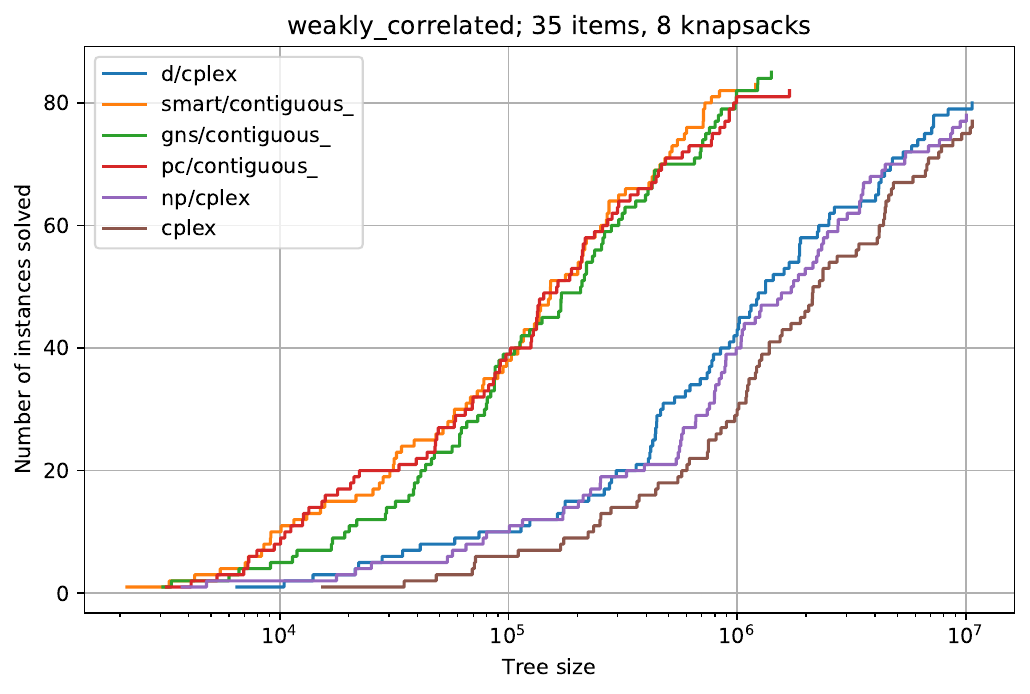}
  \label{fig:blah}
\end{subfigure}
\begin{subfigure}
  \centering
  \includegraphics[width=.39\linewidth]{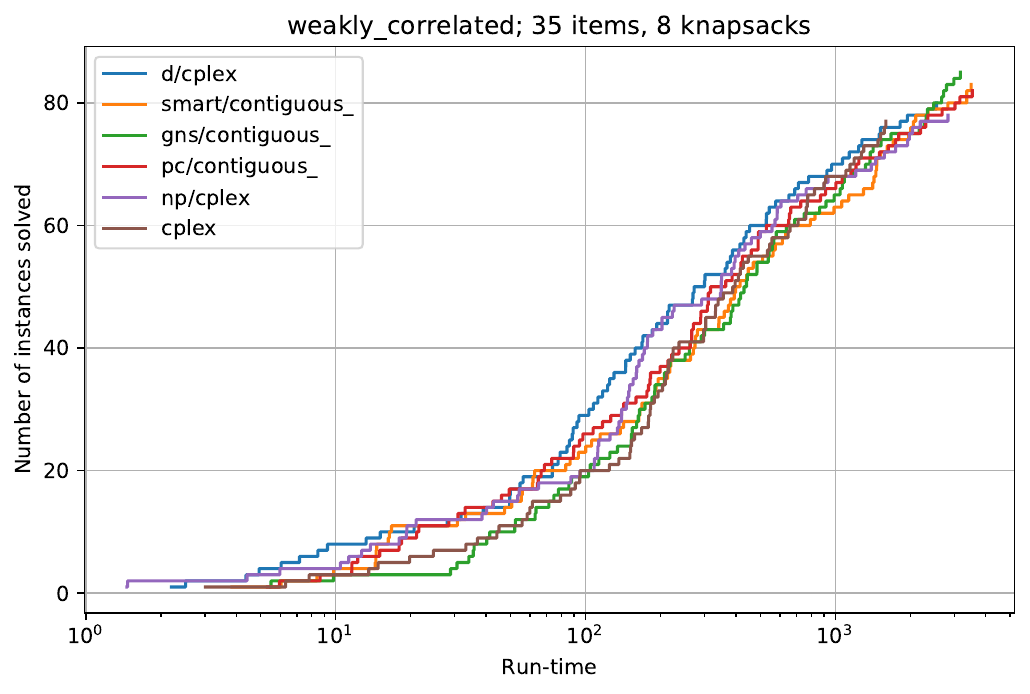}
\end{subfigure}
\begin{subfigure}
  \centering
  \includegraphics[width=.39\linewidth]{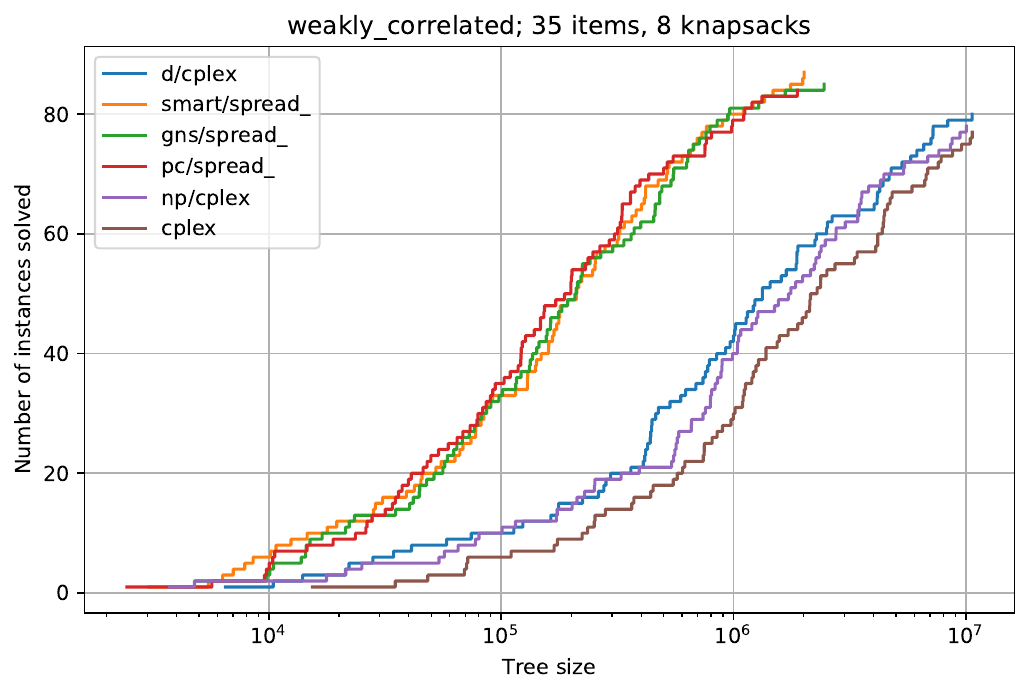}
  \label{fig:blah}
\end{subfigure}
\begin{subfigure}
  \centering
  \includegraphics[width=.39\linewidth]{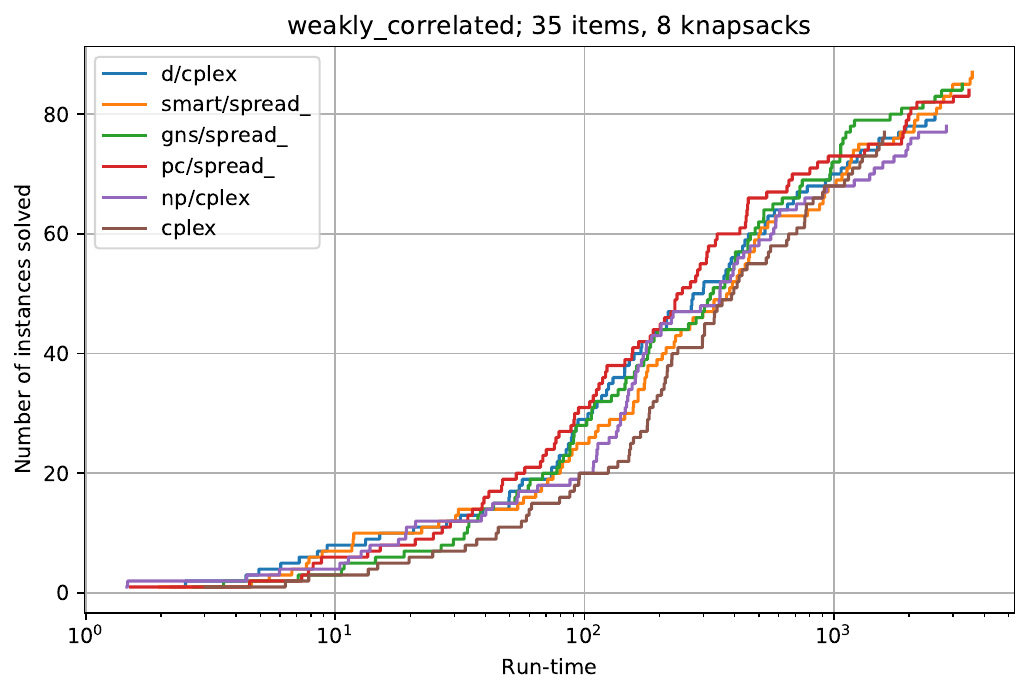}
\end{subfigure}
\begin{subfigure}
  \centering
  \includegraphics[width=.39\linewidth]{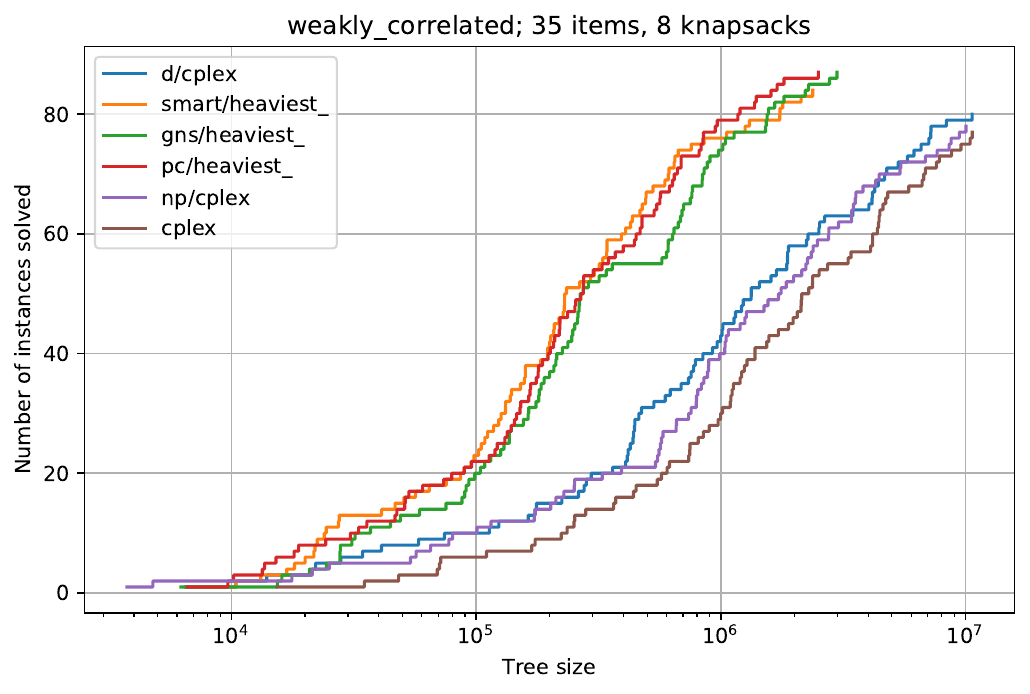}
  \label{fig:blah}
\end{subfigure}
\begin{subfigure}
  \centering
  \includegraphics[width=.39\linewidth]{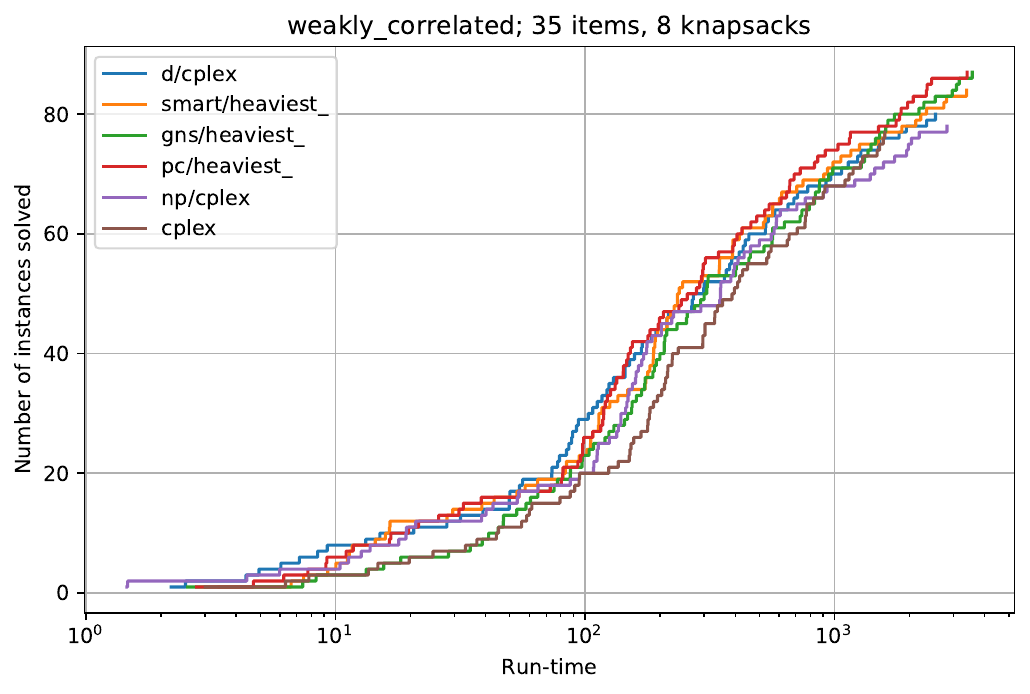}
\end{subfigure}
\begin{subfigure}
  \centering
  \includegraphics[width=.39\linewidth]{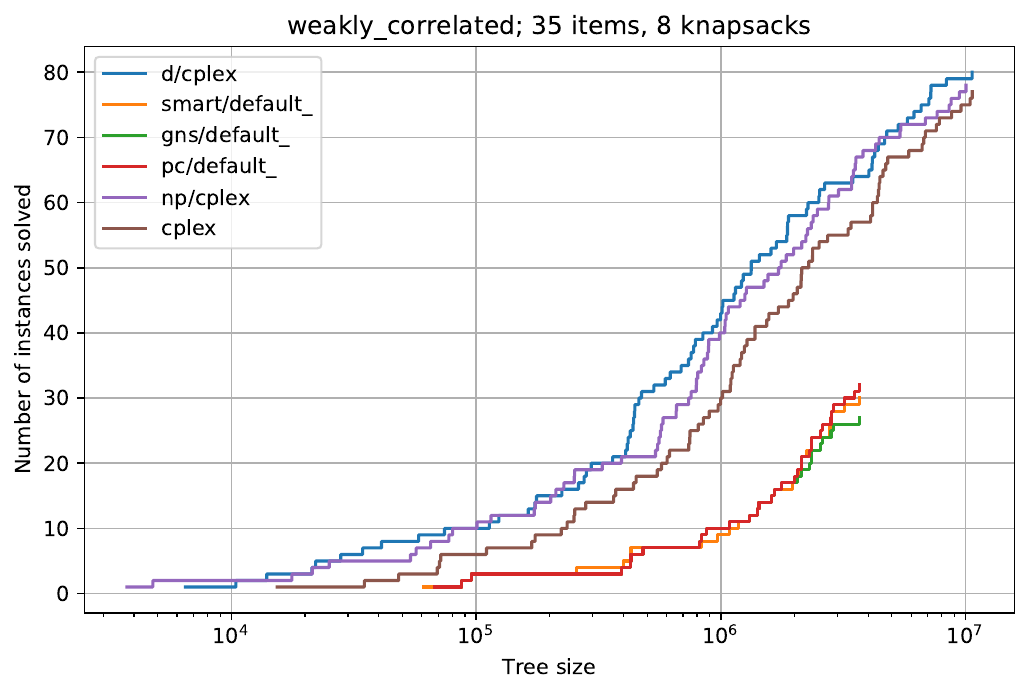}
  \label{fig:blah}
\end{subfigure}
\begin{subfigure}
  \centering
  \includegraphics[width=.39\linewidth]{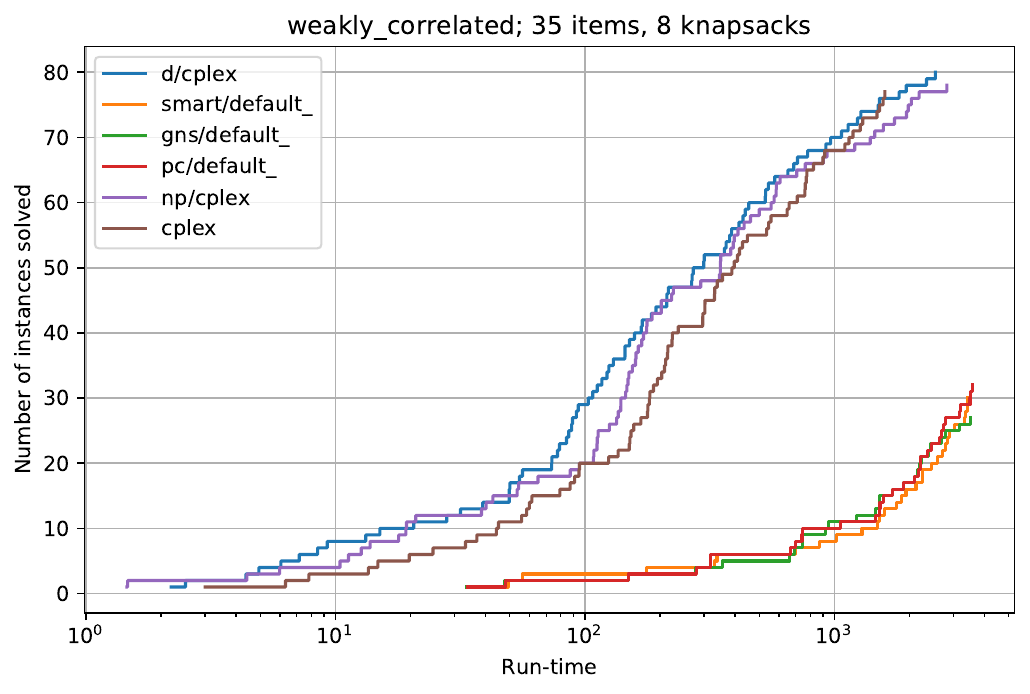}
\end{subfigure}
\begin{subfigure}
  \centering
  \includegraphics[width=.39\linewidth]{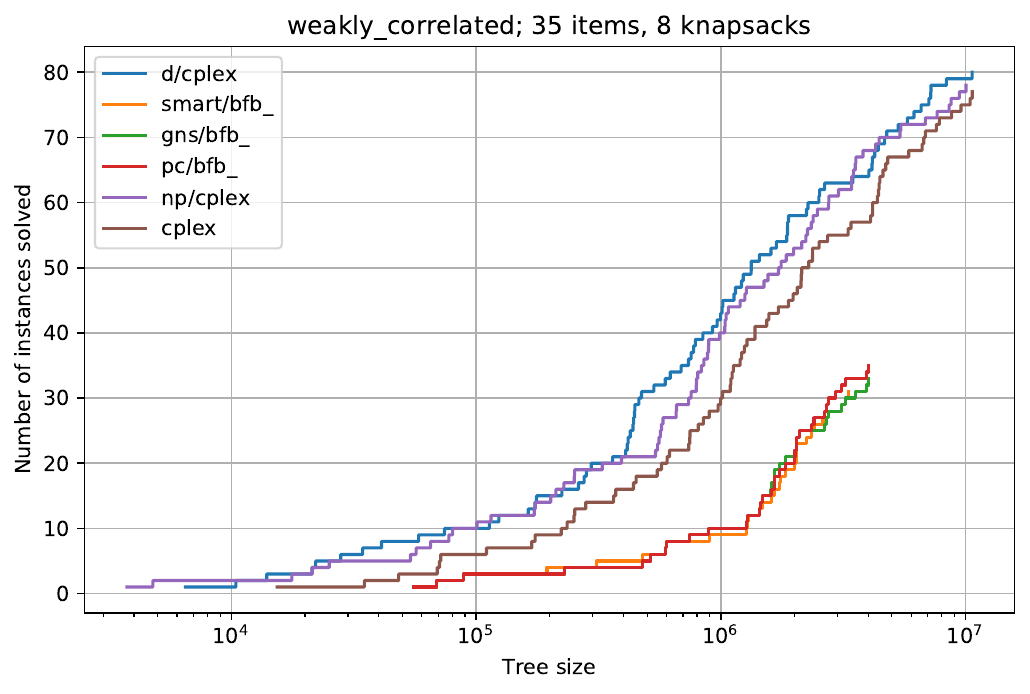}
  \label{fig:blah}
\end{subfigure}
\begin{subfigure}
  \centering
  \includegraphics[width=.39\linewidth]{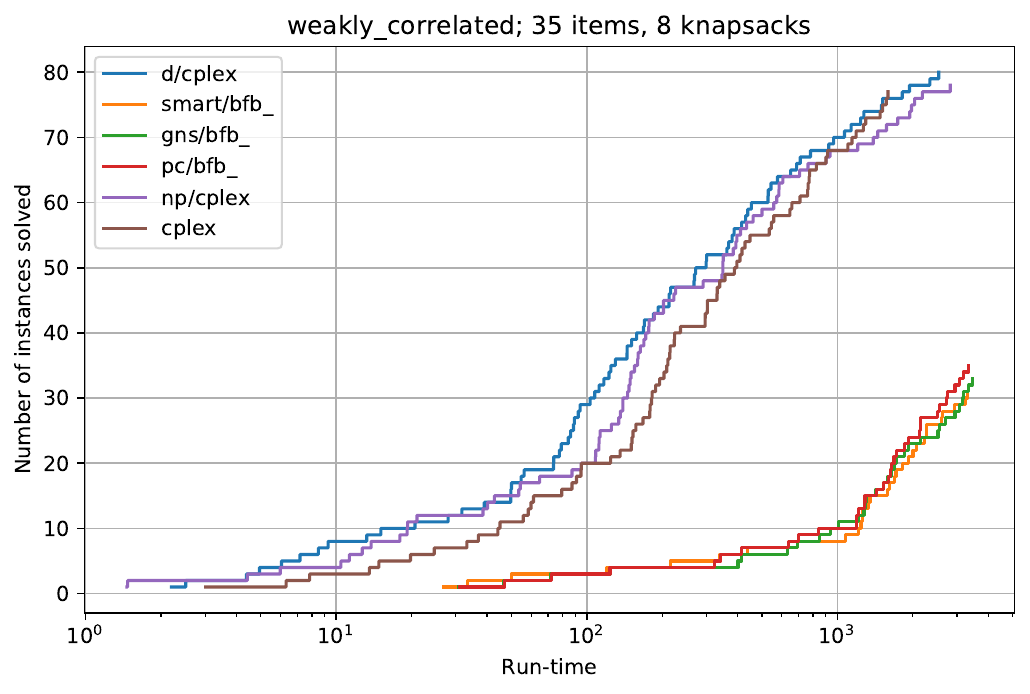}
\end{subfigure}

\caption{Weakly correlated, CPLEX cover cuts off, all other parameters off}
\label{fig:mkp_weakly_correlated_}
\end{figure}

\begin{figure}[t]
\centering
\begin{subfigure}
  \centering
  \includegraphics[width=.39\linewidth]{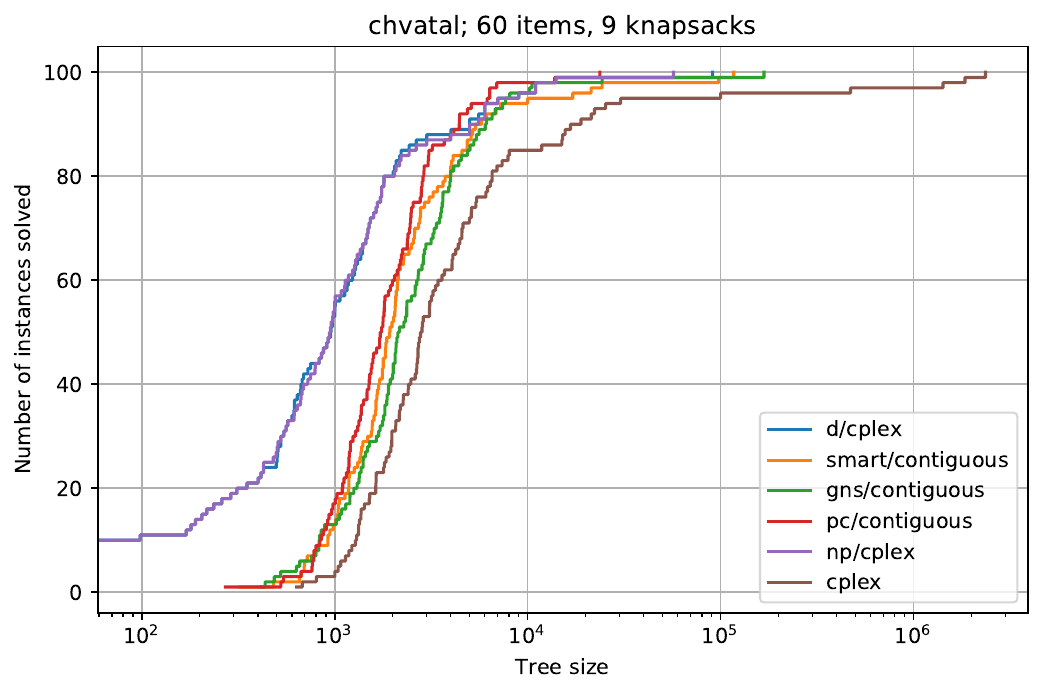}
  \label{fig:blah}
\end{subfigure}
\begin{subfigure}
  \centering
  \includegraphics[width=.39\linewidth]{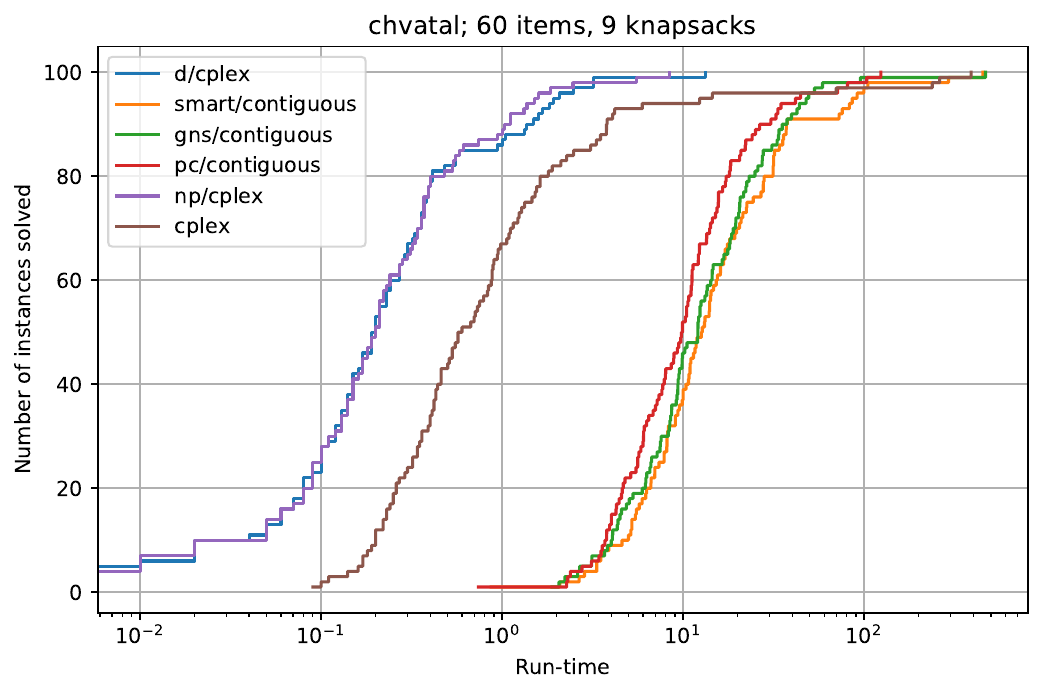}
\end{subfigure}
\begin{subfigure}
  \centering
  \includegraphics[width=.39\linewidth]{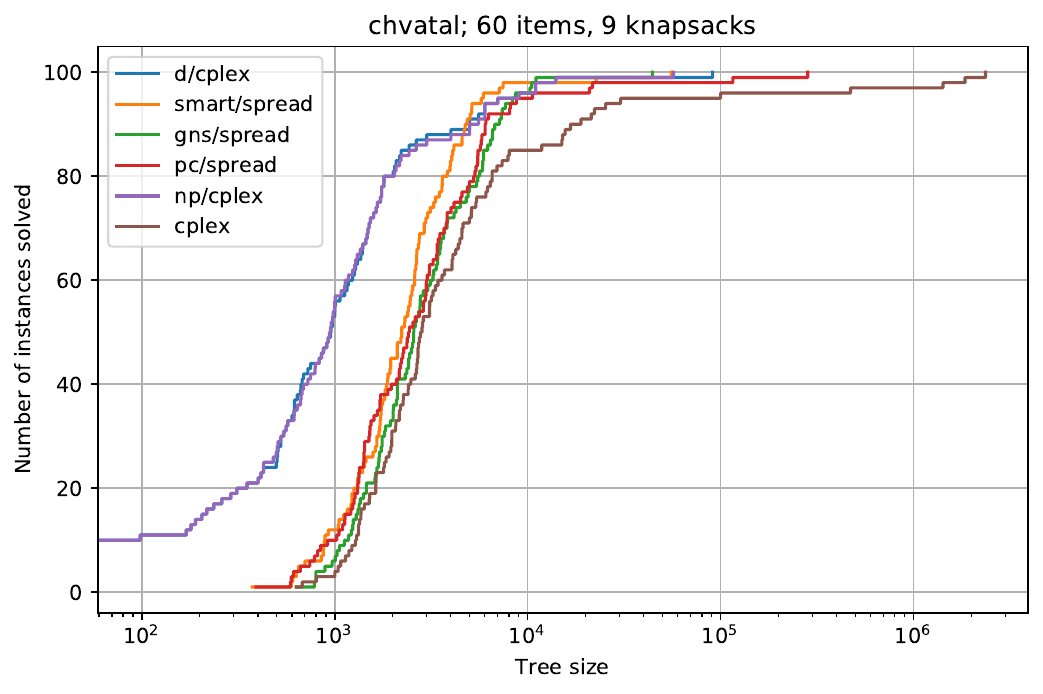}
  \label{fig:blah}
\end{subfigure}
\begin{subfigure}
  \centering
  \includegraphics[width=.39\linewidth]{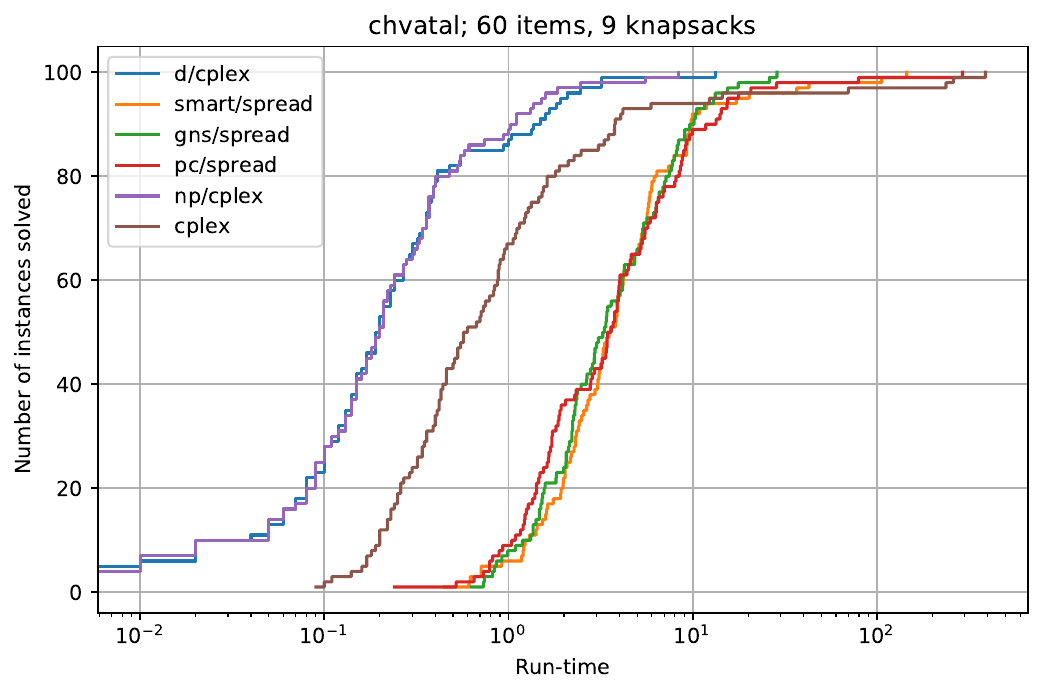}
\end{subfigure}
\begin{subfigure}
  \centering
  \includegraphics[width=.39\linewidth]{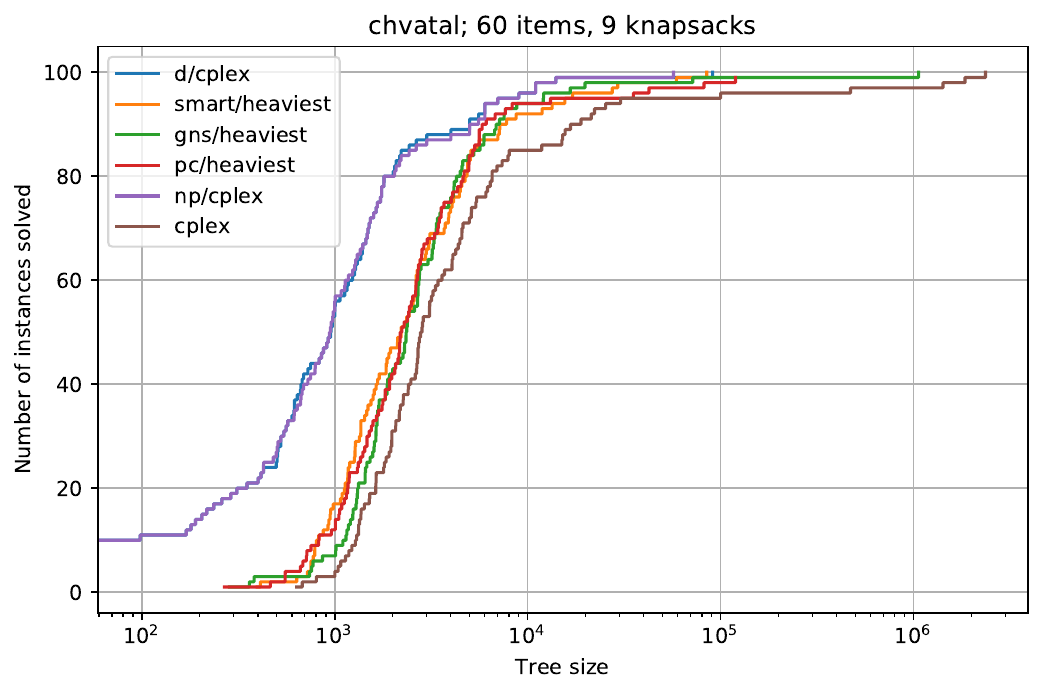}
  \label{fig:blah}
\end{subfigure}
\begin{subfigure}
  \centering
  \includegraphics[width=.39\linewidth]{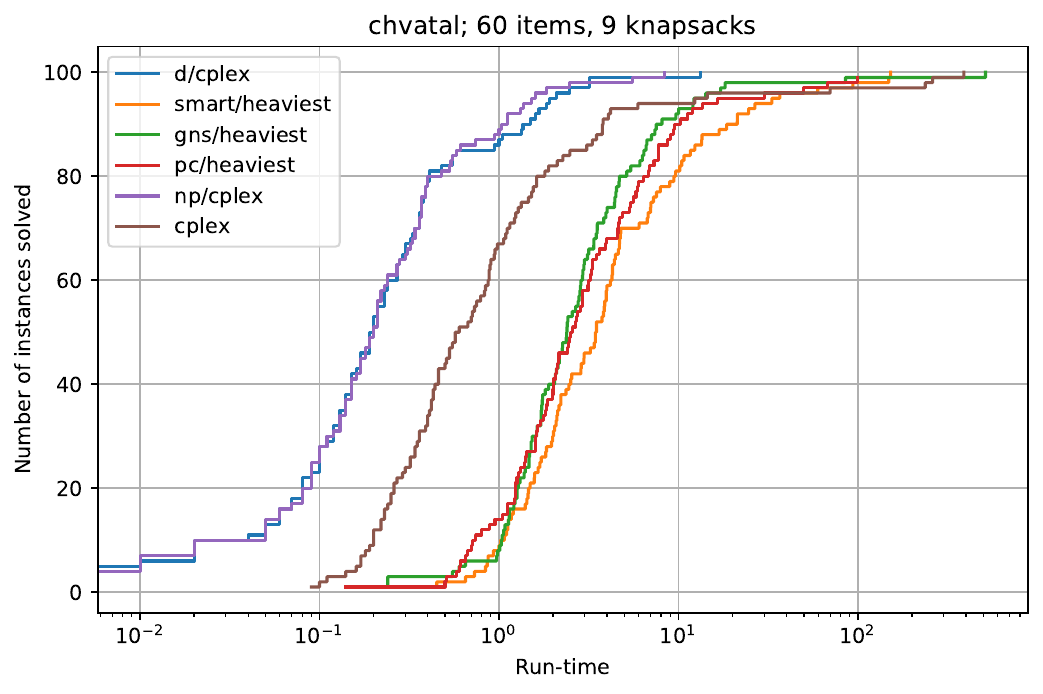}
\end{subfigure}
\begin{subfigure}
  \centering
  \includegraphics[width=.39\linewidth]{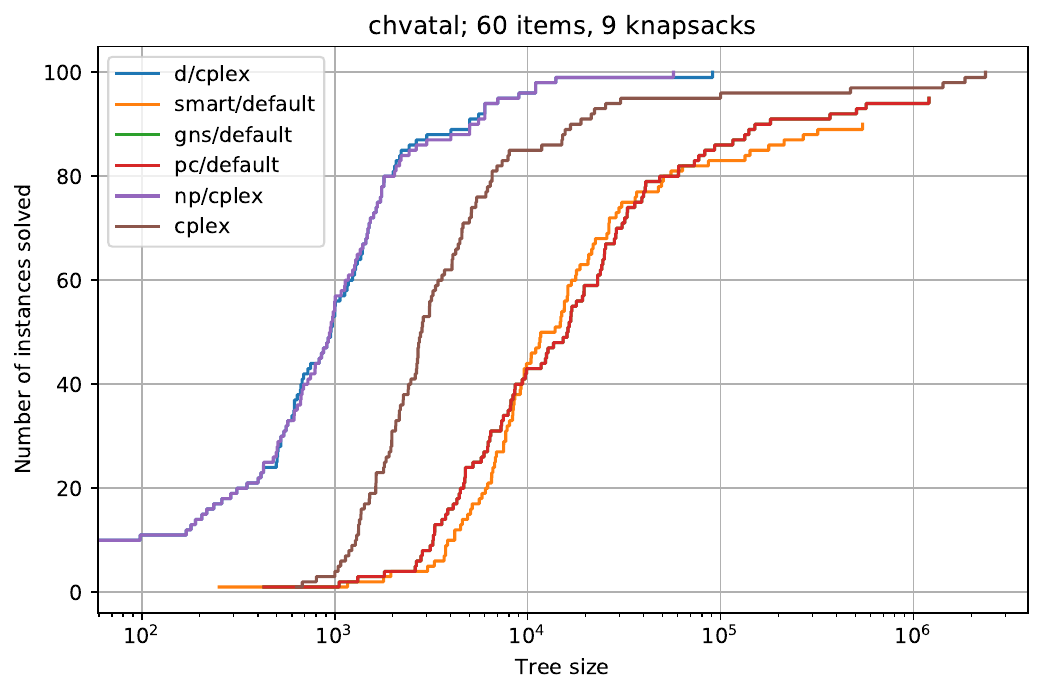}
  \label{fig:blah}
\end{subfigure}
\begin{subfigure}
  \centering
  \includegraphics[width=.39\linewidth]{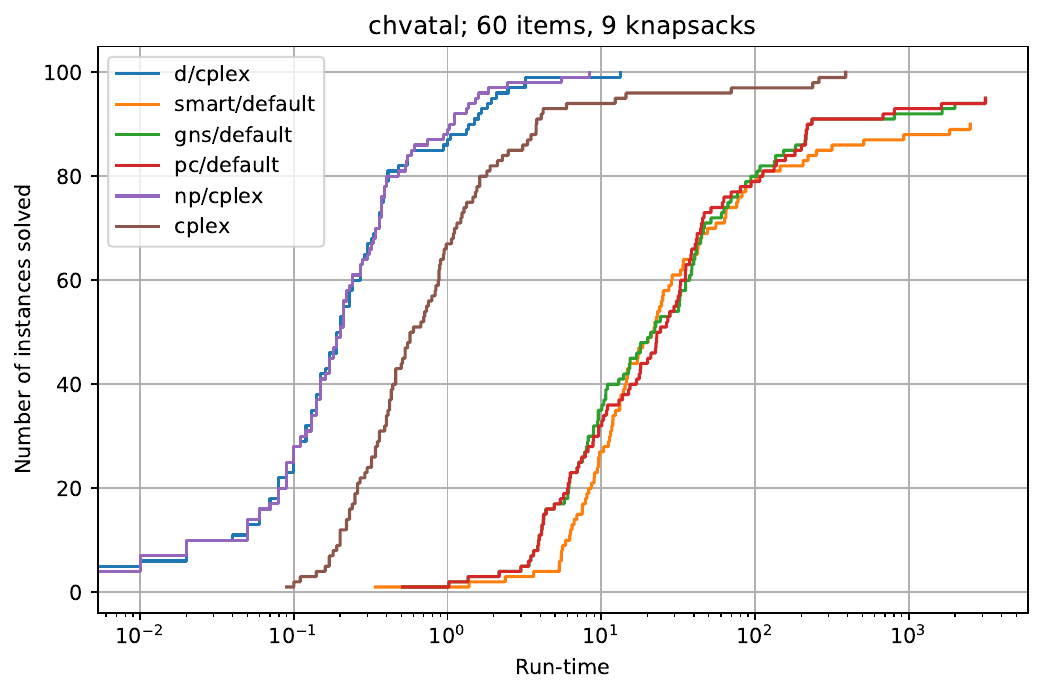}
\end{subfigure}
\begin{subfigure}
  \centering
  \includegraphics[width=.39\linewidth]{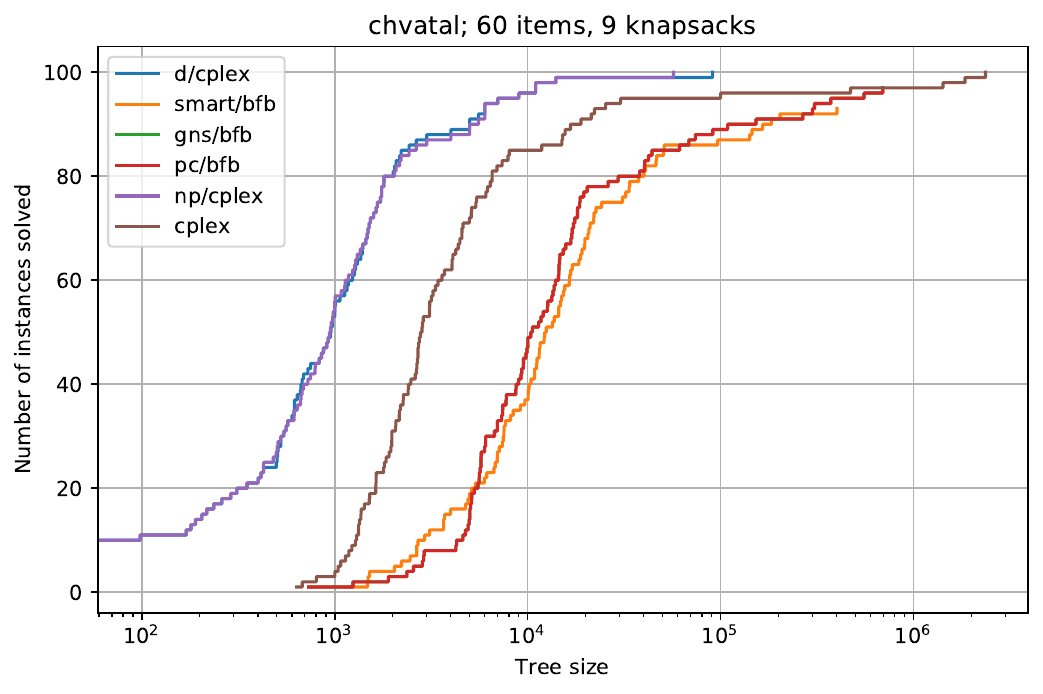}
  \label{fig:blah}
\end{subfigure}
\begin{subfigure}
  \centering
  \includegraphics[width=.39\linewidth]{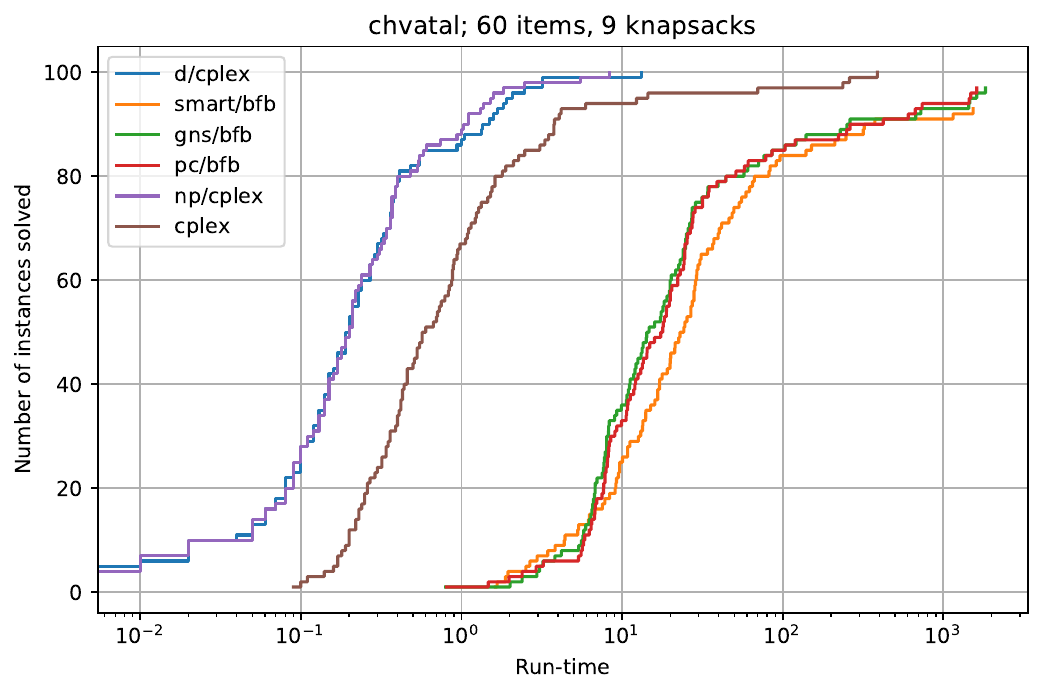}
\end{subfigure}

\caption{Chv\'{a}tal, CPLEX cover cuts on, all other parameters off}
\label{fig:chvatal}
\end{figure}

\begin{figure}[t]
\centering
\begin{subfigure}
  \centering
  \includegraphics[width=.39\linewidth]{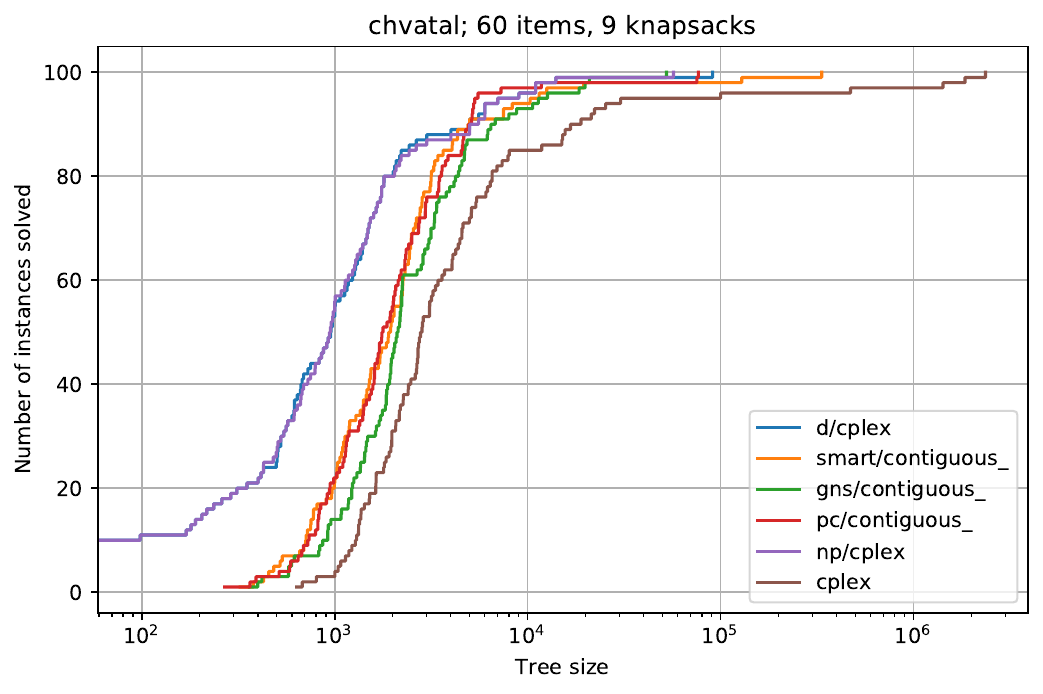}
  \label{fig:blah}
\end{subfigure}
\begin{subfigure}
  \centering
  \includegraphics[width=.39\linewidth]{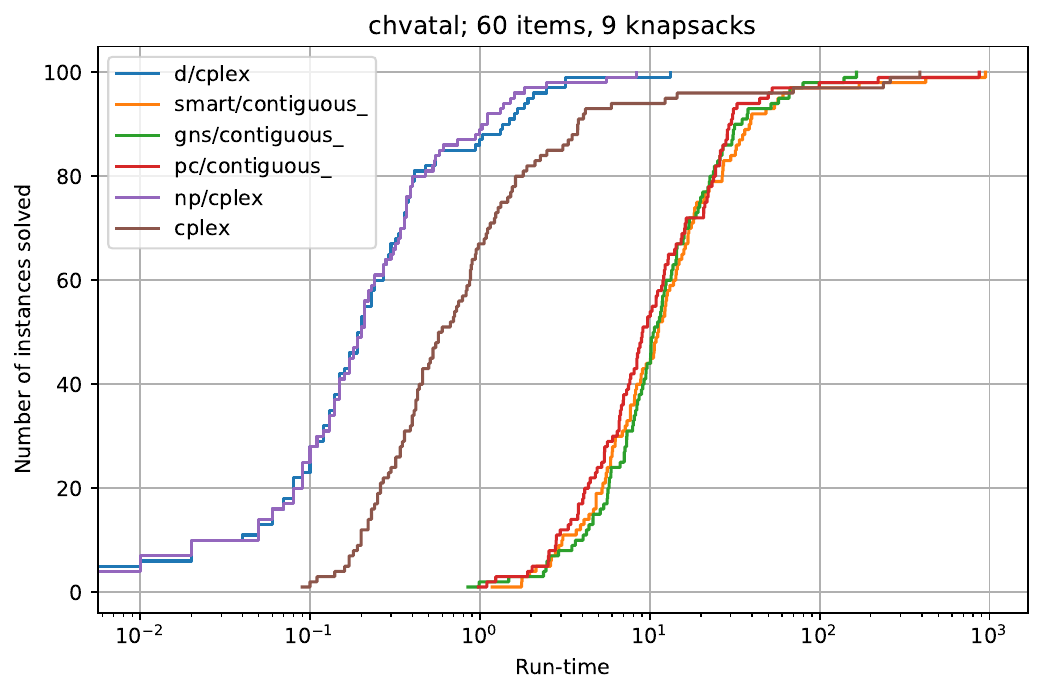}
\end{subfigure}
\begin{subfigure}
  \centering
  \includegraphics[width=.39\linewidth]{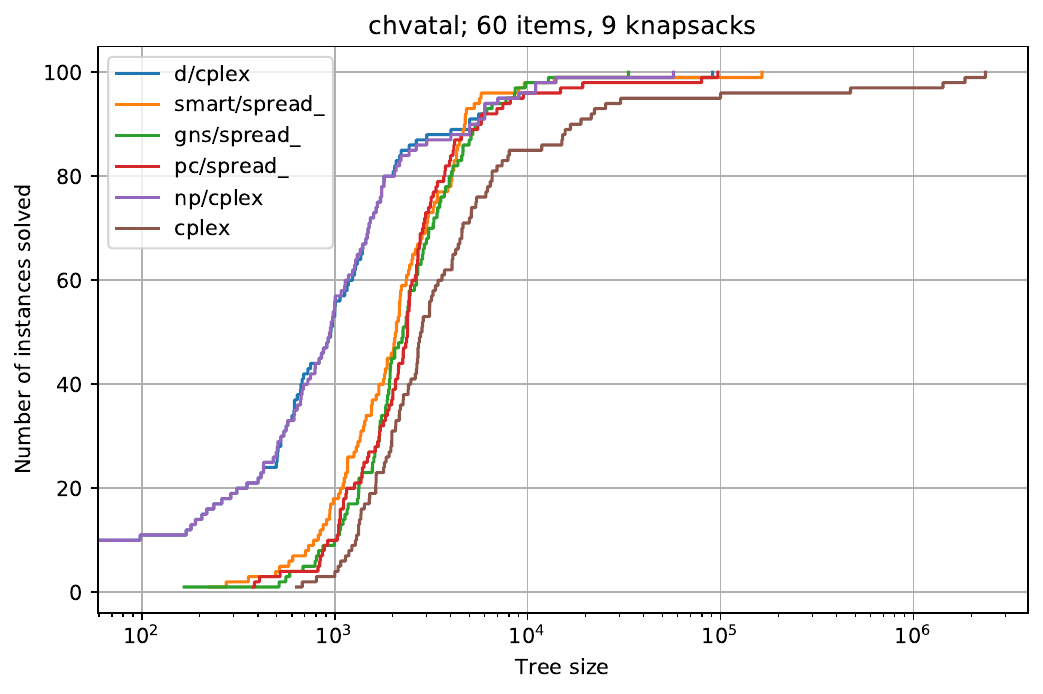}
  \label{fig:blah}
\end{subfigure}
\begin{subfigure}
  \centering
  \includegraphics[width=.39\linewidth]{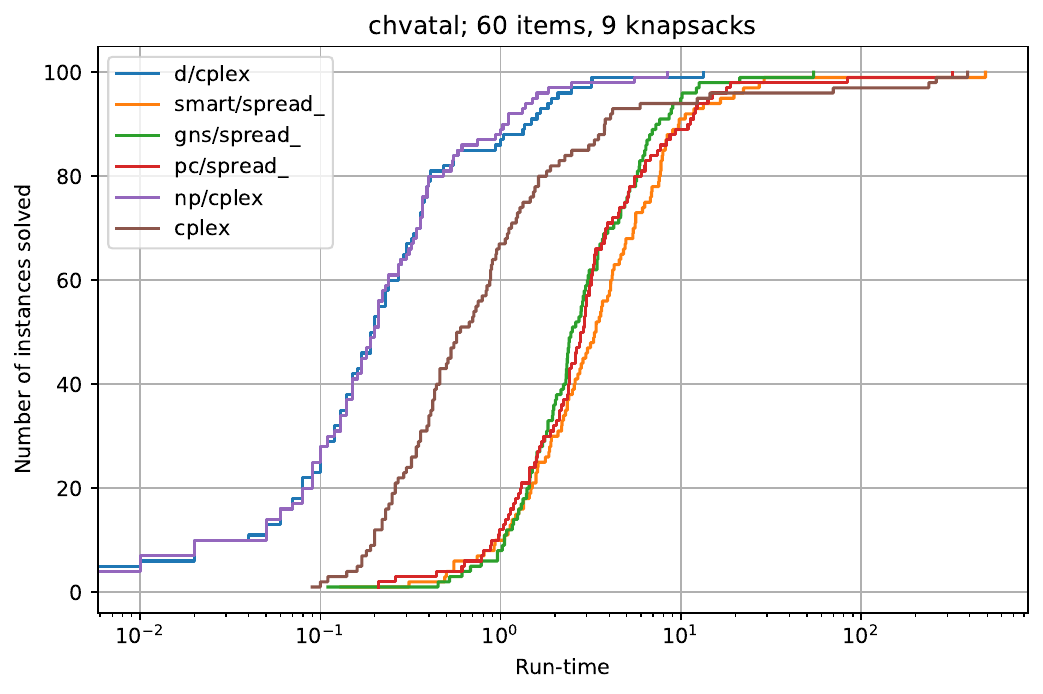}
\end{subfigure}
\begin{subfigure}
  \centering
  \includegraphics[width=.39\linewidth]{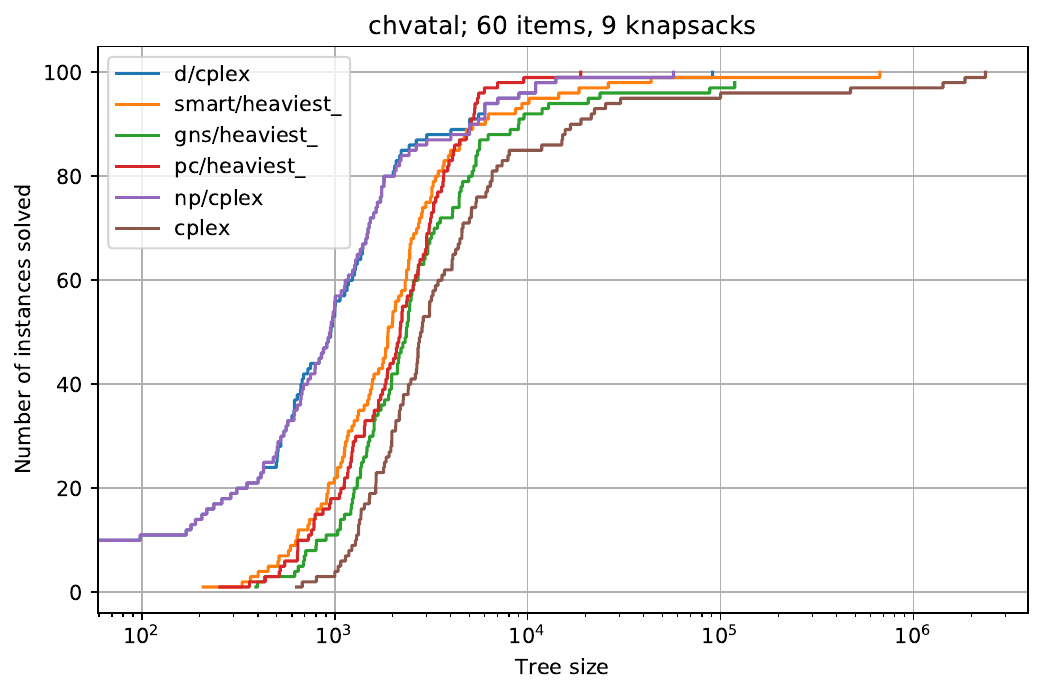}
  \label{fig:blah}
\end{subfigure}
\begin{subfigure}
  \centering
  \includegraphics[width=.39\linewidth]{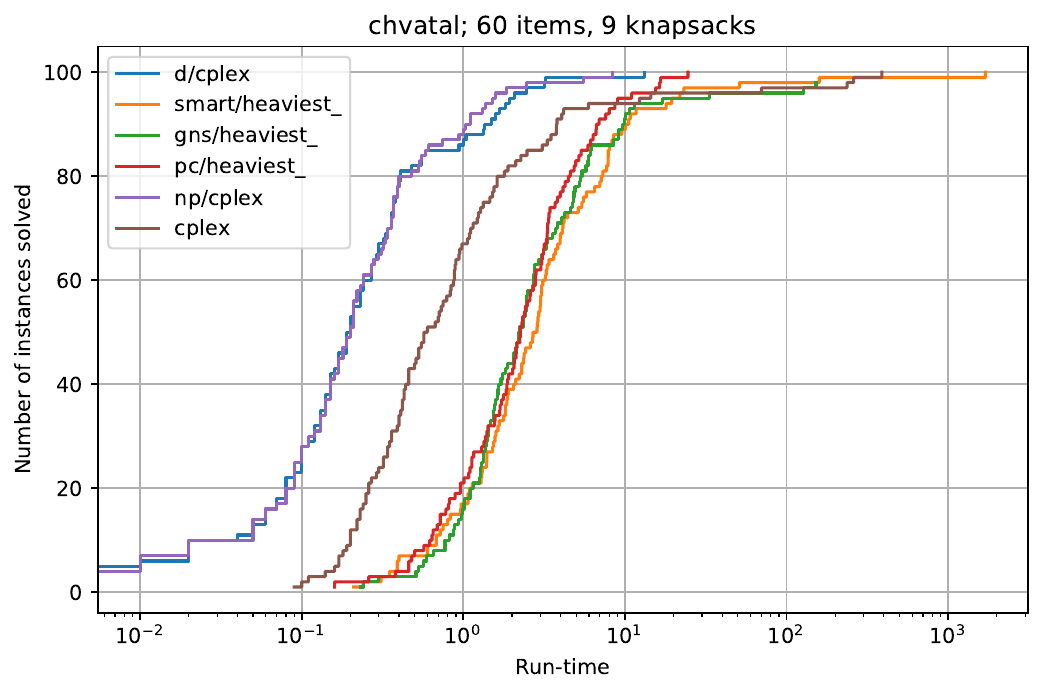}
\end{subfigure}
\begin{subfigure}
  \centering
  \includegraphics[width=.39\linewidth]{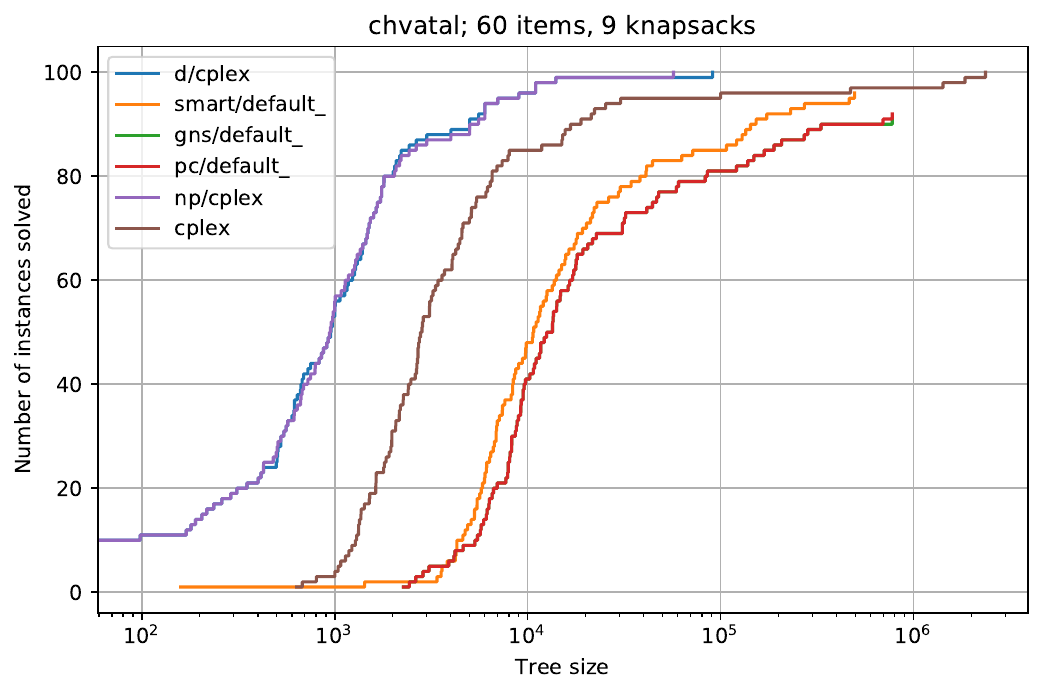}
  \label{fig:blah}
\end{subfigure}
\begin{subfigure}
  \centering
  \includegraphics[width=.39\linewidth]{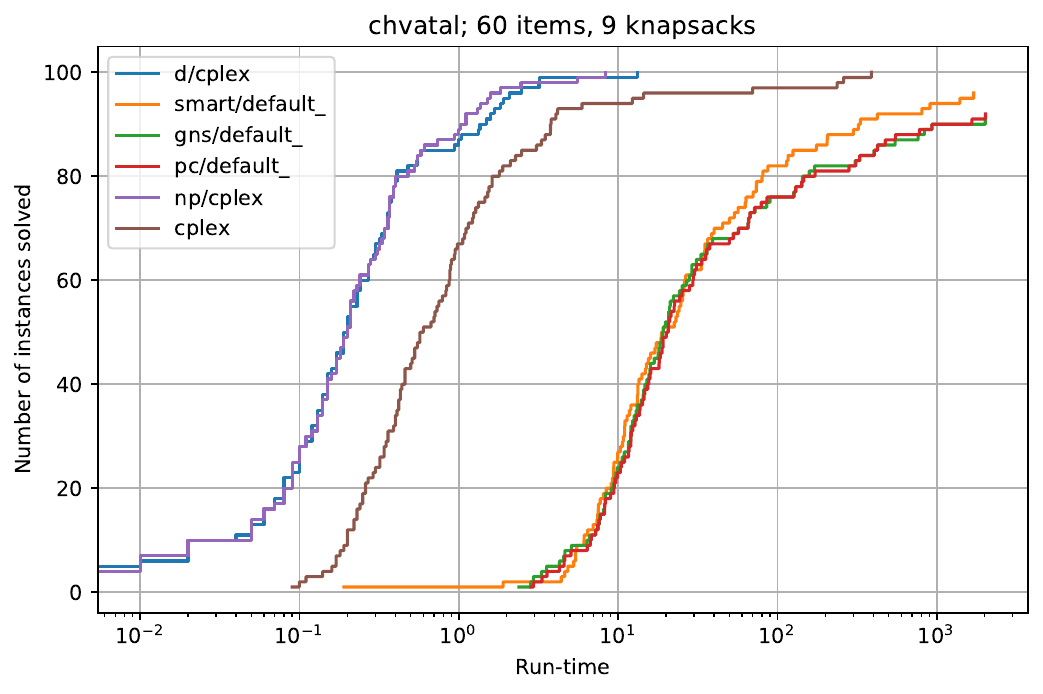}
\end{subfigure}
\begin{subfigure}
  \centering
  \includegraphics[width=.39\linewidth]{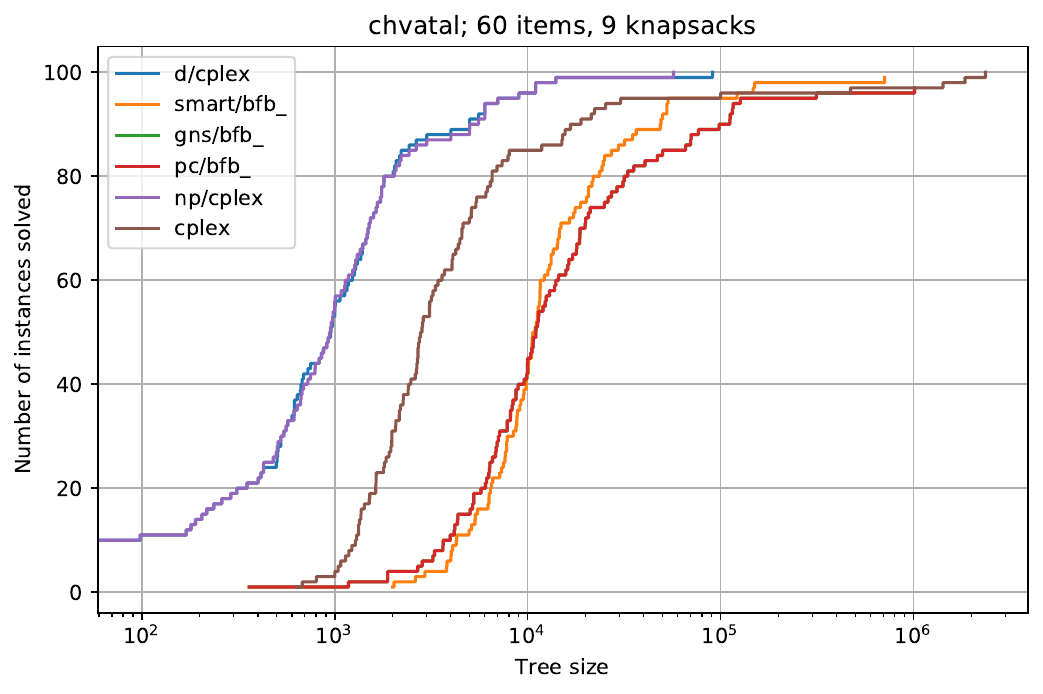}
  \label{fig:blah}
\end{subfigure}
\begin{subfigure}
  \centering
  \includegraphics[width=.39\linewidth]{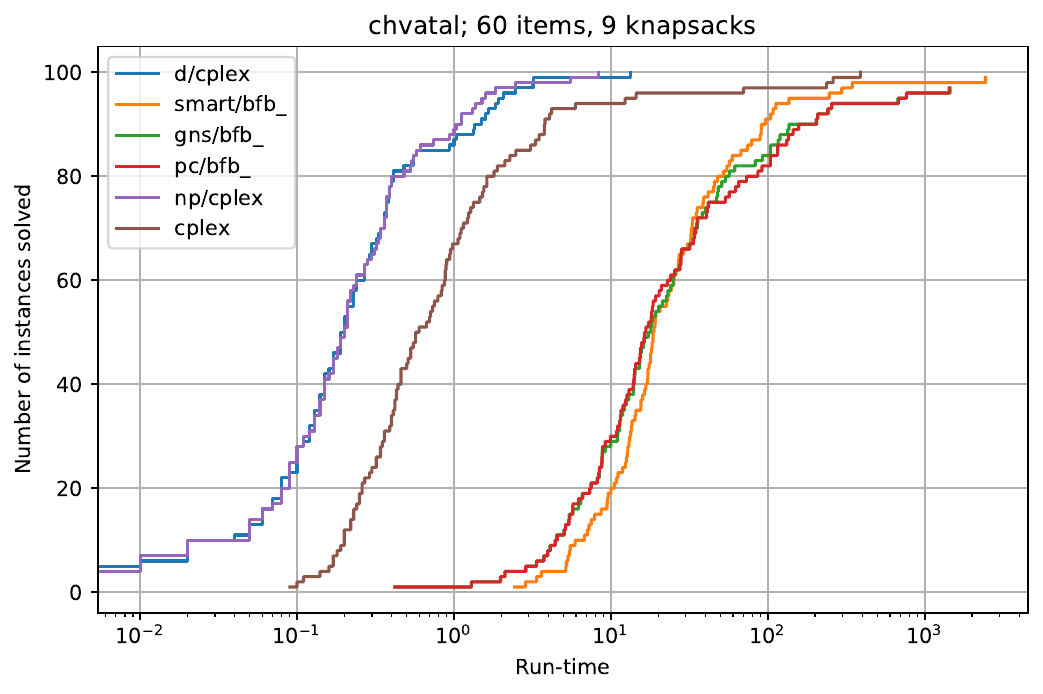}
\end{subfigure}

\caption{Chv\'{a}tal, CPLEX cover cuts off, all other parameters off}
\label{fig:chvatal_}
\end{figure}

\begin{figure}[t]
\centering
\begin{subfigure}
  \centering
  \includegraphics[width=.39\linewidth]{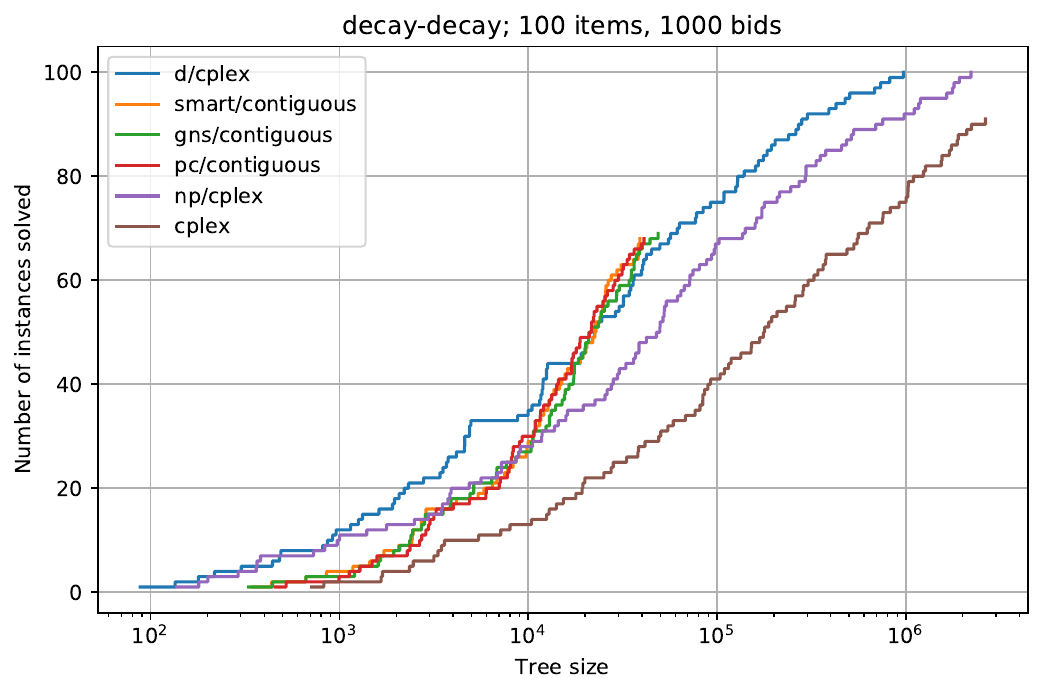}
  \label{fig:blah}
\end{subfigure}
\begin{subfigure}
  \centering
  \includegraphics[width=.39\linewidth]{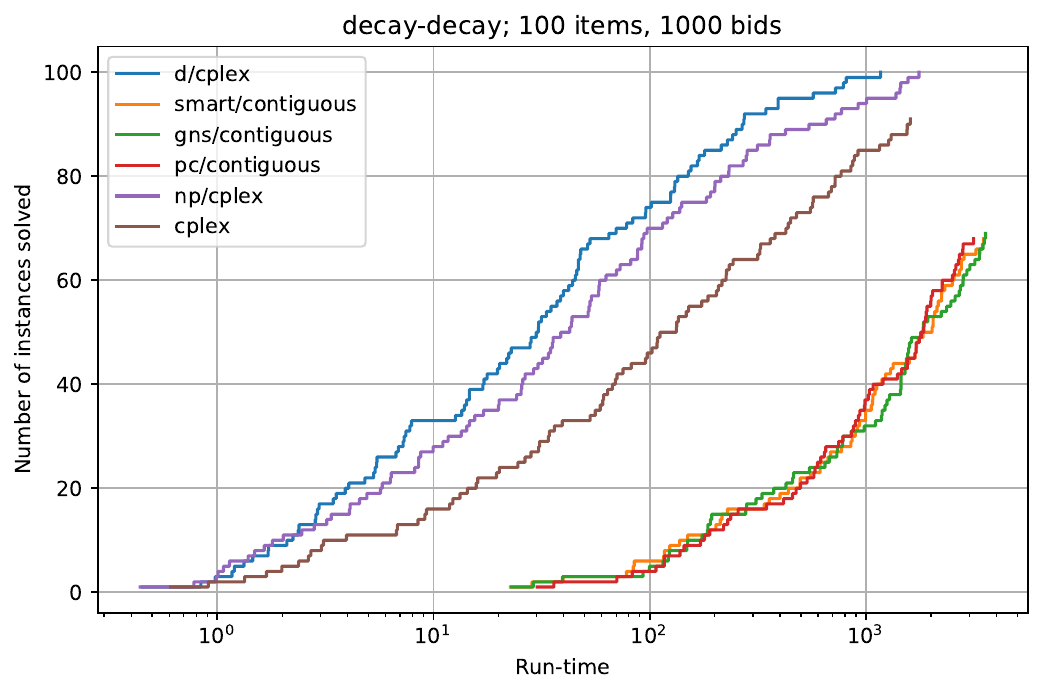}
\end{subfigure}
\begin{subfigure}
  \centering
  \includegraphics[width=.39\linewidth]{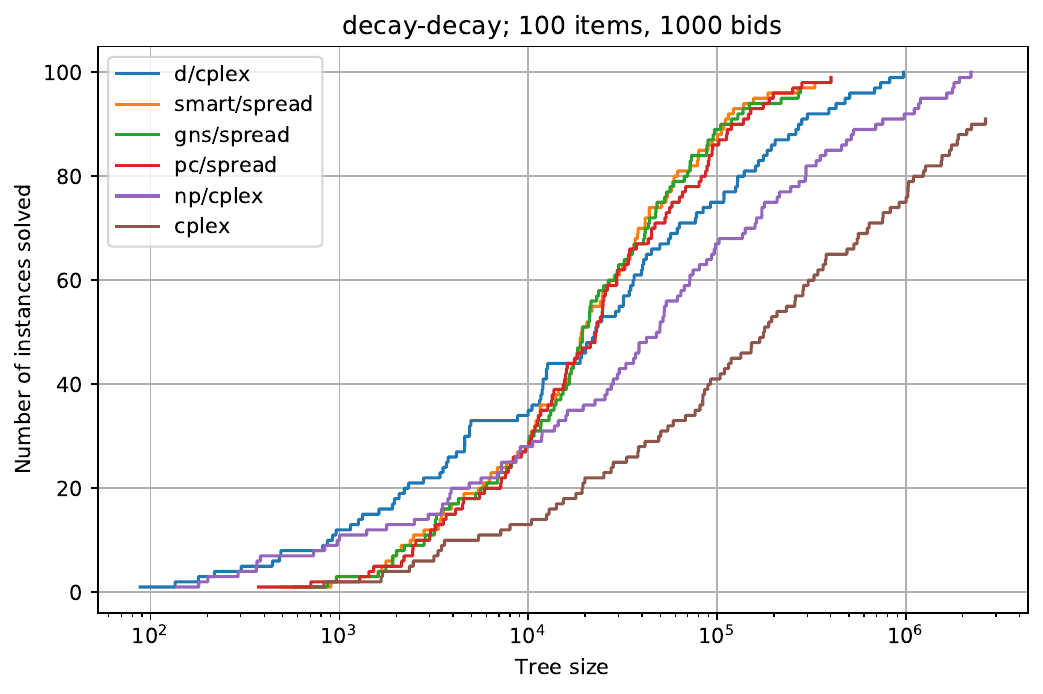}
  \label{fig:blah}
\end{subfigure}
\begin{subfigure}
  \centering
  \includegraphics[width=.39\linewidth]{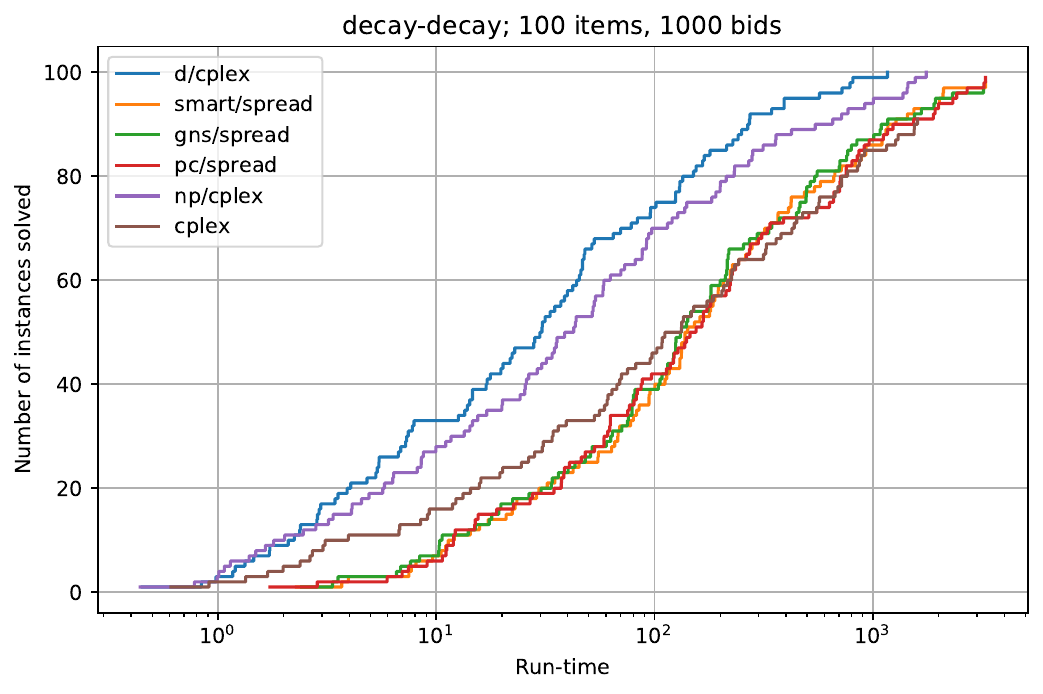}
\end{subfigure}
\begin{subfigure}
  \centering
  \includegraphics[width=.39\linewidth]{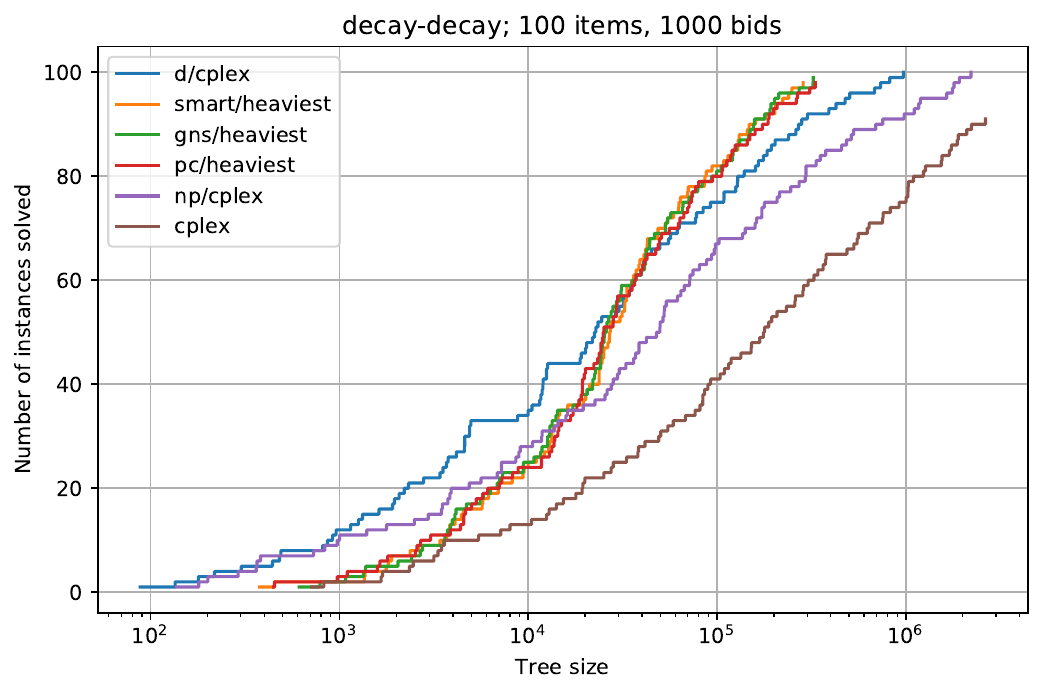}
  \label{fig:blah}
\end{subfigure}
\begin{subfigure}
  \centering
  \includegraphics[width=.39\linewidth]{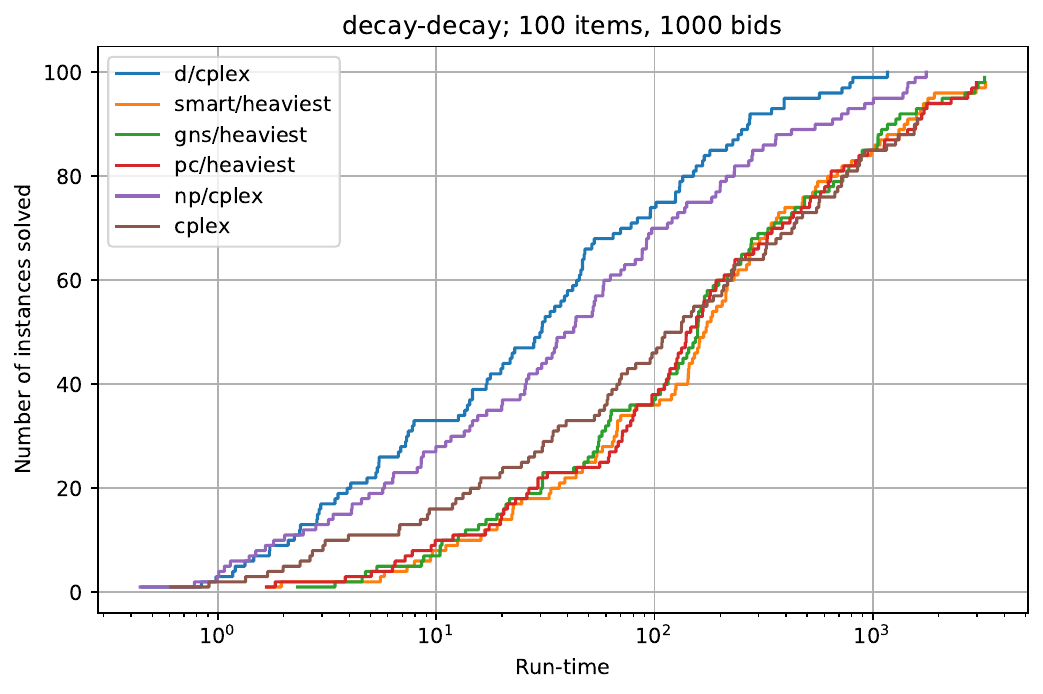}
\end{subfigure}
\begin{subfigure}
  \centering
  \includegraphics[width=.39\linewidth]{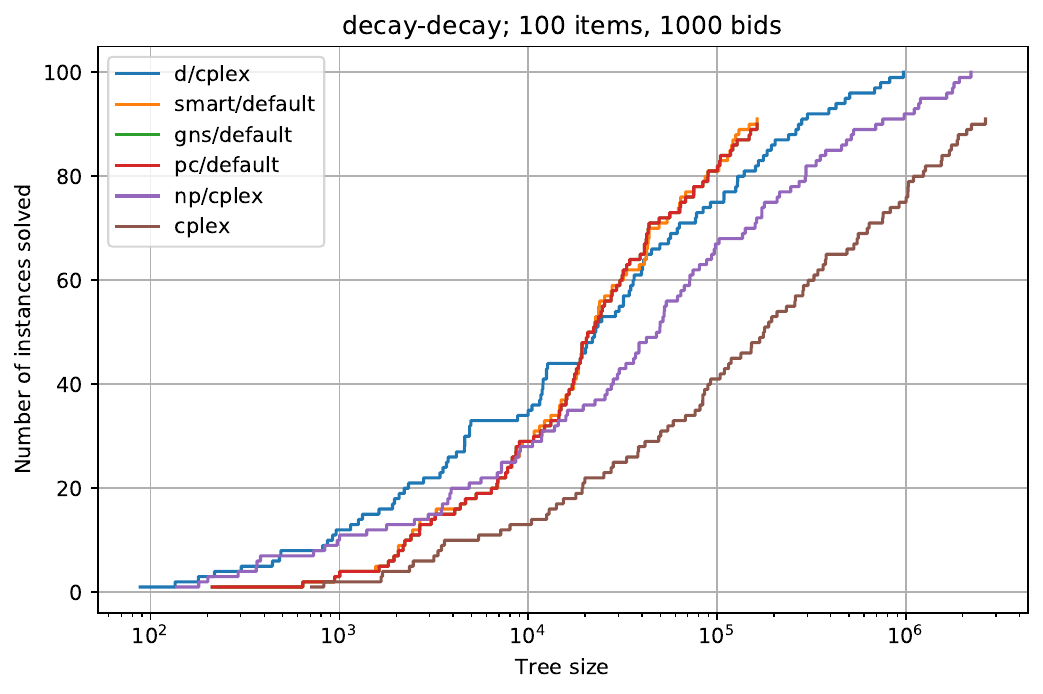}
  \label{fig:blah}
\end{subfigure}
\begin{subfigure}
  \centering
  \includegraphics[width=.39\linewidth]{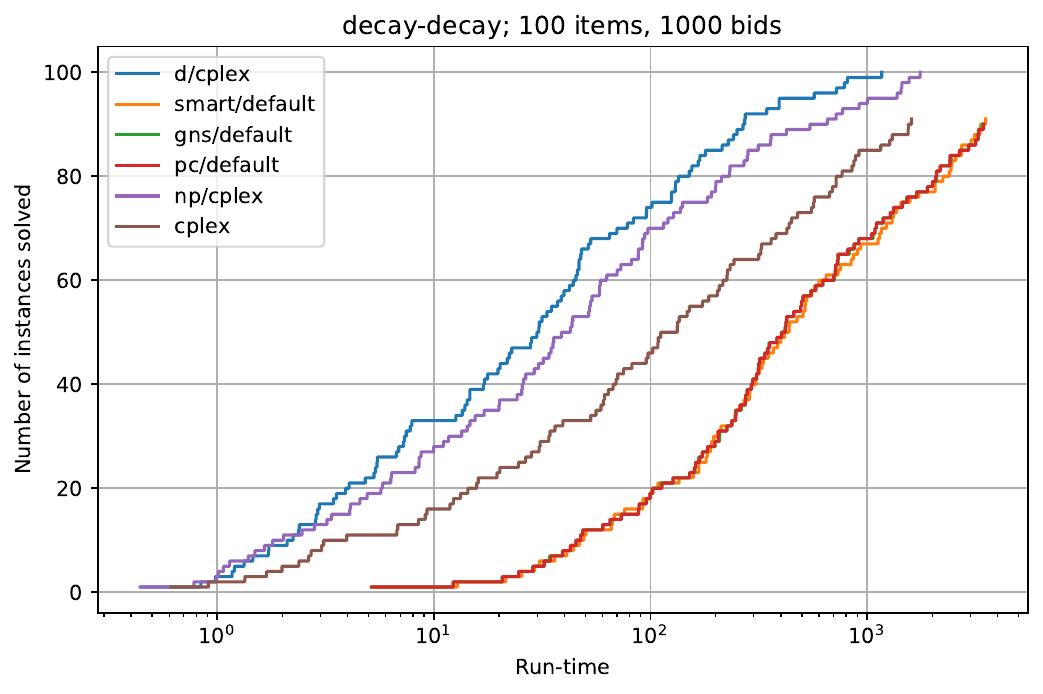}
\end{subfigure}
\begin{subfigure}
  \centering
  \includegraphics[width=.39\linewidth]{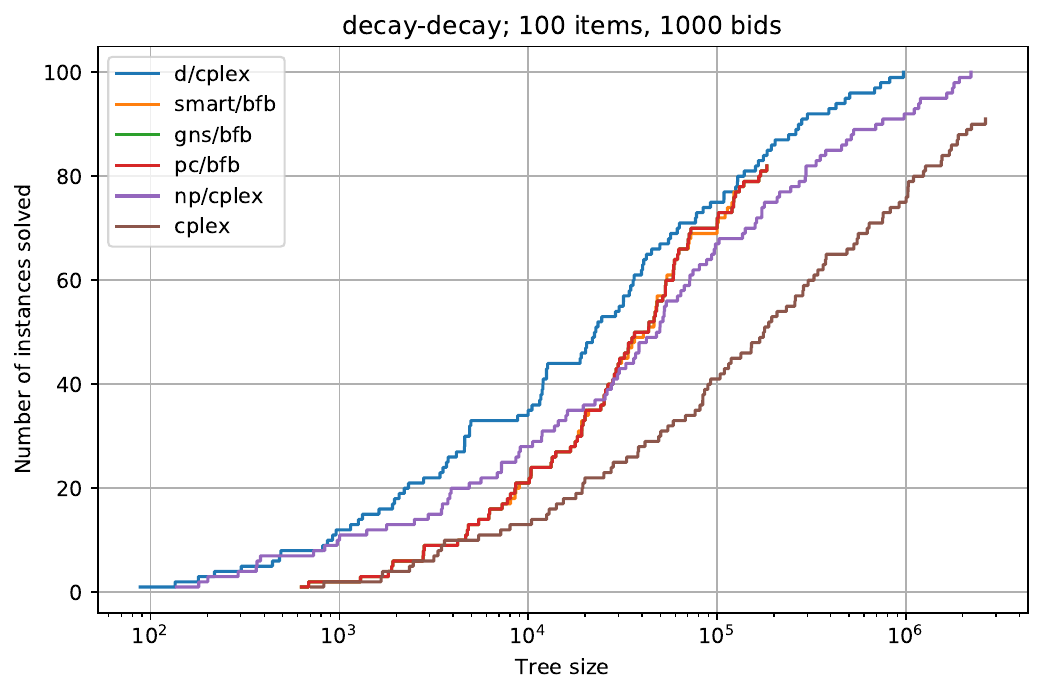}
  \label{fig:blah}
\end{subfigure}
\begin{subfigure}
  \centering
  \includegraphics[width=.39\linewidth]{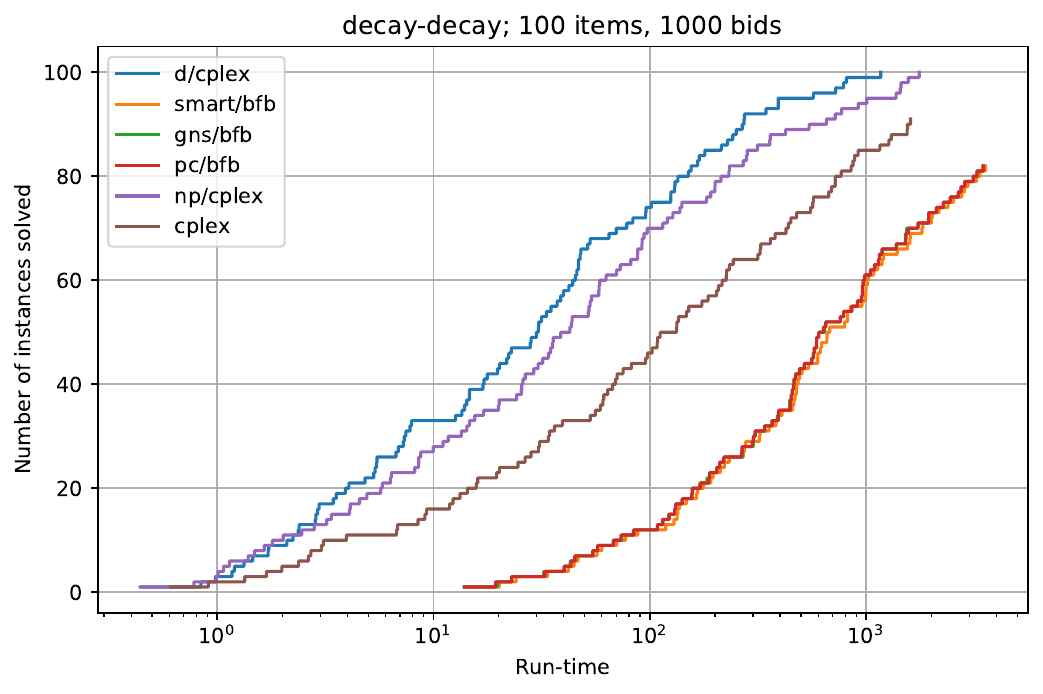}
\end{subfigure}

\caption{Decay-decay, CPLEX cover cuts on, all other parameters off}
\label{fig:muca}
\end{figure}

\begin{figure}[t]
\centering
\begin{subfigure}
  \centering
  \includegraphics[width=.39\linewidth]{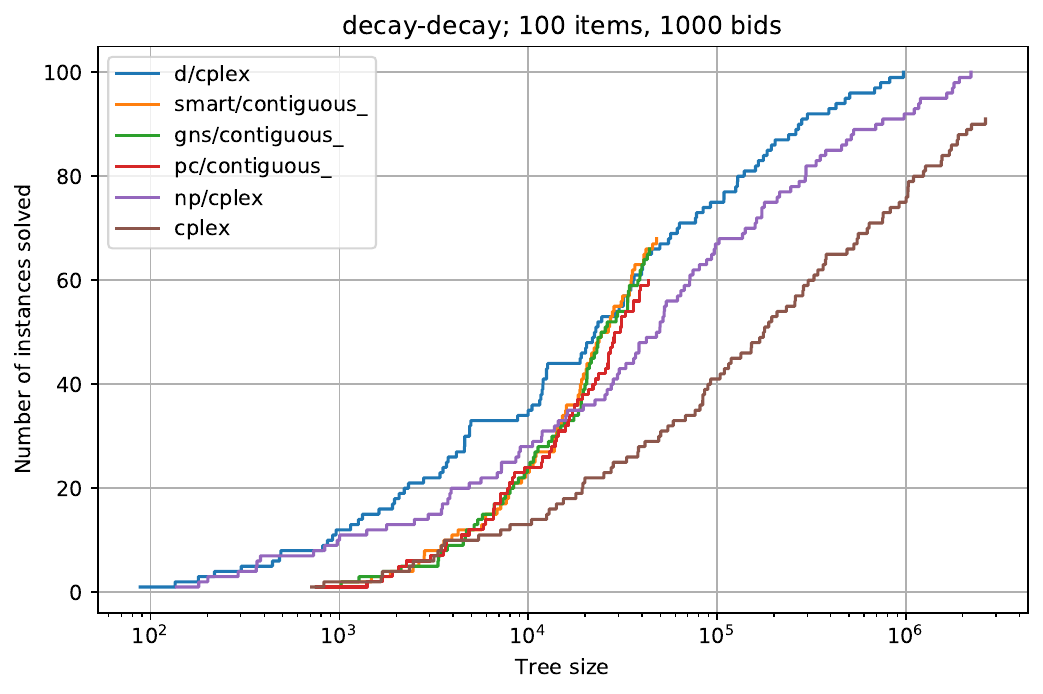}
  \label{fig:blah}
\end{subfigure}
\begin{subfigure}
  \centering
  \includegraphics[width=.39\linewidth]{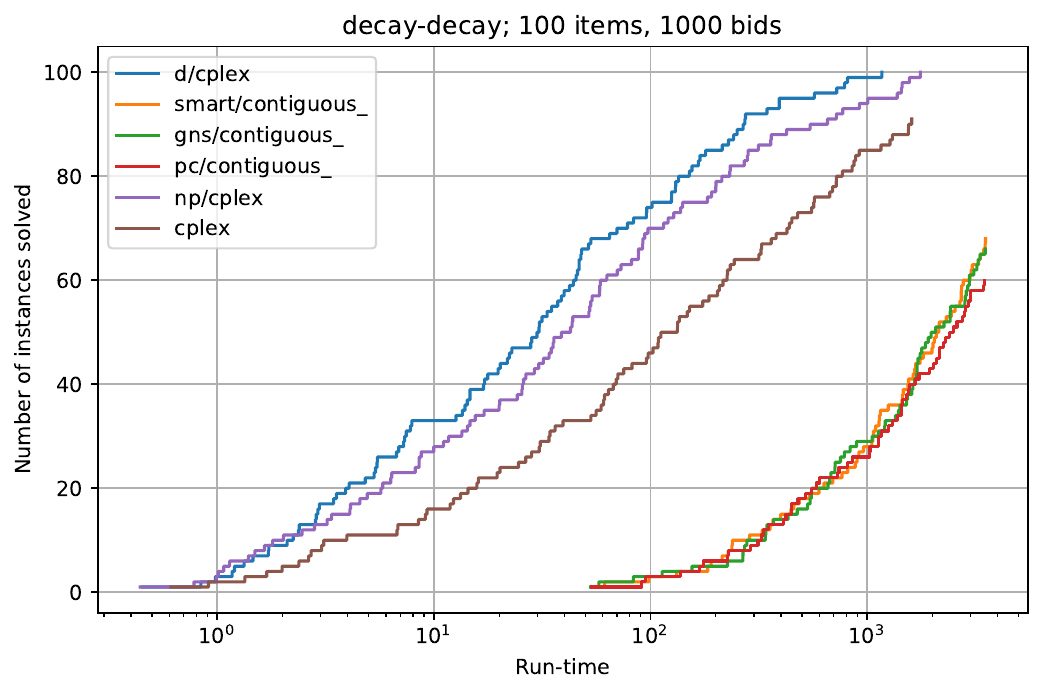}
\end{subfigure}
\begin{subfigure}
  \centering
  \includegraphics[width=.39\linewidth]{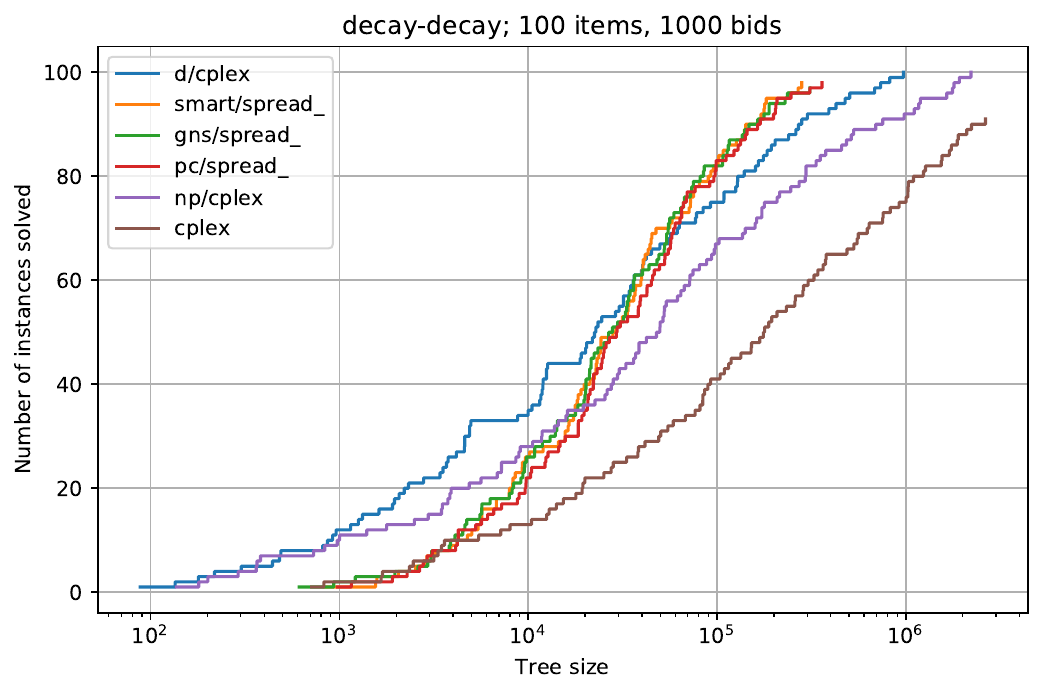}
  \label{fig:blah}
\end{subfigure}
\begin{subfigure}
  \centering
  \includegraphics[width=.39\linewidth]{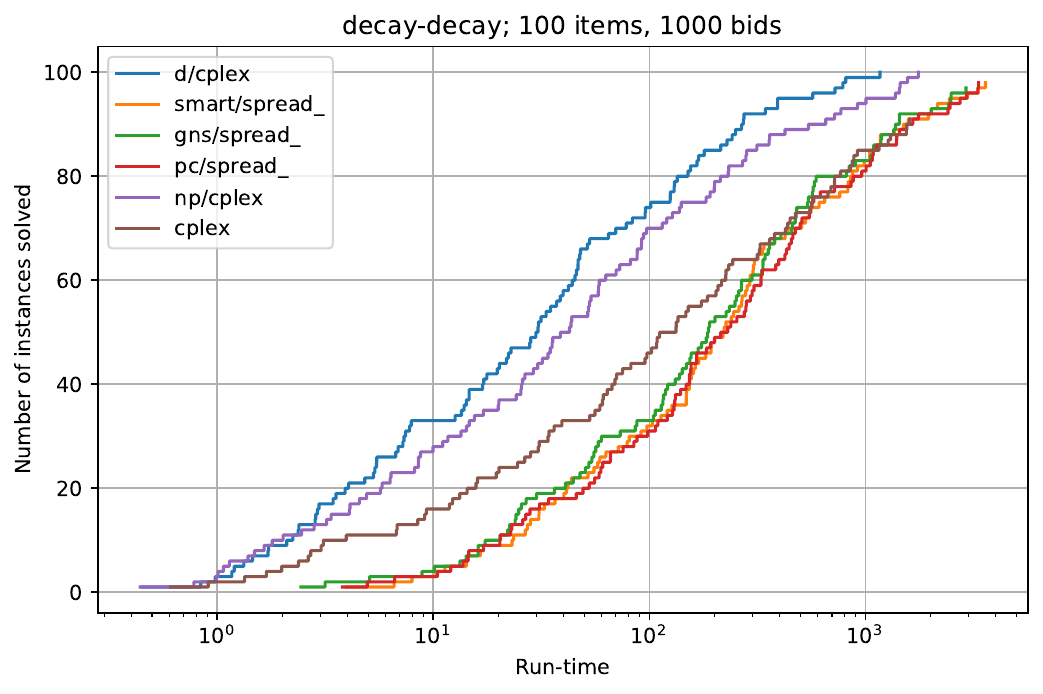}
\end{subfigure}
\begin{subfigure}
  \centering
  \includegraphics[width=.39\linewidth]{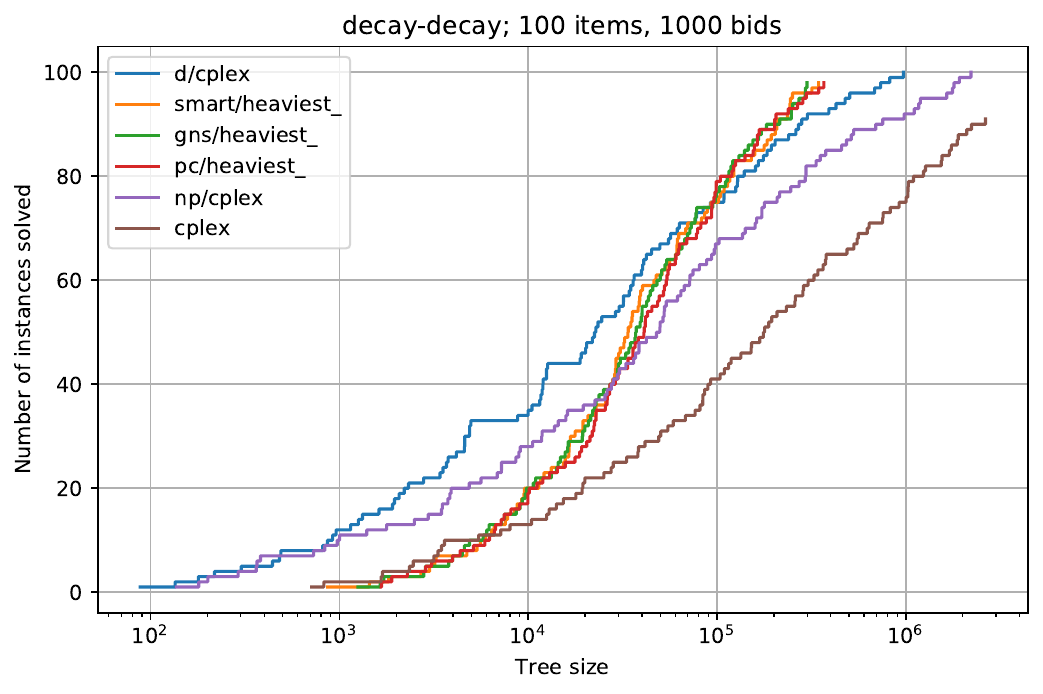}
  \label{fig:blah}
\end{subfigure}
\begin{subfigure}
  \centering
  \includegraphics[width=.39\linewidth]{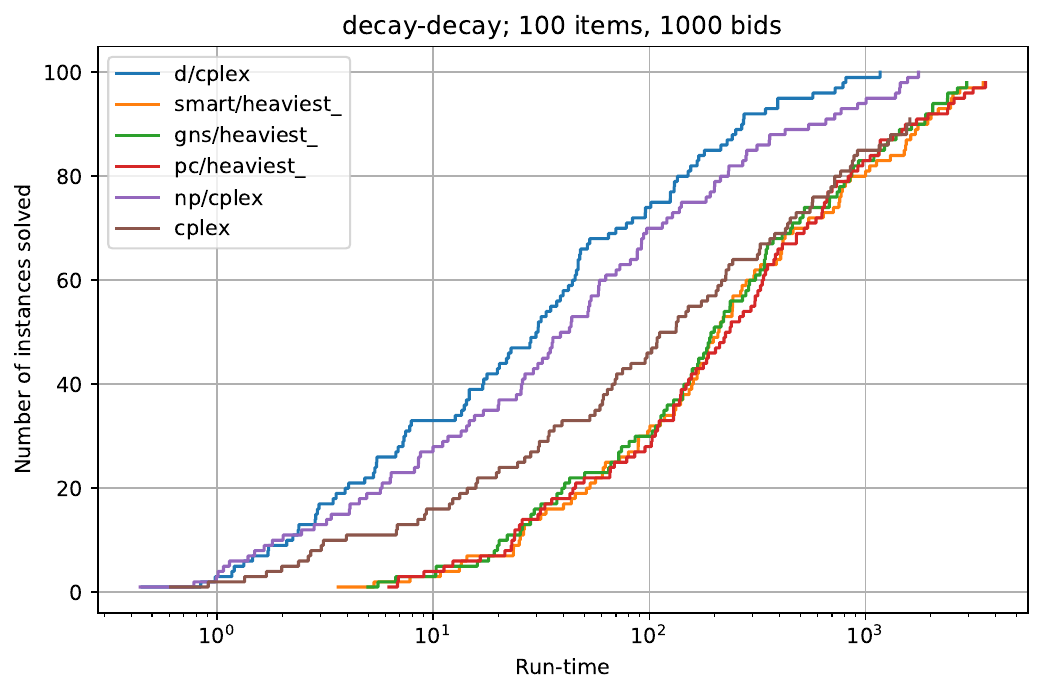}
\end{subfigure}
\begin{subfigure}
  \centering
  \includegraphics[width=.39\linewidth]{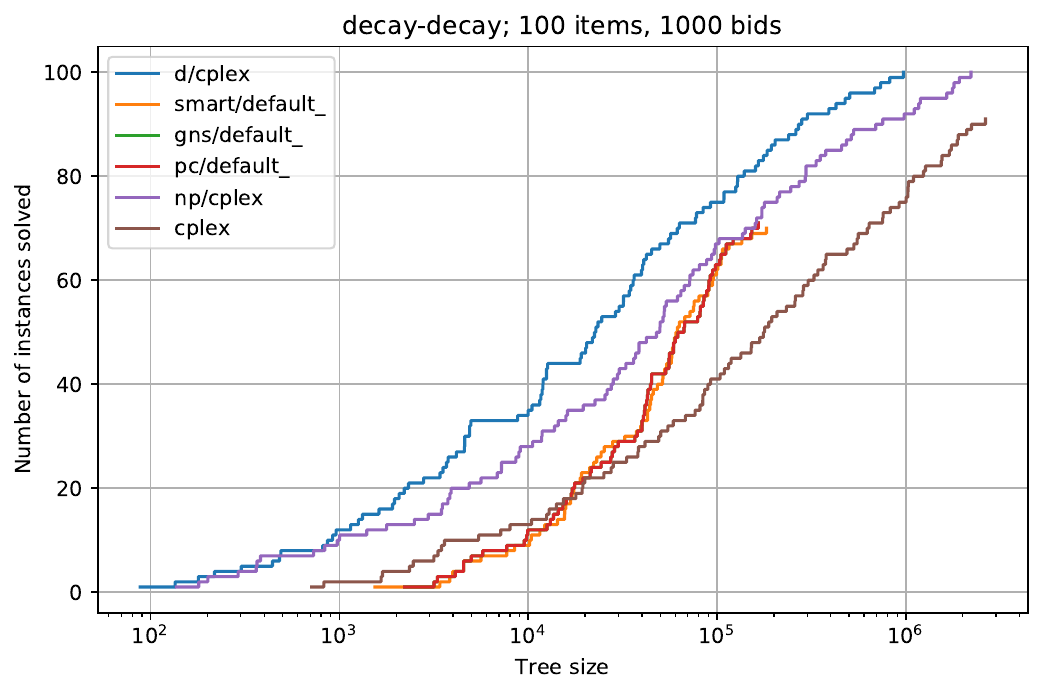}
  \label{fig:blah}
\end{subfigure}
\begin{subfigure}
  \centering
  \includegraphics[width=.39\linewidth]{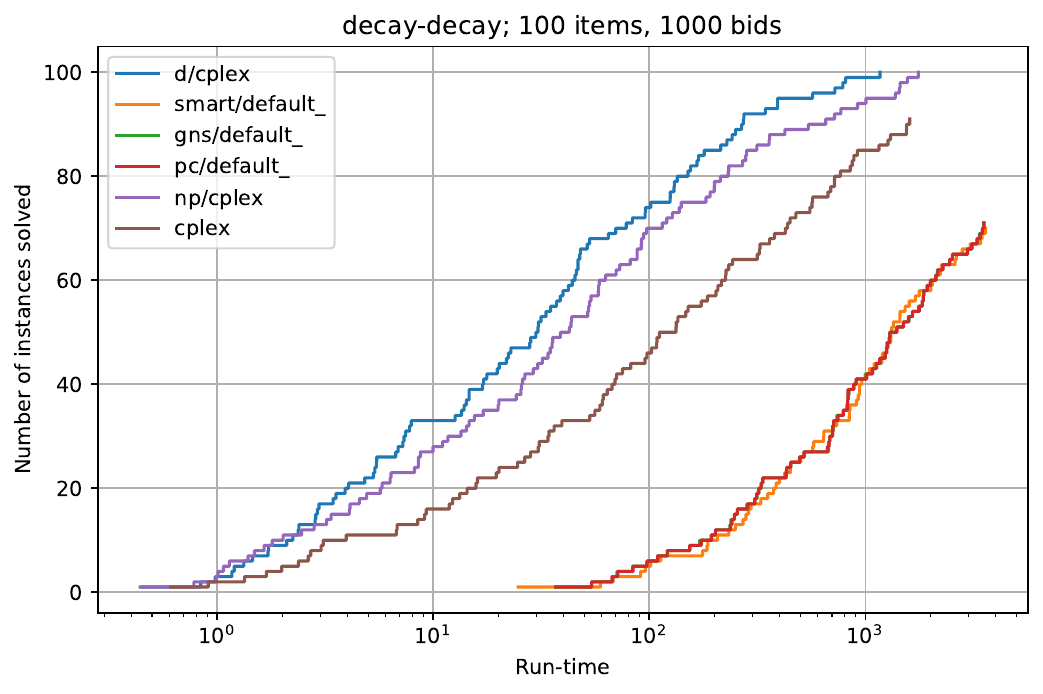}
\end{subfigure}
\begin{subfigure}
  \centering
  \includegraphics[width=.39\linewidth]{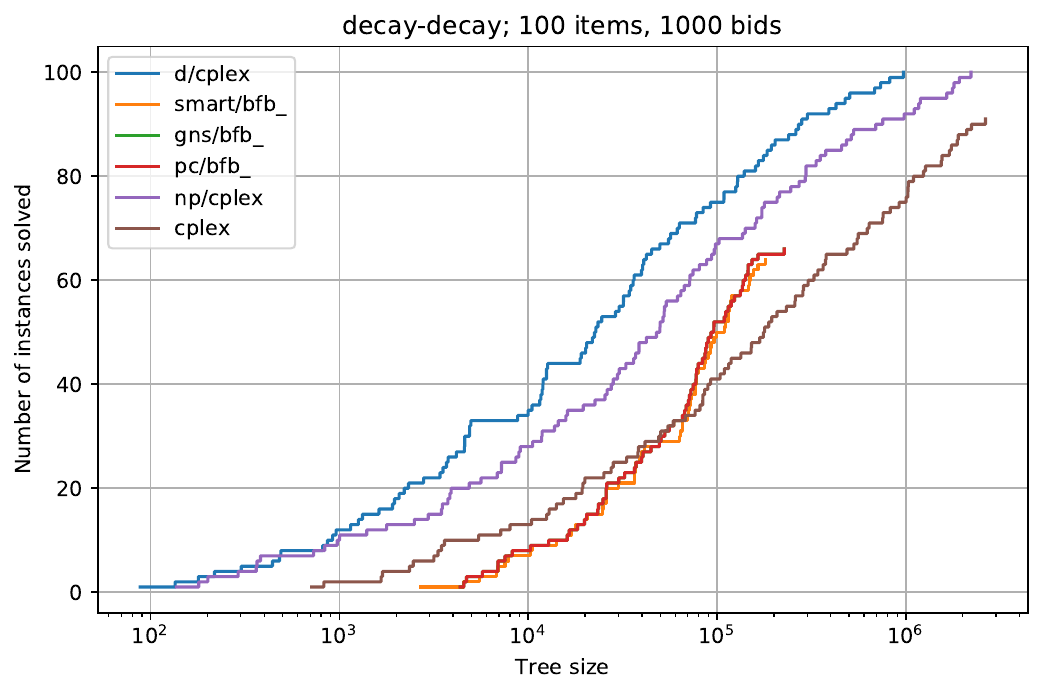}
  \label{fig:blah}
\end{subfigure}
\begin{subfigure}
  \centering
  \includegraphics[width=.39\linewidth]{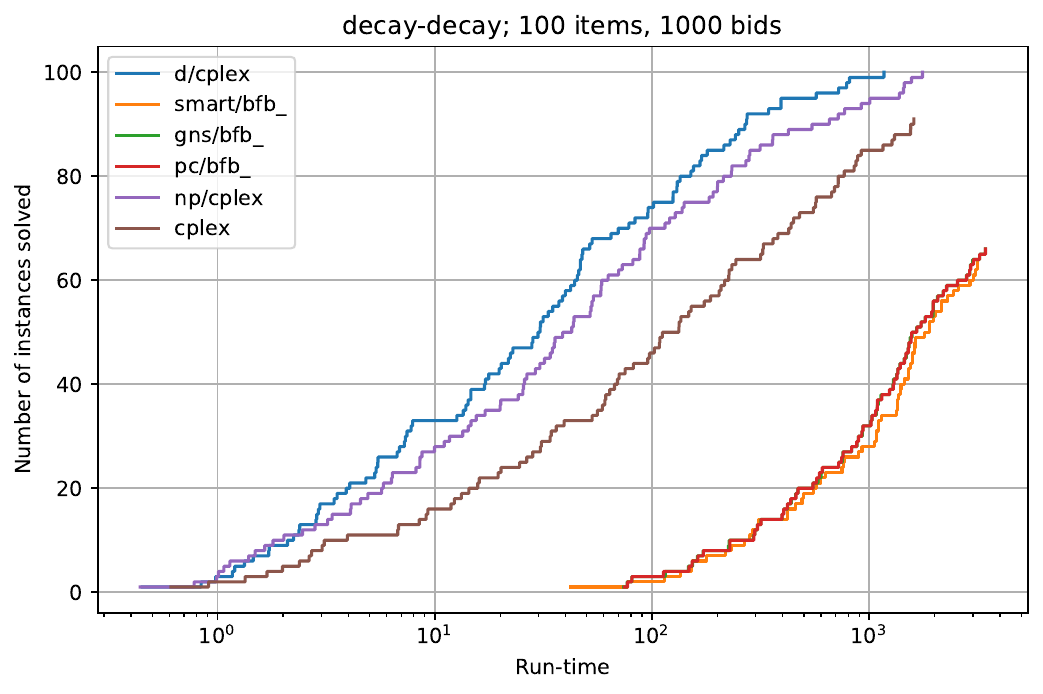}
\end{subfigure}

\caption{Decay-decay, CPLEX cover cuts off, all other parameters off}
\label{fig:muca_}
\end{figure}

\begin{figure}[t]
\centering
\begin{subfigure}
  \centering
  \includegraphics[width=.39\linewidth]{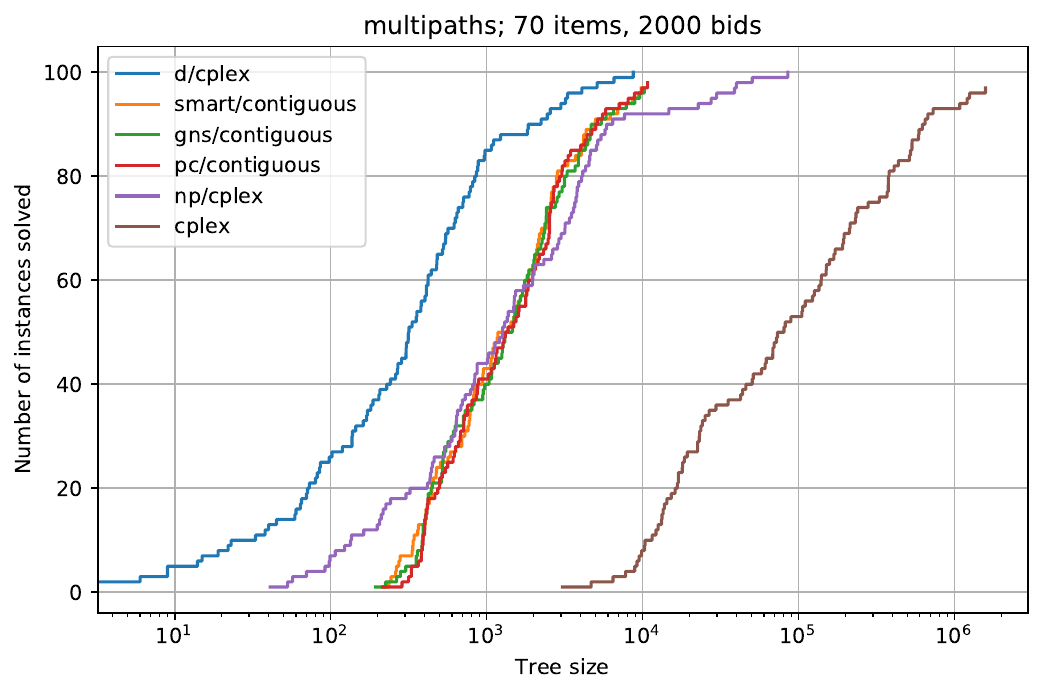}
  \label{fig:blah}
\end{subfigure}
\begin{subfigure}
  \centering
  \includegraphics[width=.39\linewidth]{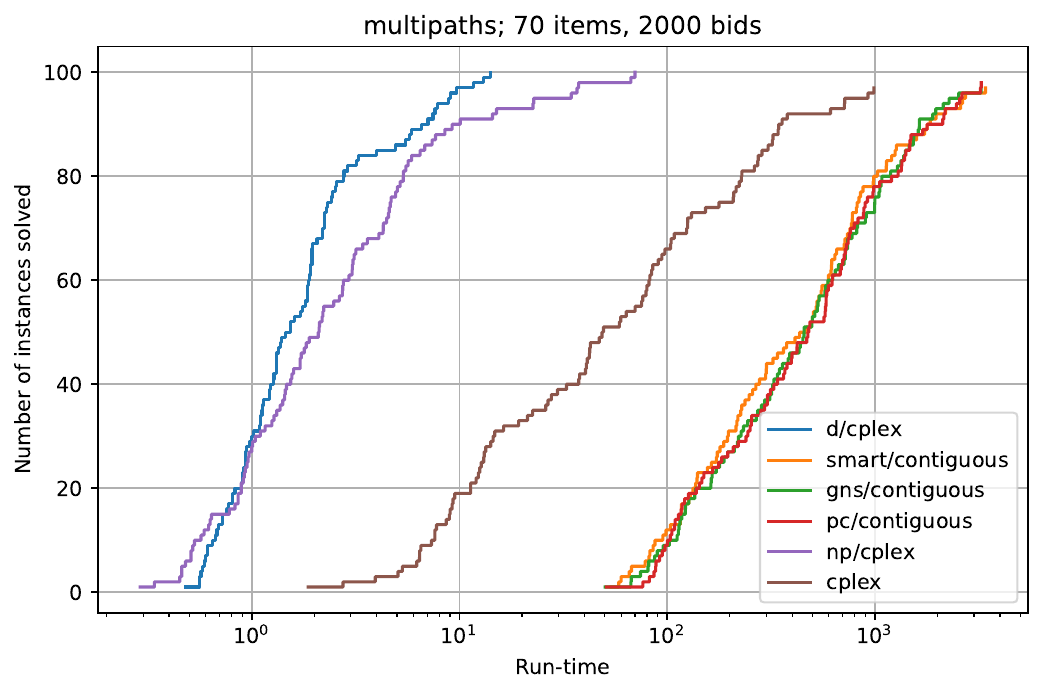}
\end{subfigure}
\begin{subfigure}
  \centering
  \includegraphics[width=.39\linewidth]{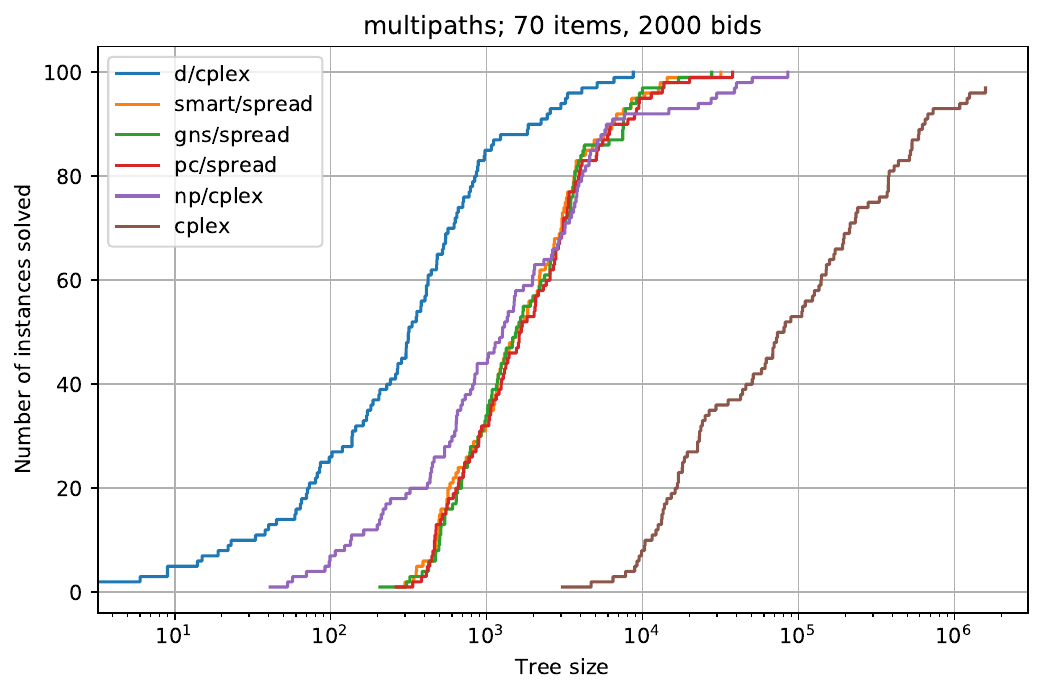}
  \label{fig:blah}
\end{subfigure}
\begin{subfigure}
  \centering
  \includegraphics[width=.39\linewidth]{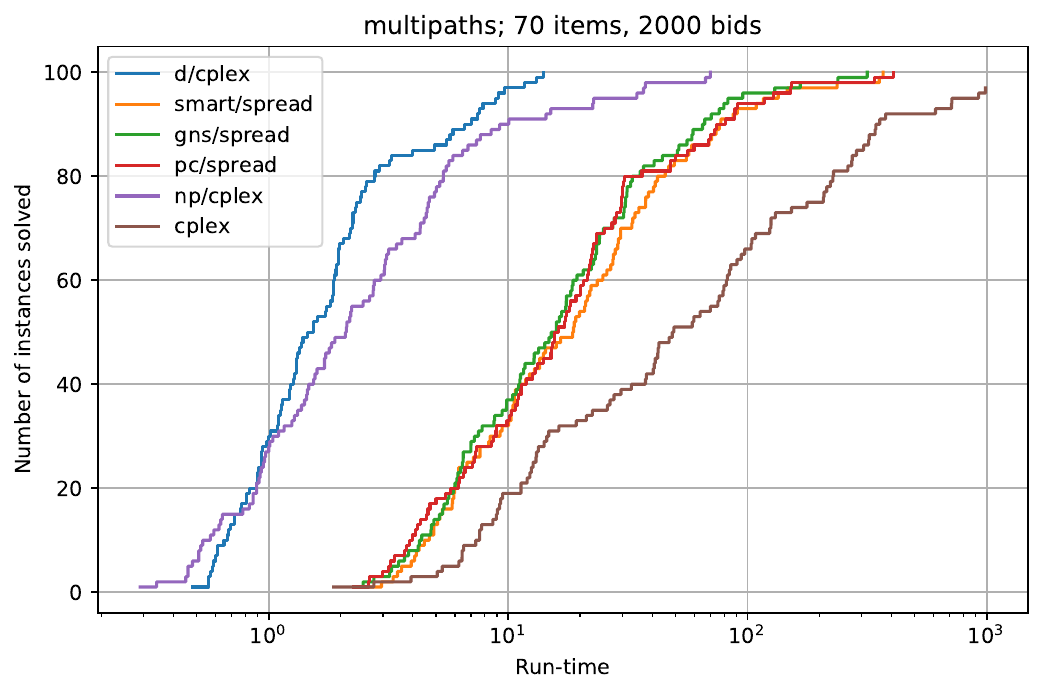}
\end{subfigure}
\begin{subfigure}
  \centering
  \includegraphics[width=.39\linewidth]{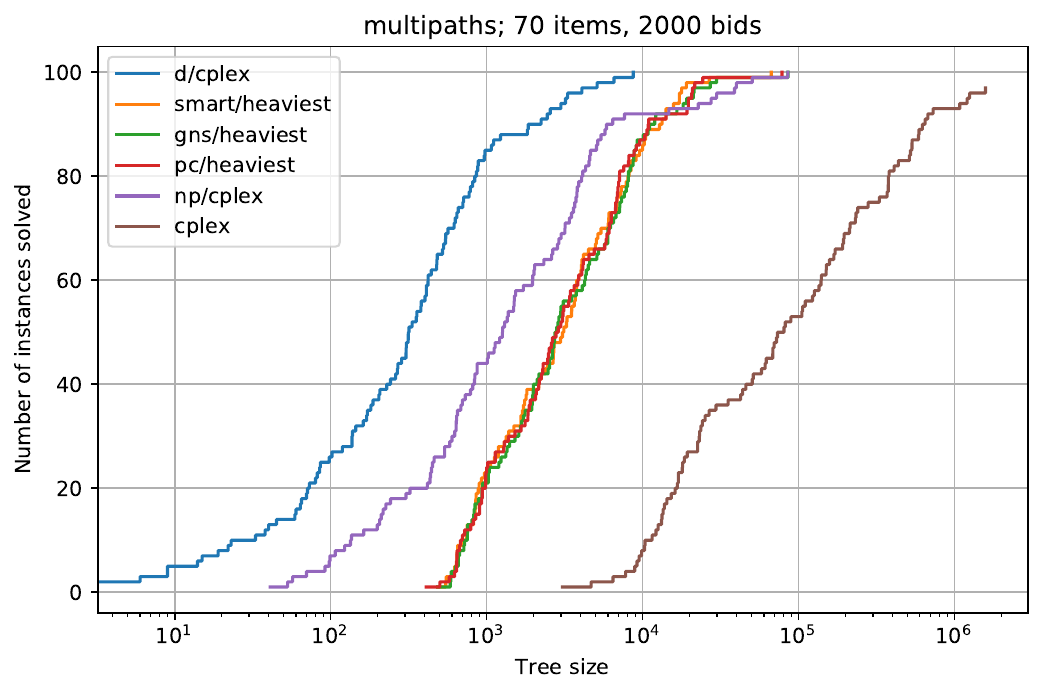}
  \label{fig:blah}
\end{subfigure}
\begin{subfigure}
  \centering
  \includegraphics[width=.39\linewidth]{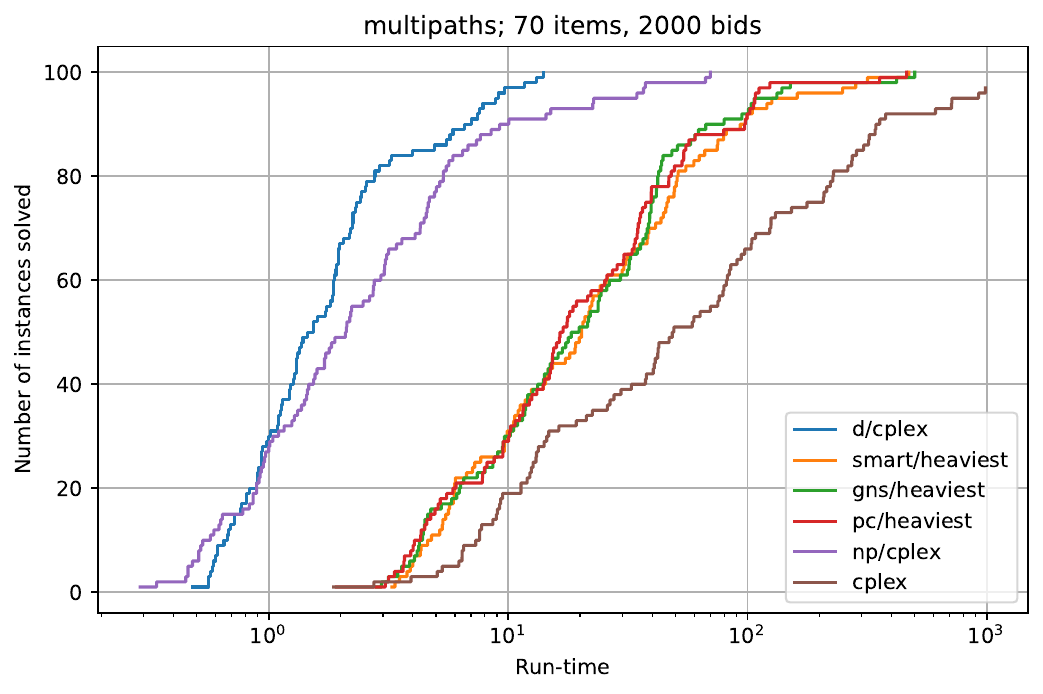}
\end{subfigure}
\begin{subfigure}
  \centering
  \includegraphics[width=.39\linewidth]{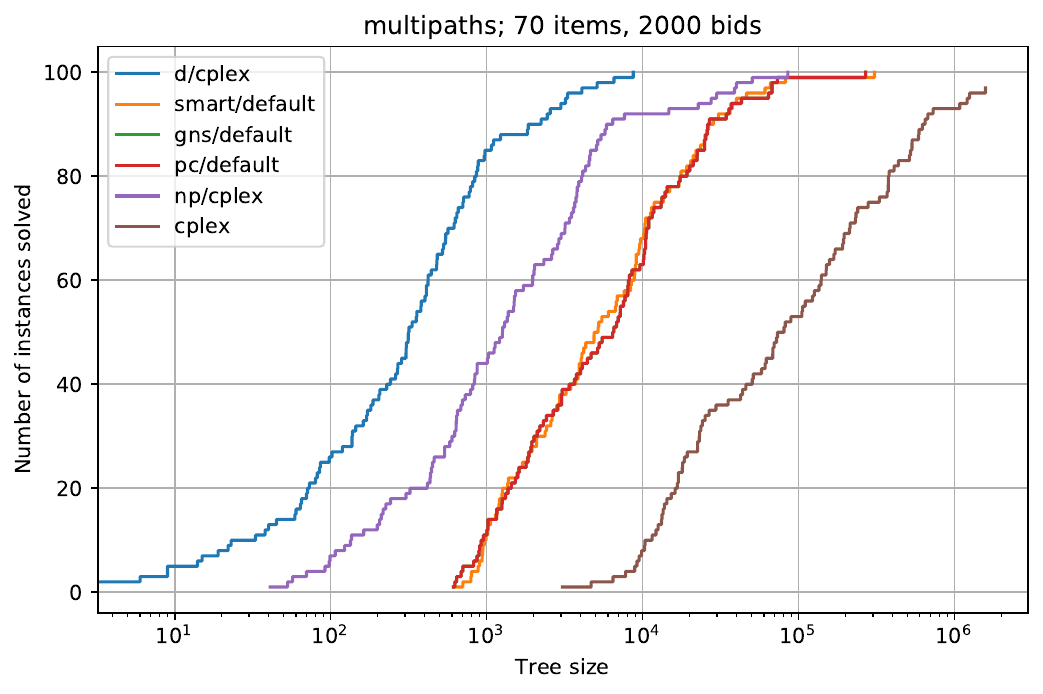}
  \label{fig:blah}
\end{subfigure}
\begin{subfigure}
  \centering
  \includegraphics[width=.39\linewidth]{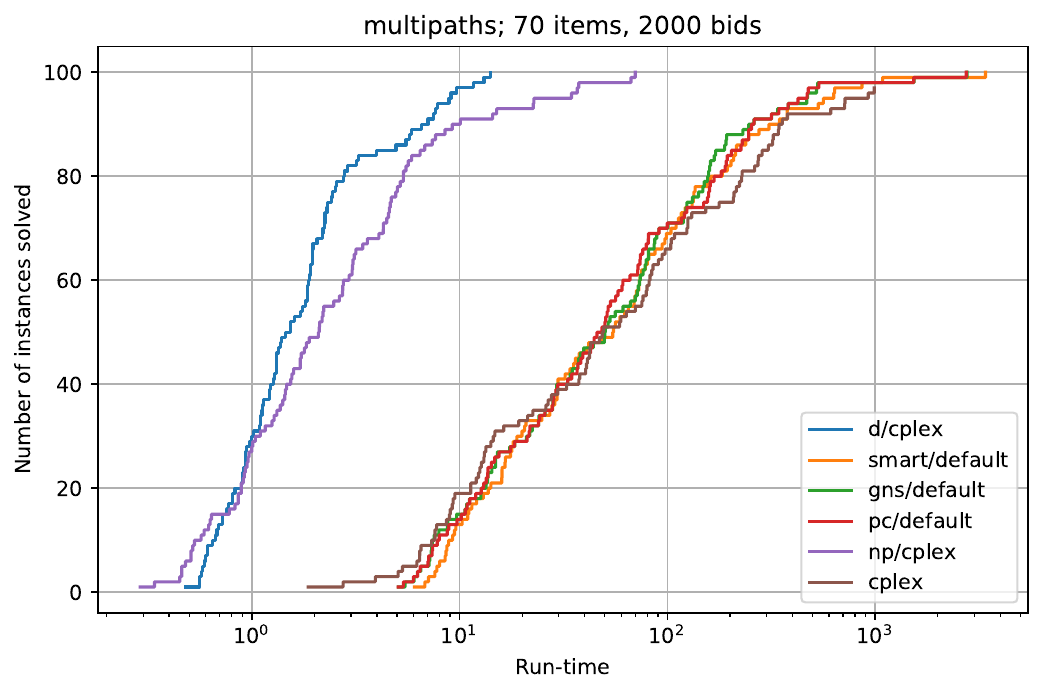}
\end{subfigure}
\begin{subfigure}
  \centering
  \includegraphics[width=.39\linewidth]{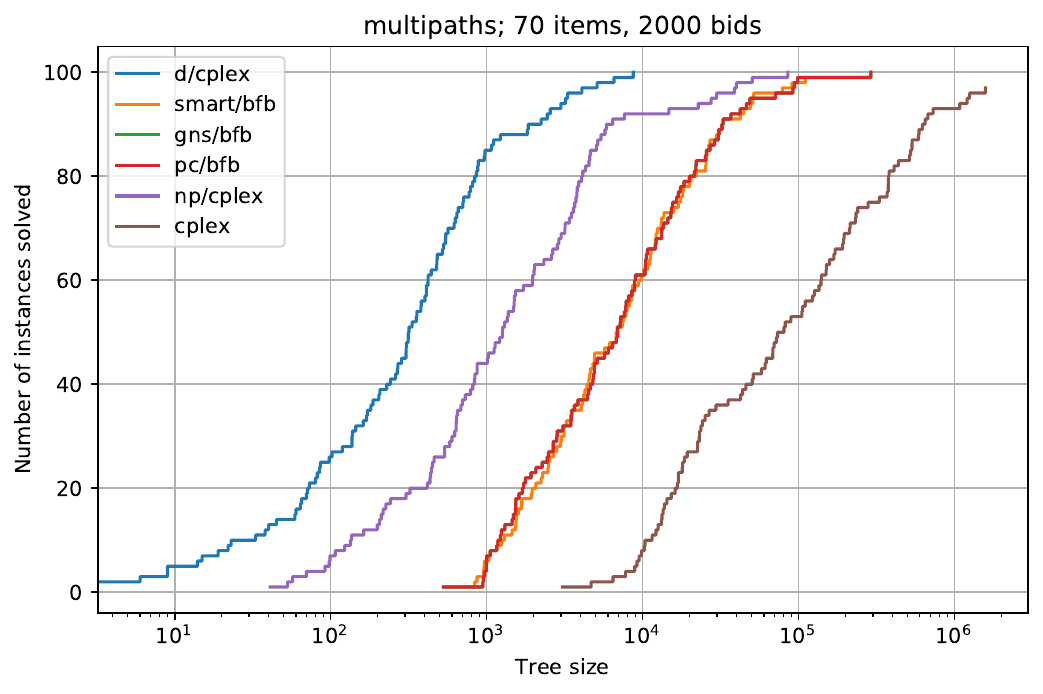}
  \label{fig:blah}
\end{subfigure}
\begin{subfigure}
  \centering
  \includegraphics[width=.39\linewidth]{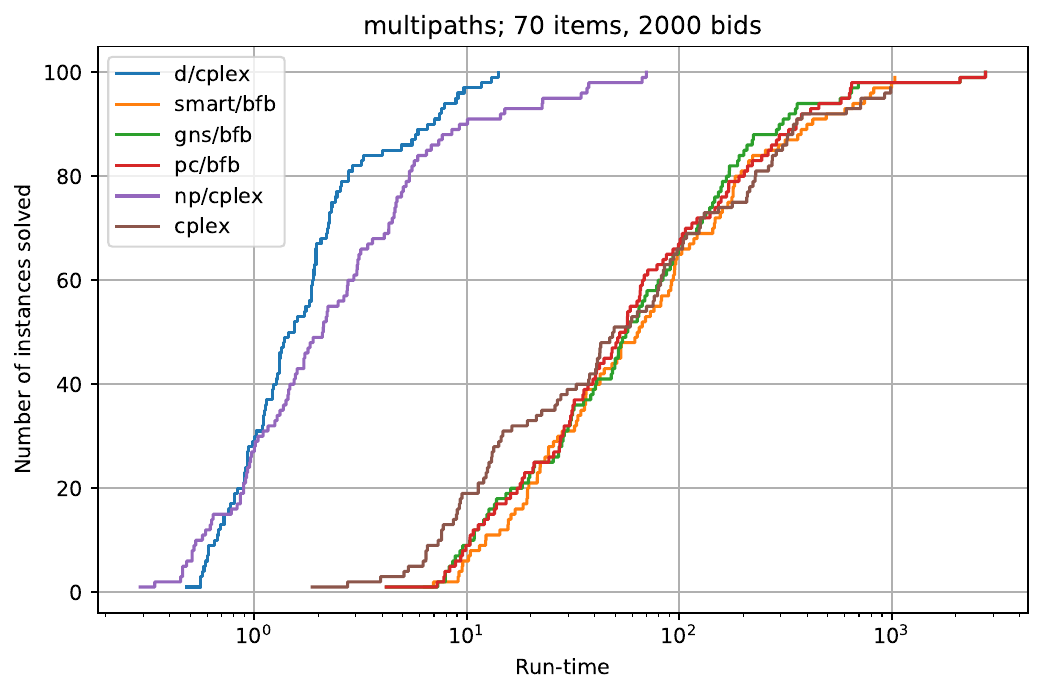}
\end{subfigure}

\caption{Multipaths, CPLEX cover cuts on, all other parameters off}
\label{fig:multipaths}
\end{figure}

\begin{figure}[t]
\centering
\begin{subfigure}
  \centering
  \includegraphics[width=.39\linewidth]{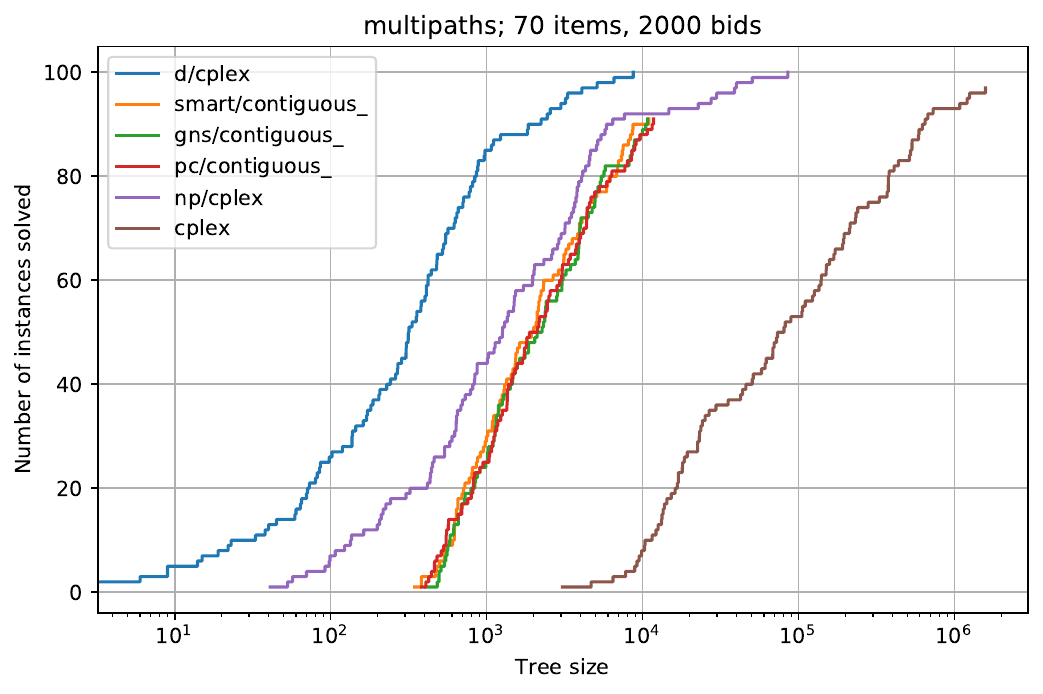}
  \label{fig:blah}
\end{subfigure}
\begin{subfigure}
  \centering
  \includegraphics[width=.39\linewidth]{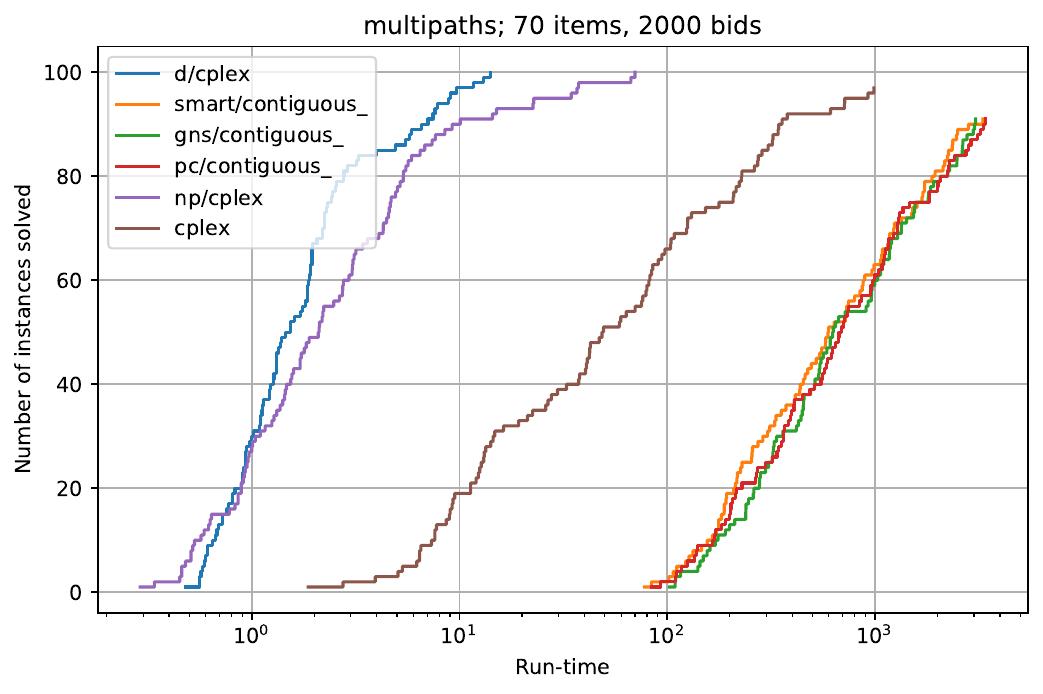}
\end{subfigure}
\begin{subfigure}
  \centering
  \includegraphics[width=.39\linewidth]{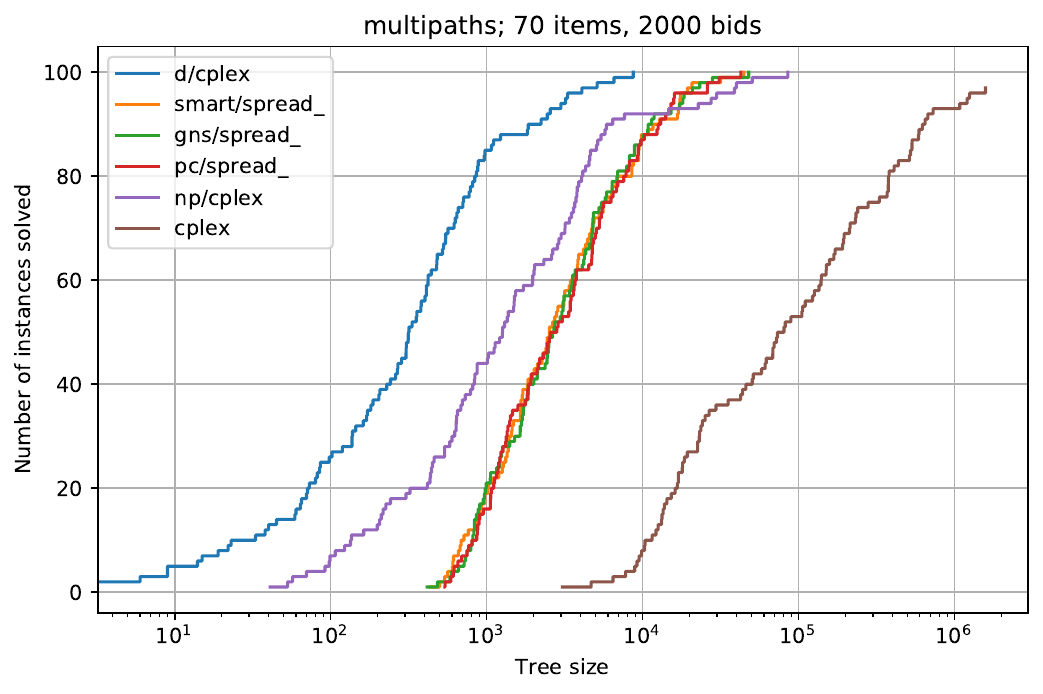}
  \label{fig:blah}
\end{subfigure}
\begin{subfigure}
  \centering
  \includegraphics[width=.39\linewidth]{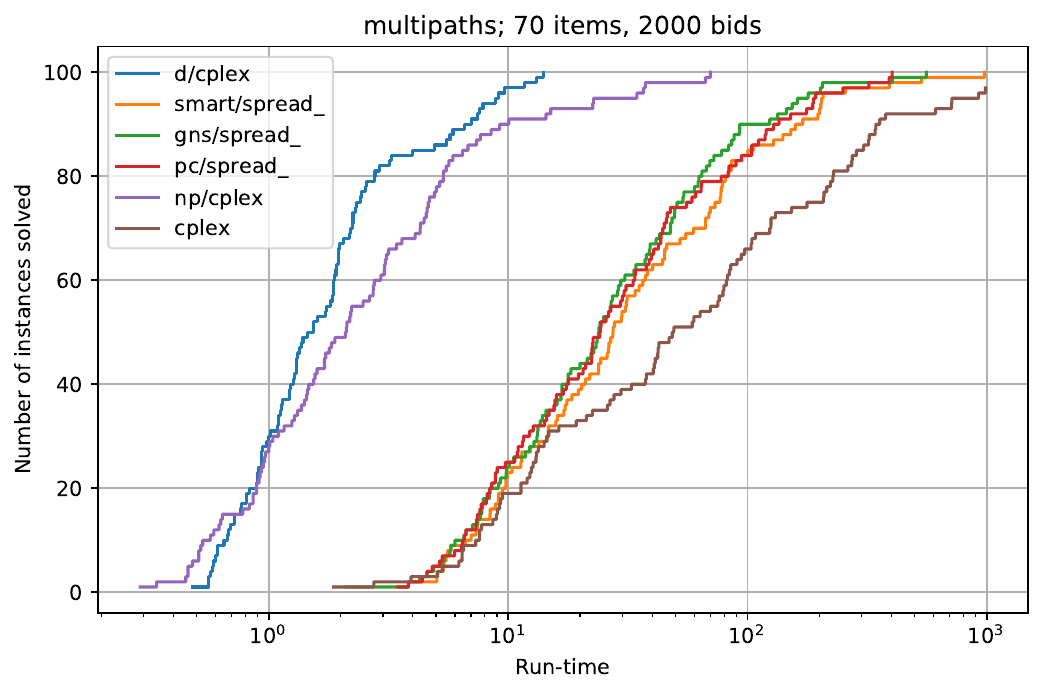}
\end{subfigure}
\begin{subfigure}
  \centering
  \includegraphics[width=.39\linewidth]{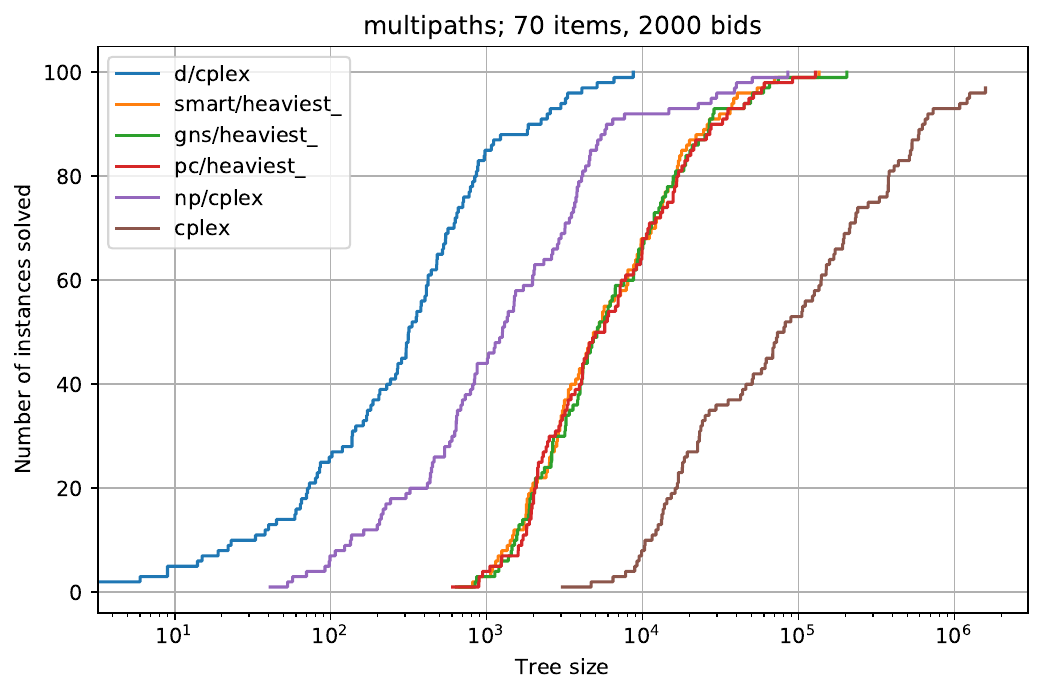}
  \label{fig:blah}
\end{subfigure}
\begin{subfigure}
  \centering
  \includegraphics[width=.39\linewidth]{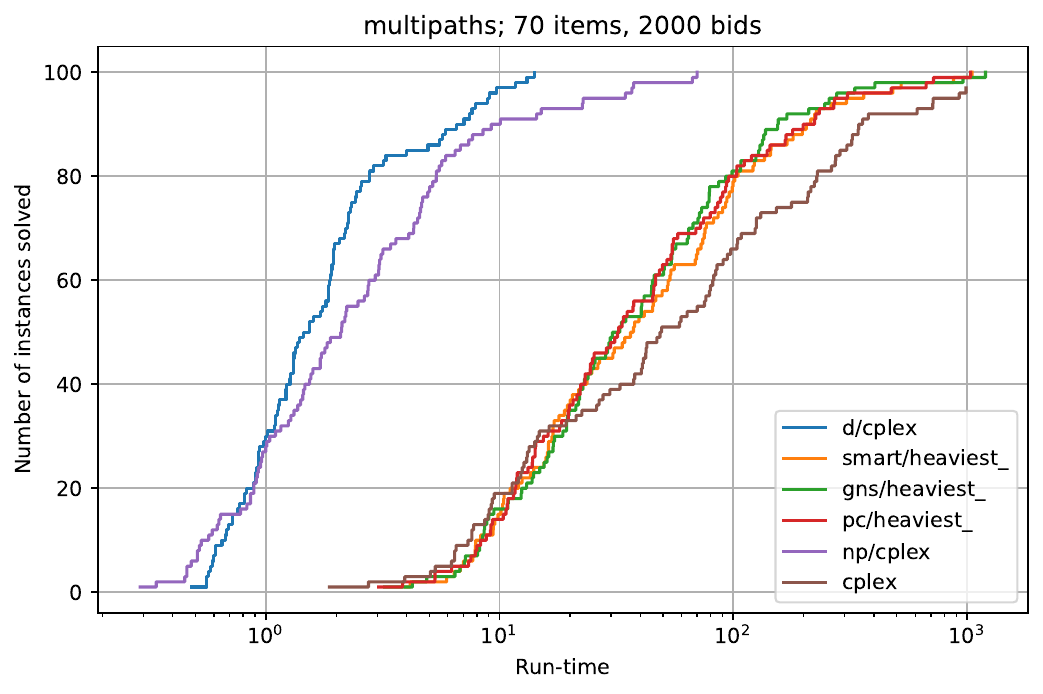}
\end{subfigure}
\begin{subfigure}
  \centering
  \includegraphics[width=.39\linewidth]{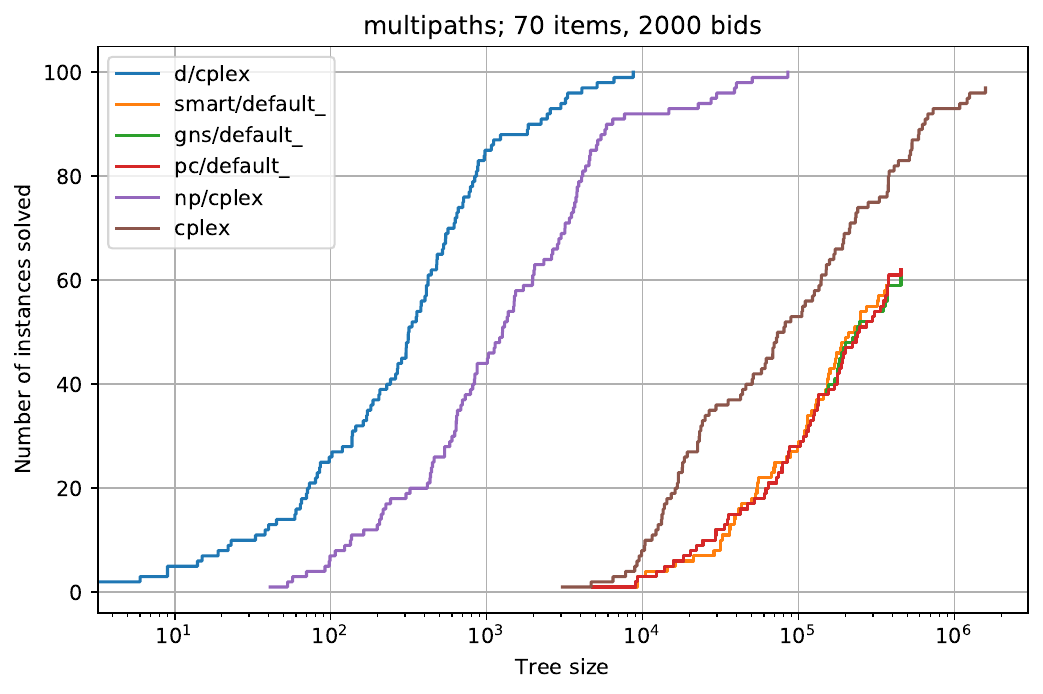}
  \label{fig:blah}
\end{subfigure}
\begin{subfigure}
  \centering
  \includegraphics[width=.39\linewidth]{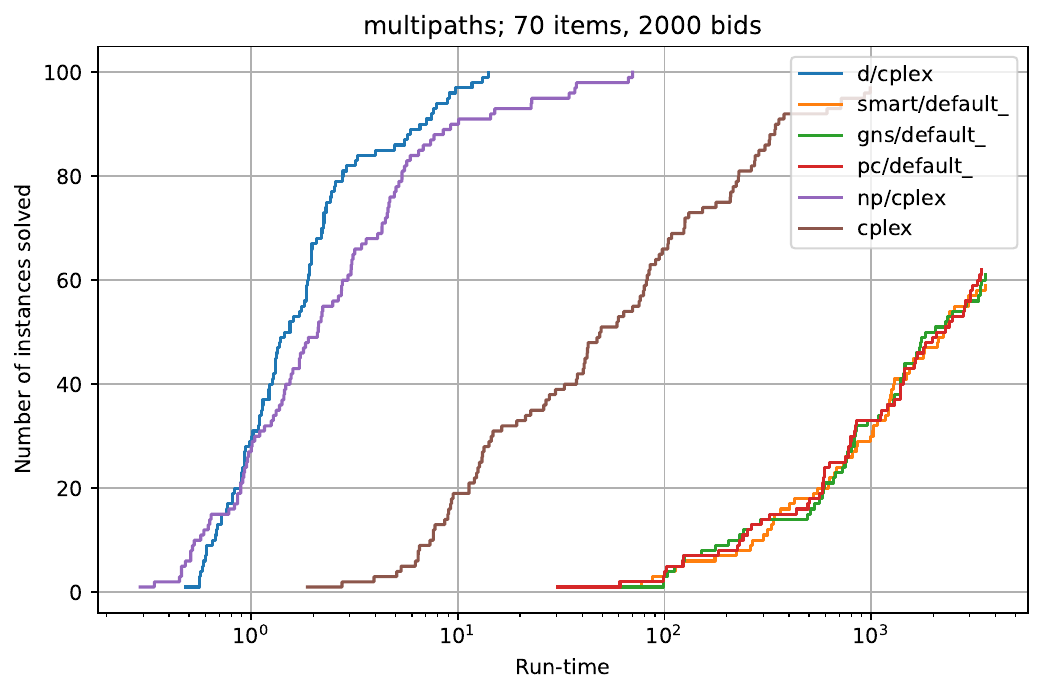}
\end{subfigure}
\begin{subfigure}
  \centering
  \includegraphics[width=.39\linewidth]{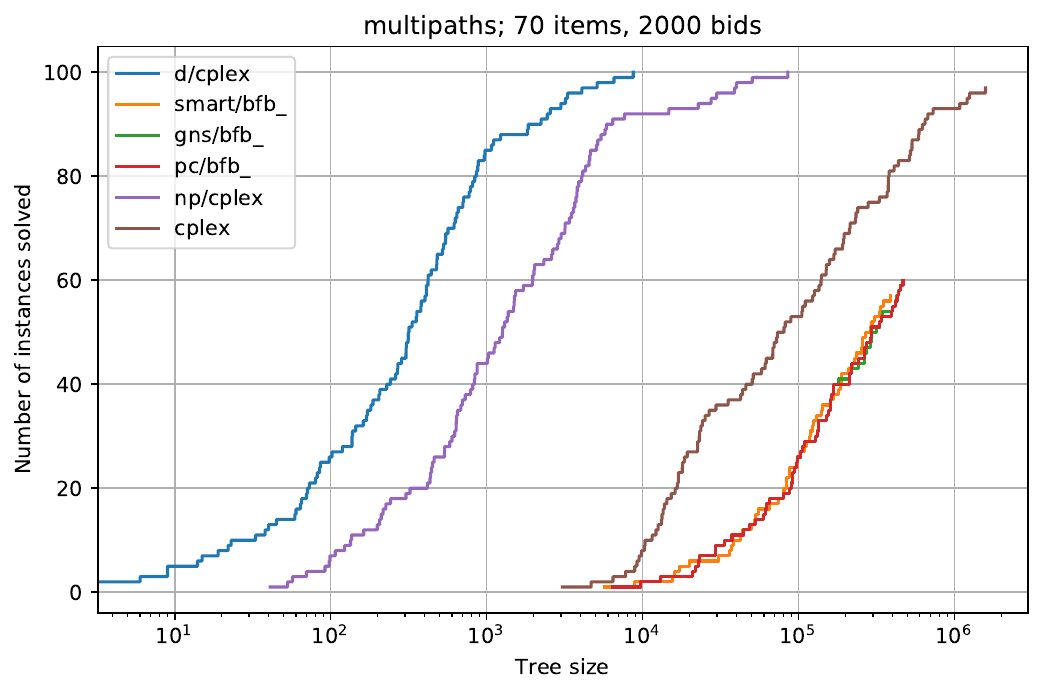}
  \label{fig:blah}
\end{subfigure}
\begin{subfigure}
  \centering
  \includegraphics[width=.39\linewidth]{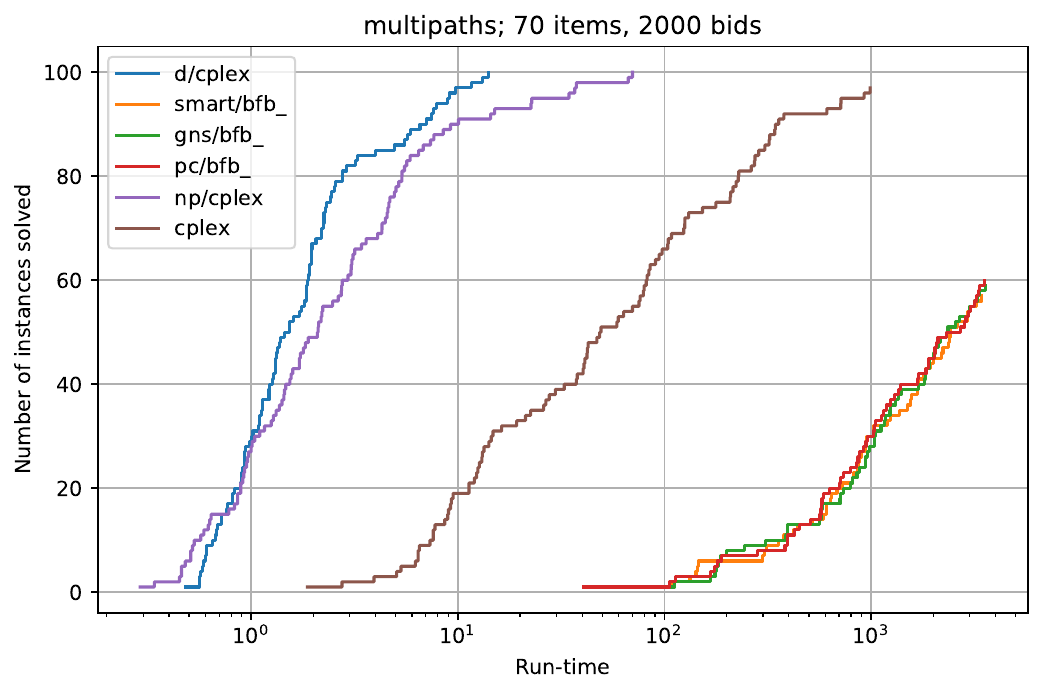}
\end{subfigure}

\caption{Multipaths, CPLEX cover cuts off, all other parameters off}
\label{fig:multipaths_}
\end{figure}


\subsection{Lifted cover cut methods atop default solver settings}

Next, we evaluate our lifting methods atop close-to-default CPLEX settings; all CPLEX settings are untouched with the exception of presolve, which is turned off. Presolve, which refers to a toolkit of optimizations that operate directly on the input formulation, is incompatible with user-implemented cutting plane callbacks. As before, We toggle CPLEX's internal cover cut generation routine on and off to measure the impact of our routines with and without the cover cuts added by CPLEX (turning CPLEX cover cuts off corresponds to an underscore label in our plots).

Figures~\ref{fig:mkp_weakly_correlatedno_presolve},~\ref{fig:mkp_uncorrelatedno_presolve},~\ref{fig:chvatal_no_presolve},~\ref{fig:muca_no_presolve}, and~\ref{fig:multipathsno_presolve} contain the tree size and run-time performance plots for the weakly correlated, uncorrelated, Chv\'{a}tal, decay-decay and multipaths distributions, respectively, when CPLEX cover cuts are turned on. Figures~\ref{fig:mkp_weakly_correlatedno_presolve_},~\ref{fig:mkp_uncorrelatedno_presolve_},~\ref{fig:chvatal_no_presolve_},~\ref{fig:muca_no_presolve_}, and~\ref{fig:multipathsno_presolve_} contain the tree size and run-time performance plots for the weakly correlated, uncorrelated, Chv\'{a}tal, decay-decay and multipaths distributions, respectively, when CPLEX cover cuts are turned off. 

\begin{figure}[t]
\centering
\begin{subfigure}
  \centering
  \includegraphics[width=.39\linewidth]{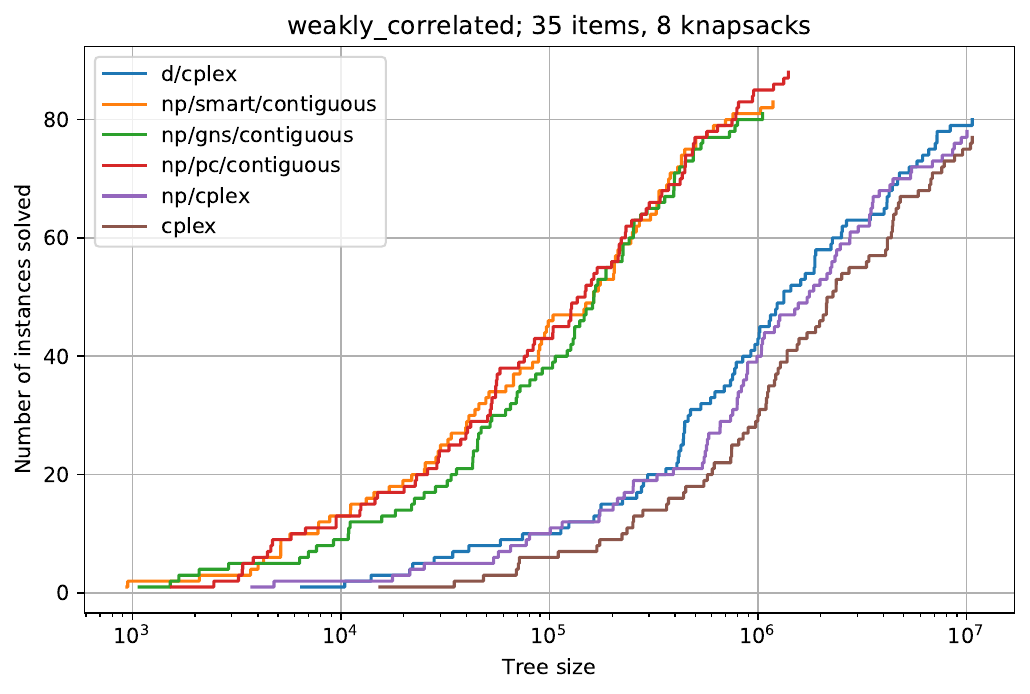}
  \label{fig:blah}
\end{subfigure}
\begin{subfigure}
  \centering
  \includegraphics[width=.39\linewidth]{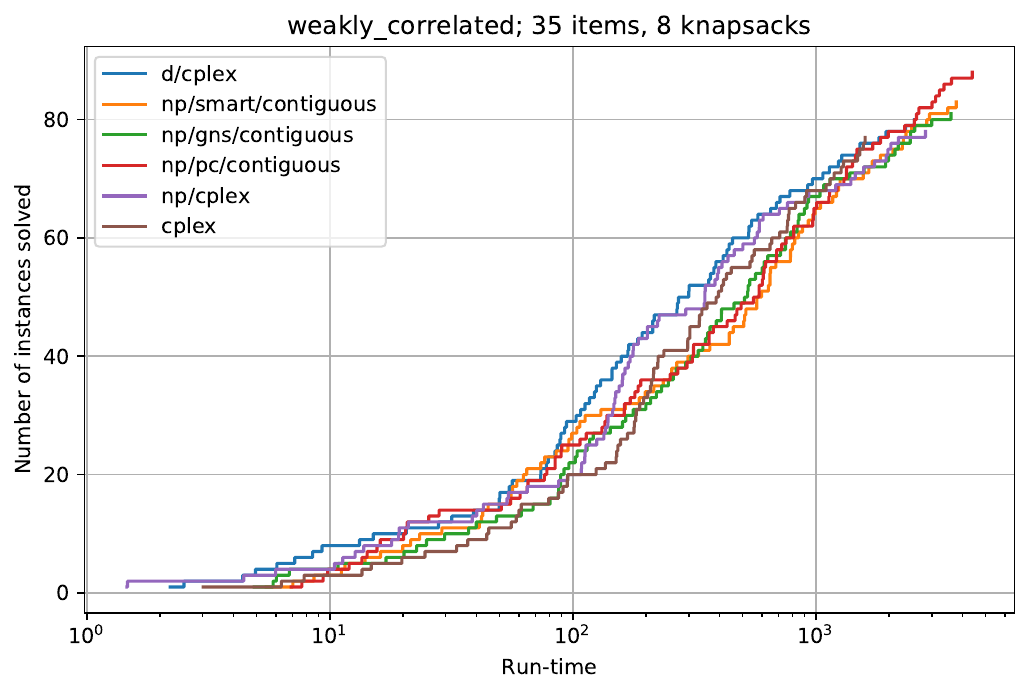}
\end{subfigure}
\begin{subfigure}
  \centering
  \includegraphics[width=.39\linewidth]{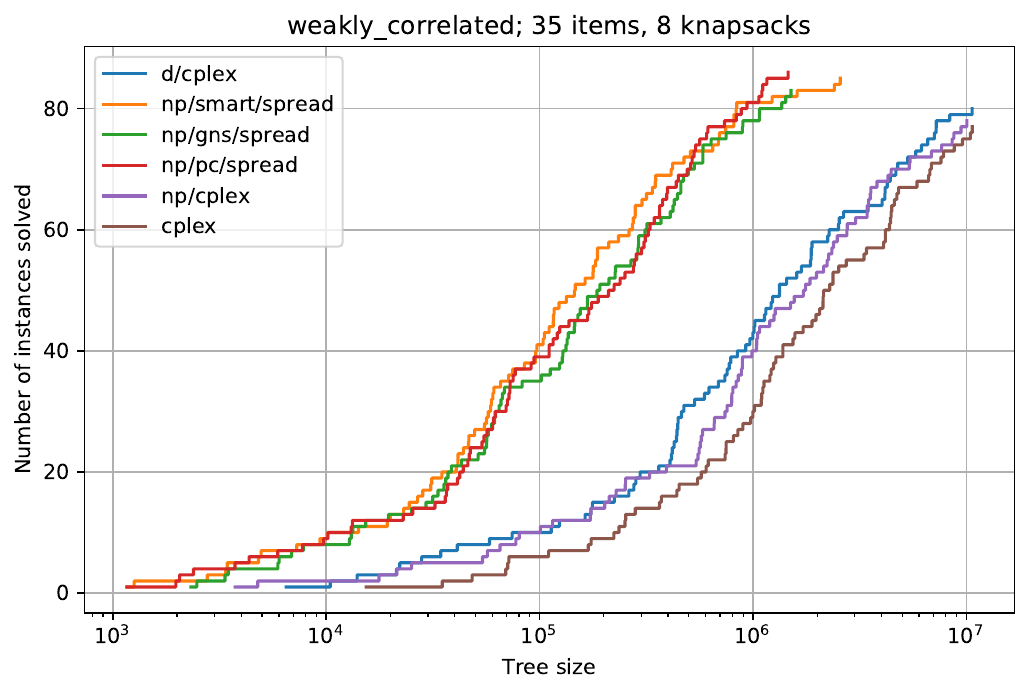}
  \label{fig:blah}
\end{subfigure}
\begin{subfigure}
  \centering
  \includegraphics[width=.39\linewidth]{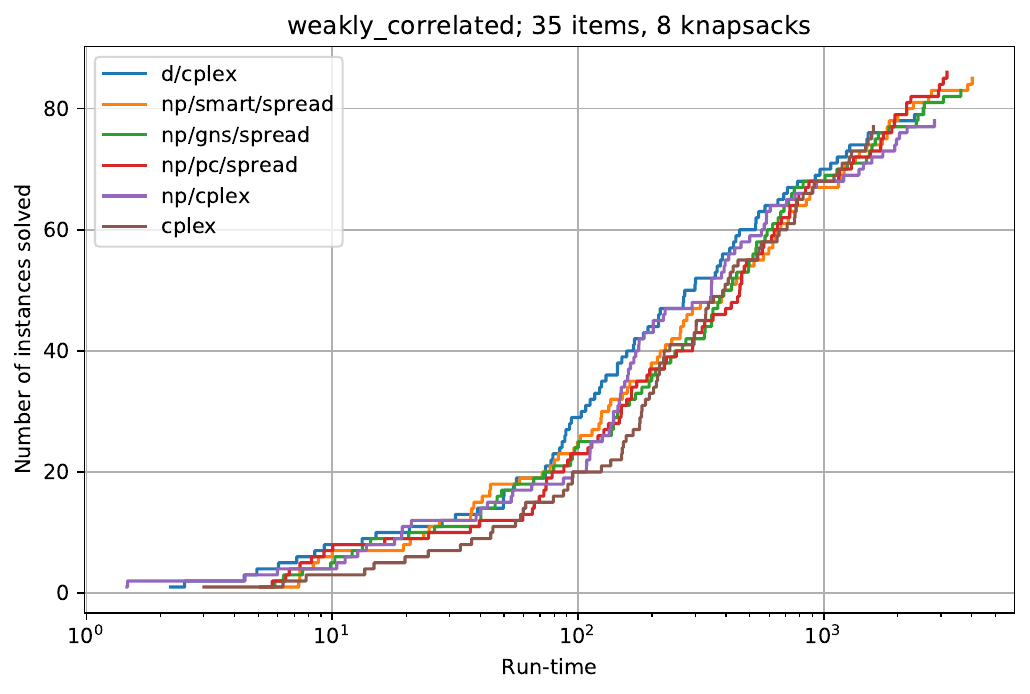}
\end{subfigure}
\begin{subfigure}
  \centering
  \includegraphics[width=.39\linewidth]{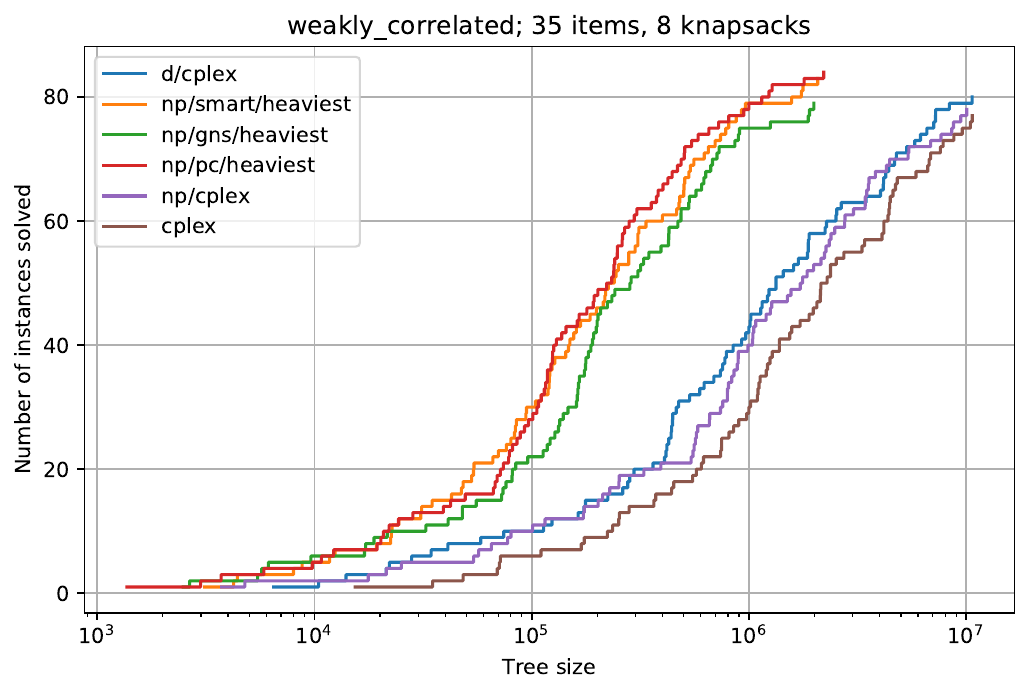}
  \label{fig:blah}
\end{subfigure}
\begin{subfigure}
  \centering
  \includegraphics[width=.39\linewidth]{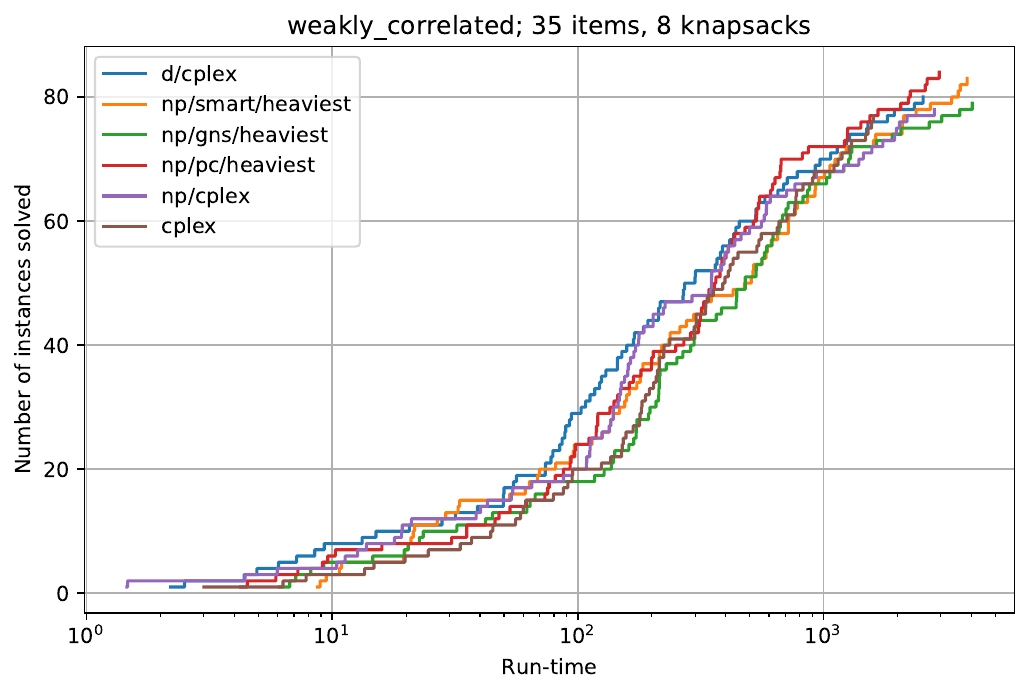}
\end{subfigure}
\begin{subfigure}
  \centering
  \includegraphics[width=.39\linewidth]{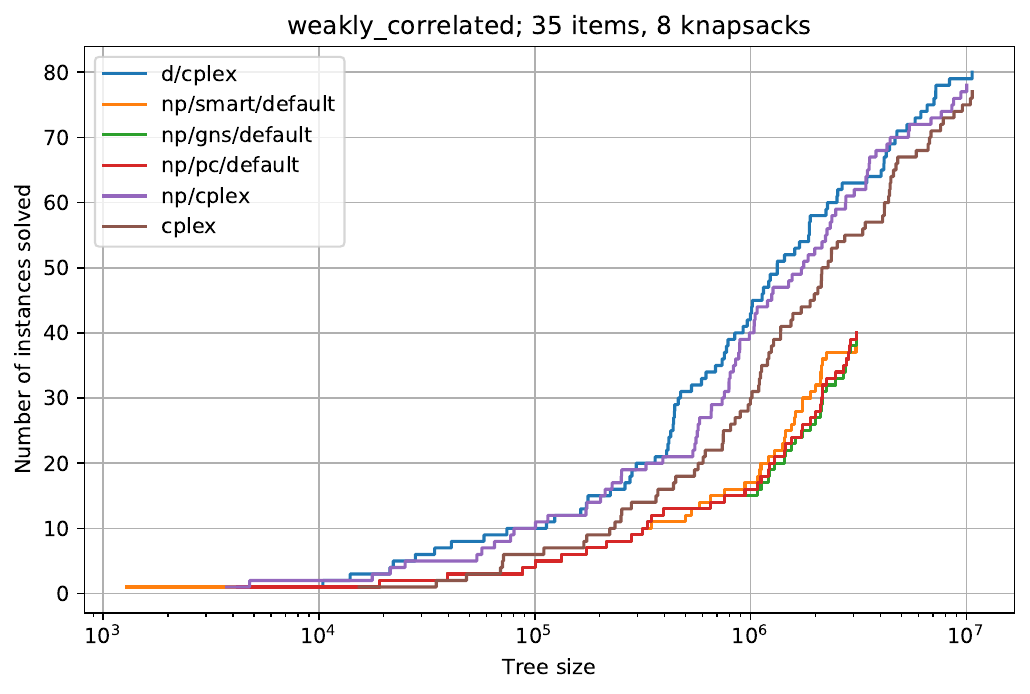}
  \label{fig:blah}
\end{subfigure}
\begin{subfigure}
  \centering
  \includegraphics[width=.39\linewidth]{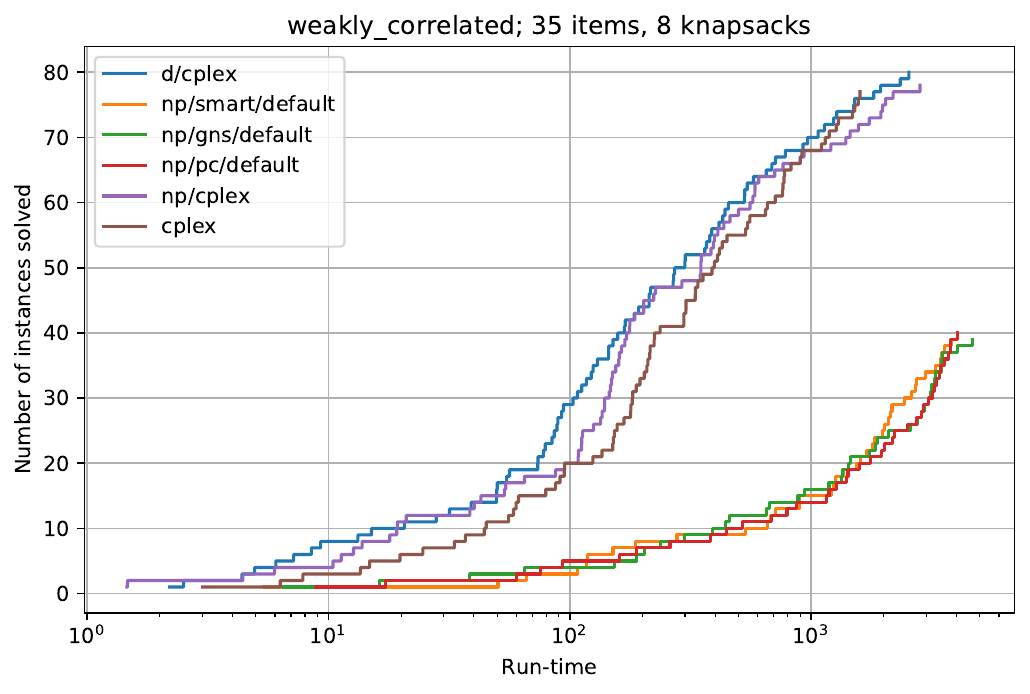}
\end{subfigure}
\begin{subfigure}
  \centering
  \includegraphics[width=.39\linewidth]{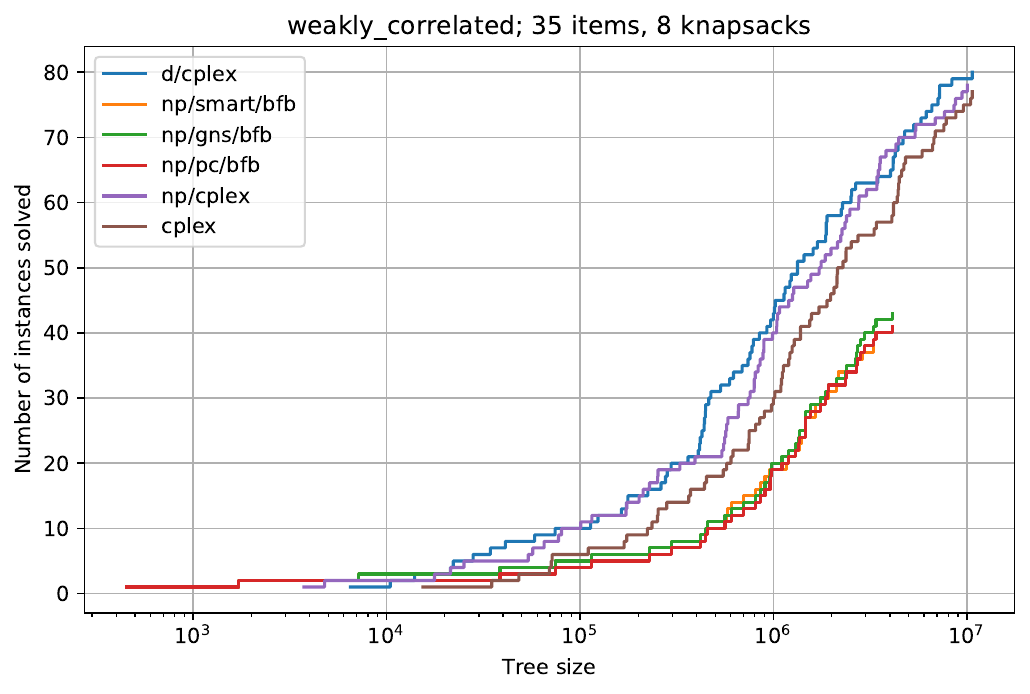}
  \label{fig:blah}
\end{subfigure}
\begin{subfigure}
  \centering
  \includegraphics[width=.39\linewidth]{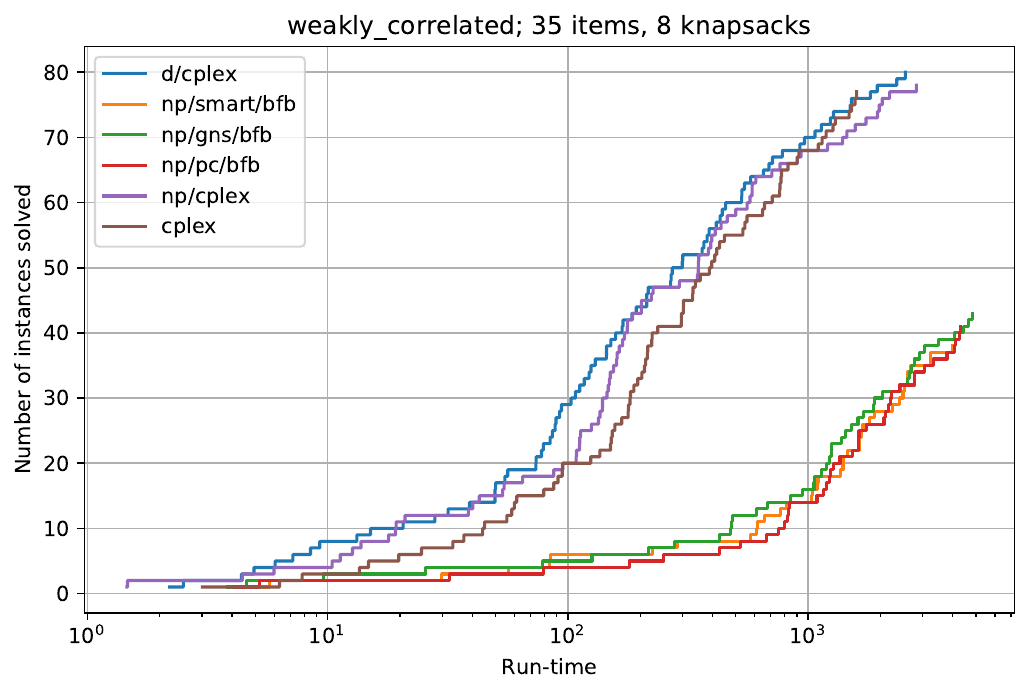}
\end{subfigure}

\caption{Weakly correlated, CPLEX cover cuts on, all other parameters but presolve on}
\label{fig:mkp_weakly_correlatedno_presolve}
\end{figure}

\begin{figure}[t]
\centering
\begin{subfigure}
  \centering
  \includegraphics[width=.39\linewidth]{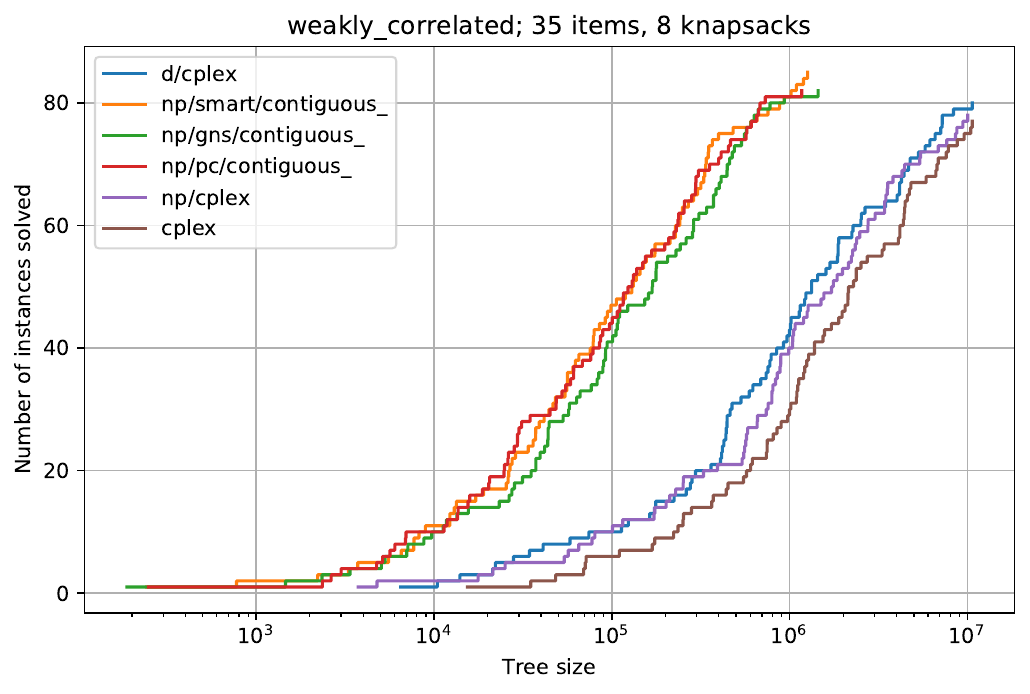}
  \label{fig:blah}
\end{subfigure}
\begin{subfigure}
  \centering
  \includegraphics[width=.39\linewidth]{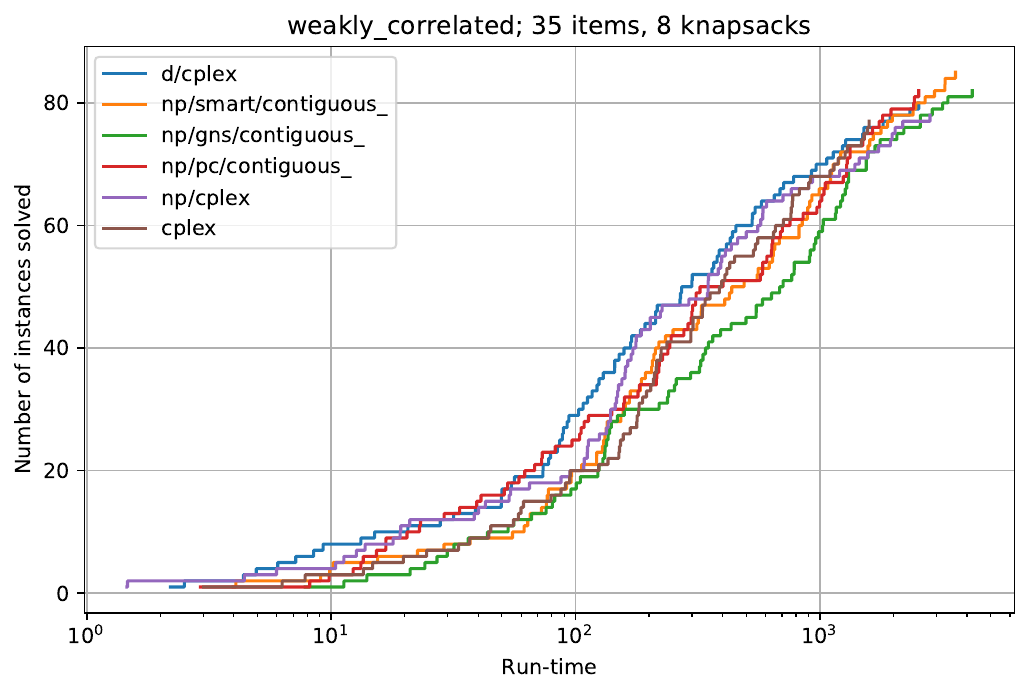}
\end{subfigure}
\begin{subfigure}
  \centering
  \includegraphics[width=.39\linewidth]{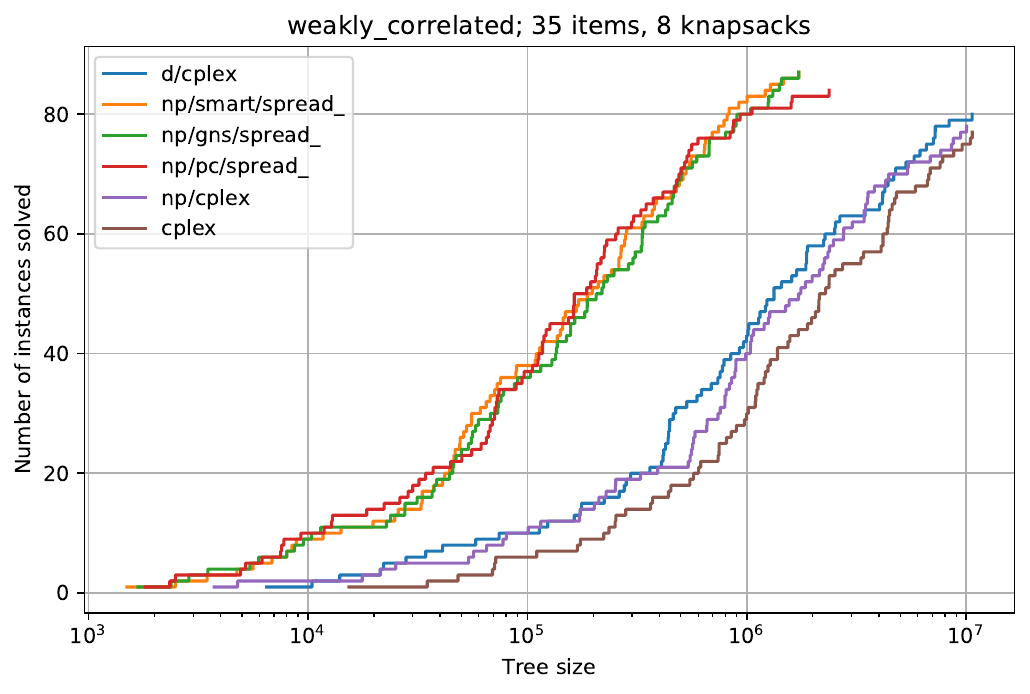}
  \label{fig:blah}
\end{subfigure}
\begin{subfigure}
  \centering
  \includegraphics[width=.39\linewidth]{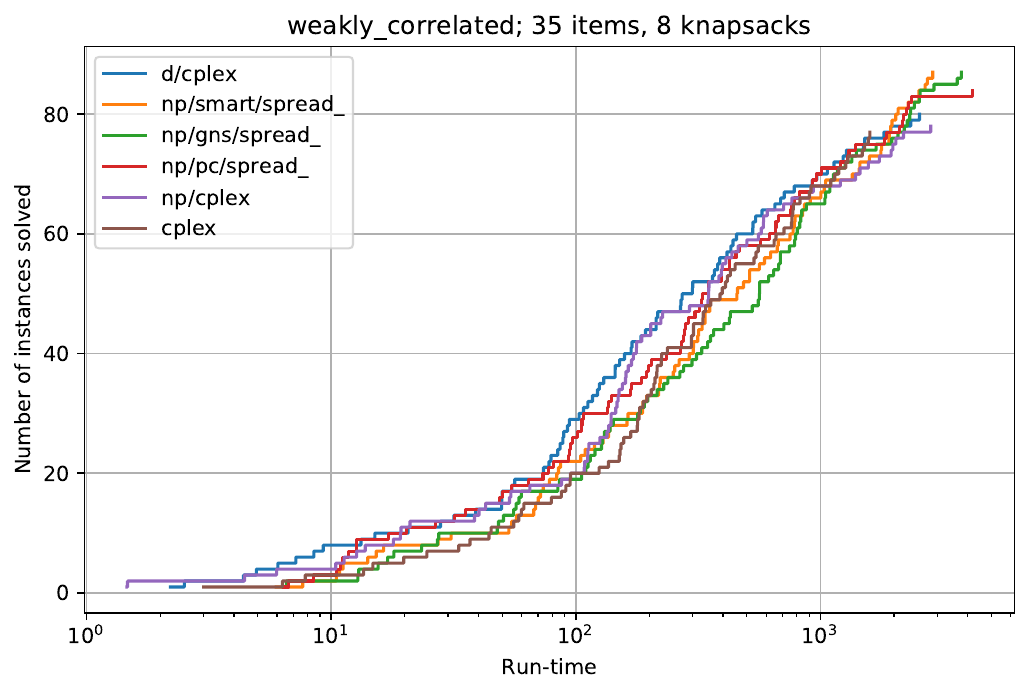}
\end{subfigure}
\begin{subfigure}
  \centering
  \includegraphics[width=.39\linewidth]{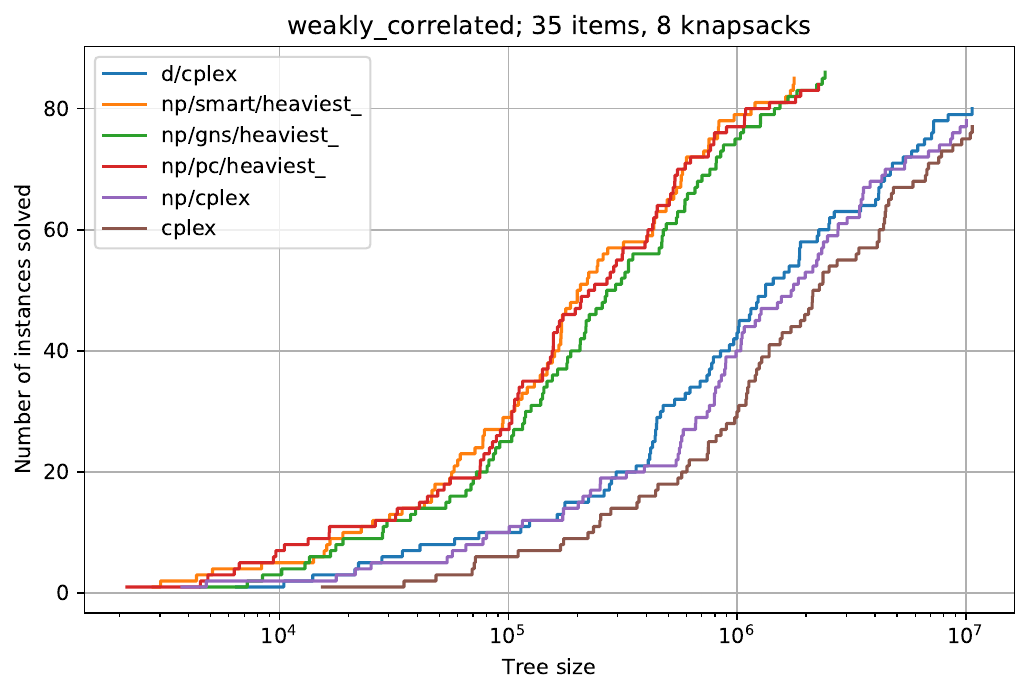}
  \label{fig:blah}
\end{subfigure}
\begin{subfigure}
  \centering
  \includegraphics[width=.39\linewidth]{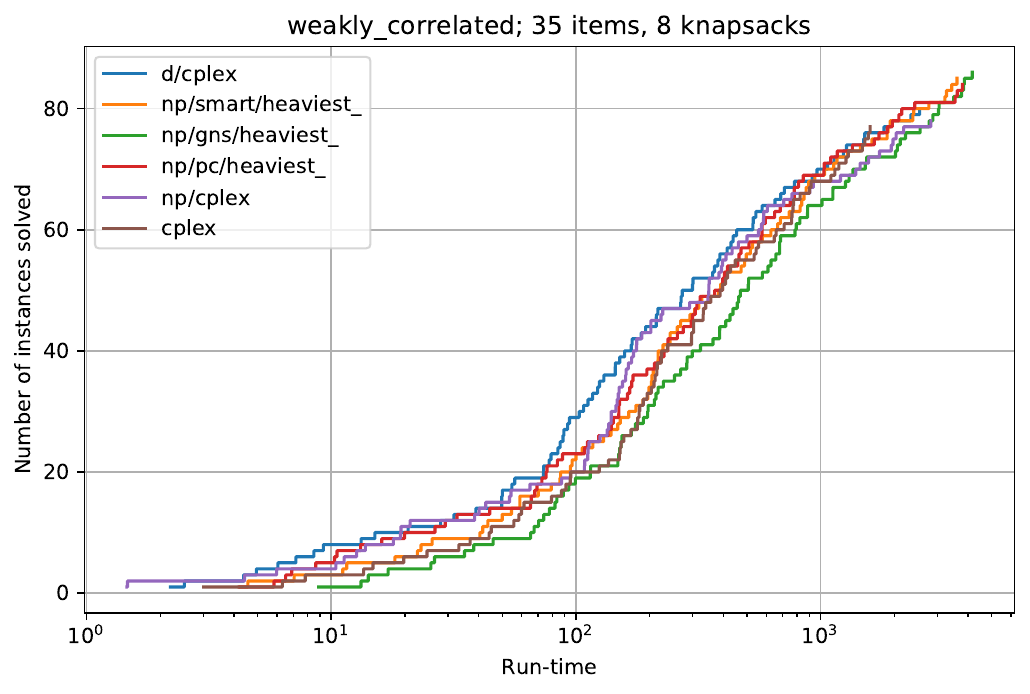}
\end{subfigure}
\begin{subfigure}
  \centering
  \includegraphics[width=.39\linewidth]{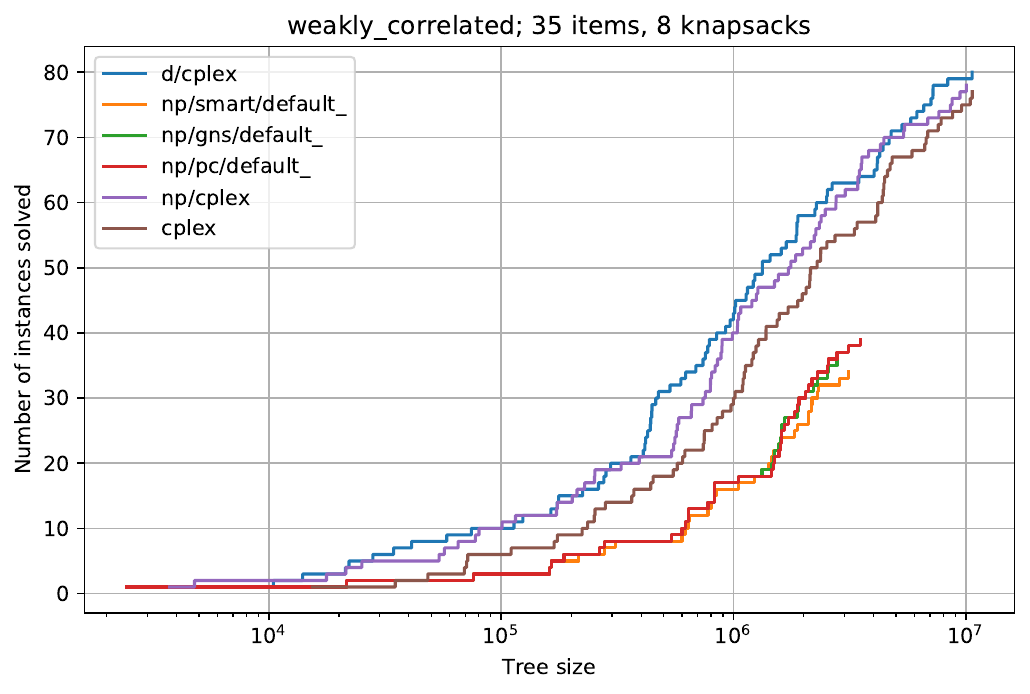}
  \label{fig:blah}
\end{subfigure}
\begin{subfigure}
  \centering
  \includegraphics[width=.39\linewidth]{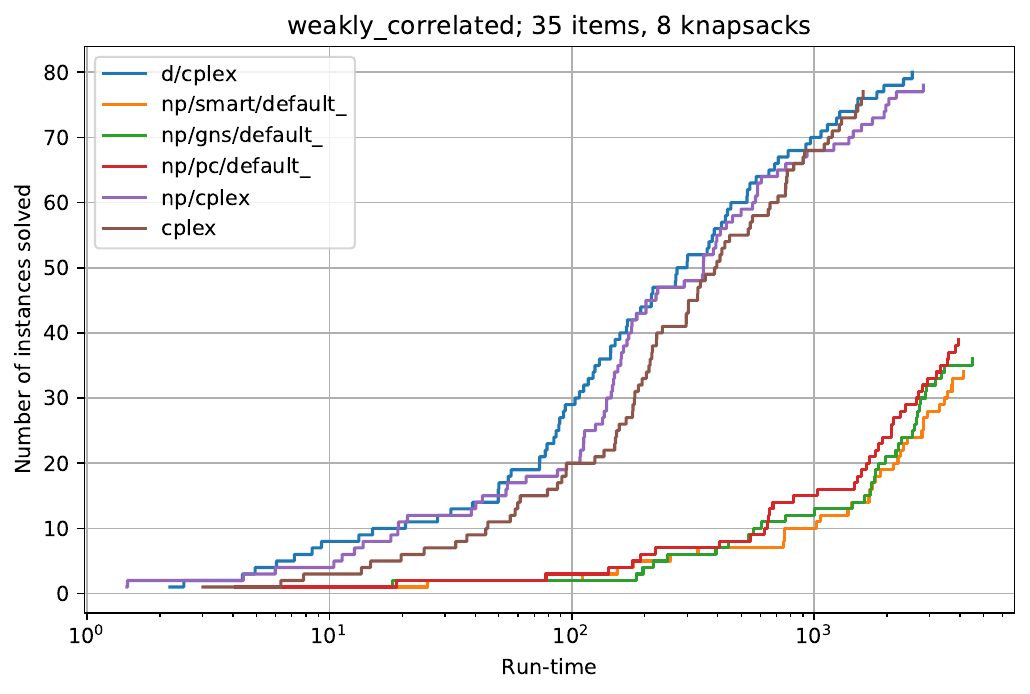}
\end{subfigure}
\begin{subfigure}
  \centering
  \includegraphics[width=.39\linewidth]{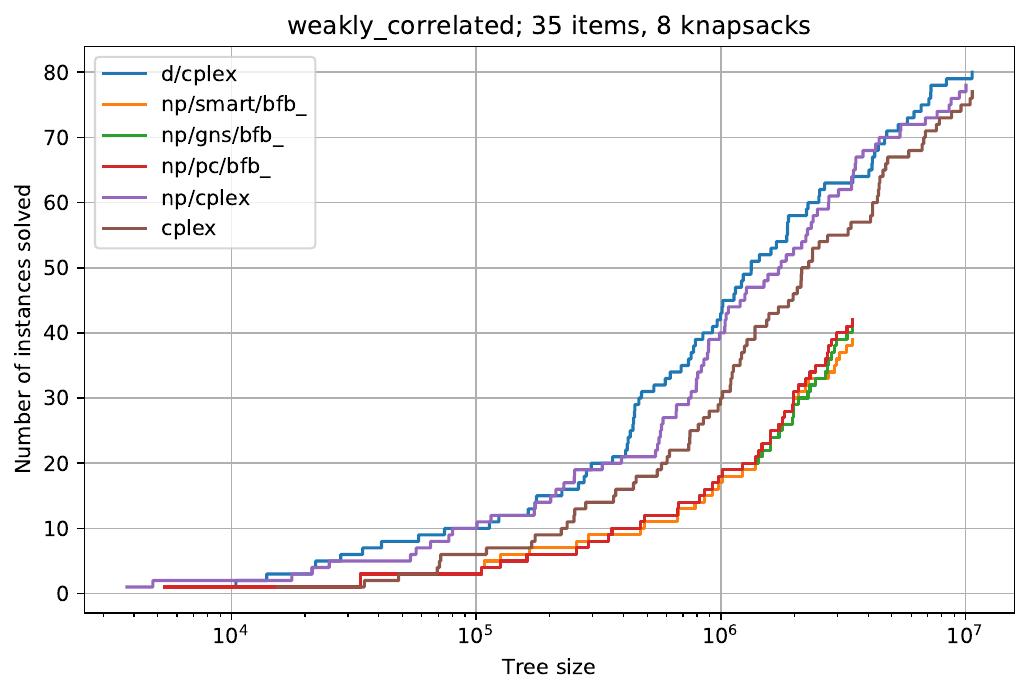}
  \label{fig:blah}
\end{subfigure}
\begin{subfigure}
  \centering
  \includegraphics[width=.39\linewidth]{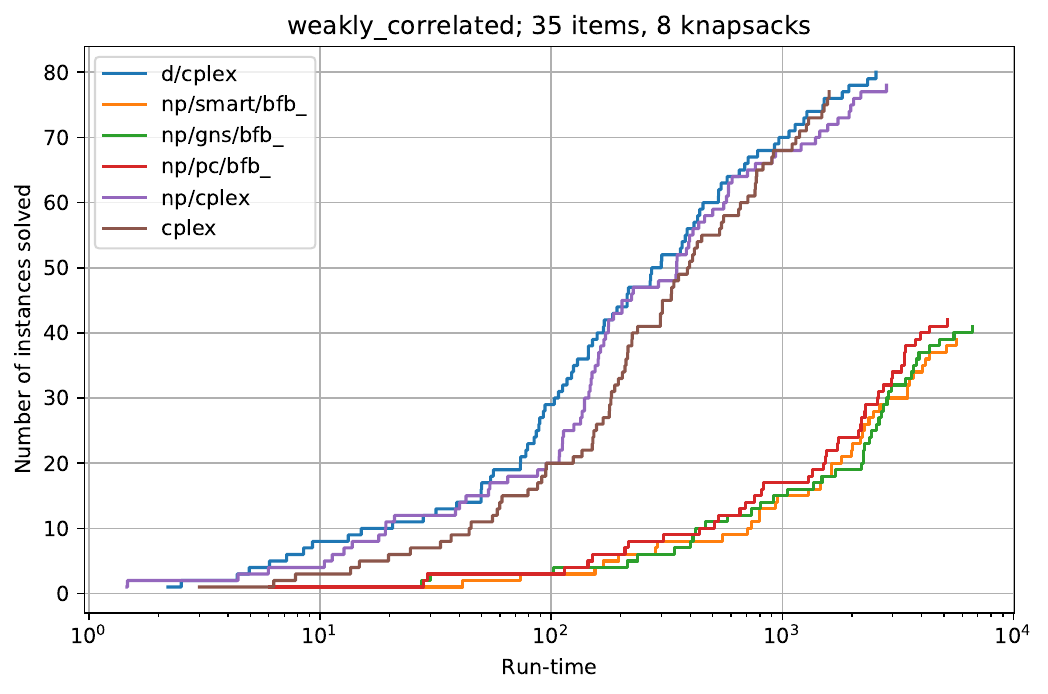}
\end{subfigure}

\caption{Weakly correlated, CPLEX cover cuts off, all other parameters but presolve on}
\label{fig:mkp_weakly_correlatedno_presolve_}
\end{figure}

\begin{figure}[t]
\centering
\begin{subfigure}
  \centering
  \includegraphics[width=.39\linewidth]{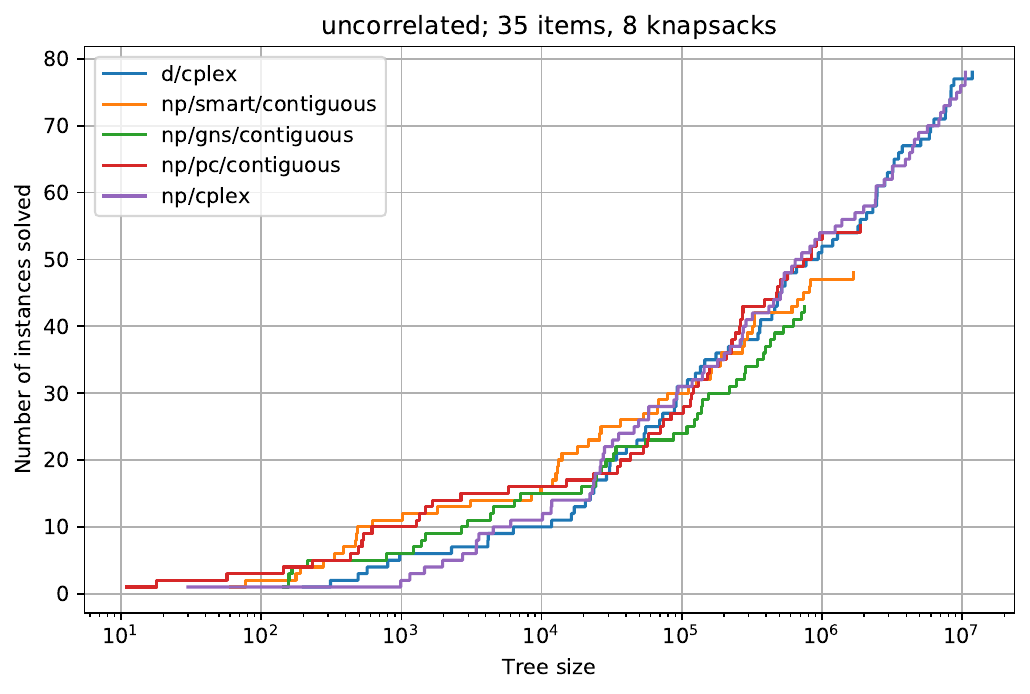}
  \label{fig:blah}
\end{subfigure}
\begin{subfigure}
  \centering
  \includegraphics[width=.39\linewidth]{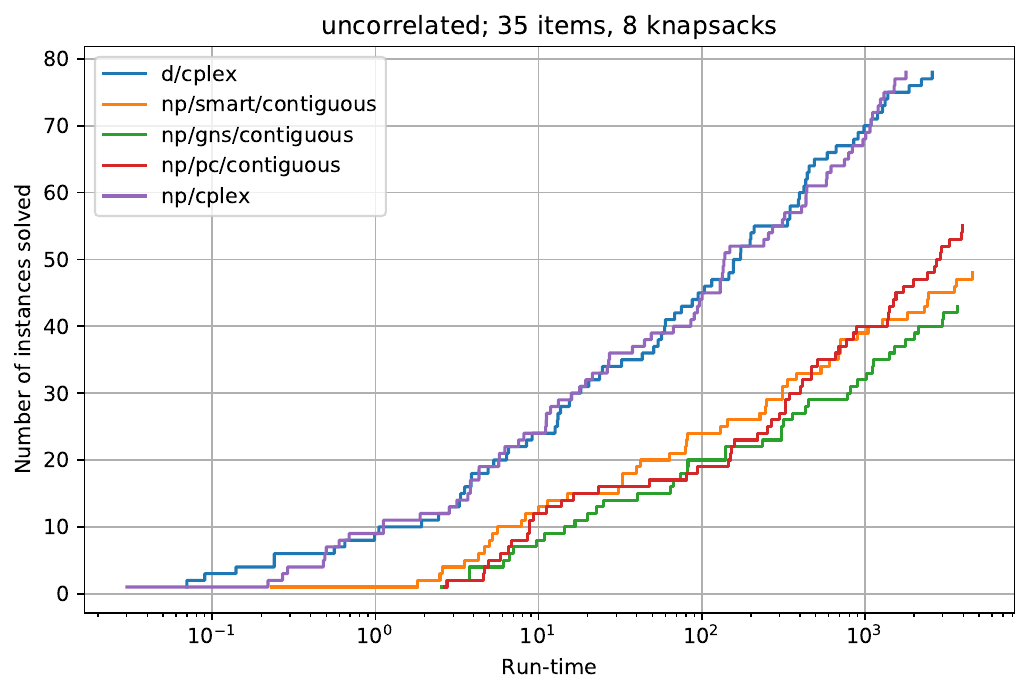}
\end{subfigure}
\begin{subfigure}
  \centering
  \includegraphics[width=.39\linewidth]{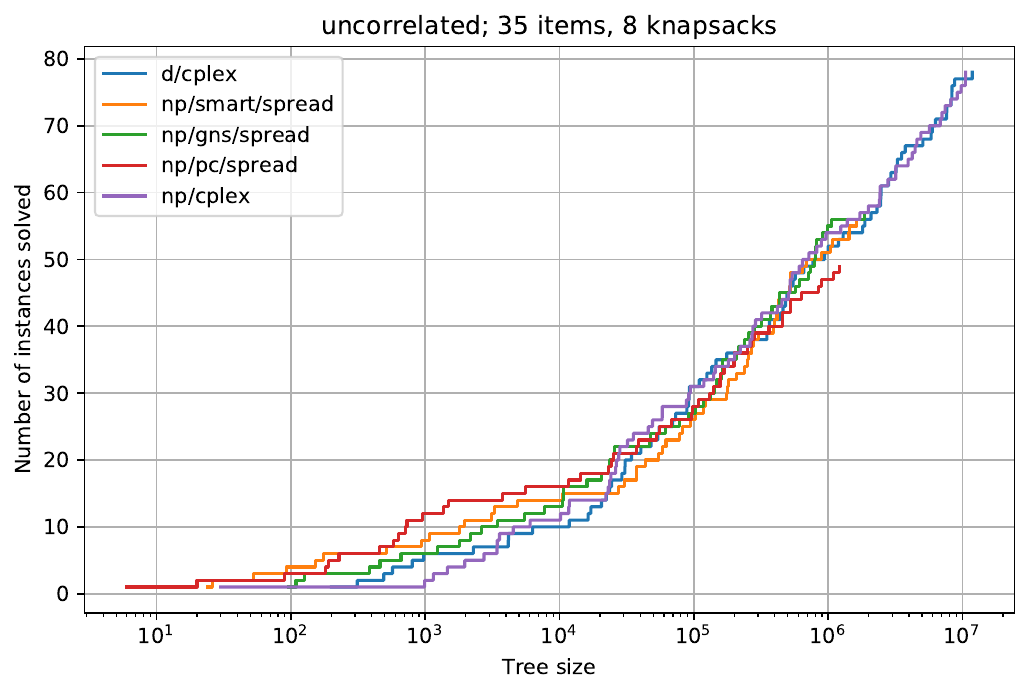}
  \label{fig:blah}
\end{subfigure}
\begin{subfigure}
  \centering
  \includegraphics[width=.39\linewidth]{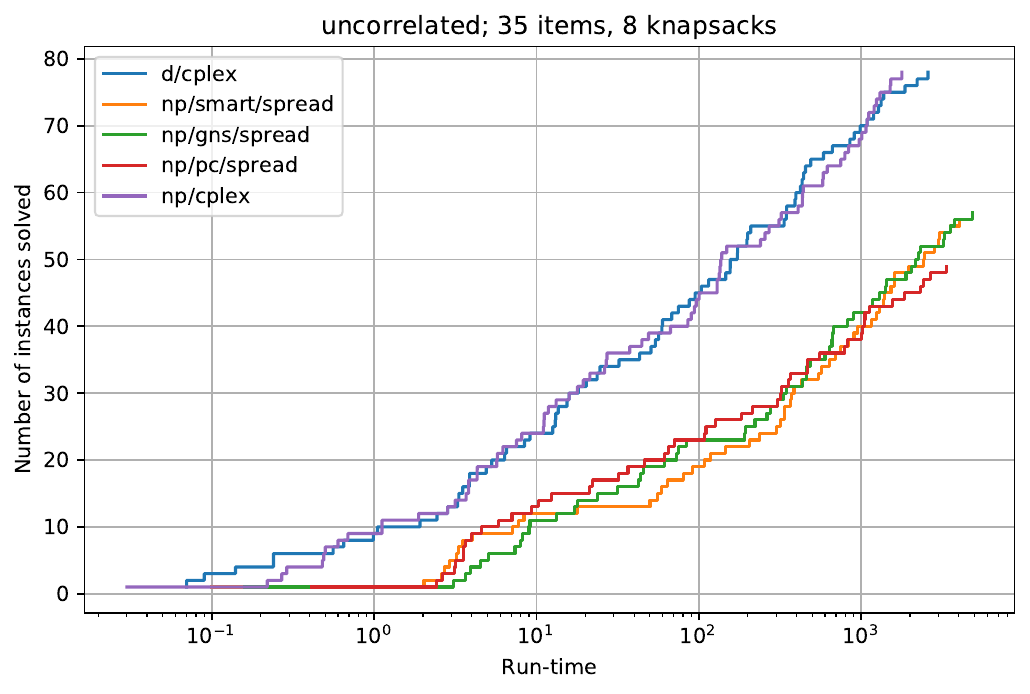}
\end{subfigure}
\begin{subfigure}
  \centering
  \includegraphics[width=.39\linewidth]{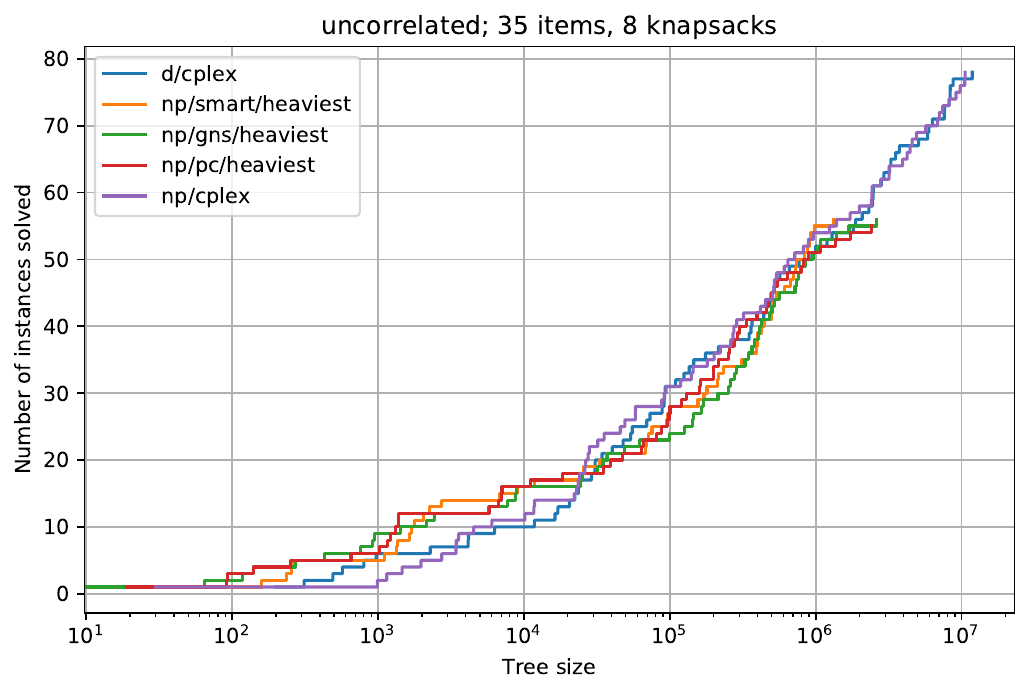}
  \label{fig:blah}
\end{subfigure}
\begin{subfigure}
  \centering
  \includegraphics[width=.39\linewidth]{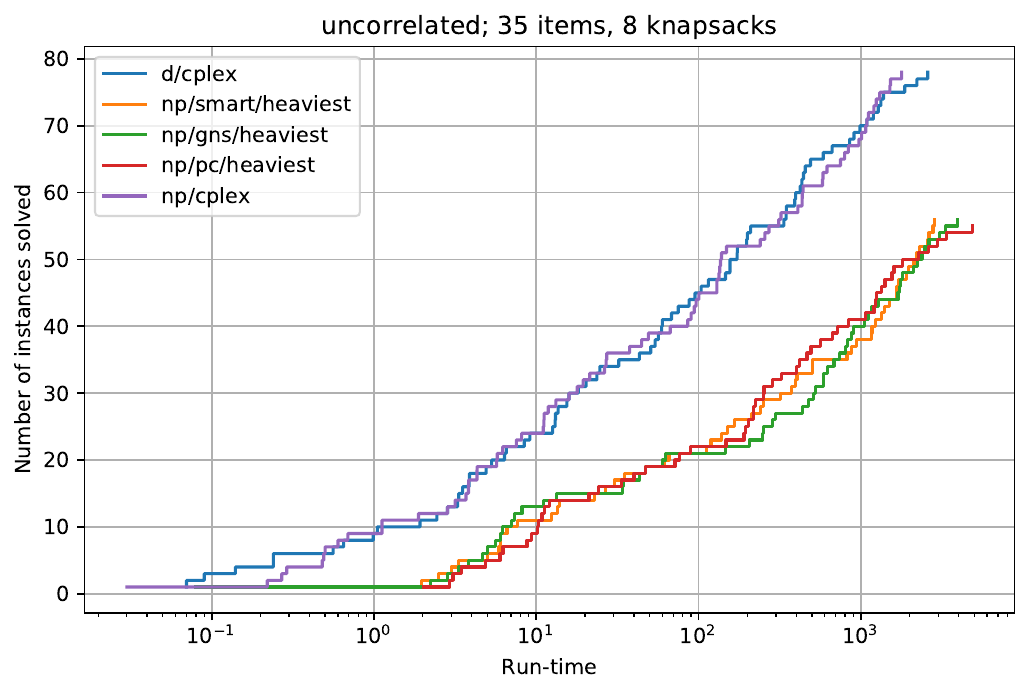}
\end{subfigure}
\begin{subfigure}
  \centering
  \includegraphics[width=.39\linewidth]{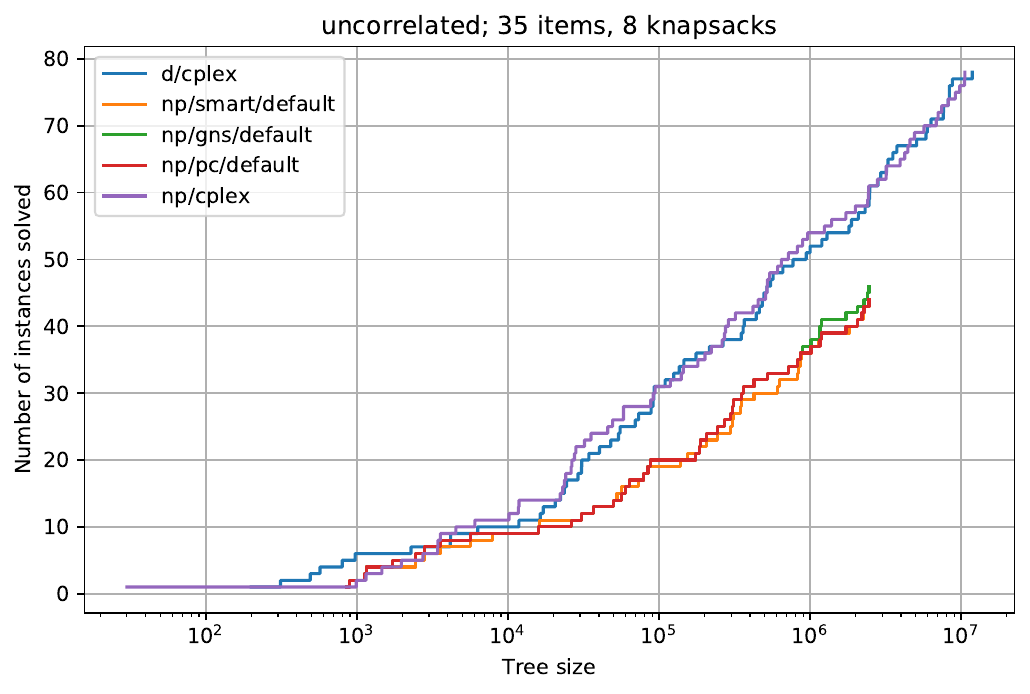}
  \label{fig:blah}
\end{subfigure}
\begin{subfigure}
  \centering
  \includegraphics[width=.39\linewidth]{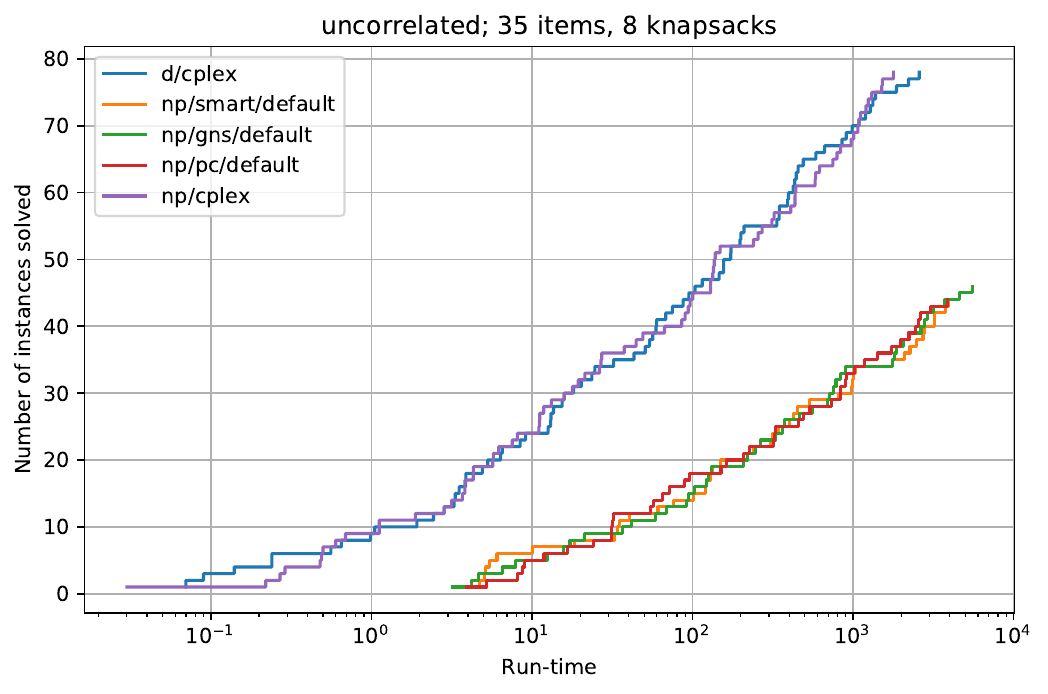}
\end{subfigure}
\begin{subfigure}
  \centering
  \includegraphics[width=.39\linewidth]{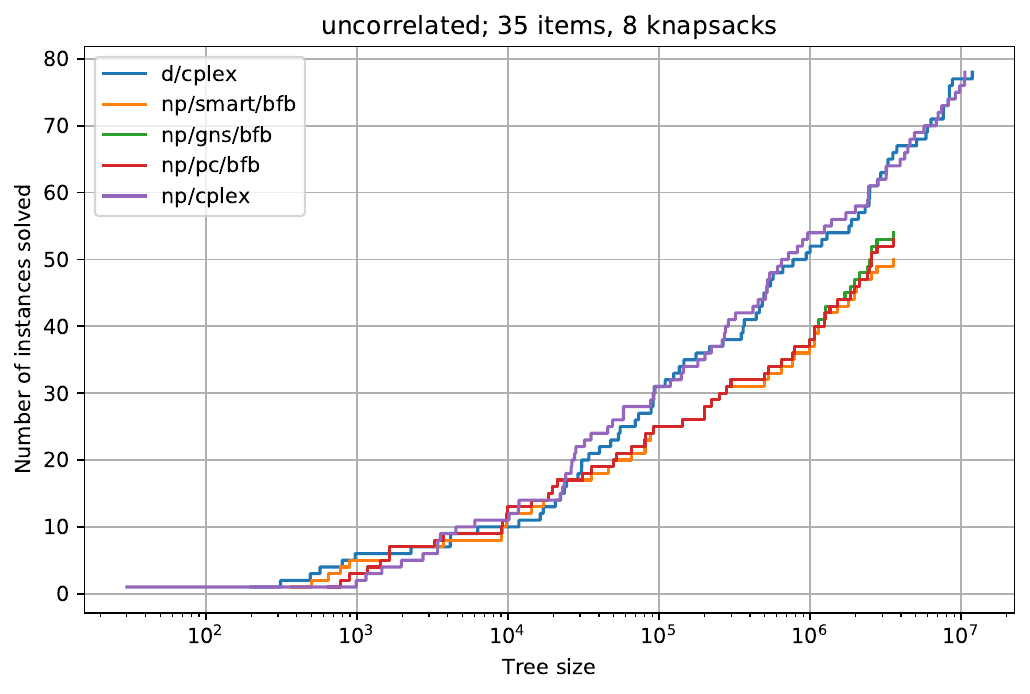}
  \label{fig:blah}
\end{subfigure}
\begin{subfigure}
  \centering
  \includegraphics[width=.39\linewidth]{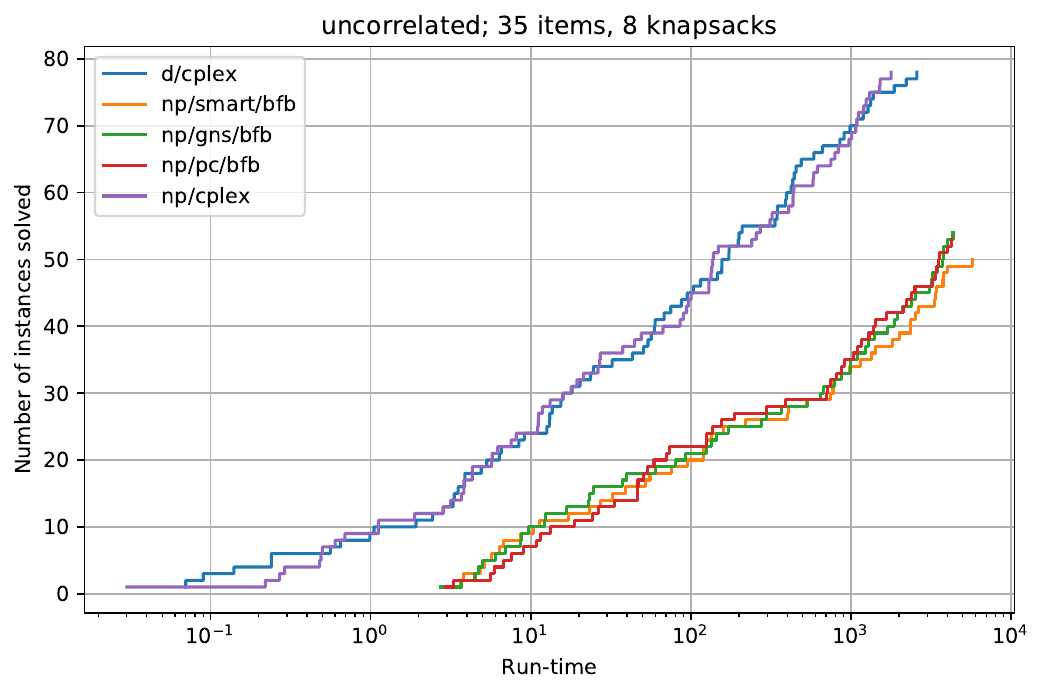}
\end{subfigure}

\caption{Uncorrelated, CPLEX cover cuts on, all other parameters but presolve on}
\label{fig:mkp_uncorrelatedno_presolve}
\end{figure}

\begin{figure}[t]
\centering
\begin{subfigure}
  \centering
  \includegraphics[width=.39\linewidth]{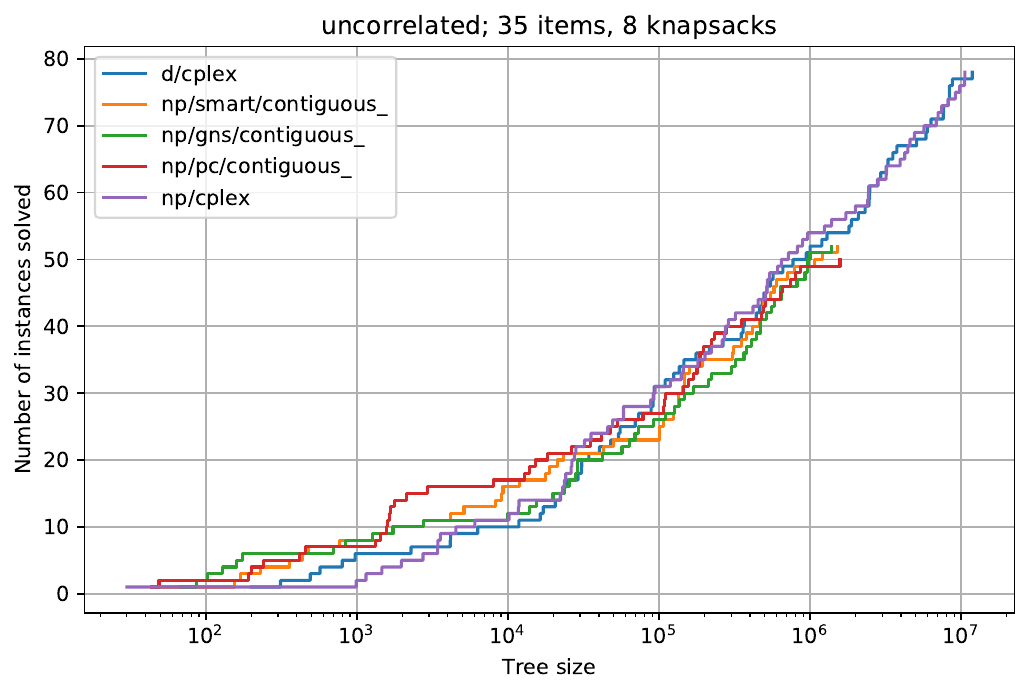}
  \label{fig:blah}
\end{subfigure}
\begin{subfigure}
  \centering
  \includegraphics[width=.39\linewidth]{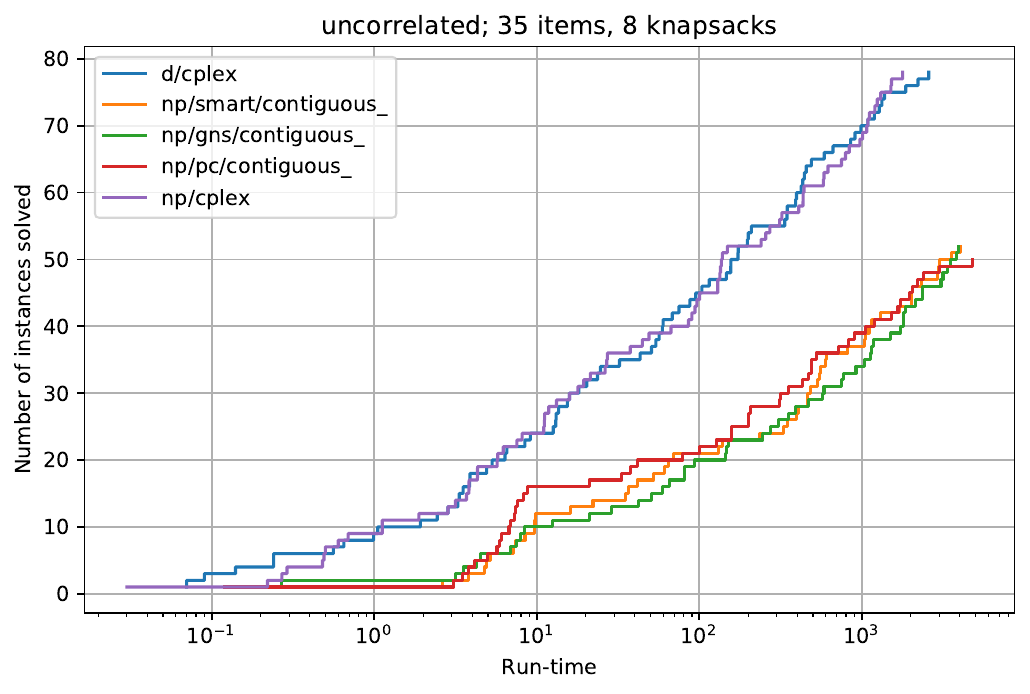}
\end{subfigure}
\begin{subfigure}
  \centering
  \includegraphics[width=.39\linewidth]{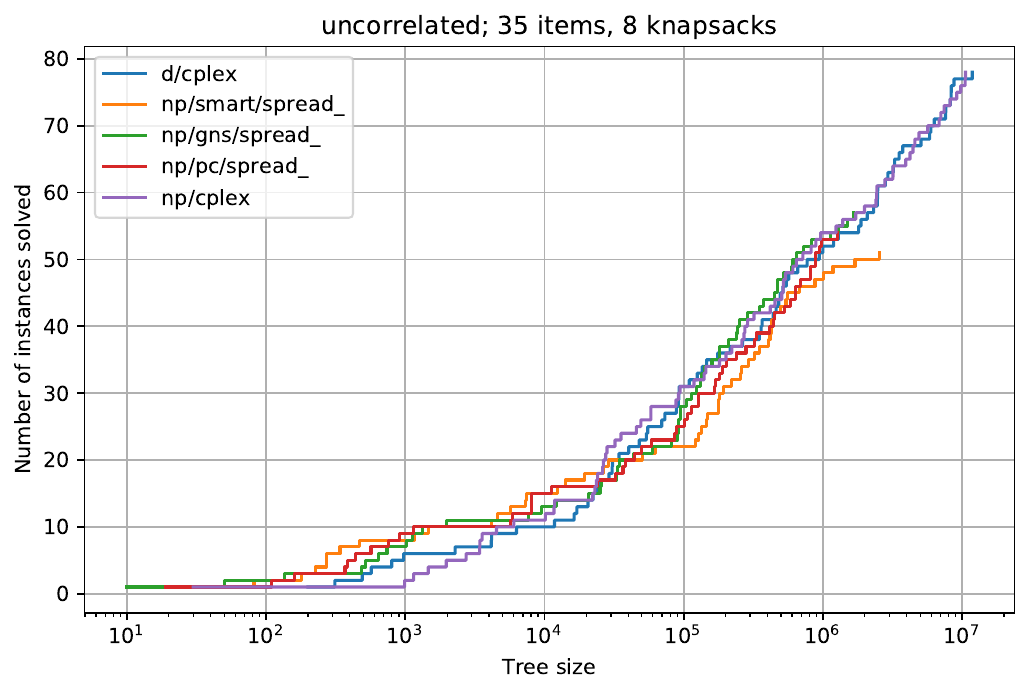}
  \label{fig:blah}
\end{subfigure}
\begin{subfigure}
  \centering
  \includegraphics[width=.39\linewidth]{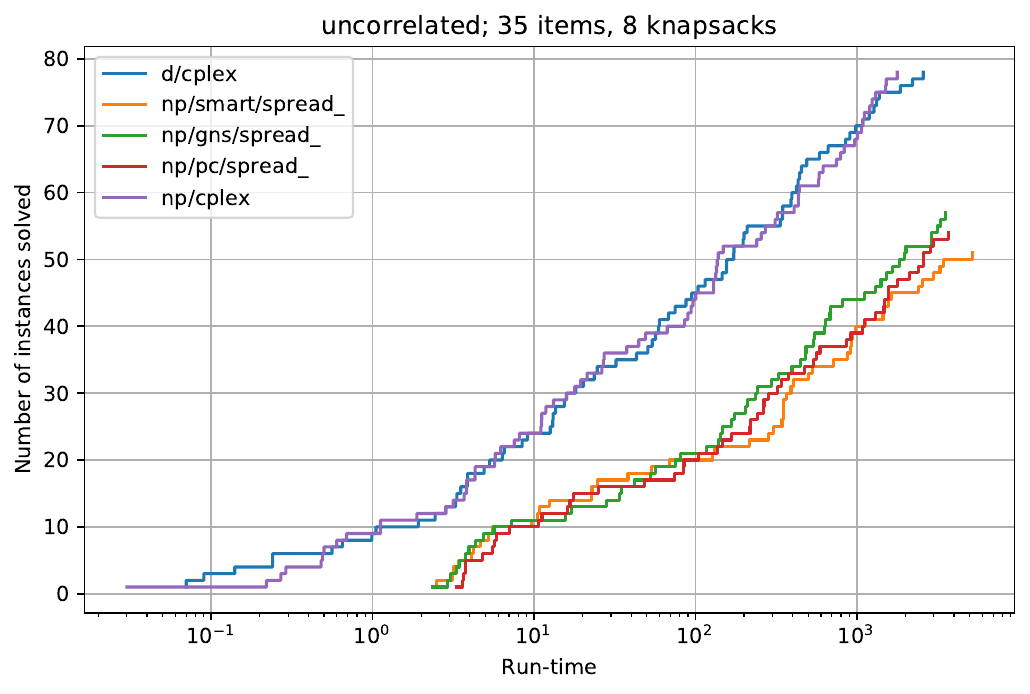}
\end{subfigure}
\begin{subfigure}
  \centering
  \includegraphics[width=.39\linewidth]{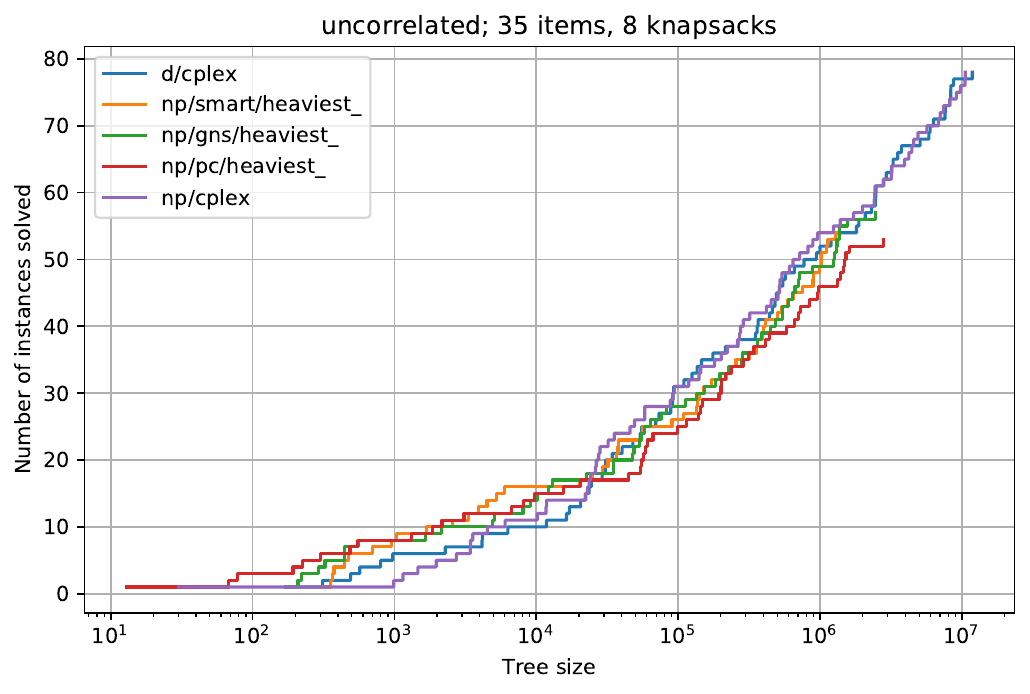}
  \label{fig:blah}
\end{subfigure}
\begin{subfigure}
  \centering
  \includegraphics[width=.39\linewidth]{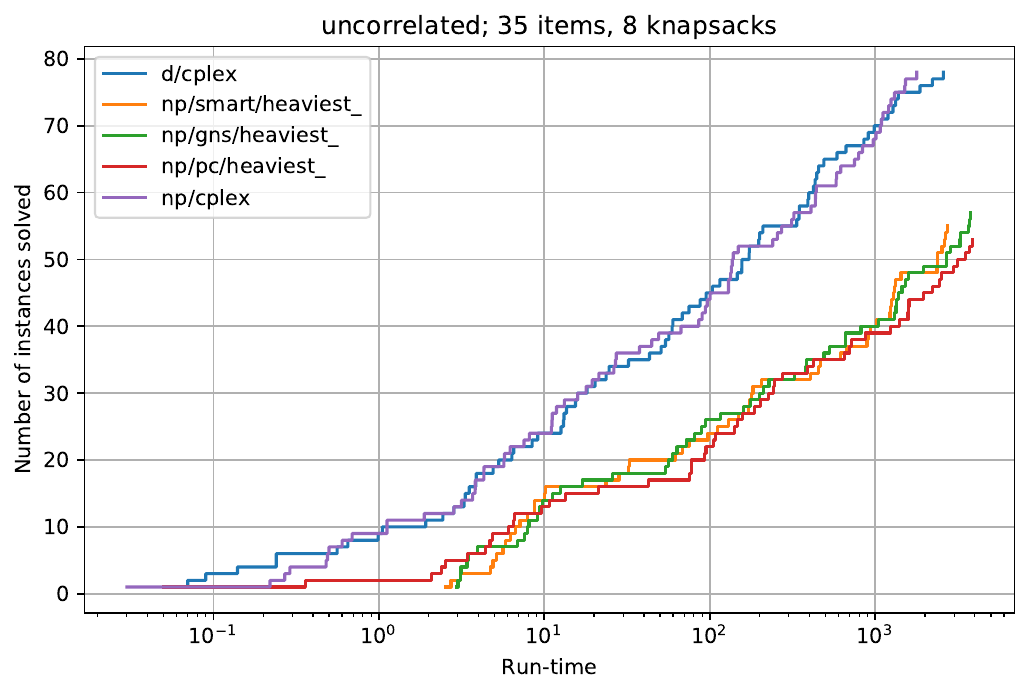}
\end{subfigure}
\begin{subfigure}
  \centering
  \includegraphics[width=.39\linewidth]{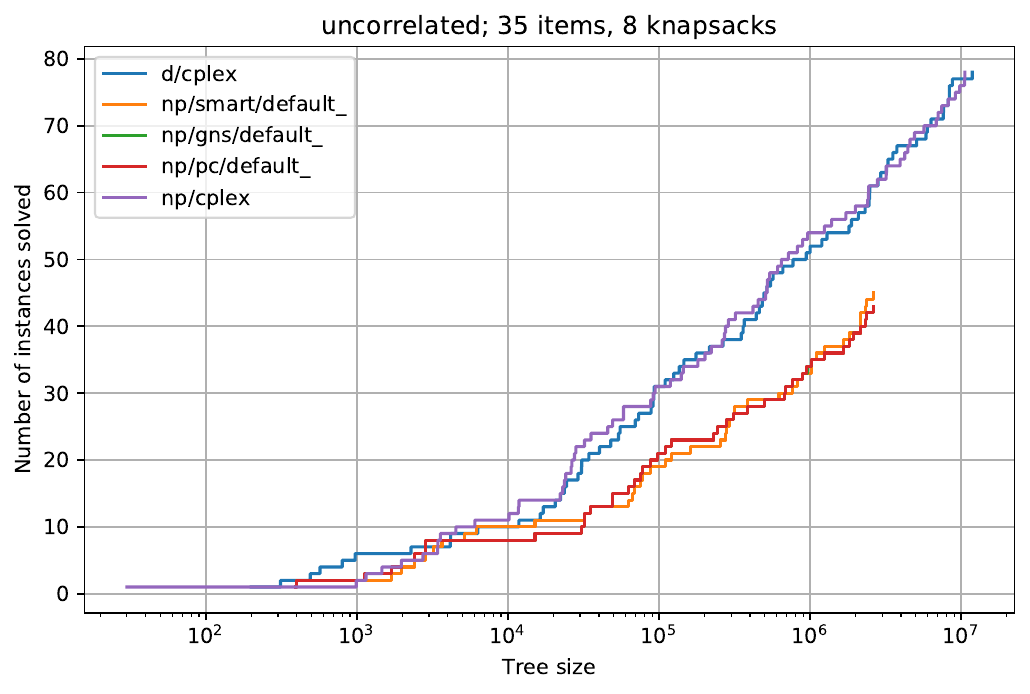}
  \label{fig:blah}
\end{subfigure}
\begin{subfigure}
  \centering
  \includegraphics[width=.39\linewidth]{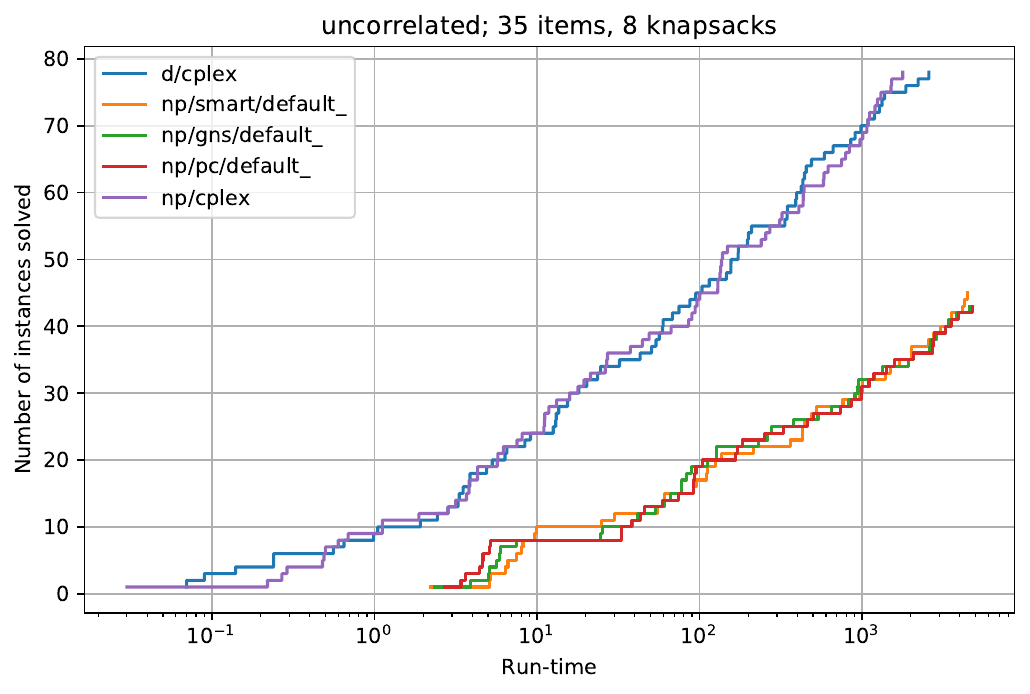}
\end{subfigure}
\begin{subfigure}
  \centering
  \includegraphics[width=.39\linewidth]{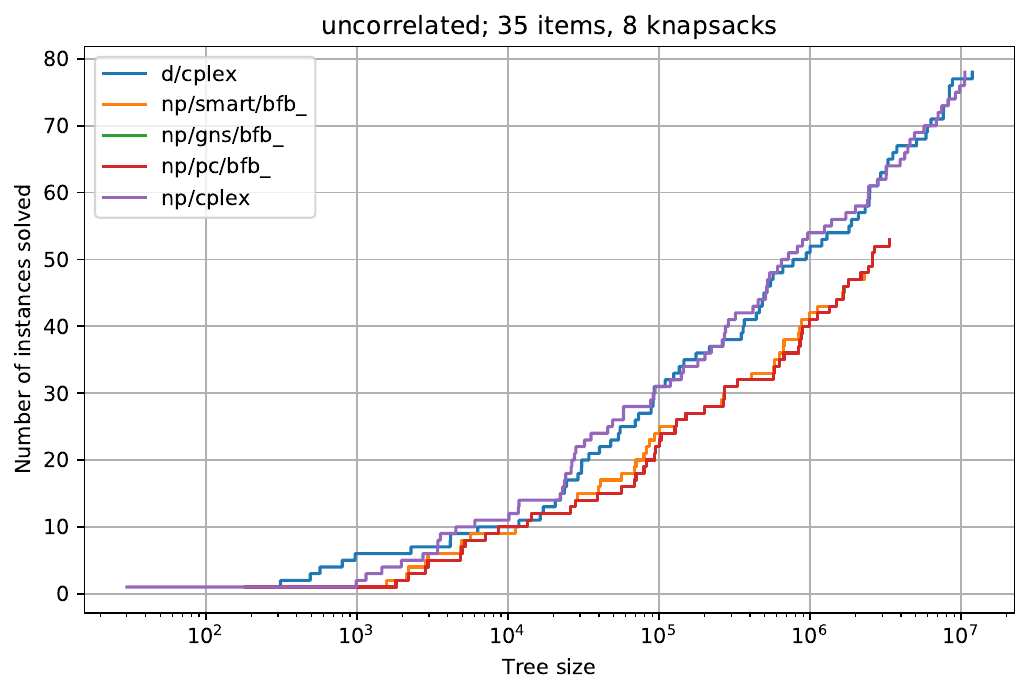}
  \label{fig:blah}
\end{subfigure}
\begin{subfigure}
  \centering
  \includegraphics[width=.39\linewidth]{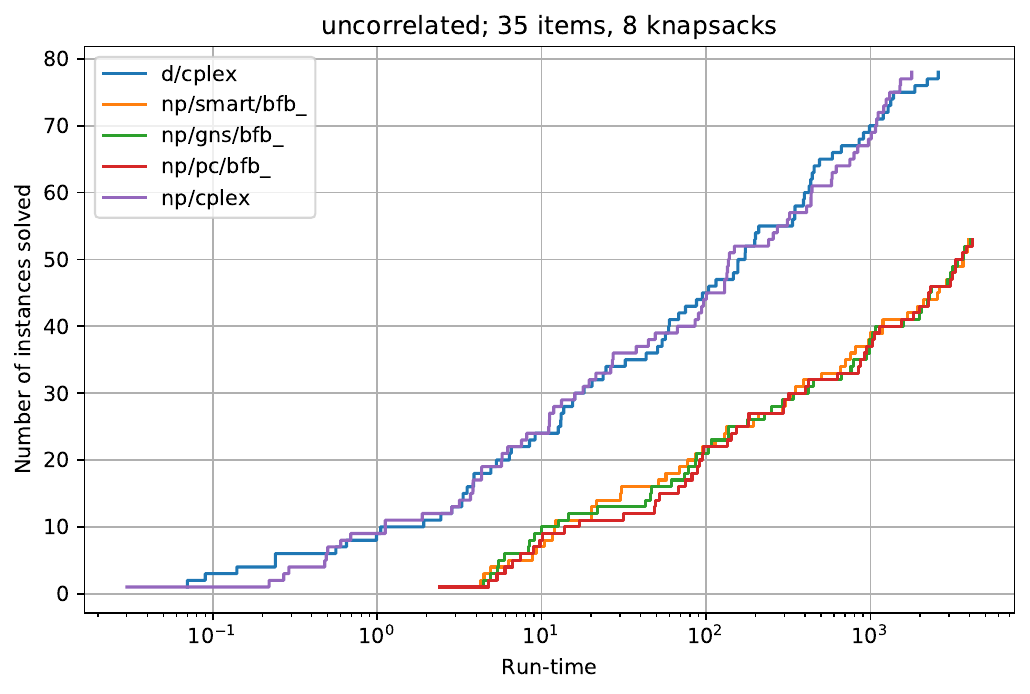}
\end{subfigure}

\caption{Uncorrelated, CPLEX cover cuts off, all other parameters but presolve on}
\label{fig:mkp_uncorrelatedno_presolve_}
\end{figure}

\begin{figure}[t]
\centering
\begin{subfigure}
  \centering
  \includegraphics[width=.39\linewidth]{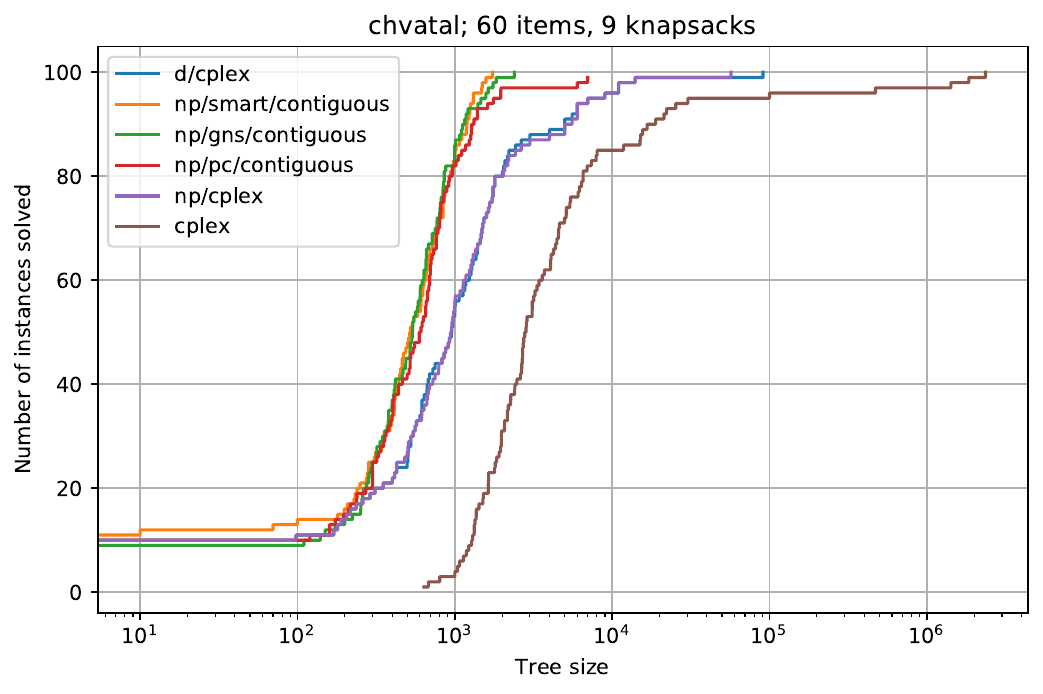}
  \label{fig:blah}
\end{subfigure}
\begin{subfigure}
  \centering
  \includegraphics[width=.39\linewidth]{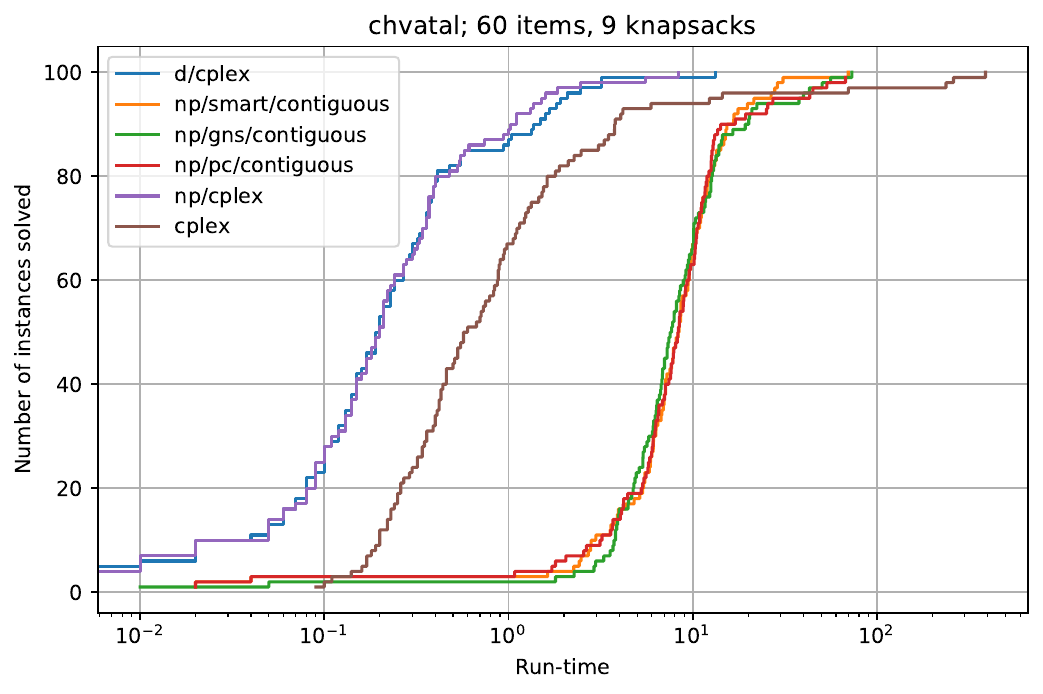}
\end{subfigure}
\begin{subfigure}
  \centering
  \includegraphics[width=.39\linewidth]{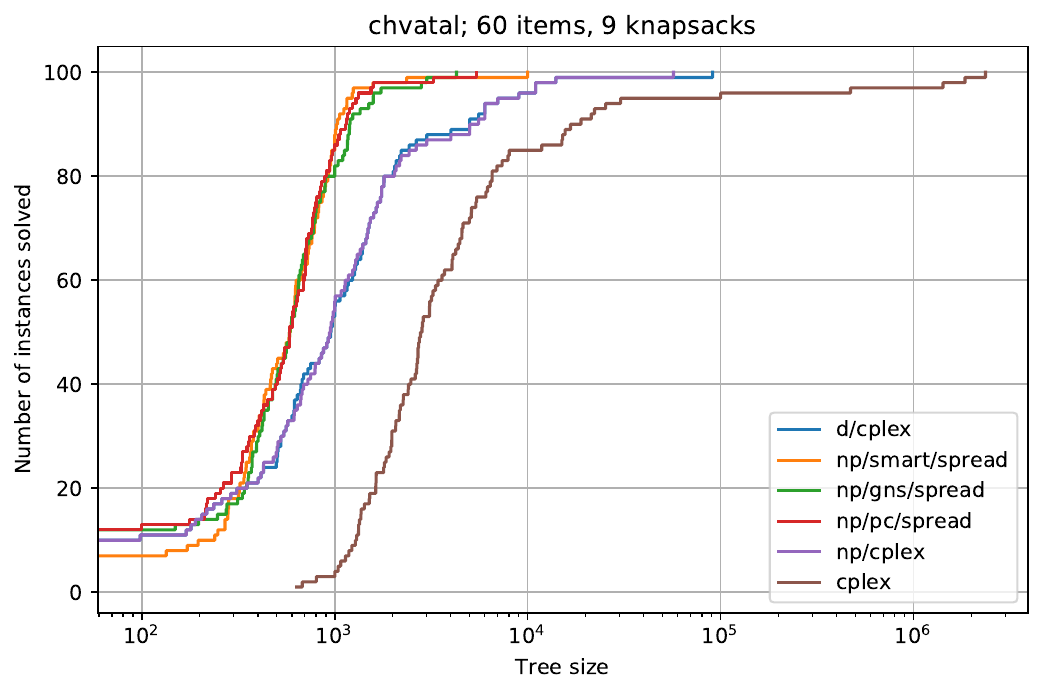}
  \label{fig:blah}
\end{subfigure}
\begin{subfigure}
  \centering
  \includegraphics[width=.39\linewidth]{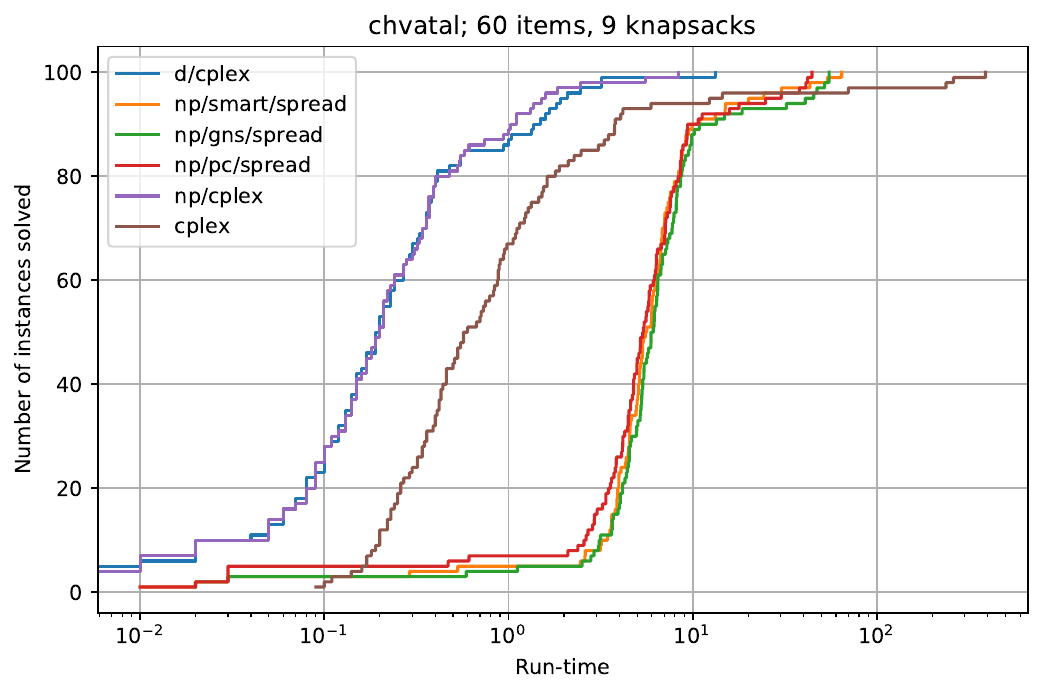}
\end{subfigure}
\begin{subfigure}
  \centering
  \includegraphics[width=.39\linewidth]{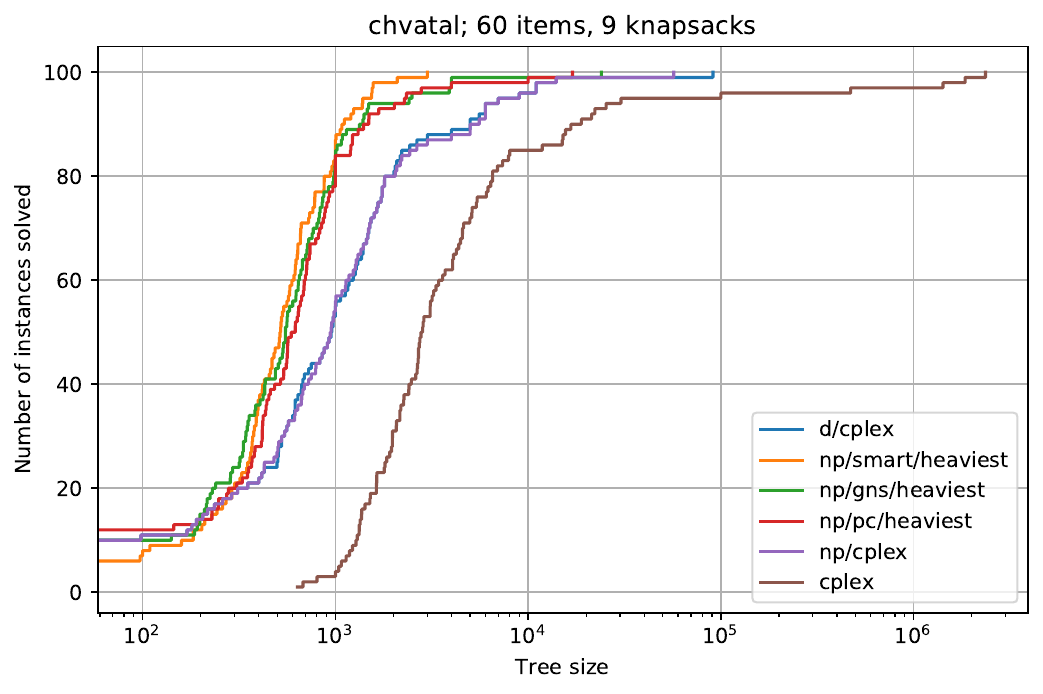}
  \label{fig:blah}
\end{subfigure}
\begin{subfigure}
  \centering
  \includegraphics[width=.39\linewidth]{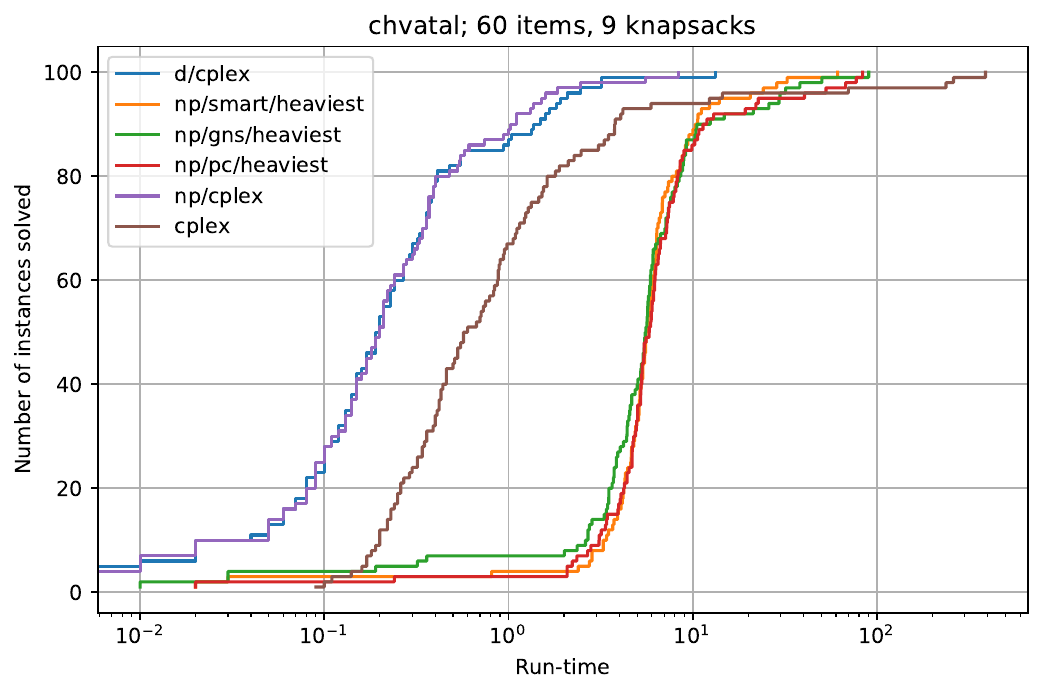}
\end{subfigure}
\begin{subfigure}
  \centering
  \includegraphics[width=.39\linewidth]{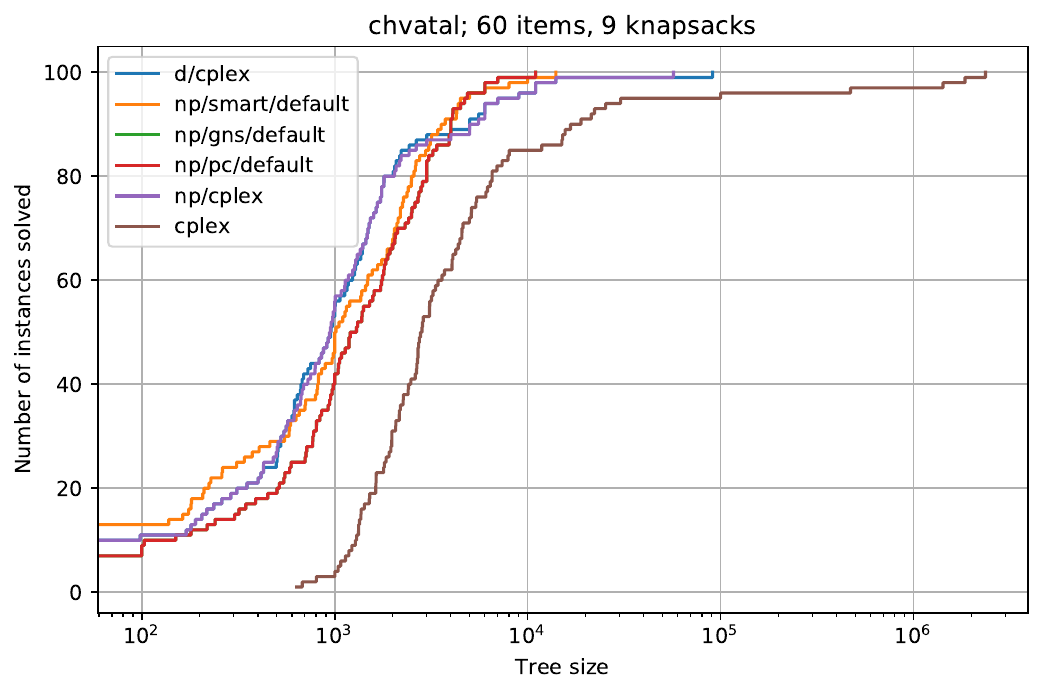}
  \label{fig:blah}
\end{subfigure}
\begin{subfigure}
  \centering
  \includegraphics[width=.39\linewidth]{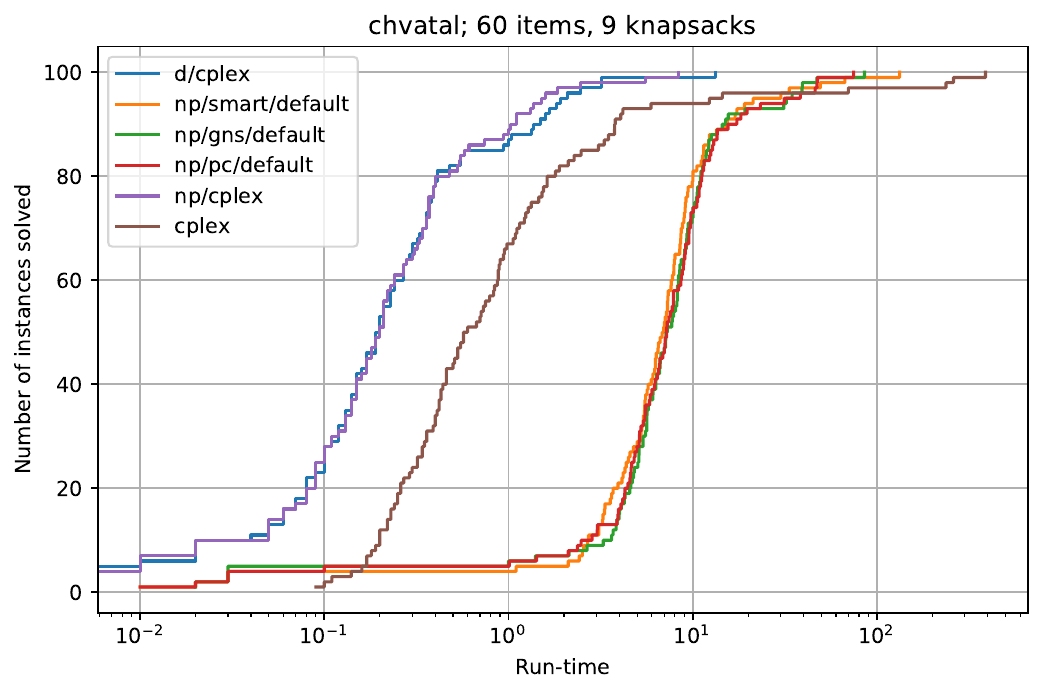}
\end{subfigure}
\begin{subfigure}
  \centering
  \includegraphics[width=.39\linewidth]{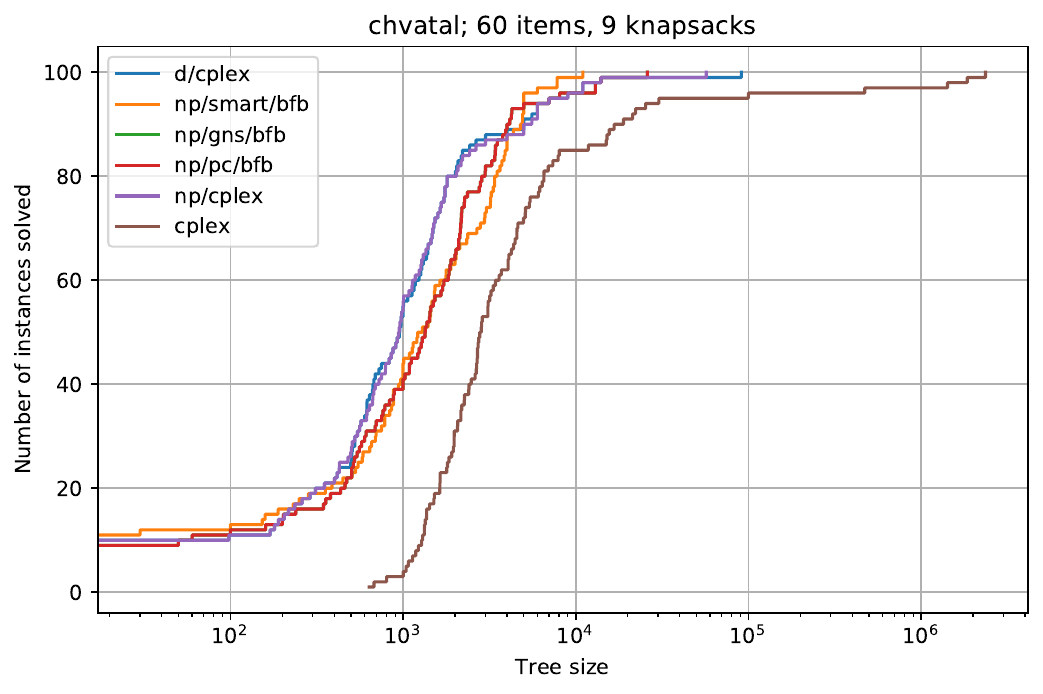}
  \label{fig:blah}
\end{subfigure}
\begin{subfigure}
  \centering
  \includegraphics[width=.39\linewidth]{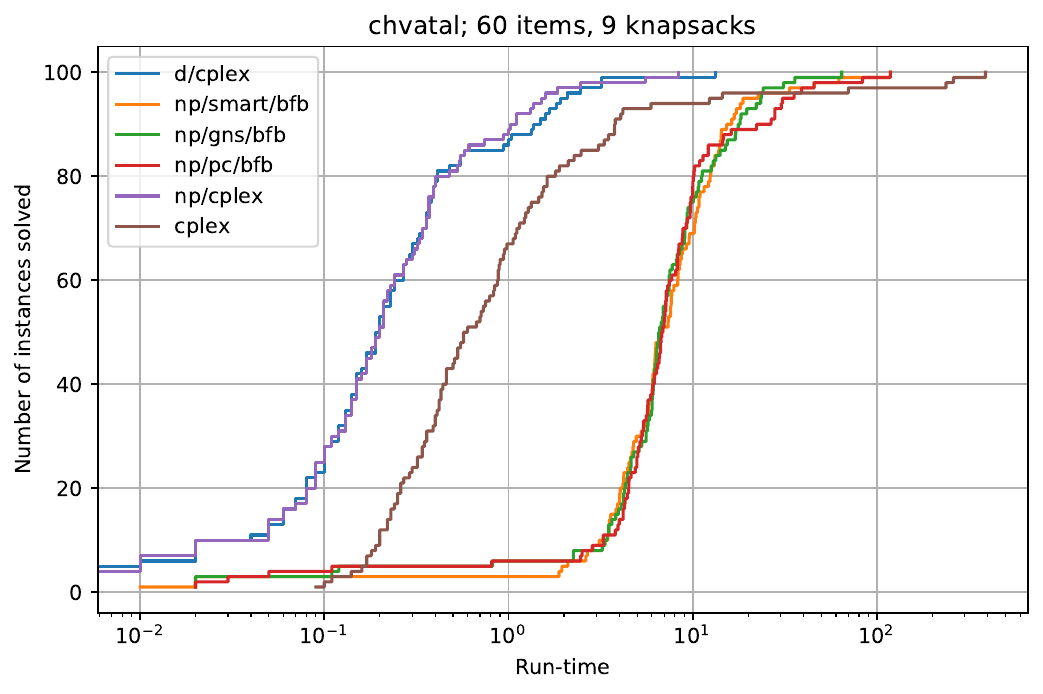}
\end{subfigure}

\caption{Chv\'{a}tal, CPLEX cover cuts on, all other parameters but presolve on}
\label{fig:chvatal_no_presolve}
\end{figure}

\begin{figure}[t]
\centering
\begin{subfigure}
  \centering
  \includegraphics[width=.39\linewidth]{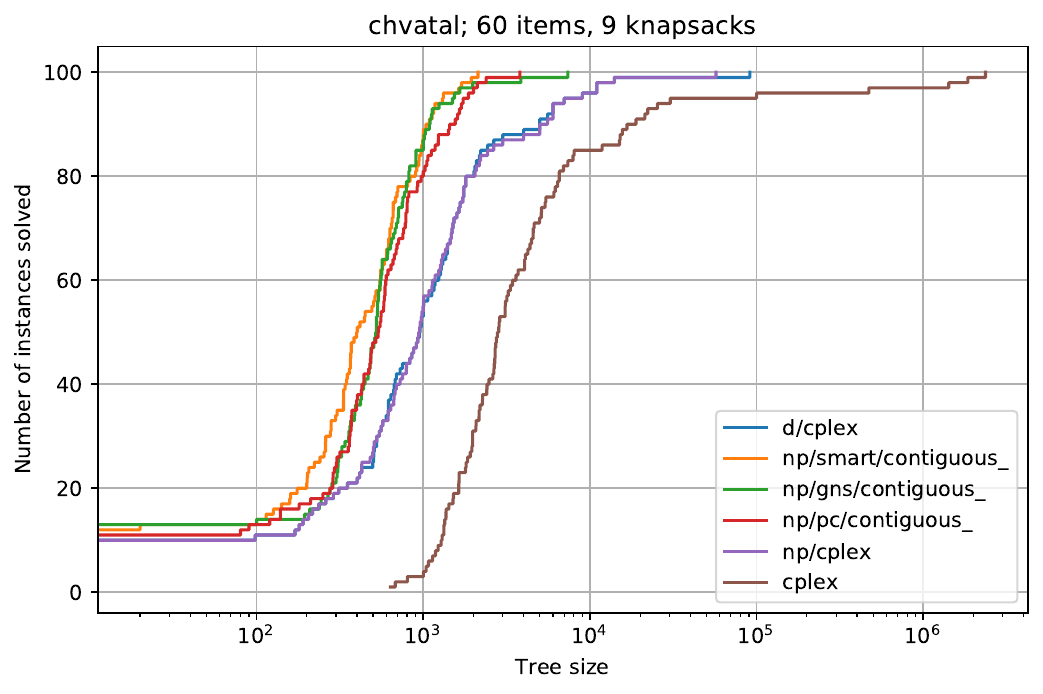}
  \label{fig:blah}
\end{subfigure}
\begin{subfigure}
  \centering
  \includegraphics[width=.39\linewidth]{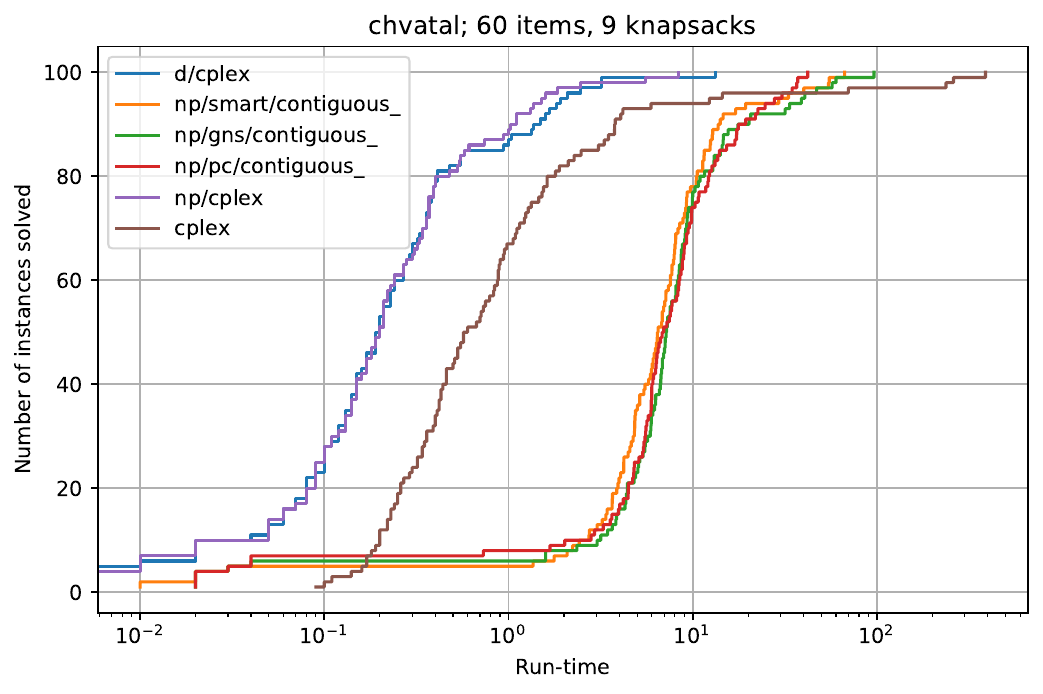}
\end{subfigure}
\begin{subfigure}
  \centering
  \includegraphics[width=.39\linewidth]{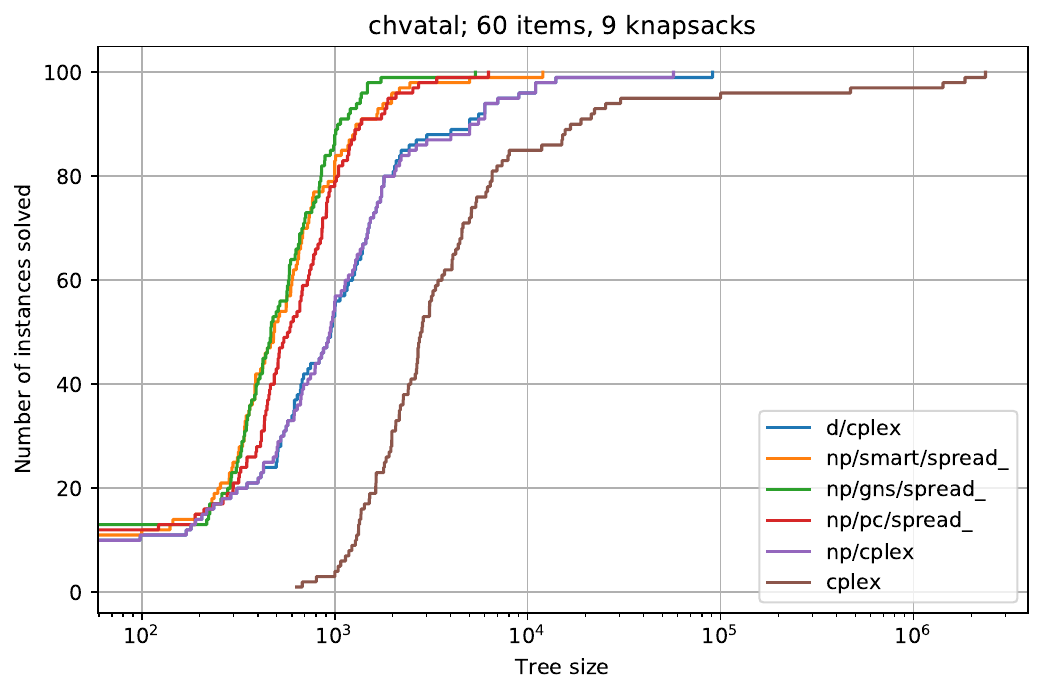}
  \label{fig:blah}
\end{subfigure}
\begin{subfigure}
  \centering
  \includegraphics[width=.39\linewidth]{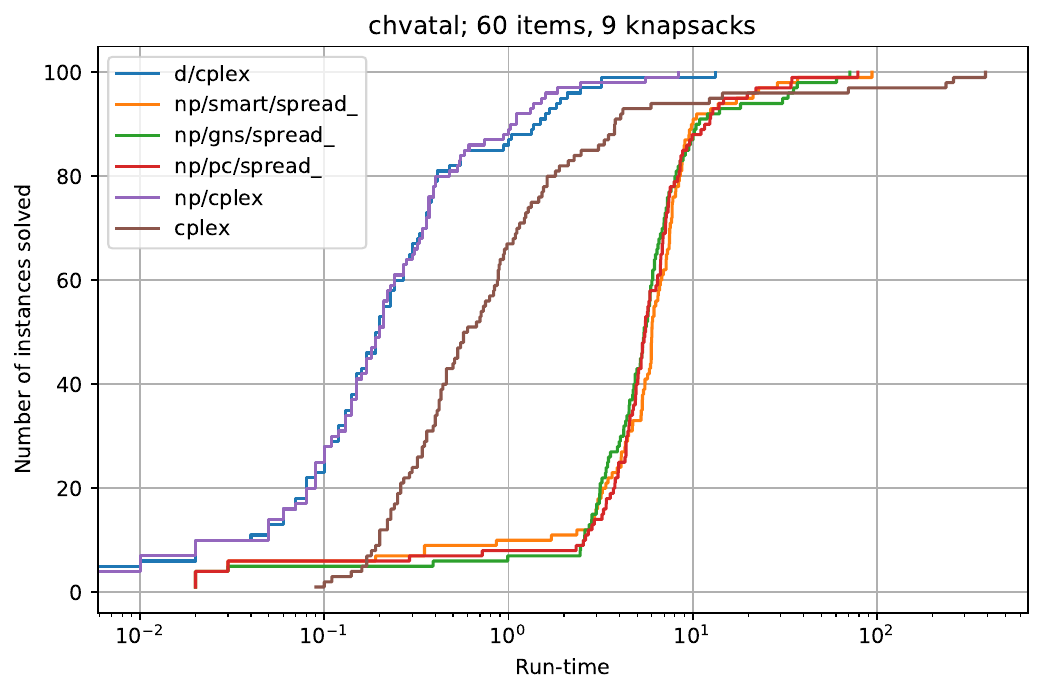}
\end{subfigure}
\begin{subfigure}
  \centering
  \includegraphics[width=.39\linewidth]{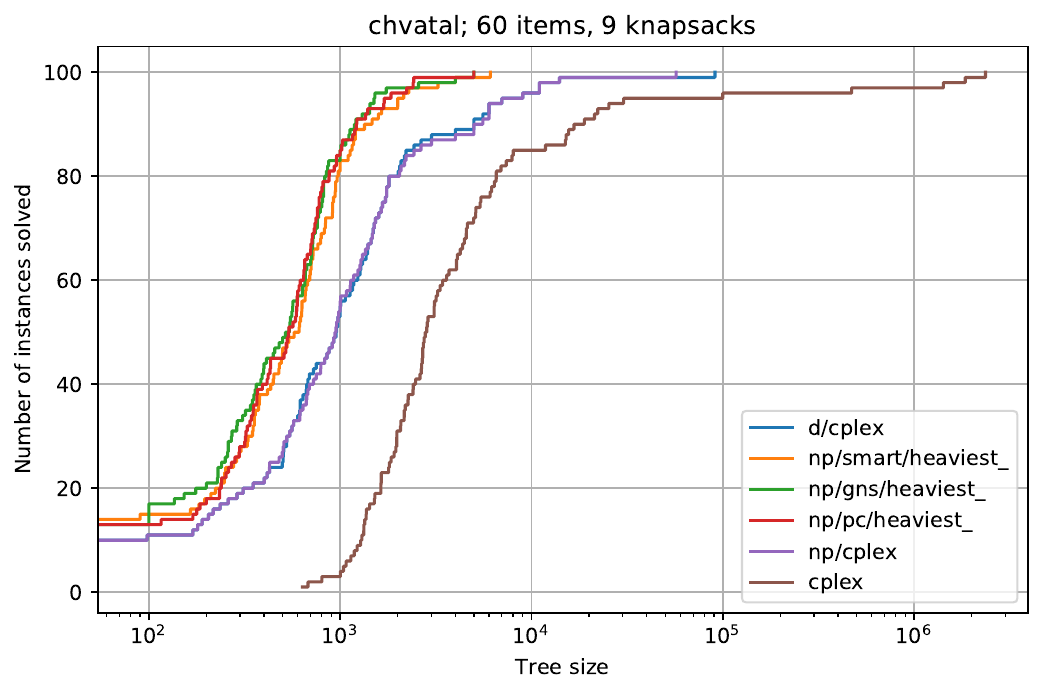}
  \label{fig:blah}
\end{subfigure}
\begin{subfigure}
  \centering
  \includegraphics[width=.39\linewidth]{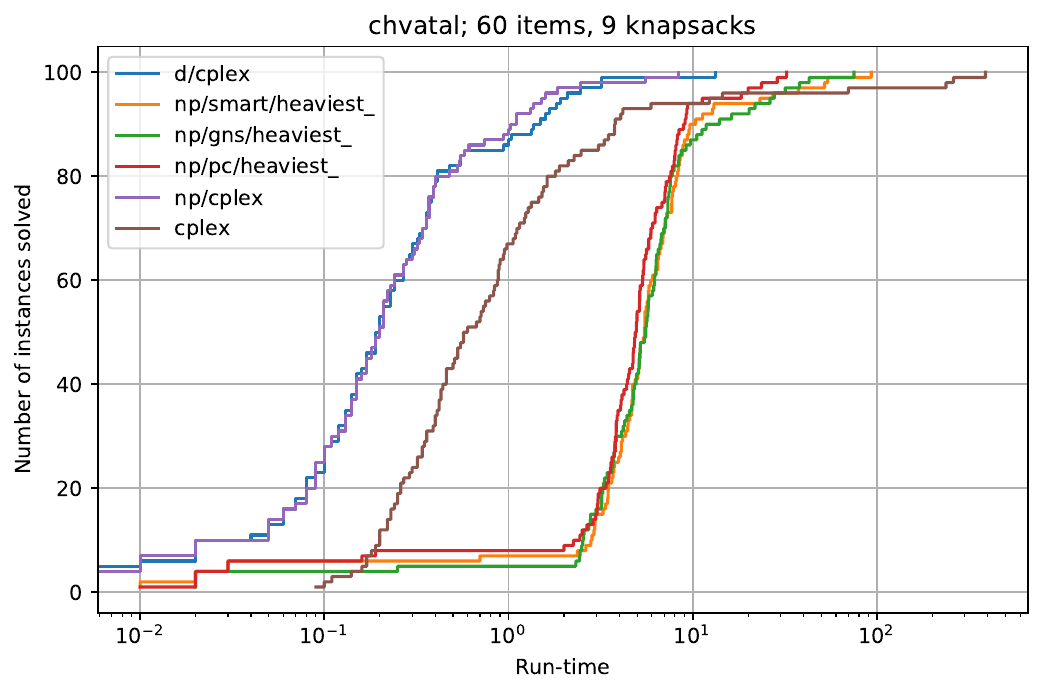}
\end{subfigure}
\begin{subfigure}
  \centering
  \includegraphics[width=.39\linewidth]{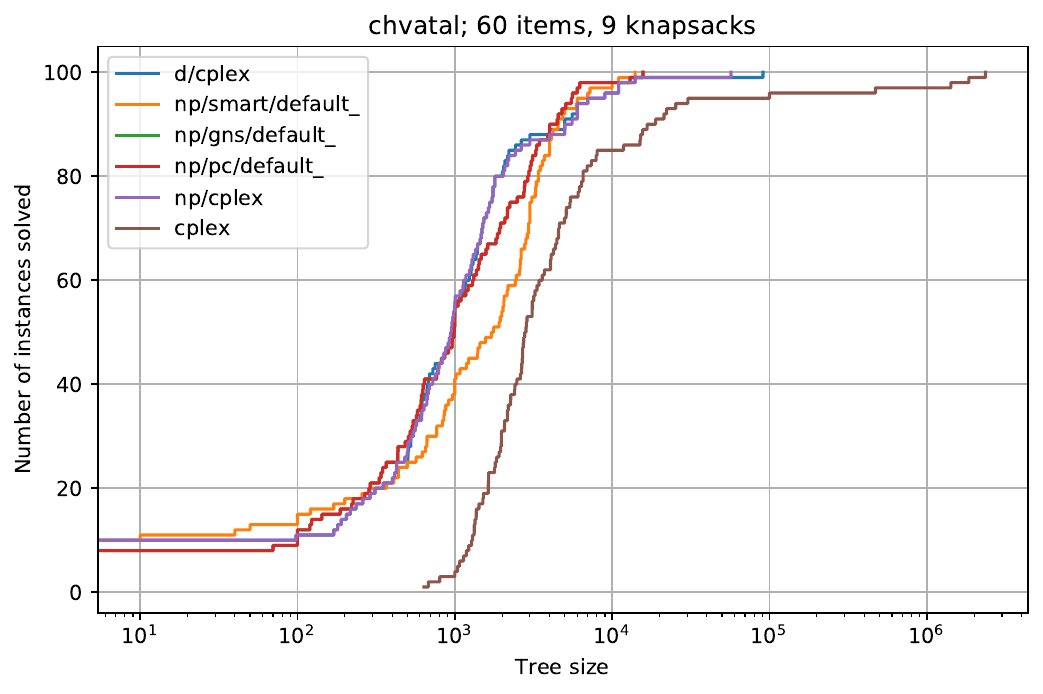}
  \label{fig:blah}
\end{subfigure}
\begin{subfigure}
  \centering
  \includegraphics[width=.39\linewidth]{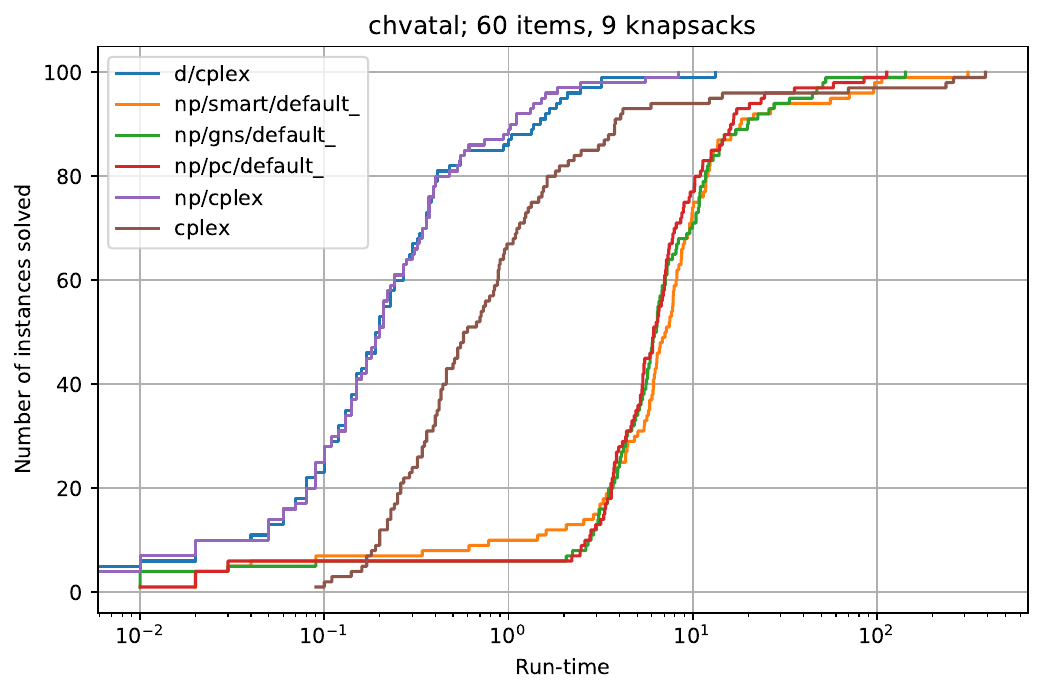}
\end{subfigure}
\begin{subfigure}
  \centering
  \includegraphics[width=.39\linewidth]{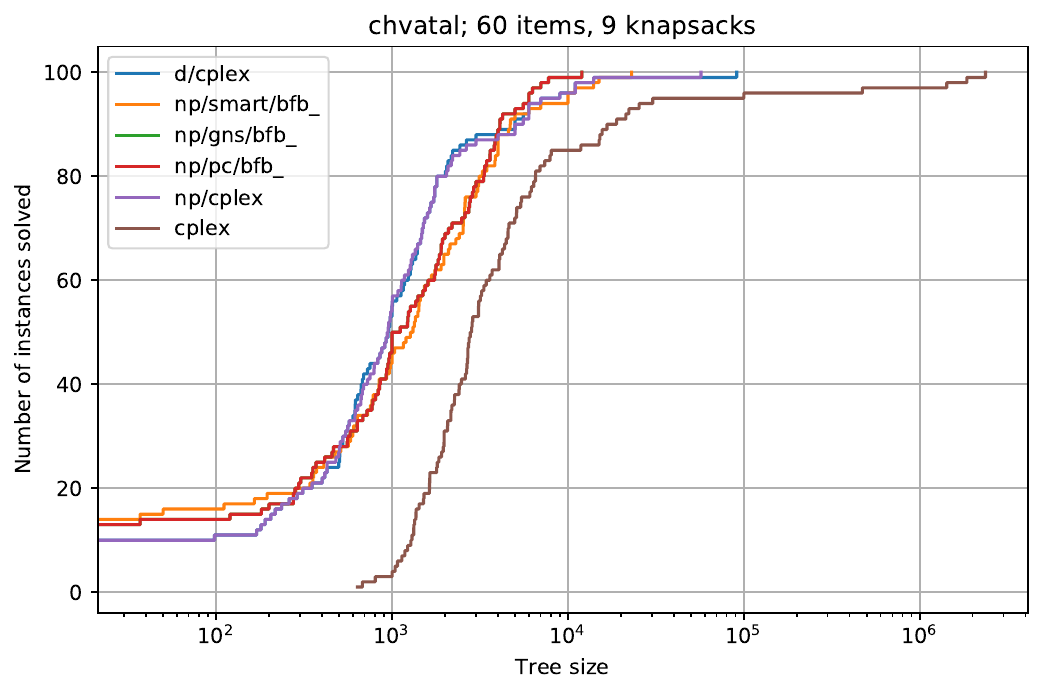}
  \label{fig:blah}
\end{subfigure}
\begin{subfigure}
  \centering
  \includegraphics[width=.39\linewidth]{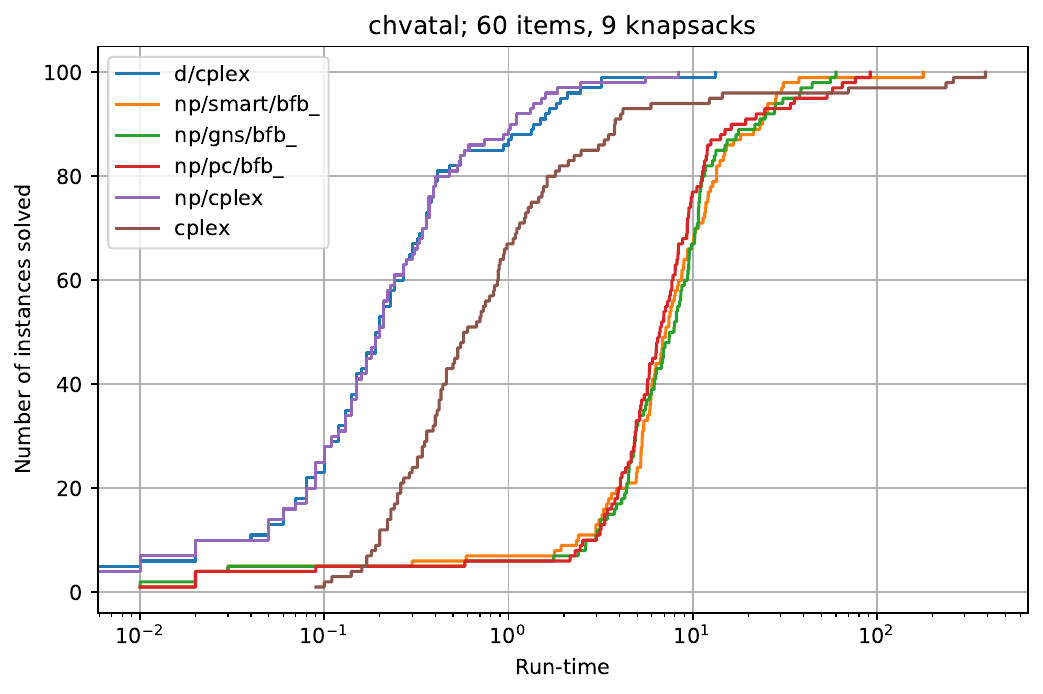}
\end{subfigure}

\caption{Chv\'{a}tal, CPLEX cover cuts off, all other parameters but presolve on}
\label{fig:chvatal_no_presolve_}
\end{figure}

\begin{figure}[t]
\centering
\begin{subfigure}
  \centering
  \includegraphics[width=.39\linewidth]{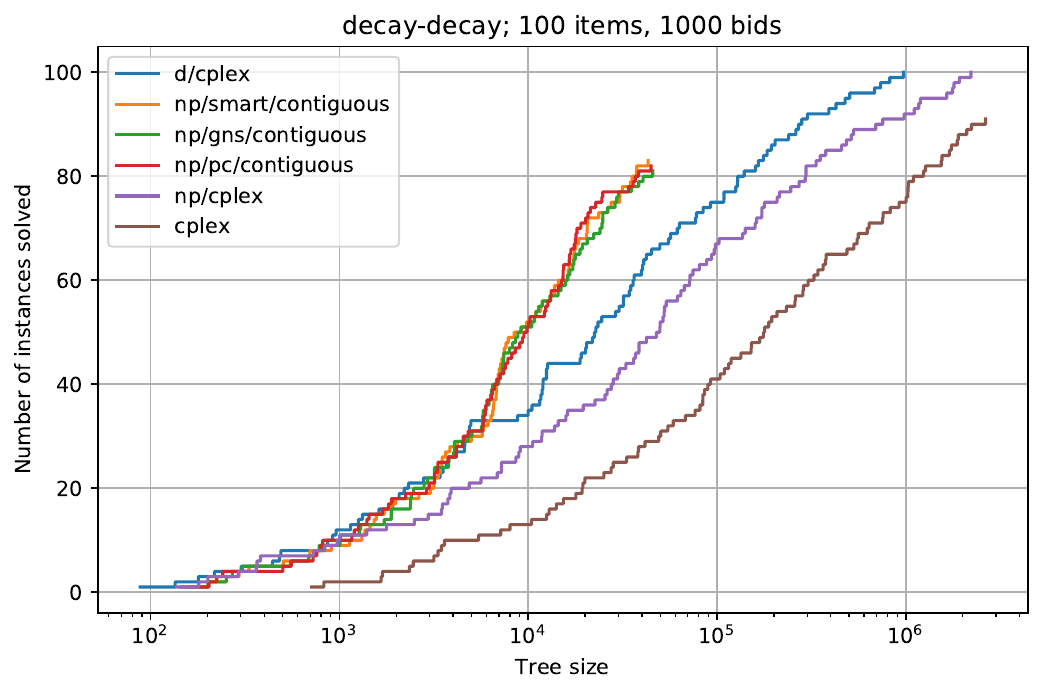}
  \label{fig:blah}
\end{subfigure}
\begin{subfigure}
  \centering
  \includegraphics[width=.39\linewidth]{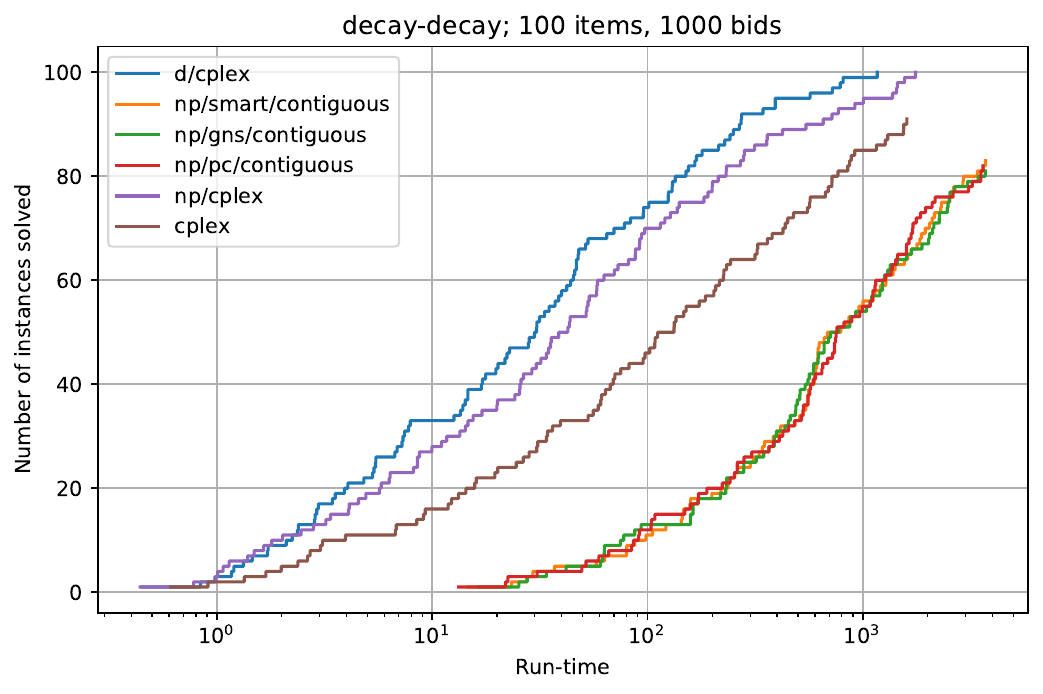}
\end{subfigure}
\begin{subfigure}
  \centering
  \includegraphics[width=.39\linewidth]{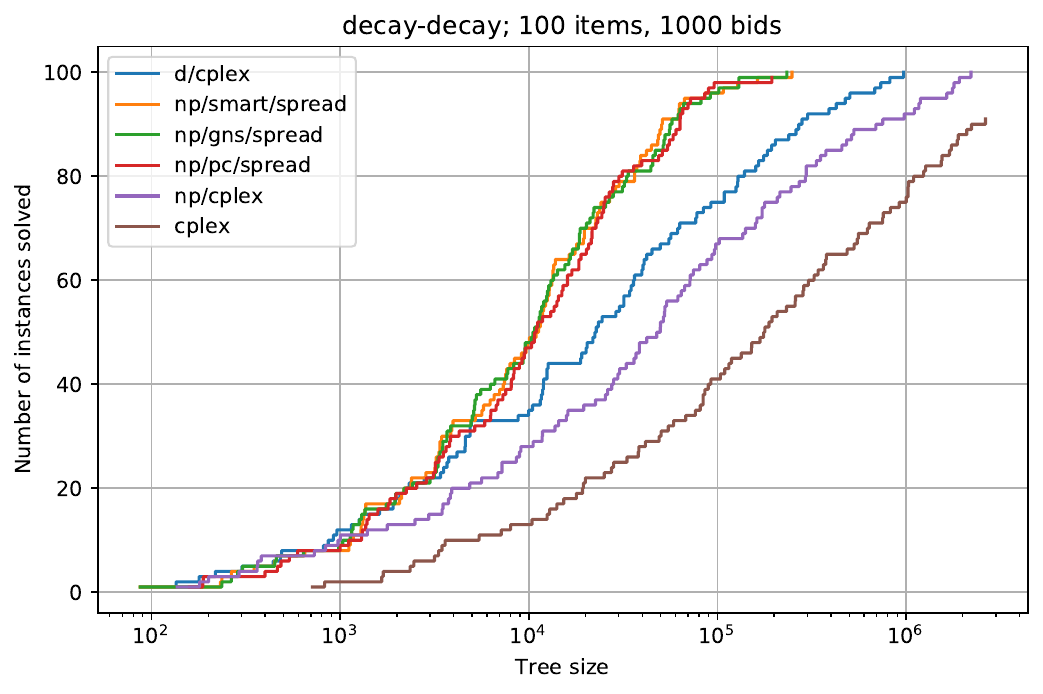}
  \label{fig:blah}
\end{subfigure}
\begin{subfigure}
  \centering
  \includegraphics[width=.39\linewidth]{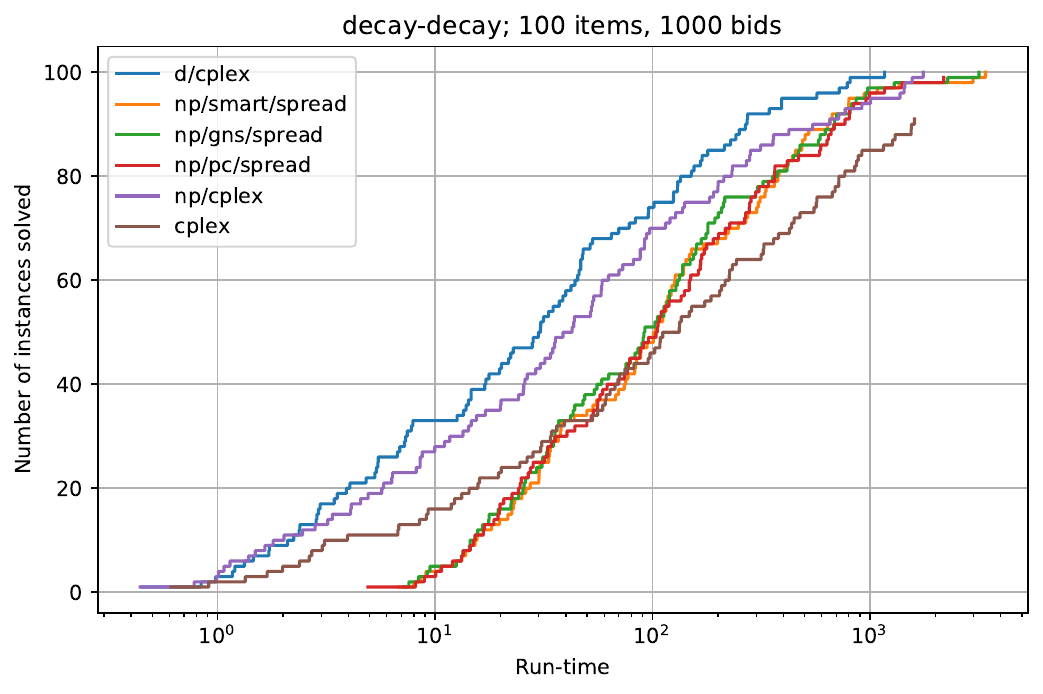}
\end{subfigure}
\begin{subfigure}
  \centering
  \includegraphics[width=.39\linewidth]{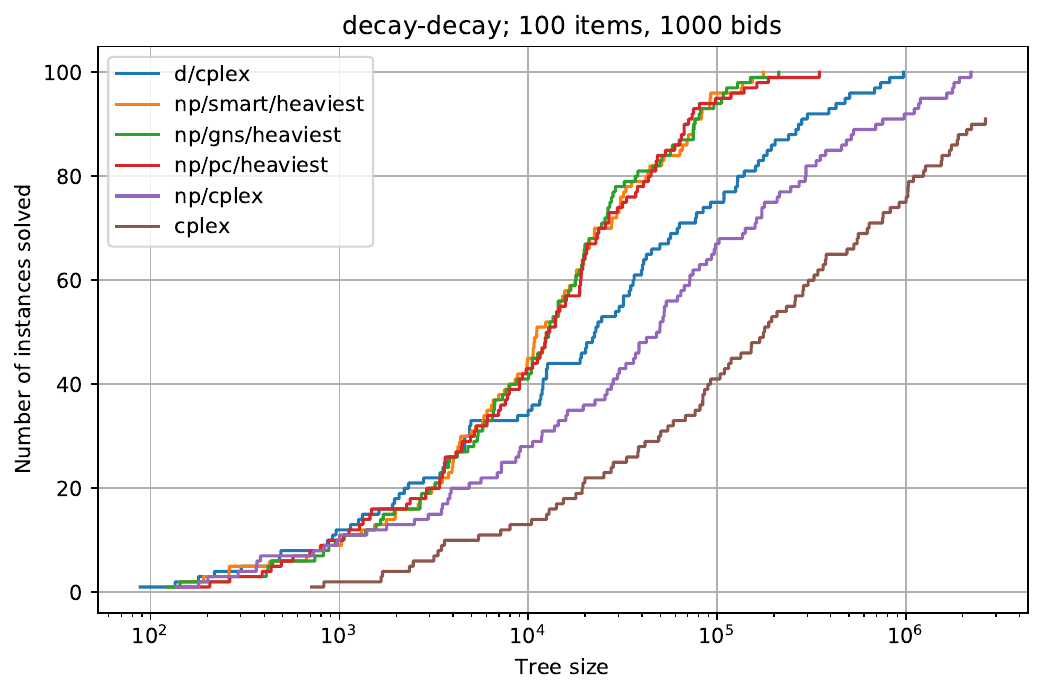}
  \label{fig:blah}
\end{subfigure}
\begin{subfigure}
  \centering
  \includegraphics[width=.39\linewidth]{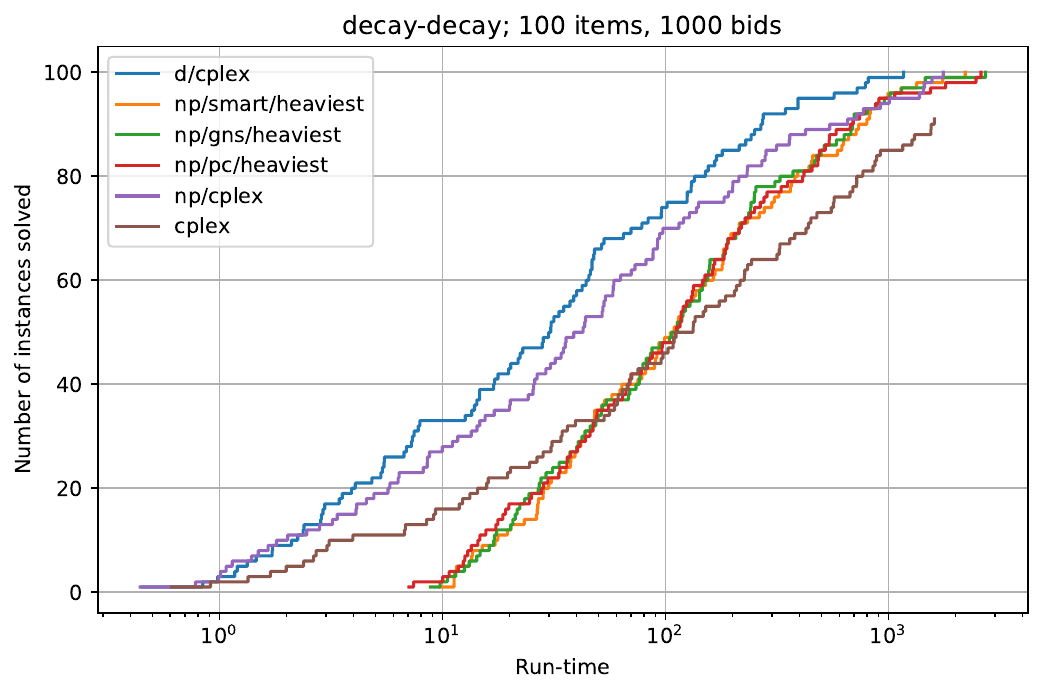}
\end{subfigure}
\begin{subfigure}
  \centering
  \includegraphics[width=.39\linewidth]{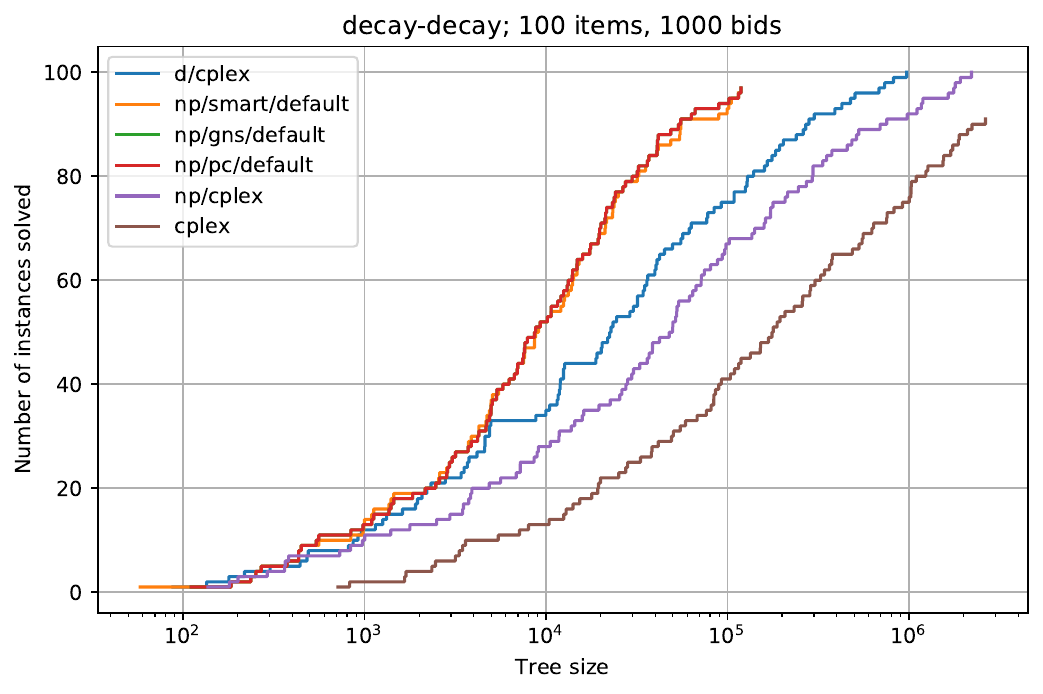}
  \label{fig:blah}
\end{subfigure}
\begin{subfigure}
  \centering
  \includegraphics[width=.39\linewidth]{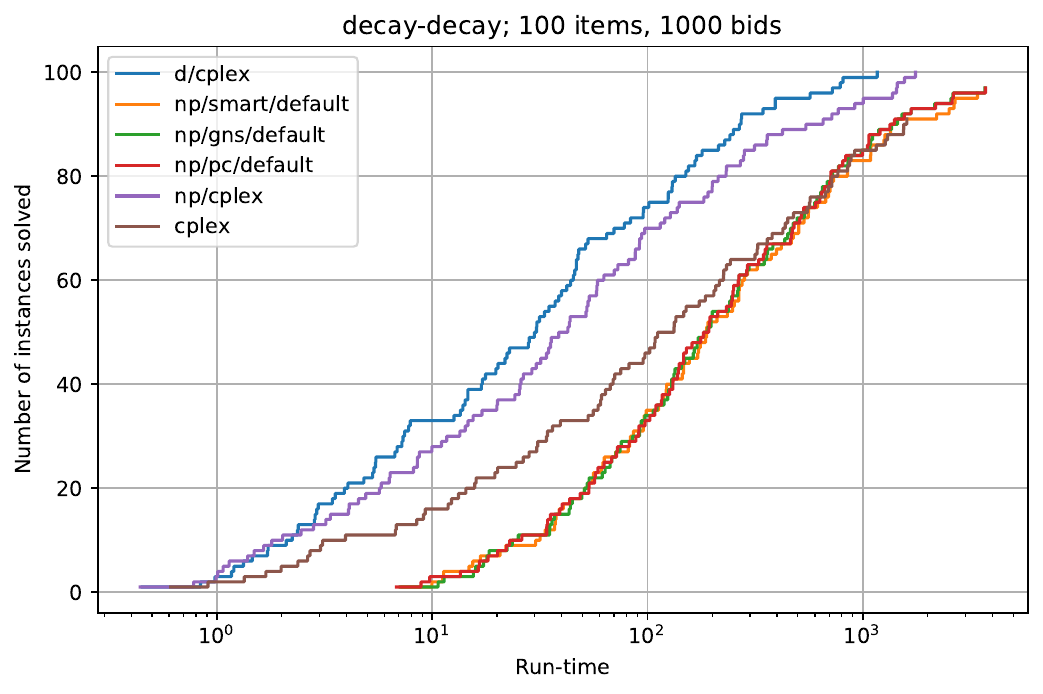}
\end{subfigure}
\begin{subfigure}
  \centering
  \includegraphics[width=.39\linewidth]{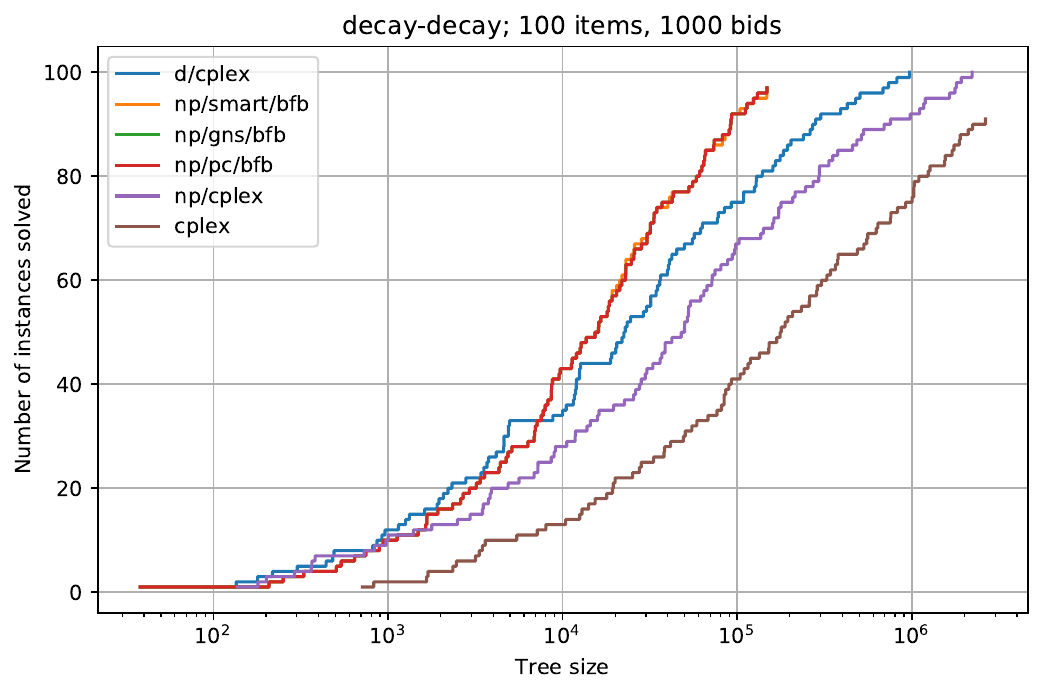}
  \label{fig:blah}
\end{subfigure}
\begin{subfigure}
  \centering
  \includegraphics[width=.39\linewidth]{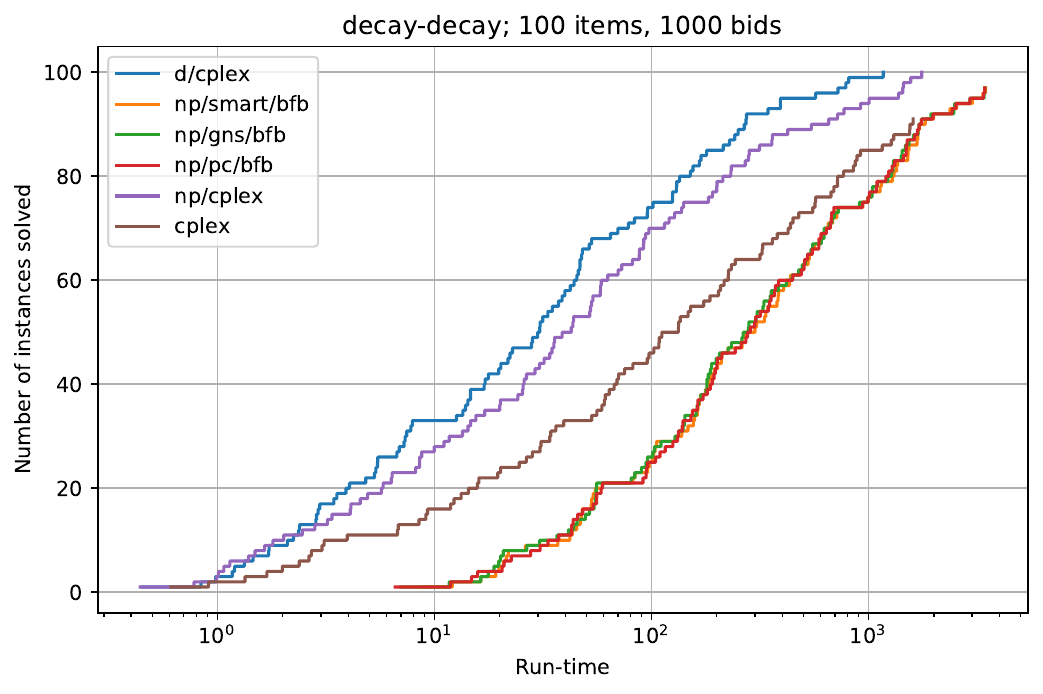}
\end{subfigure}

\caption{Decay-decay, CPLEX cover cuts on, all other parameters but presolve on}
\label{fig:muca_no_presolve}
\end{figure}

\begin{figure}[t]
\centering
\begin{subfigure}
  \centering
  \includegraphics[width=.39\linewidth]{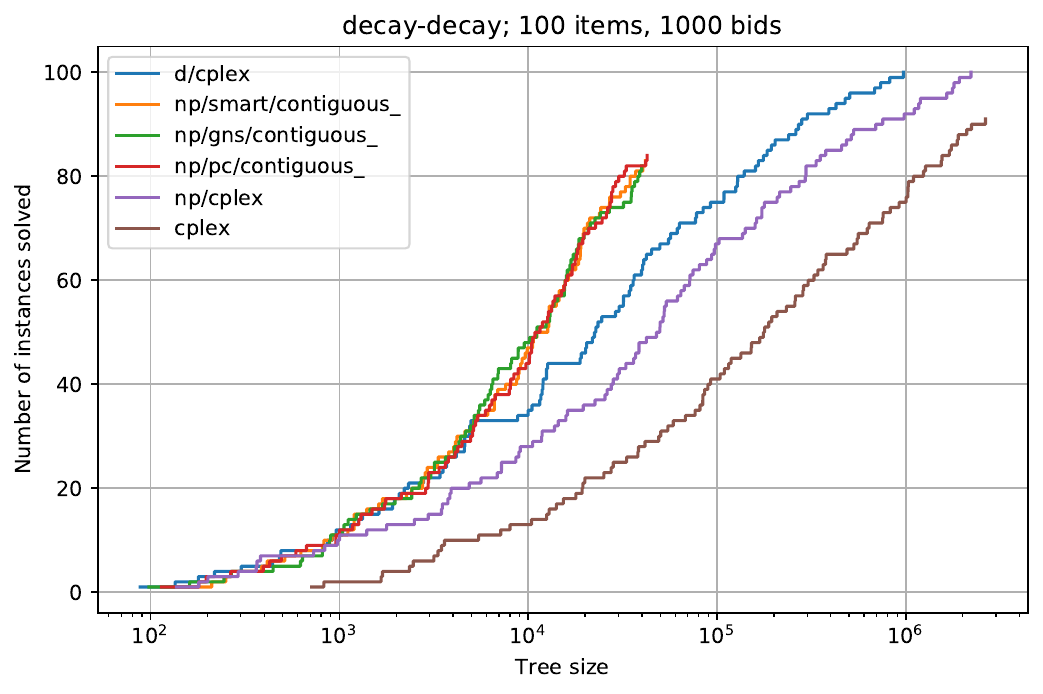}
  \label{fig:blah}
\end{subfigure}
\begin{subfigure}
  \centering
  \includegraphics[width=.39\linewidth]{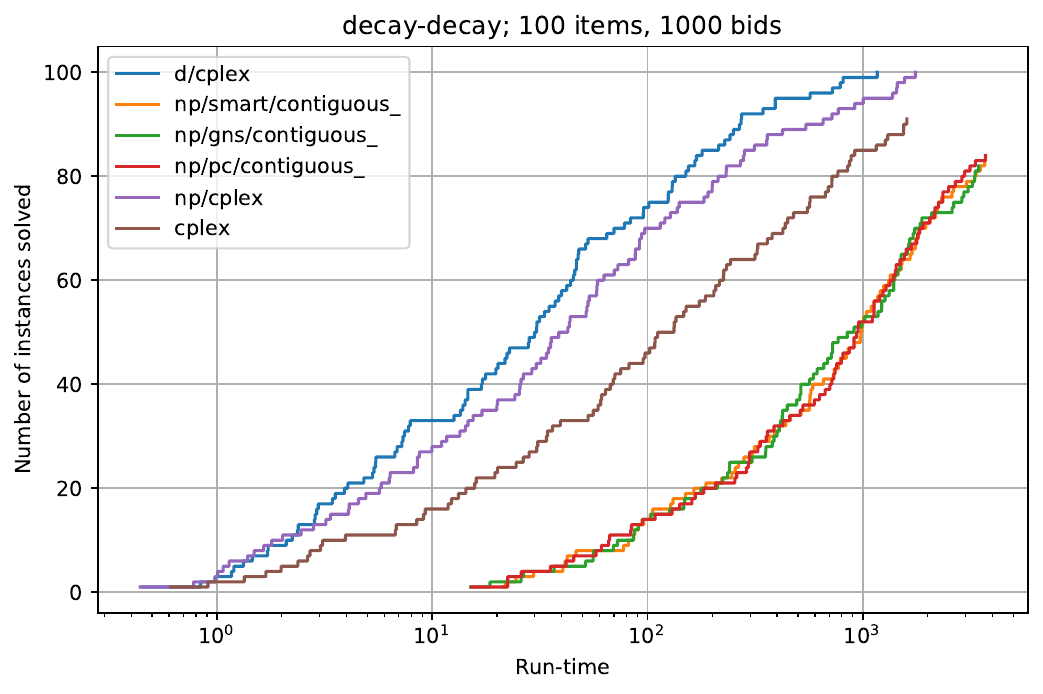}
\end{subfigure}
\begin{subfigure}
  \centering
  \includegraphics[width=.39\linewidth]{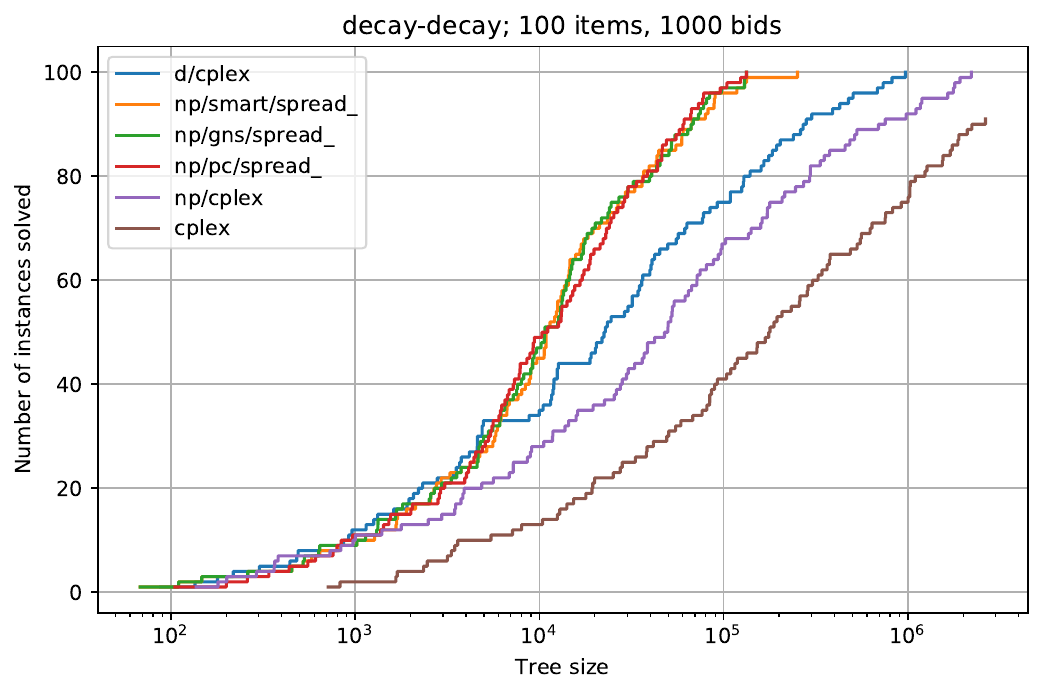}
  \label{fig:blah}
\end{subfigure}
\begin{subfigure}
  \centering
  \includegraphics[width=.39\linewidth]{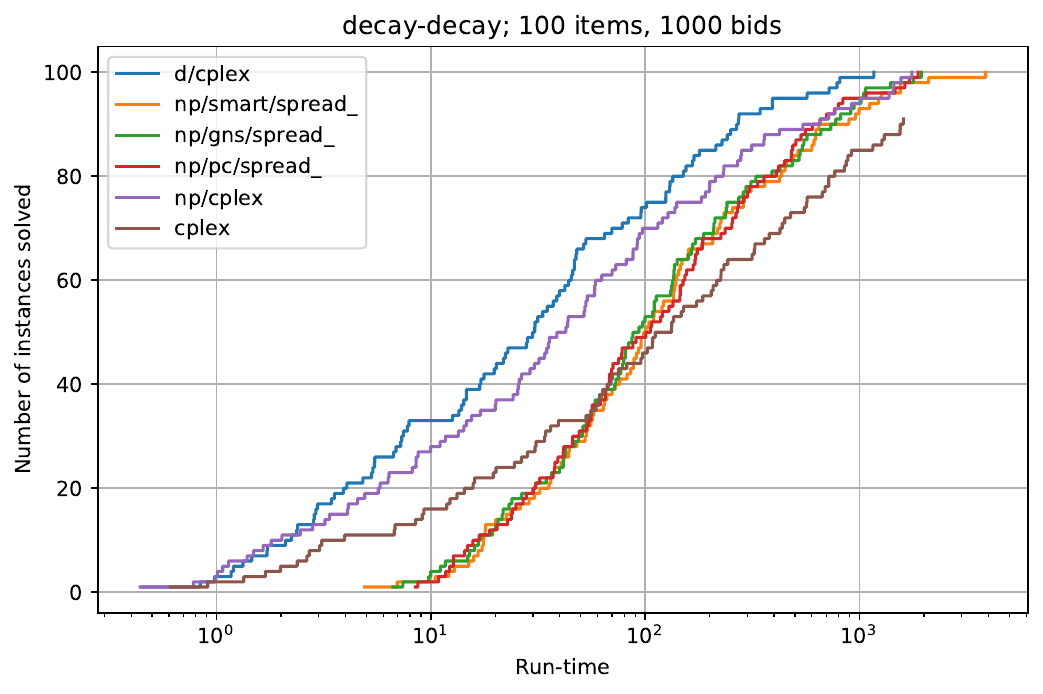}
\end{subfigure}
\begin{subfigure}
  \centering
  \includegraphics[width=.39\linewidth]{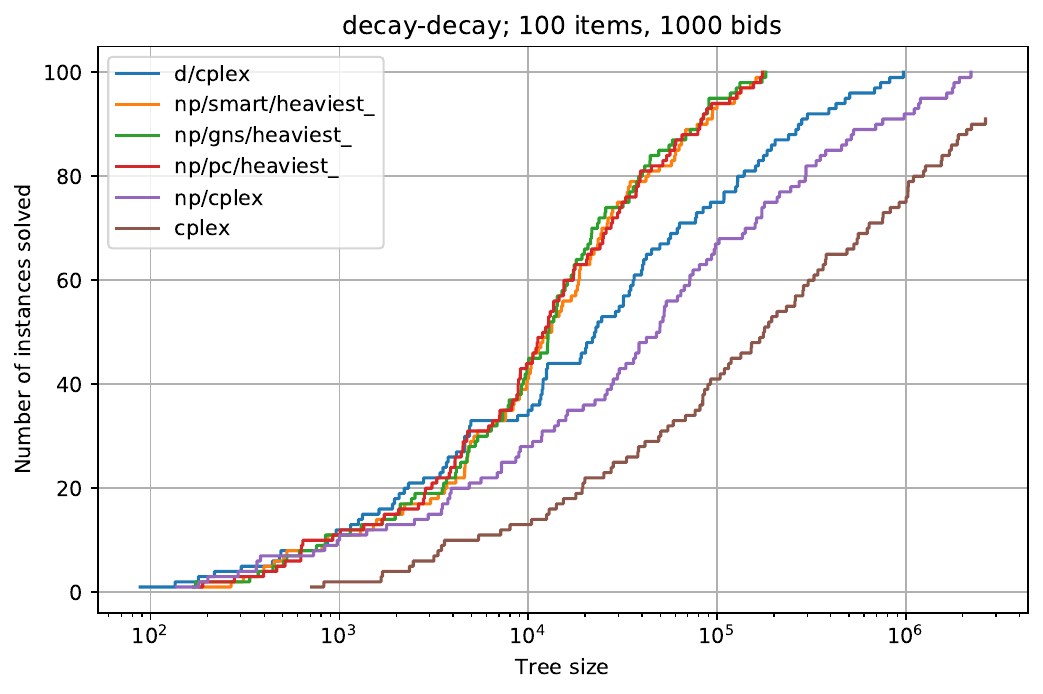}
  \label{fig:blah}
\end{subfigure}
\begin{subfigure}
  \centering
  \includegraphics[width=.39\linewidth]{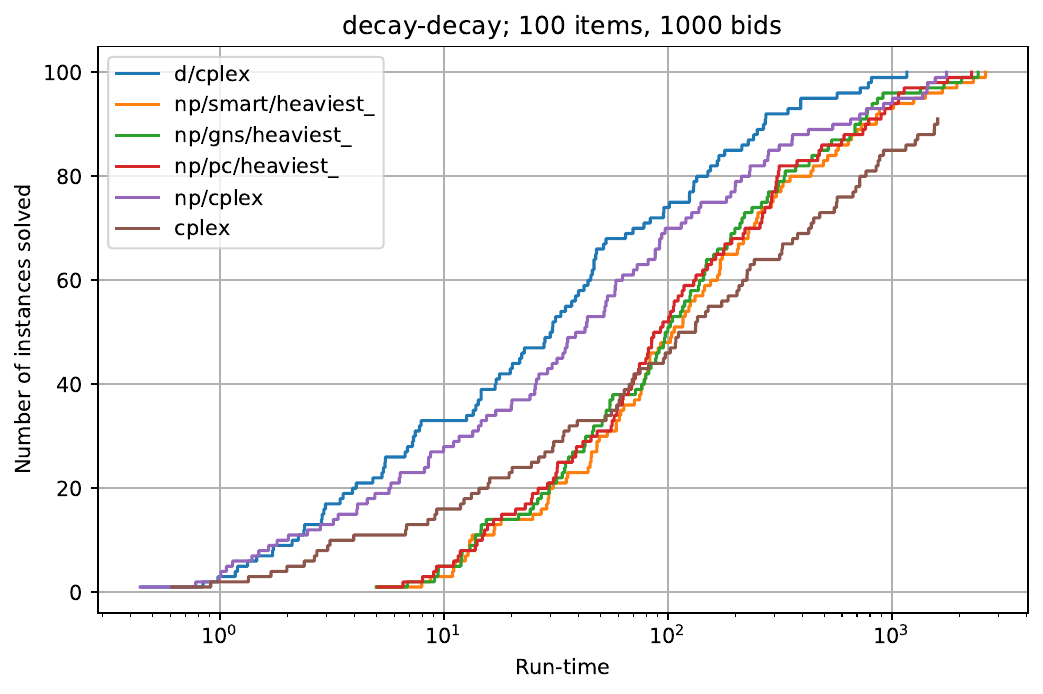}
\end{subfigure}
\begin{subfigure}
  \centering
  \includegraphics[width=.39\linewidth]{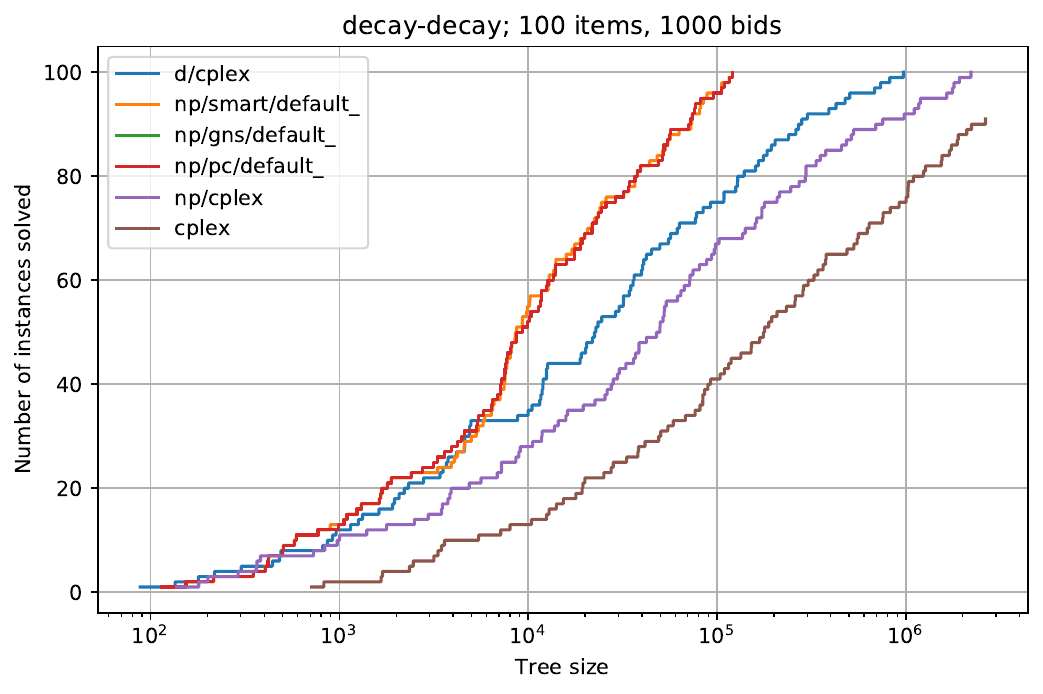}
  \label{fig:blah}
\end{subfigure}
\begin{subfigure}
  \centering
  \includegraphics[width=.39\linewidth]{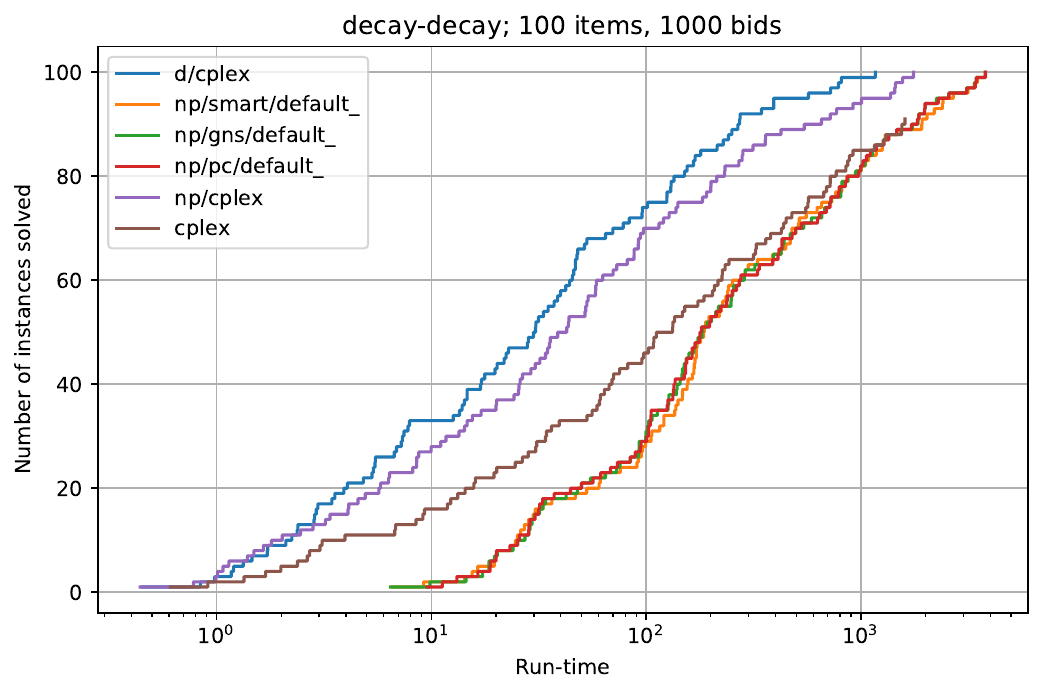}
\end{subfigure}
\begin{subfigure}
  \centering
  \includegraphics[width=.39\linewidth]{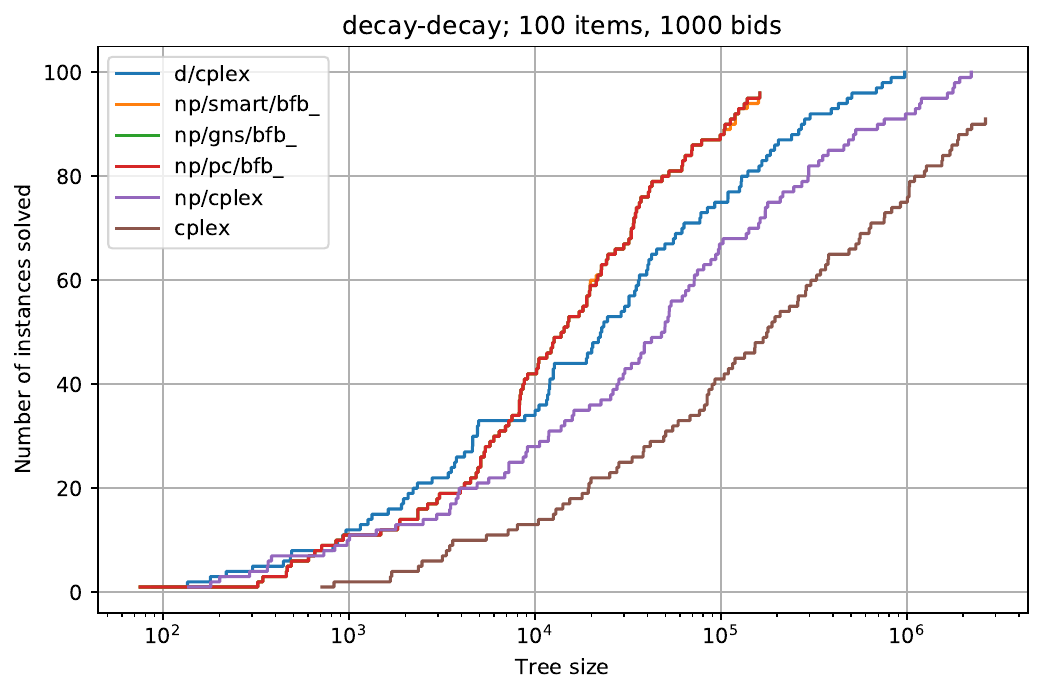}
  \label{fig:blah}
\end{subfigure}
\begin{subfigure}
  \centering
  \includegraphics[width=.39\linewidth]{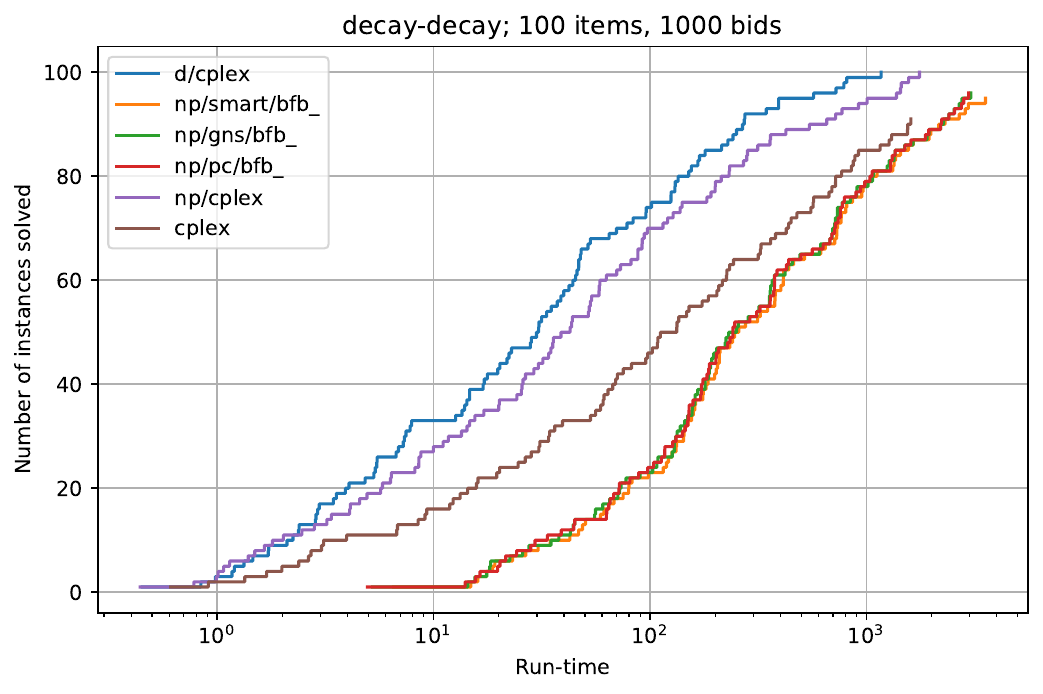}
\end{subfigure}

\caption{Decay-decay, CPLEX cover cuts off, all other parameters but presolve on}
\label{fig:muca_no_presolve_}
\end{figure}

\begin{figure}[t]
\centering
\begin{subfigure}
  \centering
  \includegraphics[width=.39\linewidth]{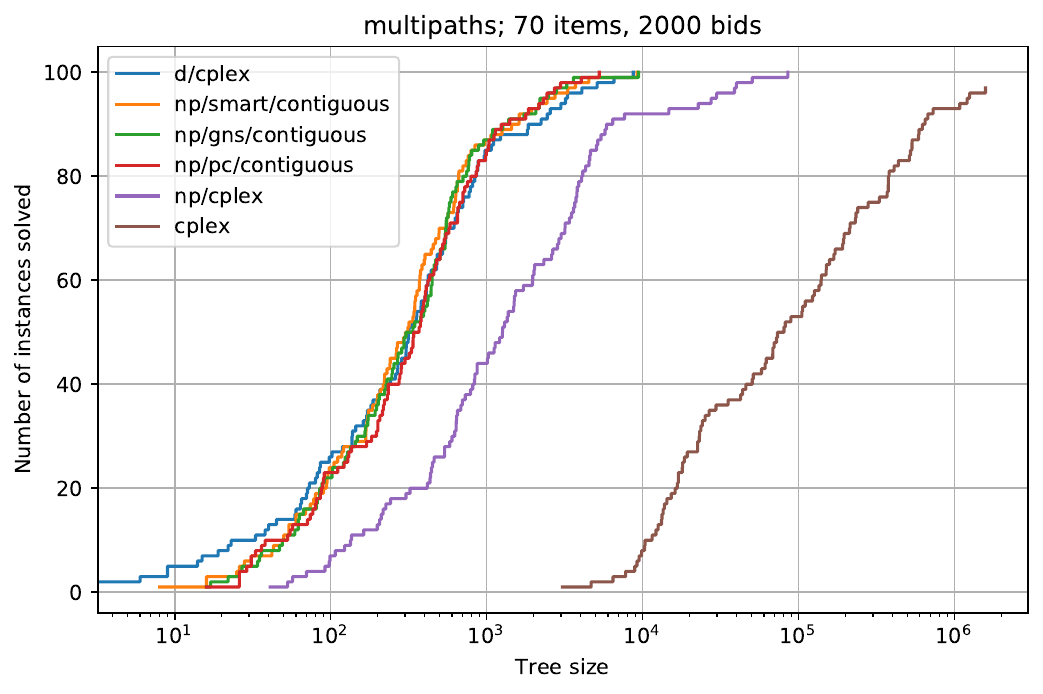}
  \label{fig:blah}
\end{subfigure}
\begin{subfigure}
  \centering
  \includegraphics[width=.39\linewidth]{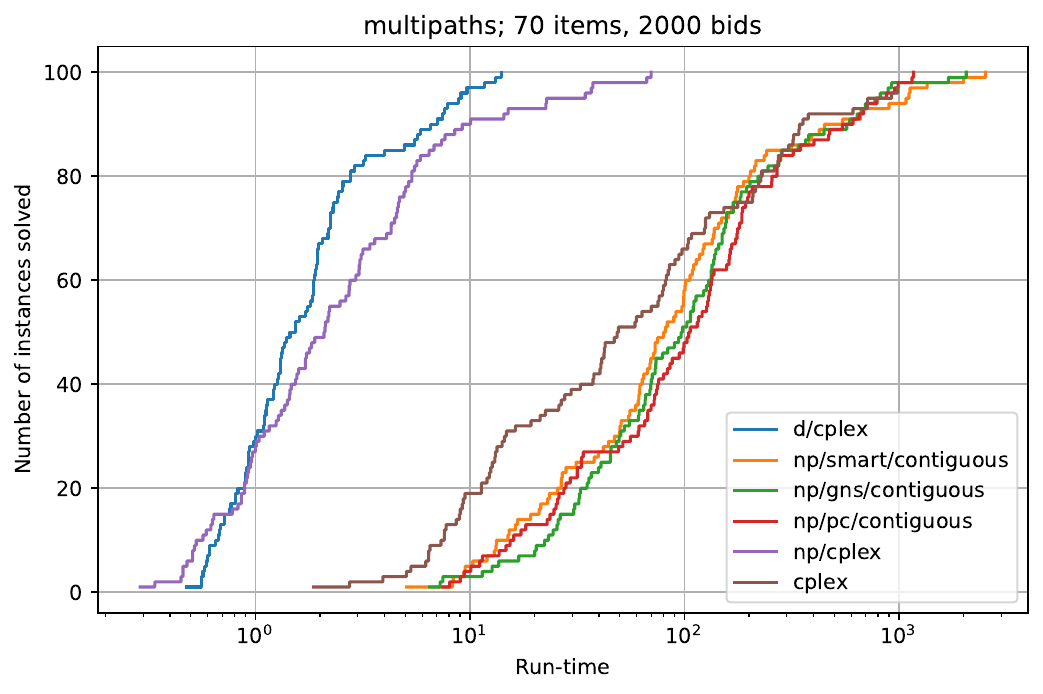}
\end{subfigure}
\begin{subfigure}
  \centering
  \includegraphics[width=.39\linewidth]{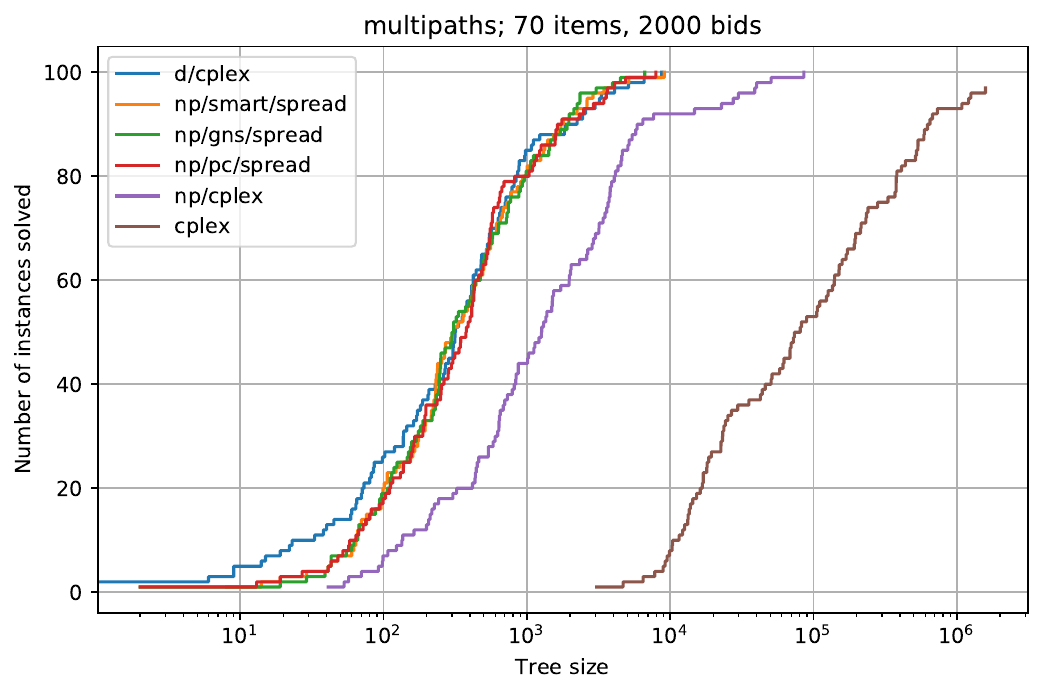}
  \label{fig:blah}
\end{subfigure}
\begin{subfigure}
  \centering
  \includegraphics[width=.39\linewidth]{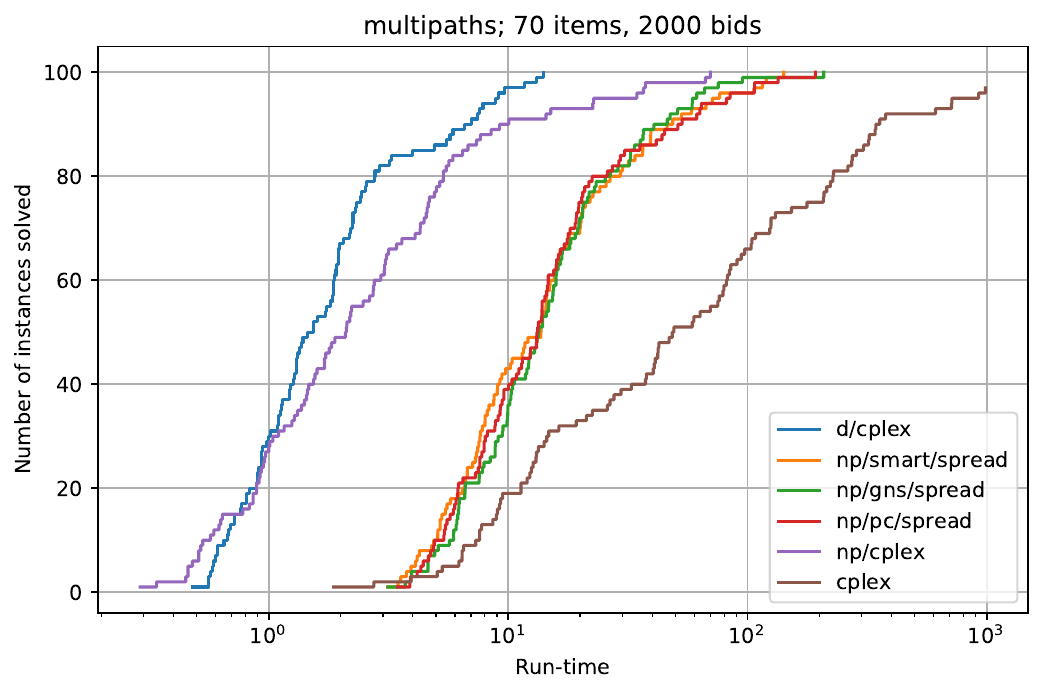}
\end{subfigure}
\begin{subfigure}
  \centering
  \includegraphics[width=.39\linewidth]{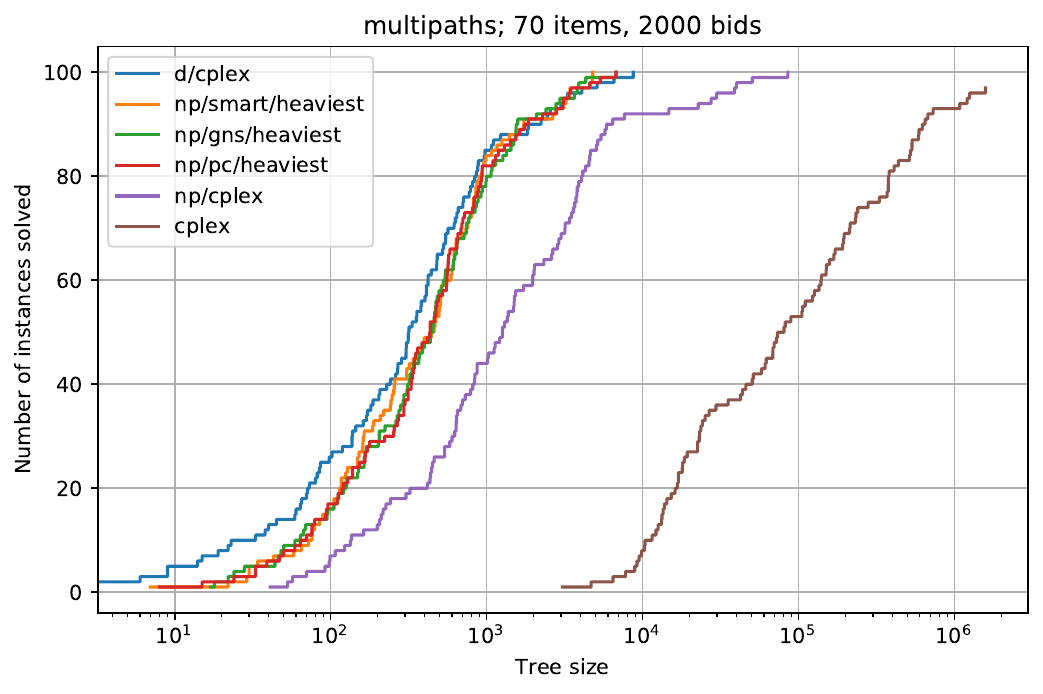}
  \label{fig:blah}
\end{subfigure}
\begin{subfigure}
  \centering
  \includegraphics[width=.39\linewidth]{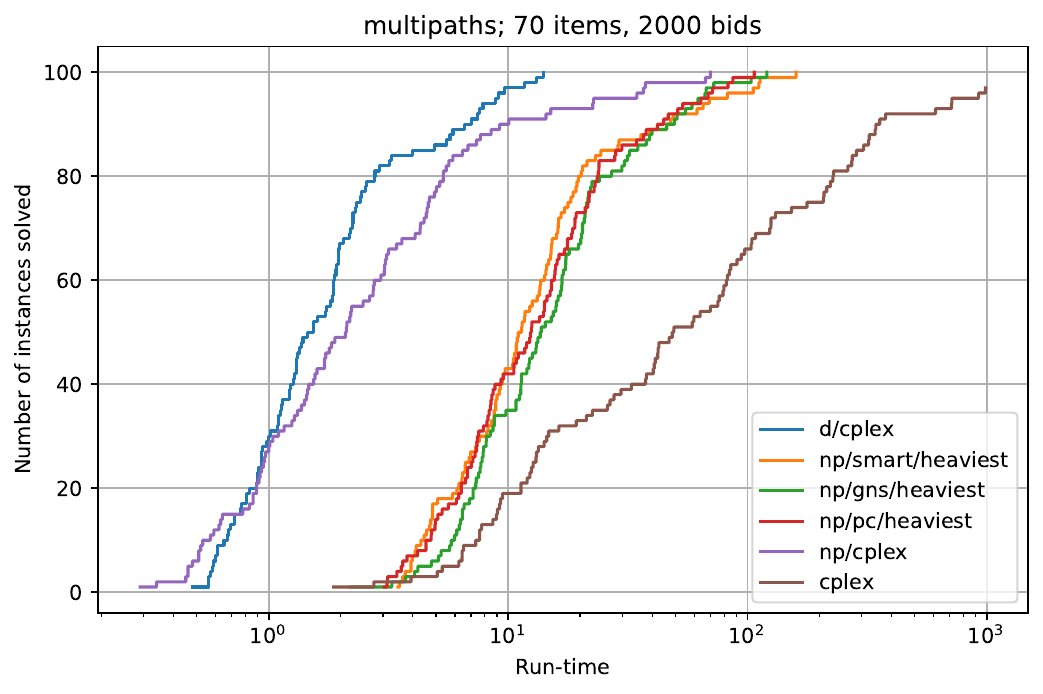}
\end{subfigure}
\begin{subfigure}
  \centering
  \includegraphics[width=.39\linewidth]{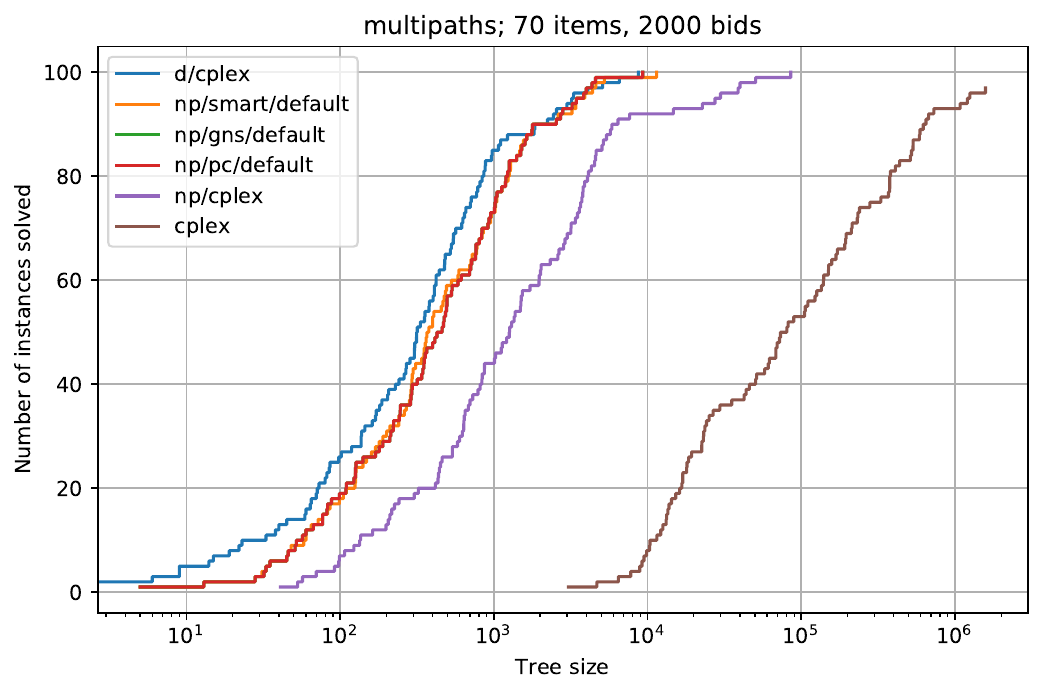}
  \label{fig:blah}
\end{subfigure}
\begin{subfigure}
  \centering
  \includegraphics[width=.39\linewidth]{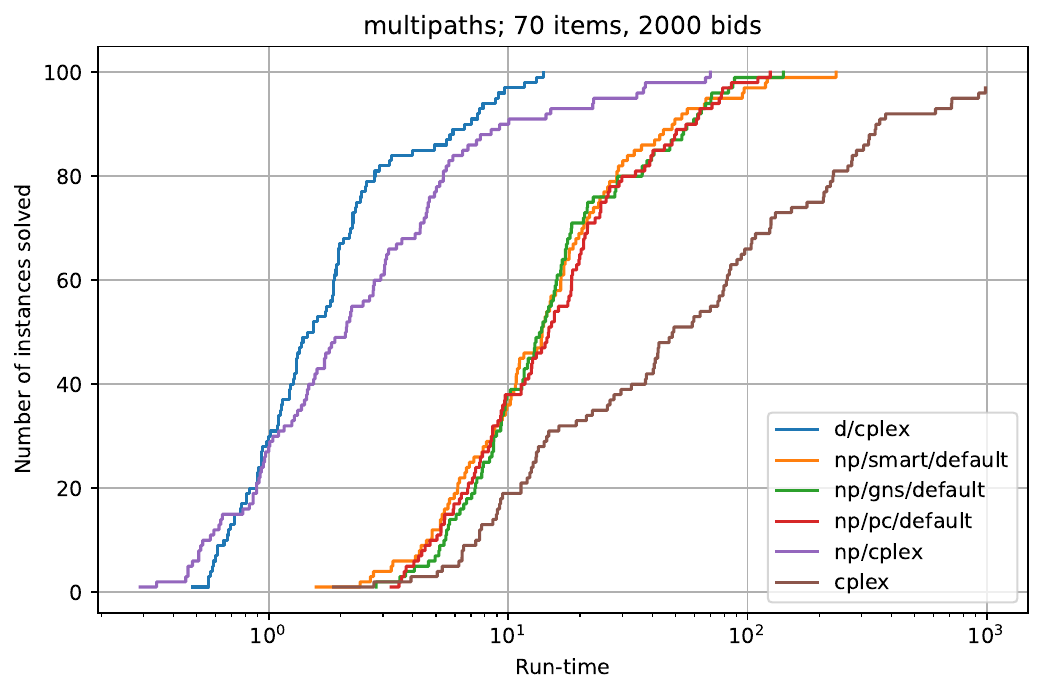}
\end{subfigure}
\begin{subfigure}
  \centering
  \includegraphics[width=.39\linewidth]{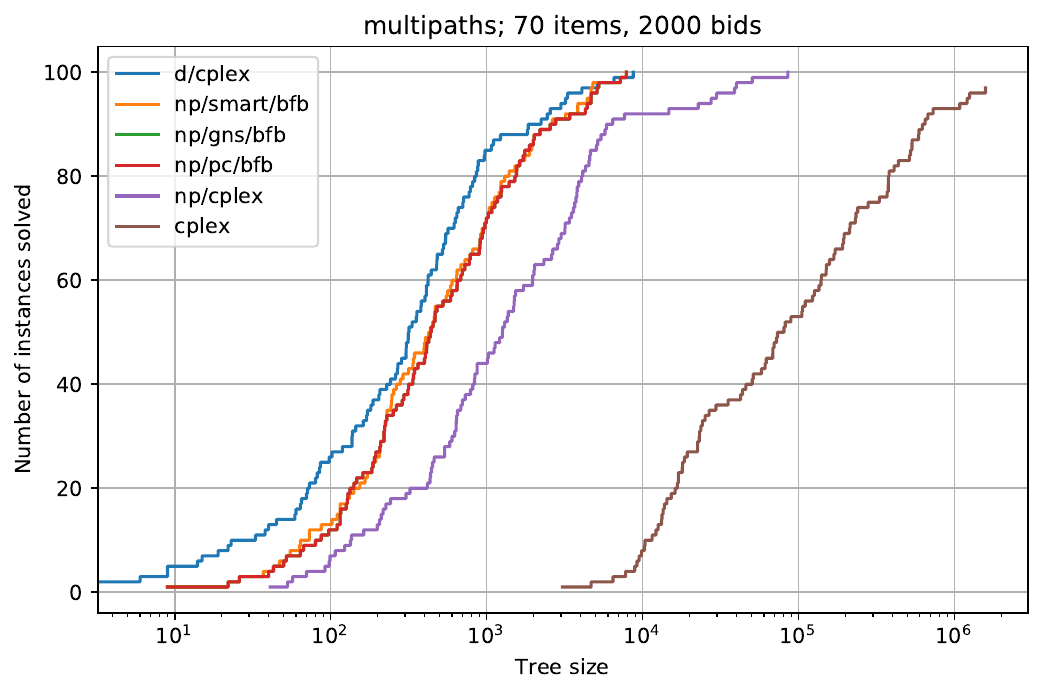}
  \label{fig:blah}
\end{subfigure}
\begin{subfigure}
  \centering
  \includegraphics[width=.39\linewidth]{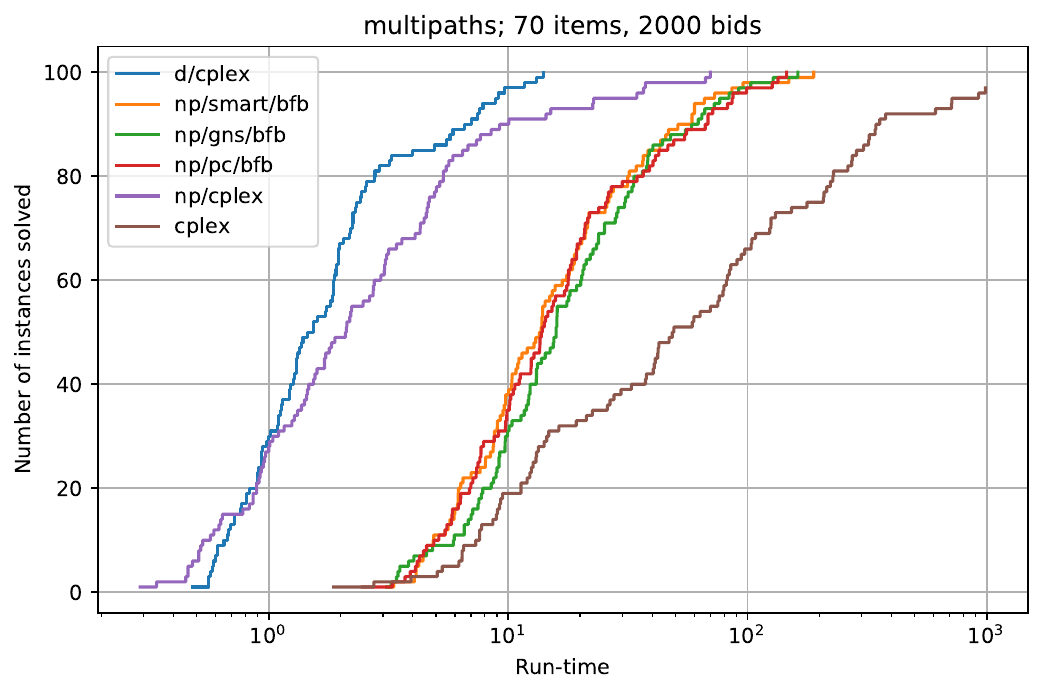}
\end{subfigure}

\caption{Multipaths, CPLEX cover cuts on, all other parameters but presolve on}
\label{fig:multipathsno_presolve}
\end{figure}

\begin{figure}[t]
\centering
\begin{subfigure}
  \centering
  \includegraphics[width=.39\linewidth]{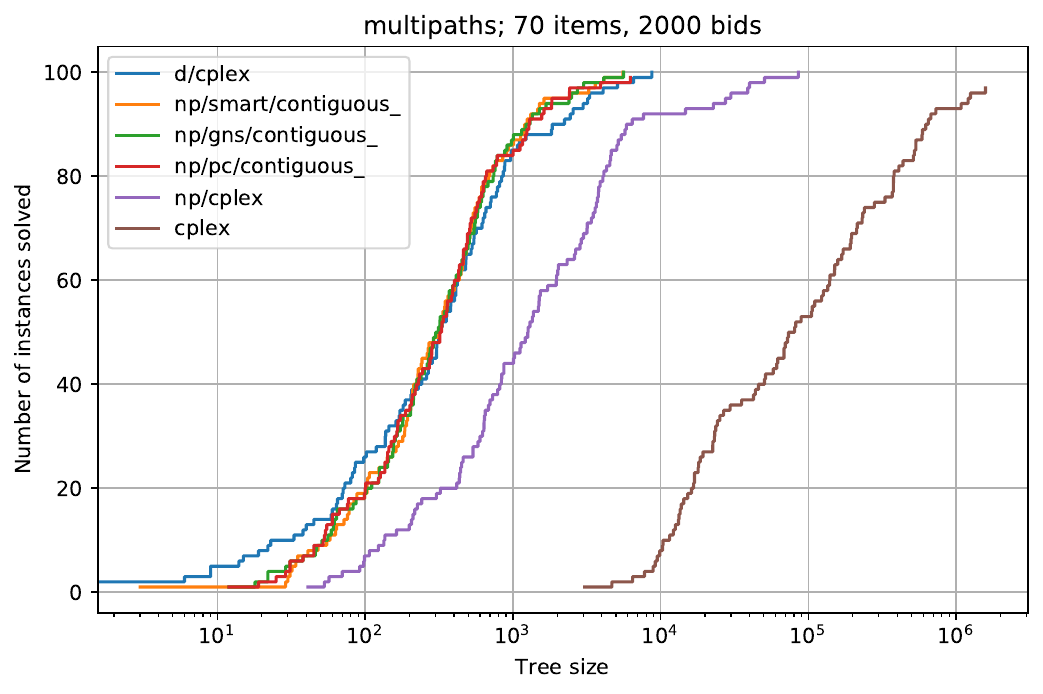}
  \label{fig:blah}
\end{subfigure}
\begin{subfigure}
  \centering
  \includegraphics[width=.39\linewidth]{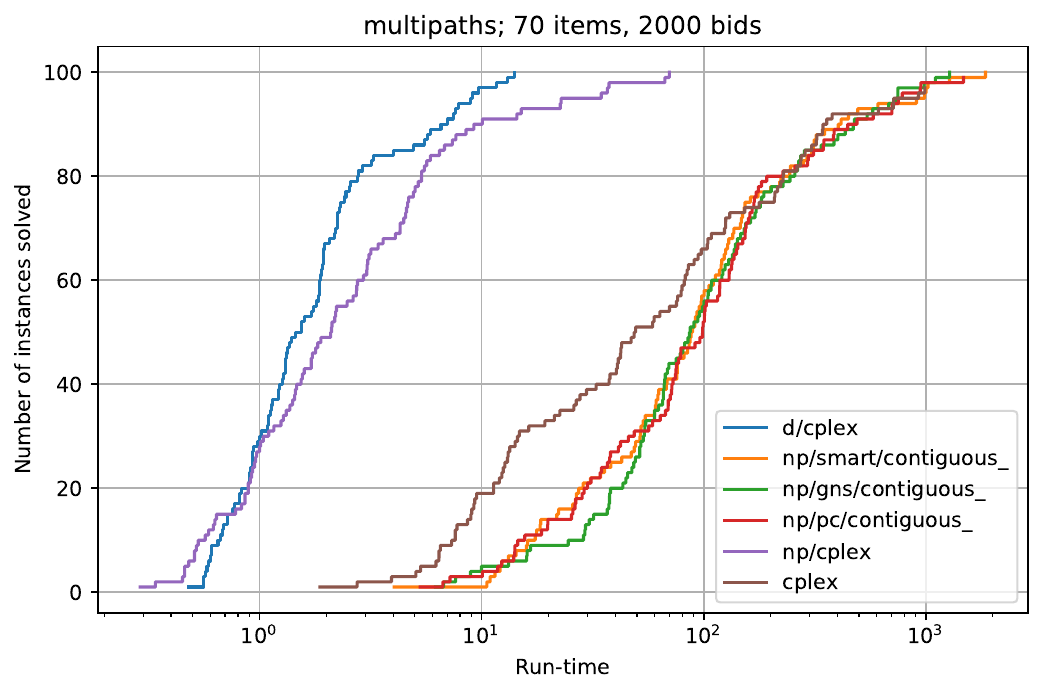}
\end{subfigure}
\begin{subfigure}
  \centering
  \includegraphics[width=.39\linewidth]{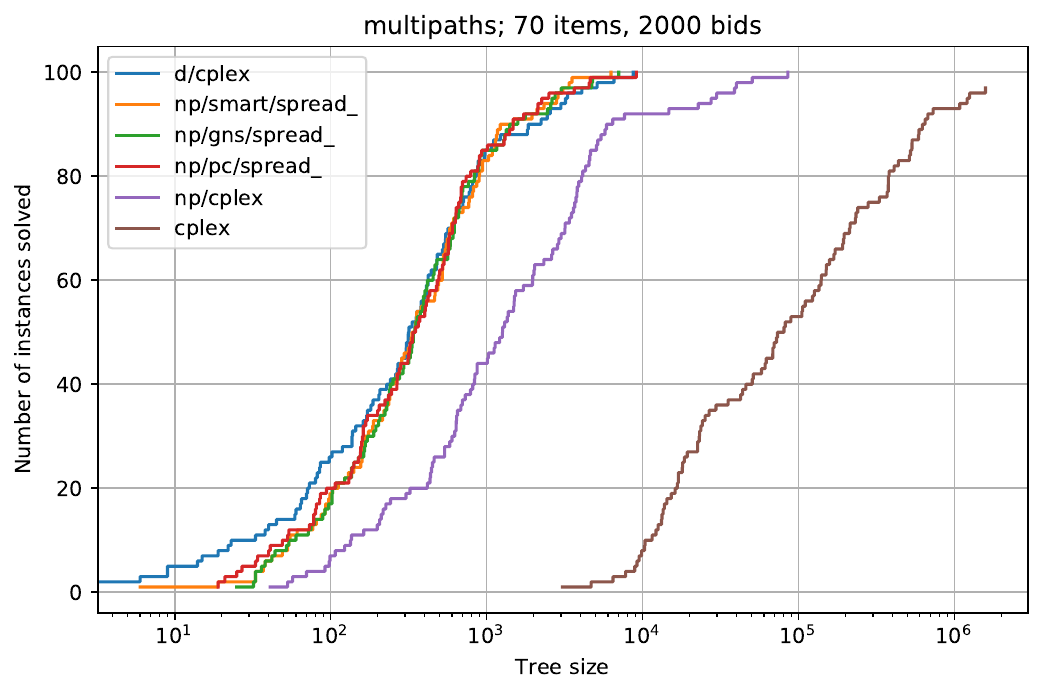}
  \label{fig:blah}
\end{subfigure}
\begin{subfigure}
  \centering
  \includegraphics[width=.39\linewidth]{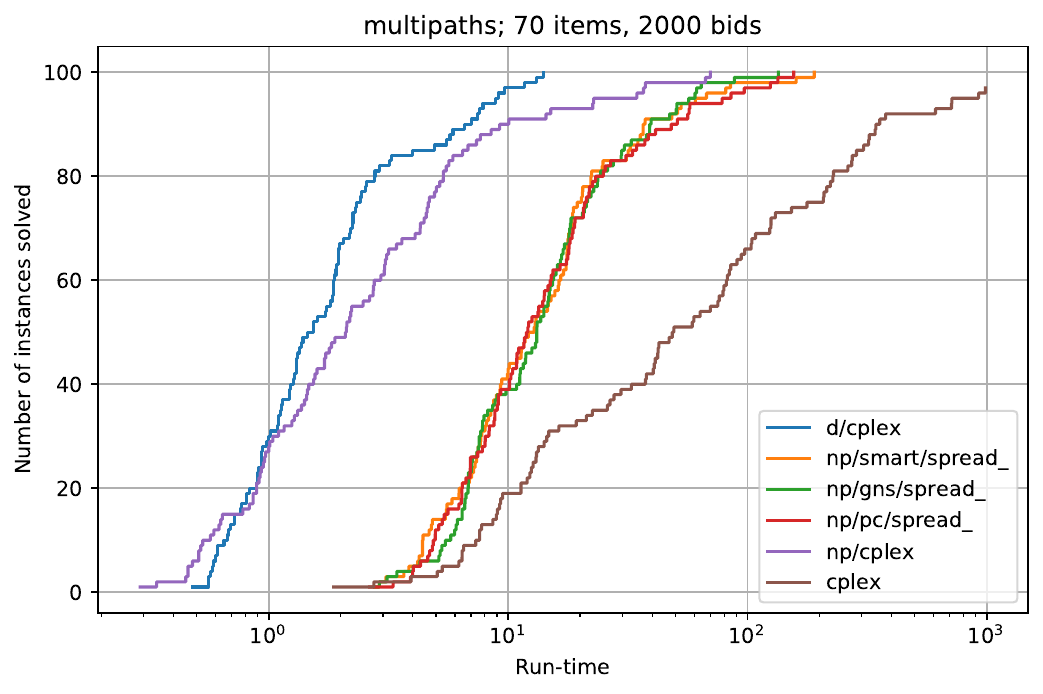}
\end{subfigure}
\begin{subfigure}
  \centering
  \includegraphics[width=.39\linewidth]{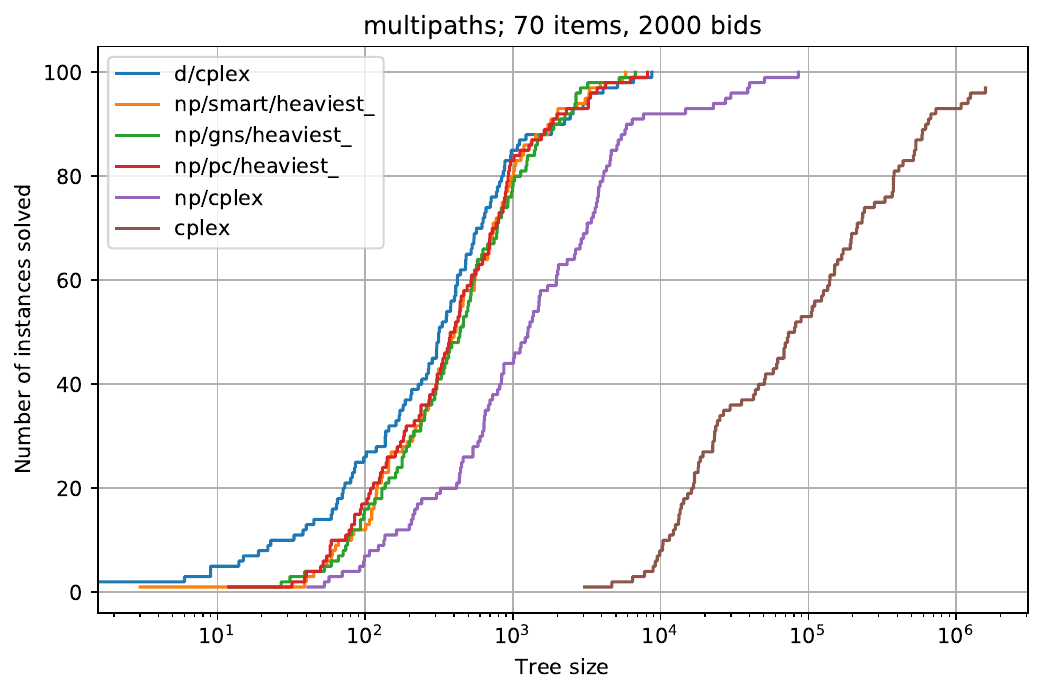}
  \label{fig:blah}
\end{subfigure}
\begin{subfigure}
  \centering
  \includegraphics[width=.39\linewidth]{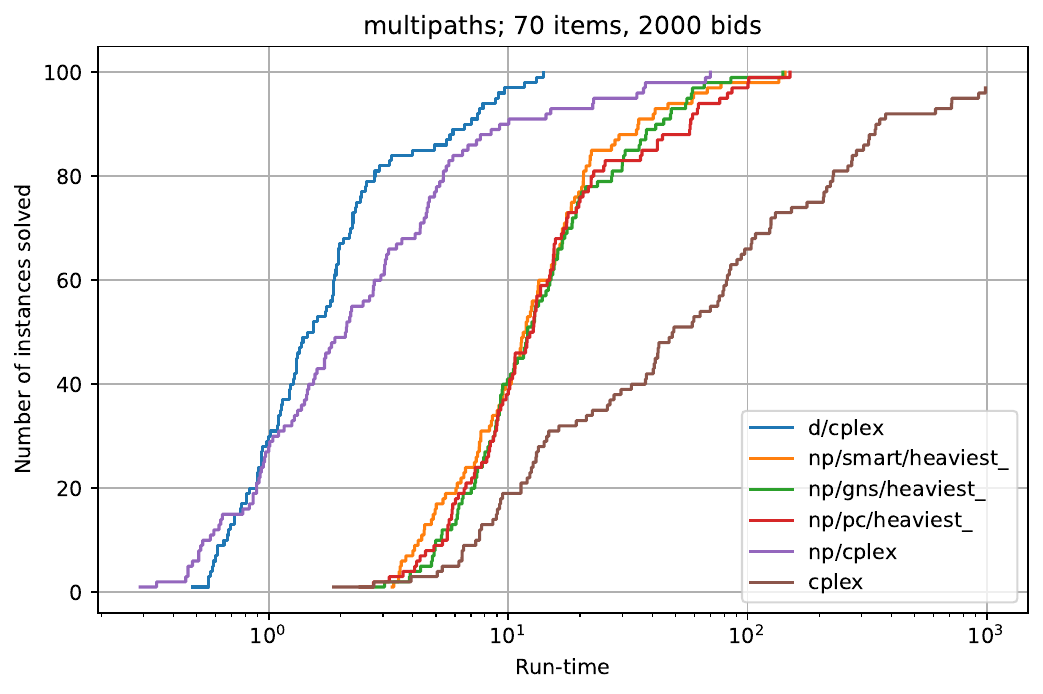}
\end{subfigure}
\begin{subfigure}	
  \centering
  \includegraphics[width=.39\linewidth]{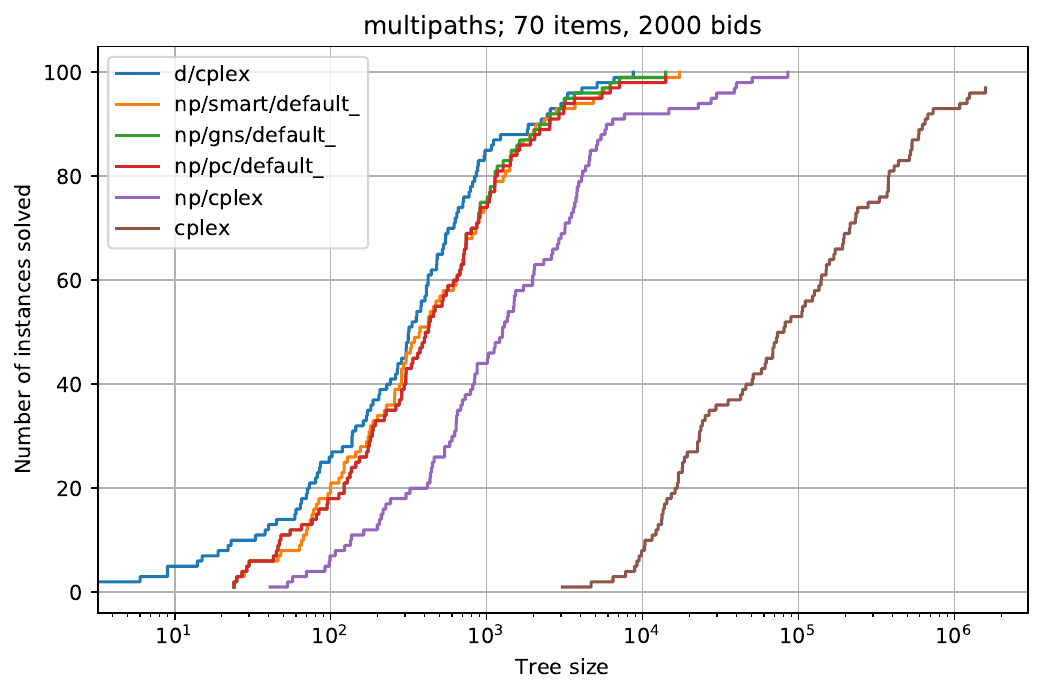}
  \label{fig:blah}
\end{subfigure}
\begin{subfigure}
  \centering
  \includegraphics[width=.39\linewidth]{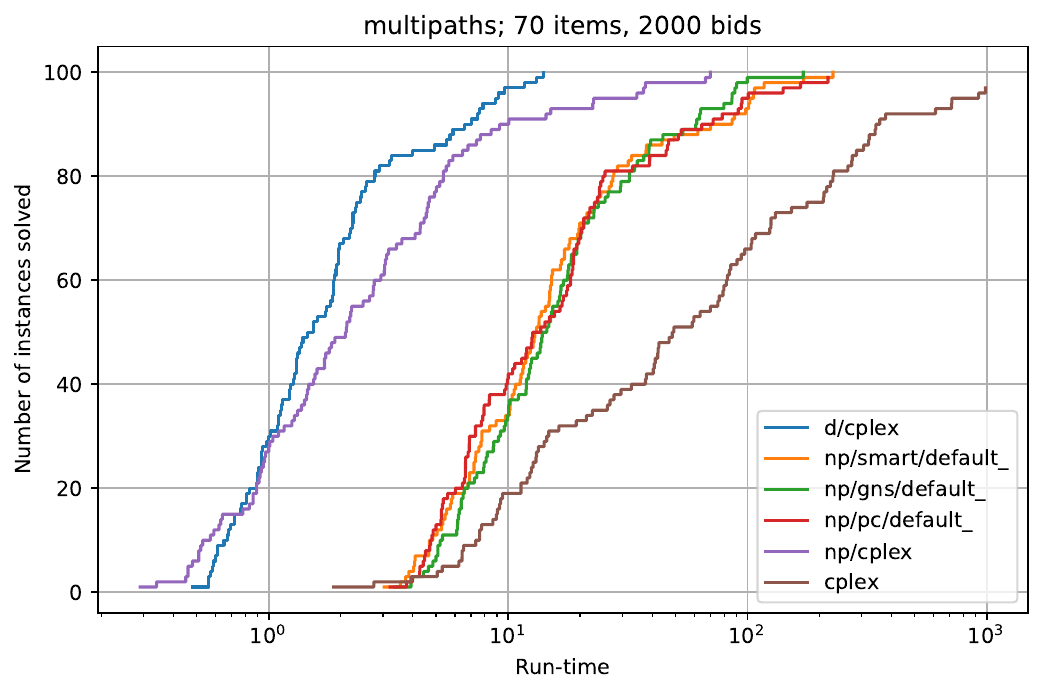}
\end{subfigure}
\begin{subfigure}
  \centering
  \includegraphics[width=.39\linewidth]{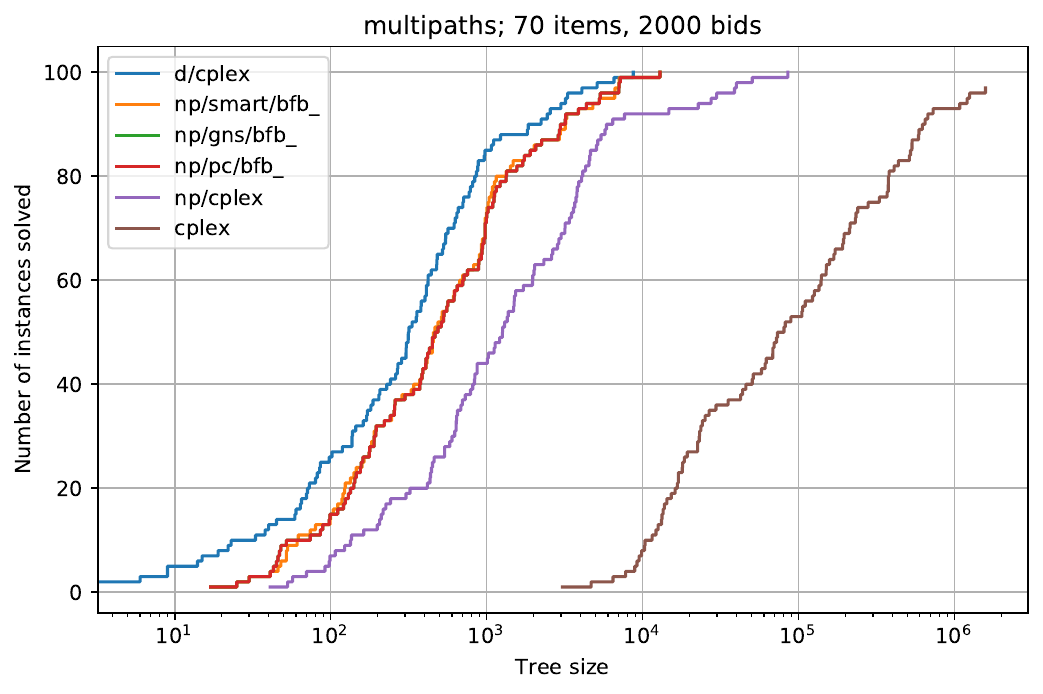}
  \label{fig:blah}
\end{subfigure}
\begin{subfigure}
  \centering
  \includegraphics[width=.39\linewidth]{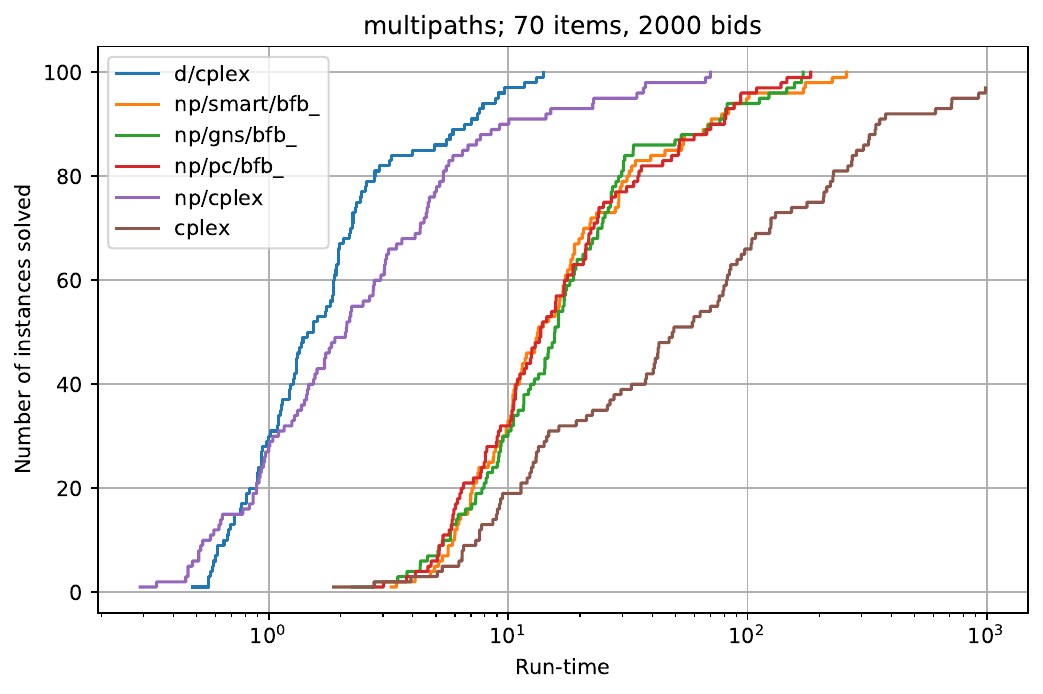}
\end{subfigure}

\caption{Multipaths, CPLEX cover cuts off, all other parameters but presolve on}
\label{fig:multipathsno_presolve_}
\end{figure}

\end{document}